\documentclass[letterpaper,12pt,oneside,onecolumn,reqno]{amsart}

\usepackage{amsmath, amssymb, amsaddr}
\usepackage{bm}
\usepackage{comment}
\usepackage{hyperref}
\usepackage{enumerate}
\usepackage[left=1in,top=1in,right=1in,bottom=1in]{geometry}
\usepackage{graphicx}
\usepackage[usenames,dvipsnames]{xcolor}
\usepackage[all]{xy}
\usepackage[T1]{fontenc}

\newcommand{\A}{\mathcal{A}}
\newcommand{\B}{{\mathcal{B}}}
\newcommand{\C}{\mathcal{C}}
\renewcommand{\d}{\mathrm{d}}
\newcommand{\D}{\mathcal{D}}
\newcommand{\e}{\mathrm{e}}
\newcommand{\E}{\mathbb{E}}

\renewcommand{\H}{H}

\newcommand{\I}{\mathcal{I}}

\newcommand{\M}{\mathcal{M}}

\newcommand{\PP}{{\boldsymbol{\mathcal{P}}}}

\newcommand{\R}{\mathbb{R}}

\newcommand{\W}{\mathcal{W}}

\newcommand{\X}{\vec{X}}

\newcommand{\Y}{\vec{Y}}

\newcommand{\mockalph}[1]{}

\newcommand{\oo}{\infty}

\renewcommand{\bar}[1]{\overline{#1}}
\renewcommand{\hat}[1]{\widehat{#1}}
\renewcommand{\tilde}[1]{\widetilde{#1}}
\renewcommand{\vec}[1]{\mathbf{#1}}
\newcommand{\cov}{\operatorname{cov}}

\newcommand{\quasienvironmentality}{{minimal selectivity}}

\newcommand{\tot}{{\operatorname{tot}}}

\newcommand{\Id}{\operatorname{Id}}

\newcommand{\Lonelim}{\mbox{$L^1$--$\operatorname{lim}$}}
\newcommand{\Ltwolim}{\mbox{$L^2$--$\operatorname{lim}$}}

\newcommand{\Tr}{\operatorname{Tr}}

\newcommand{\var}{\operatorname{var}}

\newcommand{\NS}{{\operatorname{NS}}}
\newcommand{\EC}{{\operatorname{EC}}}
\newcommand{\KS}{{\operatorname{KS}}}
\newcommand{\dis}{{\operatorname{dis}}}
\newcommand{\mix}{{\operatorname{mix}}}

\newcommand{\symdiff}{\,\triangle\,}
\newcommand{\opleft}{{\operatorname{left}}}
\newcommand{\opright}{{\operatorname{right}}}
\newcommand{\opin}{{\operatorname{in}}}
\newcommand{\opout}{{\operatorname{out}}}

\numberwithin{equation}{section}

\theoremstyle{definition}
\newtheorem{thm}{Theorem}[section]

\newtheorem{cor}[thm]{Corollary}
\newtheorem{lem}[thm]{Lemma}
\newtheorem{sublem}[thm]{Sublemma}
\newtheorem{pro}[thm]{Proposition}
\newtheorem{conj}[thm]{Conjecture}
\newtheorem{defn}[thm]{Definition}

\newtheorem{rem}[thm]{Remark}
\newtheorem{exa}[thm]{Example}
 \title{The Mathematics of Evolution: The Price Equation, Natural Selection, and Environmental Change}
\author[T. LaGatta]{Tom LaGatta}
\email{tlagatta@gmail.com}

\date{\today}
\keywords{evolution, Price equation, natural selection, environmental change, entropy, population dynamics}
\subjclass[2010]{92D15}

\begin{document}

\begin{abstract} 

%\textbf{DRAFT VERSION. DO NOT SHARE. \today} 

George Price introduced his famous equation to study selective and environmental effects in discrete populations. We extend Price's evolutionary framework to the measure-theoretic and quantum cases, showing that all evolutionary processes decompose into selective and environmental components. 
%The selective change of a random variable is its covariance with relative fitness, and the environmental change is the measure-preserving redistribution of population after weighting by the relative fitness. 
We also extend Fisher's fundamental theorem, showing that selective change of relative fitness equals the variance of relative fitness. % We analyze multiple functionals with convex analysis, showing monotonicity across selective changes. 

To further quantify selective and environmental effects, we introduce selective and environmental entropy functionals. Selective entropy is non-positive, representing biological negentropy, and environmental entropy is non-negative, representing physical entropy. The selective entropy vanishes if and only if the selective change operator vanishes, and environmental entropy vanishes if and only if the environmental change operator vanishes. The environmental entropy further decomposes into dispersion and mixing entropies, which in general are not realized by change operators.

%The selective or environmental entropy vanishes if and only if the selective or environmental change operator vanishes.

We prove four novel Laws of Natural Selection, showing that selection consistently acts in a manner to increase selection, but which can be disrupted by environmental change. Our methodology is to apply convex analysis to variance and entropy functionals and their selective changes, a technique which applies to both theoretical models and empirical data. These laws are inspired by but distinct from the classical Thermodynamic Laws. %, and provide a rigid 

Our Zeroth Law is a refinement of Fisher's theorem, showing that variance of relative fitness is bounded below by $1/p_* - 1$, for $p_*$ the proportion of the child-bearing population. This inequality is saturated in the case of ``life and death'' selective-equilibrium populations, and otherwise is a strict inequality for non-equilibrium populations. %Thus any two processes in selective-equilibrium with same $p_*$ have the same variance of relative fitness.
Our First Law shows that selective acceleration of relative fitness is also bounded by a non-negative quantity, which is optimized for the same selective-equilibrium populations. This is a non-conservative, selective version of the Thermodynamic First Law. % The variance of relative fitness represents a dynamic ``energy'' which is amplified by the process of selection.
These results show that natural selection speeds up natural selection, regardless of biological, physical, or mathematical domain.

Our Second Law shows that the selective change of selective entropy and its selective acceleration are similarly bounded by non-positive constants, and these inequalities are saturated in the selective equilibrium case. This is a formal, rigorous version of the Thermodynamic Second Law, specialized to the case of selective entropy always growing under natural selection. 
%, illustrating the importance of Darwin's theories across arbitrary types of processes.
% \var(U) \left(1 + \var(U) \right)

We also introduce a class of environmental-equilibrium processes, where dispersion and mixing effects are perfectly balanced. Our Third Law shows that for environmental-equilibrium processes, selective change of environmental entropy vanishes, and for non-equilibrium processes, it may vary in a certain open window around zero. The environmental-equilibrium case corresponds to ``zero temperature'' processes, and thus this is a selective version of the Third Law of Thermodynamics: environmental entropy is constant under selection only when environmental temperature is at absolute zero.

% We introduce novel entropy functionals to quantify selective and environment effects, and further decompose environmental entropy into dispersion and mixing entropies. These are convex 

% Our variance and entropy functionals tend to be  concave and convex, as are many selective and environmental changes, allowing us to produce novel inequalities using convex 

%Our Third Law provides constraints on the selective change of environmental, dispersion, and mixing entropies, which are saturated for a different type of equilibrium population. 
\end{abstract}

\maketitle
\tableofcontents

\part{The Price Equation and Its Consequences} \label{part_price}

\section{Introduction}

%\textbf{NEW VERSION OF PAPER. Started from scratch. Goal is to be super concise. State the new results, then describe how it fits into context. } \newline

%\subsection{Literature Review}

%The goal of this article is to put the Price equation from mathematical biology on a firm, measure-theoretic foundation, and derive general consequences from this result. We introduce new entropy functionals corresponding to natural selection and environmental change, and we analyze  

% The real value of the Price equation comes from decomposing different systems or quantities in terms of their selective and environmental parts, and seeing how they interact. 

%In Part \ref{part_price}, we state the prove the Price equation in generality, prove a new result about representations of evolutionary processes, and state Fisher's fundamental theorem in appropriate form. 

%\textbf{DRAFT VERSION. DO NOT SHARE. COMPILED \today} 

George Price introduced his famous equation \cite[(4)]{price1970selection} to analyze evolutionary processes acting on discrete populations. The Price equation states that any change decomposes into ``natural selection'' and ``environmental change''. We extend Price's equation to the general measure-theoretic and quantum cases, decomposing evolutionary transition kernels into selective and environmental components. %We also prove a quantum Price equation. 
The general effect of natural selection is to grow and scale population sizes, and environmental change is to redistribute those populations. 

We introduce novel %selective and environmental
entropy functionals to further quantify the amount of selection and environmental change of a process. We define the selective entropy as the Kullback-Leibler divergence of relative fitness, modeling biological effects of growth and variation of fitness. This selective entropy (or ``negentropy'') is non-positive, and vanishes for purely environmental processes. We define environmental entropy as the one-step Kolmogorov-Sinai entropy of a process, modeling physical effects of dispersion and mixing of populations. Environmental entropy is non-negative, and vanishes for purely environmental processes.

Fisher's theorem states that selective change of relative fitness (the selective velocity) is the variance of relative fitness. We present four Laws of Natural Selection extending Fisher's theorem, showing that selective functionals tend to be monotone under the act of selection. Thus ``selection begets selection'' in the absence of environmental effects, with selective effects amplifying over time, and environmental effects either disrupting or amplifying selection. 

We apply an optimization perspective to the Price equation, and prove four novel Laws of Natural Selection. Our Zeroth Law (Proposition \ref{pro_zerothlaw}) extends Fisher's theorem, providing a lower bound for selective velocity of $1/p_* - 1$, where $p_*$ is the proportion of child-bearing population, which is strengthened by Corollary \ref{cor_strongzerothlaw}. These bounds are saturated when the process is in ``selective equilibrium'': the Darwinian life-or-death processes where fitness is either zero or non-zero. Thus any two processes in selective-equilibrium with same $p_*$ have the same variance of relative fitness, in analogy with the Thermodynamic Zeroth Law. 

Our First Law (Theorem \ref{thm_firstlaw}) shows that the selective change of relative variance (the selective acceleration) is non-negative, with a lower bound given by variance times second-moment of fitness. This bound is saturated again in selective equilibrium. Consequently, selective non-equilibrium processes tend to ``speed up'' and become more selective over multiple iterations, though this selective acceleration can be disrupted or amplified by environmental effects. %If fitness and population are used to model energy and mass, then our First Law describes a non-conservative version of the First Law of Thermodynamics. 

Our Second Law (Theorem \ref{thm_secondlaw}) shows that the selective change of selective entropy is non-positive, amplifying selective effects since selective entropy is itself non-positive. The bounds are saturated in selective equilibrium, meaning that selective non-equilibrium processes become more selectively entropic over time. This is a selective form of the Second Law of Thermodynamics: selective entropy is monotonic under the act of selection. 

We identify a case called environmental equilibrium, corresponding to processes whose dispersive and mixing effects are perfectly balanced. Our Third Laws (Theorem \ref{thm_weakthirdlaw}, \ref{thm_thirdlaw}) show that for environmental-equilibrium processes, selective change of environmental entropy vanishes, and otherwise it may may fluctuate in a certain open window around zero. This models selective and environmental interactions. Environmental equilibrium corresponds to ``zero temperature'' processes, or systems at absolute zero. Thus this result is an selective version of the Third Law of Thermodynamics: environmental entropy is constant under selection only when environmental temperature is at absolute zero.

We derive similar results in the case of quantum processes, showing that the evolutionary framework covers both classical and quantum systems. We also prove a version of the Price equation in the case of open processes.%, i.e., a generalized version of the classic (known as the Kerr-Godfrey-Smith equation) in the op consider the case of open evolutionary processes. %, proving a variant of the Price equation in that context. also present a version of the Kerr-Godfrey-Smith equation for open processes. 

% This is an analytical, formal approach to evolutionary dynamics. Our hope is that practicing scientists can use this work to compute selective changes and entropy functionals for their models, and thus better approximate true distributions... Statistics takes probabilistic ensembles, and updates them based on estimating various functionals. Our hope is that evolutionary statisticians find the changes and entropies we compute to ....
% with empirical data \cite{mccullagh2002statistical}, based on estimated % . Our hope is that scientists can use this formalizi

\subsection{Literature Review} \label{sect_litreview} 

The statistician Ronald A. Fisher 
%\cite{fisher1930genetical, fisher1958genetical} 
was the first to introduce a quantitative theory of selection. Fisher's ``fundamental theorem'' stated ambiguously that ``the rate of increase in fitness of any organism at any time is equal to its genetic variance in fitness at that time'' \cite[p.~35]{fisher1930genetical}, \cite[p.~37]{fisher1958genetical}. Fisher claimed his theorem was a biological version of the Second Law of Thermodynamics, in that ``natural selection requires a `reservoir' of additive genetic variance'' \cite{plutynski2006fisher}. However, an exact quantitative reading of Fisher's statement eluded the biology community until the work of Price in the 1970s. 
Price \cite{price1970selection} recognized that Fisher's statement could be quantified as population covariance against fitness. This enabled him to convert Fisher's regression statistics \cite{fisher1930genetical} into a discrete probabilistic framework, substituting populations for probability. Robertson \cite{robertson1966mathematical} independently identified the covariance formula in his work on dairy farming.

Price stated his eponymous equation as \cite[(4)]{price1970selection}, decomposing arbitrary change into two terms and nothing else, which later \cite{price1972fisher} he would describe as selective change and environmental change. Fisher's fundamental theorem follows as an immediate consequence: the selective change of relative fitness is equal to its variance. In Price's view, ``The main cause of misunderstanding about the theorem is that everyone has supposed that Fisher was talking about the total change rather than just the fraction of this due to natural selection'' \cite[p.~130]{price1972fisher}. In \cite{price1971extension}, Price provides a formal extension of the Hardy-Weinberg mating law using his framework. In the posthumous \cite{price1995nature}, Price argues for a general theory of selection and frames what some of its properties might be. In \cite{price1972extension}, Price further extended the covariance-selection mathematics to the multi-level and continuous-time cases.

Price hoped to build a theory of altruism \cite{harman2011price}. He worked with Maynard Smith \cite{price1972fishersmith,smith1973logic} on evolutionary stable strategies, and he communicated with Hamilton \cite{hamilton1996narrow1} who used the covariance-selection mathematics to build a unified theory of kin selection \cite{hamilton1970selfish} and group selection \cite{hamilton1975innate} using the Price equation.

% Price proved the Price equation by elementary algebra in \cite{price1970selection}, introduced the partial derivative and population-covariance notation in \cite{price1972fisher}, and expanded on his formalism in \cite{price1972extension}

Price's work lingered in the annals of evolutionary biology as an interesting sidenote, enabling many authors to apply the general theory to their mathematical models of interest (cf.~e.g., \cite{hamilton1970selfish, lewontin1974genetic, crow1976rate,grafen1985geometric,michod2000darwinian,loreau2001partitioning,fox2006using}). The research program of Frank has centered the Price equation at forefront of evolutionary theory \cite{frank1985hierarchical,frank1986hierarchical,frank1995george,frank1997price,frank2009natural,frank2018price}. Frank \cite{frank1986dispersal,frank1987demography,frank1992kin} shows how to apply the Price equation to evolutionary-stable strategies, by taking a variational derivative of the Price equation in a manifold of parametrized model constraints. 
% This has lead him to showing the Price equation is equivalent to many classical theorems in evolutionary biology, by choosing context-specific ways to parametrize the statistical model constraints and using calculus. 
This has enabled Frank to show that natural selection maximizes Fisher information \cite{frank2012natural}, that the Price equation is equivalent to d'Alembert's principle in physics \cite{frank2015d}, and that the Price equation is equivalent to the statistical equation of model of Nicholson et al. \cite{nicholson2020time,frank2020fundamental}. 

Page and Nowak \cite{page2002unifying} show that the Price equation in continuous settings is equivalent to the Replicator-Mutator and Lotka-Volterra equations. Most common biological mechanisms (e.g., genetic, epigenetic, behavioral, and symbolic) can be expressed in a unified manner using the Price equation; see \cite{helantera2010price} and \cite{lehtonen2020fifty}. Rice \cite{rice2008stochastic,rice2020universal} describes stochastic evolutionary processes with the multi-level Price equation. Week et al. \cite{week2021white} present a stochastic partial differential equation version of the Price equation for a Gaussian allelic model of mutation, approximating large populations of discrete individuals by diffusion limits. 

%  % simulating individual-based models, approximating their diffusion limits, solving associated SPDE and solving SDE corresponding to an eco-evolutionary model of a competitive community

% Philosophy of math of Price equation:
% Ewens 1989
% Lessard 1997

% Fisher saw his theorem as a biological equivalent of the Second Law of Thermodynamics \cite{price1972fisher,plutynski2006fisher}. Plutynski \cite{plutynski2006fisher} frames this as, ``Natural selection requires a `reservoir' of additive genetic variance''. Just as heat is converted into temperature, additive genetic variance $\var(U)$ is converted into relative fitness via the selective change $\partial_\NS(U)$. % We use an analogy between relative fitness and temperature throughout this article

Nowak \& Highfield \cite{nowak2011supercooperators} criticize the universal applicability of the Price equation as a mere tautology. Frank \cite{frank2009natural,frank2012natural} counters that this a strength of the Price equation. To better understand total change of a process, we transform it into selective and environmental parts, calculate and reason about each separately, then use the Price equation to combine insights additively, e.g., by summing equalities or inequalities. We cannot expect a general mathematical theorem to have explanatory power in and of itself, but we can use it as a vessel for interpreting empirical data and conclusions about real-world phenomena. 

% Ewens 1989, Lessard 1997 and Plytinski 2004 defend 

Grafen and Batty et al. \cite{grafen2000developments,grafen2007formal,batty2014foundations,grafen2015biological} have built an topological-analytical framework for working with Price's equation and Fisher's fundamental theorem, based on measure theory with common topological assumptions (in particular, Borel measurability). An early paper of Grafen \cite{grafen2002first} considered the case of arbitrary measurable populations, related by an integral kernel $w_i(\d i) = \frac{1}{N} \int w(i,i') \mu'(\d i')$ with unspecified regularity assumptions, and only considered the selective change, not the full Price equation.
%, rather than arbitrary regular conditional distributions; and did not consider the full general Price equation, just the selective form. %Grafen \cite{grafen2002first} also explores a class of fitness-optimizing processes.  

Kerr and Godfrey-Smith \cite{kerr2009generalization} relax the assumption that all children be accounted for by parents, and prove an extended Price equation with a third term for those orphaned children. Brown and Field \cite{brown2021extended} recognize that this has novel interpretations around migration and mixed asexual/sexual reproduction. %We do not consider this extension, instead treating migration and mixed asexual and sexual reproduction using the environmental-entropy formalism of Part 3. 

Luque et al. \cite{luque2021mirror,baravalle2021towards} argue for the Price equation at the center of a general theory of evolution, including cultural evolutionary theory. Aguilar and Ak{\c{c}}ay \cite{aguilar2018gene} use the multi-level Price equation to analyze processes in terms of genetic  and cultural factors. Reiskind et al. \cite{reiskind2021nothing} use the Price equation to describe the selective change of trait and allele frequencies across generations. %Frank and Bruggeman \cite{frank2020fundamental} recognize that the discrete Price equation corresponds to the statistical equation of motion of Nicholson et al. \cite[(1)]{nicholson2020time}.

\subsection{Justification of Abstraction and Biological Examples} 
\label{sect_justificationofabstraction}
%\begin{rem}[Justification of Abstraction] \label{rem_justificationofabstraction}

Before we present the abstract framework for the Price equation, we share a brief justification for why this level of abstraction is useful, and discuss some biological implications. Historically, Price's discrete framework has helped scientists analyze simple populations of differentiated individuals, and revealed new biological insights. In the study of more complex systems like continuous, hierarchical, and stochastic populations, researchers have introduced alternative versions of the Price equation, as described in Section \ref{sect_litreview}.

%Researchers have needed to introduce alternative versions of the Price equation to analyze , as described in Section \ref{sect_litreview}. 

While these models are each useful in their specific domains, approximations become difficult when dealing with complex multi-scale systems, especially those with very small and very large scales. For example, metacognition arises from competing evolutionary time scales, and has resisted a quantitative modeling via the Price equation \cite{kuchling2022metacognition}. Measure theory provides an effective way to integrate different models of the Price equation into a coherent whole, as illustrated by Grafen \cite[Sect 2.4]{grafen2002first}:
% presents a convincing argument for building a measure-theoretic framework for the Price equation:
\begin{quote}
The first reason to be general is to show that the optimization link with natural selection is not just a coincidence in a special case, but a fundamental fact about a class of selection processes. Furthermore, the formal Darwinism unification project aims to provide a technical representation of the commonsense, informal, arguments first proposed by Darwin (1859), and accepted by generations of biologists since. The formal argument should work in the same way for finite and infinite populations; for haploid populations, diploid populations and mixtures; for one-, two- and multi-locus traits; and for cases with and without environmental stochasticity, with finite or infinite sets of possible environments. Darwin did not take these cases separately, and neither should we. It is worth noting that, although the apparatus is complex, the argument is simple, reflecting the persuasive nature of the original verbal argument.

Another advantage of generality is that the theoretical developments here can be viewed as ``meta-models'', that is, as models of models. The aim is to show that a wide class of existing population genetic models admit of an optimization interpretation, and to show how to construct the corresponding optimization model. This purpose is fulfilled in proportion to the generality of the model.

Finally, the model is not yet general enough. A general argument provides a better source for further development than a special case. For example, inclusive fitness and ESS theory could be incorporated with careful extensions of the model, and ideally both would be incorporated simultaneously.
\end{quote}

Our abstract framework takes Grafen's next steps of generalization. In Part 1, we describe inclusive fitness via an evolutionary transition mapping \eqref{def_process}, categorizing and quantifying all measurable parent-child relationships in rearing. We show that natural selection is given by absolutely continuous scalings of measures, and environmental change by Markov chains, familiar tools to applied mathematicians. %While Fisher's theorem shows that selective change of relative fitness is its variance, we further decompose the environmental ... 
Our Zeroth Law bounds the selective velocity, and the First Law bounds the selective acceleration. These laws quantify how quickly selection speeds up the process of selection.

In Parts 2 and 3, we introduce an optimization theory based on new entropy functionals, quantifying the degree of selective and environmental change of a process. These entropies satisfy universal quantitative law: our Second Law further quantifies how selective change drives selection, and our Third Law shows how environmental change amplifies or counters selection. 
%Equilibrium populations are those which optimize these inequalities, and non-equilibrium populations can drive additional selective growth. 
% Populations which optimize these inequalities are in selective or environmental equilibrium, and non-equilibrium populations can exploit additional selective and environmental benefits depending on niche. 
We also extend these laws into the quantum realm, which has implications for selection in quantum biology \cite{lambert2013quantum,cao2020quantum}. The primary method we use throughout our analysis is Jensen's inequality, as applied to convex and concave entropy functionals and their changes. 

The equilibrium cases for the entropy optimization inequalities correspond to evolutionarily stable strategies. These are characterized by the saturation condition for Jensen's inequality, meaning that we do not have to compute partial differential equations to solve for the variational principle. With additional specifications on the model, these equilibrium cases can be analyzed using calculus and methods from evolutionary game theory. Thus the entropy functionals provides a measure of model fit to empirical data, where real-world populations can be analyzed and approximated relative to their nearest equilibrium neighbors. 

We summarize some biological examples where the abstract framework can be used, extending techniques currently available in the literature. The abstract framework applies universally both to concrete models and empirical data, but it cannot provide biological insights in the absence of models or data. % As McCullagh says, ``mathematics knows nothing about anything except mathematics, so mathematics must be instructed in the facts of rural life'' \cite{mccullagh2002statistical}.
Nonetheless in applications, the Laws of Natural Selection will manifest as constraints on observed selective and environmental growth. 

Population niches which are stable over generations can be modeled locally by processes in selective or environmental equilibrium. In the absence of environmental effects, populations engage in pure selective growth, with non-equilibrium populations obtaining faster rates of selection. When combined with environmental effects, populations can interact to optimize their selective growth via dispersal and mixing, with non-equilibrium populations having a bigger impact for or against selection. 

% Use of convex functionals like entropies provide a streamlined way to solving for evolutionarily stable strategies, as we circumvent differential calculus by using the saturation condition given by Jensen's inequality. 

% In these settings, the Selective Laws will always manifest as ``speed limits'' available for selective and environmental change:

\begin{exa}[Biological Examples] \label{exa_justificationofabstraction} \textbf{ }
\begin{enumerate}
    \item Differentiated individuals on a smooth spacetime, %, where each individual $i$ exists at a continuously-varying time $t$ and place $q$ on some topological manifold, represented as a tuple $(i,t,q)$ % \in I$. 
    such as predator-prey models in continuous geographical ecosystems \cite{zhou2013dynamics} or general evolutionary games \cite{friedman2016evolutionary}. Differential calculus and differential games \cite{isaacs1999differential} can be used to study dynamics in these environments, especially for loss/gain functions that populations are optimizing against. The Price framework extends the discrete Page-Nowak dynamic framework \cite{page2002unifying} to the case of populations with very small allele differences, as well as arbitrarily large, multi-scale populations. 
    %The entropy functionals serve as objective functions for selective and environmental change.
    
    %The theory of differential games \cite{isaacs1999differential} can be applied to problems of conflict, and the general Price equation and entropy functionals can be used to incorporate 
    
    \item Non-differentiated entities on a continuous spacetime, such as plants, fungi, and molds, as illustrated by Fox \cite{fox2006using} to study biodiversity loss in a partitioned spatial environment. %, nor does it take into account complex social communication structures as indicated by \cite{crespi2001evolution}. 
    The abstract framework provides a way to integrated microscopic and macroscopic flora into a single model, as described as shapes extended over a continuous spacetime, organized by genotypic and phenotypic properties. The Price equation describes selective growth of these flora, and the environmental change from dispersion and mixing after spora leave the originating parent. This provides a quantitative framework for the qualitative work of Hamilton and Lenton \cite{hamilton1998spora}, who showed how microbes of the atmosphere (spora) use dispersion and mixing to drive selective growth in their populations,
    %and Crespi \cite{crespi2001evolution} extends this analysis to social behaviors. 
    and Lenton and Oijen \cite{lenton2002gaia}, who provided an simple probabilistic model for discrete daisy populations. % However, their model does not include multi-scale dynamics like the transmission of daisy seed and pollen.
    
    \item Stochastically-varying populations, incorporating empirical position, stochastic fluctuations, random strategies, and uncertain states of nature into one distribution describing the system \cite{freund2000stochastic,grafen2002first, taitelbaum2020population}. This is because statistical models are parametrized distributions \cite{mccullagh2002statistical}, and models can be integrated using copulas, as is done in the ecological literature \cite{charpentier2007estimation,ghosh2020copulas}. The abstract framework extends Rice's stochastic Price equation \cite{rice2008stochastic}, enabling the multi-level Price equation to simultaneously describe stochasticity and selection at multiple scales.
    
    \item Hierarchical bioinformatics like protein folding, which combines microscopic genetic codes in amino-acid sequences, mesoscopic configurations of protein as atom configurations, and macroscopic effects arising from protein interactions  \cite[Supplementary Material]{jumper2021highly}. Rice \cite{rice2020universal} shows how to use the stochastic Price equation to analyze bioinformatic codes like genetic sequences, and Reiskind et al. \cite{reiskind2021nothing} use the Price equation to predict selective changes of genetic frequencies. The abstract framework allows us to integrate Rice and Reiskind et al.'s coding theory with Fox's shape-based analysis \cite{fox2006using} for a more complete model of protein folding. Here, entities consist simultaneously of strands of DNA along with folded protein configurations in 3-dimensional space. Transition mappings consist of substitutions of DNA bases, snippings of DNA strands, and reconfigurations and interactions of proteins. 
    
    %The abstract framework extends the Reiskind et al. framework \cite{reiskind2021nothing} of pure selective growth to the case of specimens dispersing and mixing. 
    
%    \item Biophysical populations with dynamics far away from thermodynamic equilibrium \cite{fang2020nonequilibrium}. 
    
    \item Approximations of large-population systems by continuous models and hydrodynamic limits, where increasing sequences of finite populations are embedded in a uniform topological space where limits are defined, and transition mappings satisfy partial differential equations \cite{demasi2006mathematical}. Hydrodynamic limits have been historically applied in physics \cite{rezakhanlou1991hydrodynamic}, economics \cite{scalas2006application}, and political science \cite{dehaan2008dynamics}, and more recently have been used in neuroscience to describe large systems of interacting neurons \cite{demasi2015hydrodynamic}, and in crowd dynamics to understand behaviors of herds \cite{banda2020recent}. Week et al. \cite[(10)]{week2021white} present a limiting Price equation to approximate large populations with purely selective growth, which holds under sufficiently strong regularity conditions as the population size goes to infinity. The abstract framework is robust enough to handle infinite population sizes as the limit of large populations, and we present a continuous-time Price equation \eqref{eqn_timevaryingprice} which generalizes that of Week et al. Equilibrium conditions can be analyzed by taking partial derivatives of model parameters, as with the comparative-statics method in political science \cite{little2015elections}. In large-but-finite models, decision-making can be analyzed using Poisson games \cite{smith2017group}.
    
    \item Computer vision, where entities are represented at a microscopic scale as shaded pixels, while simultaneously organized as macroscopic shapes and objects \cite{samet1989hierarchical,van2020blending}. This can be represented hierarchically, where population individuals combine pixelated images and collections of features on those images. Transition functions can include changes in the image size and coloration, as well as the addition, subtraction, and merging of object structures. The abstract framework provides a practical way to integrate high-dimensional empirical data of ecosystems with classic evolutionary models that use the discrete Price equation. Nowozin \cite{nowozin2014optimal} analyzes optimal decision-making problems in computer vision, by leveraging Rice's stochastic Price equation to approximate ratios of random variables \cite{rice2008stochastic}. 
    
    \item The Price equation provides an alternate approach to thermodynamics, as illustrated in recent work in the evolutionary biology and statistical physics literature. Nicholson et al. \cite[(1)]{nicholson2020time} and Frank and Bruggeman \cite{frank2020fundamental} identify the discrete Price equation as the equation of motion for systems with finitely-many energy states. They interpret the Price equation as a stochastic First Law of Thermodynamics, decomposing motion into ``flux of heat'' (selective change) and ``flux of work'' (environmental change). Nicholson et al. \cite[(13)]{nicholson2020time} and Frank \cite[\S 6.4,\,12]{frank2018price} explore inequalities for the rate of Shannon entropy production, representing stochastic versions of the Second Law of Thermodynamics. These methods directly extend to the general situation via the abstract framework.
    
    Our Laws of Natural Selection are related to the Laws of Thermodynamics in subtle ways, which should be further explored by future researchers. When fitness and population are used to model energy and mass, then our First Law describes a non-conservative version of the First Law of Thermodynamics. Our Second Law shows that selective entropy is monotone under selective change, relating to the monotonicity of the Second Law of Thermodynamics. Our Third Law provides extremes for environmental change, and therefore environmental equilibrium corresponds to a ``zero temperature'' case. 
\end{enumerate}
\end{exa}
% Just as Price's theoretical discrete work paved the way for numerous applications, summarized in Section \ref{sect_litreview}, our hope is that this article can serve as a foundation for applications across these domains. In each domain, the general Price equation provides a description of selective and environmental processes, and the entropy functionals of Parts 2 and 3 allow for more quantitative analyses of growth, dispersion, and mixing phenomena.
%\end{rem}

 % Consequently, the three Selective Laws provide a novel method for scientists to estimate model fit, by measuring the degree to which measured entropy functionals compare to calculated ones. %We look forward to generations of mathematical biologists 

%\section{Review of Discrete Price Equation and Article Summary} \label{sect_discretepriceequation}

\subsection{Review of Discrete Price Equation} \label{sect_discretepriceequation}

%\subsection{Discrete Price Equation}

We recall Price's discrete work \cite{price1970selection,price1972fisher} and express it in modern probability formalism. We summarize the general measure-theoretic and quantum frameworks in Section \ref{sect_summary}, and present in detail in Sections \ref{sect_price} and \ref{sect_quantum}.

Let $I = (i_1, \cdots, i_K)$ be a finite set, and let $\mu = (N_1, \cdots, N_K)$ and $\mu' = (N'_1, \cdots, N'_K)$ be two finite measures on $I$, representing separate populations of interest (i.e., $N_k = \mu(\{i_k\})$ and $N'_k = \mu'(\{i'_k\})$). Write the total population sizes $\mu(I) = N$ and $\mu'(I) = N'$. Let $X = (X_1, \cdots, X_K)$ be a measurable function (an ``observable''), with average values $\bar X := \E[X] := \frac{1}{N} \sum_k X_k N_k$ and $\bar X' := \E'[X] := \frac{1}{N'} \sum_k X_k N_k'$. 

Price \cite{price1970selection} introduced the average change operator as the difference of the average values:
\begin{equation}
    \Delta(\bar X) := \bar X' - \bar X = \E'[X] - \E[X] = \sum_k X_k(N_k' - N_k).
    %\Delta(\bar X) := \bar X' - \bar X = \E'[X] - \E[X] = \sum_k X_k(N_k' - N_k).
\end{equation}

%We will also use the notation $\Delta(\bar X) := \Delta(\bar X)$, in order to describe changes of functionals.{For example, 
%    \begin{equation}
%        \Delta \var(X) := \var'(X) - \var(X) = \Delta(\bar X^2-\E[X]^2, X^2-\E'[X]^2) = \Delta \left( \overline{X^2-\E[X]^2}, \overline{X^2-\E'[X]^2}\right).
%    \end{equation} }

% (e.g., $\Delta \var(X) := \Delta(\bar X^2 - \bar X^2) = \Delta(\overline{X^2 - \bar X^2}) = \E'[X^2-\bar X^2] - \var(X)$). 

% Notational Remark.
% We will use the notation $\Delta(\bar X) := \Delta(\bar X)$ in order to describe changes of functionals (e.g., $\Delta(\var(X),\var'(Y)) := \var'(Y) - \var(X)
% \Delta(\bar X^2-\bar X^2) := \E'[X^2 - (\bar Y')^2] \E[X^2-\bar X^2])$). We will also use the notation $\Delta(\bar X, \bar Y) := \Delta(\bar X, \bar Y) := \bar Y - \bar X = \E'$

Price took as given that the populations $\mu$ and $\mu'$ be somehow related. We formalize this with the concept of an evolutionary process, representing a full accounting of the child population $\mu'$ in terms of the parent population $\mu$. For the discrete setting, we express an evolutionary process as an arbitrary transition kernel $w(i,i')$, satisfying 
\begin{equation} \label{eqn_discretedisintegration}
    \mu'(i') = \sum_k w(i_k,i') N_k.
\end{equation}
i.e., the second population can be decomposed as a weighted sum against the first population. This is a discrete disintegration equation, in the sense of regular conditional probabilities \cite{leao2004regular,bogachev2007measure,lagatta2013continuous}. The kernel $w$ represents an \emph{evolutionary process}, where $w(i,i')$ is the contribution of type $i$ toward child $i'$, and equation \eqref{eqn_discretedisintegration} ensures that all children are accounted for in terms of parents. For example, in sexually-reproductive populations, every child has two biological parents, so $w(i,i') = 1/2$ for each parent $i$ of $i'$. In asexually-reproductive populations, every child $i'$ has a unique parent $i$, so $w(i,i') = 1$.

%\subsubsection*{Natural Selection}

Write the total fitness function $W_k := W(i_k) := \sum_{k'} w(i_k, i_{k'})$, i.e., the total contribution to all children $i_{k'}$ from parent $i_k$. Let $\bar W := \frac{N'}{N}$ be the population ratio, satisfying $\bar W = \E[W] = \sum_k W_k N_k$. Define the relative fitness $U_k := U(i_k) := \frac{W(i_k)}{\bar W}$. %Write $W_k := W(i_k)$ and $U_k := U(i_k) = \frac{W_k}{\bar W}$. 

Price defined the selective change of $X$ as the covariance against relative fitness:
\begin{equation}
    \partial_\NS(X) := \cov(X,U) := \E[(X-\bar X)(U - 1)] = \frac{1}{N} \sum_k (X_k - \bar X) (U_k - 1) N_k.
\end{equation}

%\subsubsection*{Environmental Change} 

Define the local average $\langle X \rangle_w(i_k) := \frac{1}{W(i_k)} \sum_{k'} X(i_{k'}) w(i_k,i_{k'}) N_k$, i.e., the average of $X$ across all children of $i_k$, normalized by fitness. Using this, we can express the process definition \eqref{eqn_discretedisintegration} in terms of a tower property:
\begin{equation}
    \E'[X] = \E[U \langle X \rangle_w].
\end{equation}
i.e., the expected future value is given by taking the scaled population average of the local average weighted by relative fitness. Define the local change $\Delta_w(X)(i) := \langle X \rangle_w(i) - X(i)$ as the difference between the local average and the original value of $X$. 

Price defined the environmental change of $X$ as the average local change, weighted by $U$:
\begin{equation}
    \partial_\EC(X) := \E[\Delta_w(X) U] = \E[(\langle X \rangle_w - X)U] = \E'[X] - \E[UX]. %= \frac{1}{N} \sum_k \left( \frac{1}{W_k} \sum_{k'} X(i'_{k'} w(i'_{k'}|i_k) N_k - X(i_k)  \right) U_k.
\end{equation}
% \subsubsection*{Discrete Price Equation}

The Price equation states that the average change is the sum of selective change and environmental change, with no additional components:

\begin{equation} \label{eqn_discreteprice}
    \Delta(\bar X) = \partial_\NS(X) + \partial_\EC(X) = \cov(X,U) + \E[\Delta_w(X) U].
\end{equation}

The proof of the discrete Price equation is simple given the definitions and tower property:
\begin{eqnarray}
    \Delta(\bar X) &=& - \E[X] + \E'[X] = \left( \E[UX] - \E[X] \right) + \left( \E[U\langle X \rangle_w] - \E[UX] \right) \nonumber \\
    &=& \cov(X,U) + \E[\Delta_w(X) U]. \label{eqn_discretepriceproof}
\end{eqnarray}

Fisher's fundamental theorem \cite{price1972fisher} states that the selective change of relative fitness is equal to its variance: 
\begin{equation}
    \partial_\NS(U) = \cov(U,U) = \var(U). 
\end{equation}

% We revisit this argument for the general Price equation in Section \ref{sect_price}. 

%In the remainder of this article, we use this foundation to further quantify selective and environmental effects. 

\subsection{Article Summary} \label{sect_summary}

We now %describe the abstract framework, and 
summarize our results and contributions.
%I LIKE THIS SENTENCE: We introduce a theory of functional change based on the Price equation, and apply this theory to key functionals like mean relative fitness, variance, and selective and environmental entropies. We show that selective changes of selective functionals tend to be monotone, illustrating that selection begets more selection in the absence of environmental effects. 
% We introduce two novel entropy functionals: selective entropy, integrating informatic entropy to quantify selective growth processes; and environmental entropy, integrating physical entropy to quantify environmental redistributive processes. We prove multiple inequalities for the selective and environmental entropies and their changes. We further split environmental entropy in terms of dispersion and mixing entropies. 
% We present the Price equation and some new consequences. Our framework includes the Grafen-Batty-et al. topological framework \cite{grafen2015biological}, the Grafen measurable framework \cite{grafen2002first}, the Page-Nowak differential framework \cite{page2002unifying}, and the Week et al. diffusion-limit framework \cite{week2021white} as special cases. 

\subsubsection*{Part 1 (Price Equation)} 

%We summarize the mathemtaics of the abstract framework. 

In Section \ref{sect_price}, we introduce the abstract framework for the Price equation. We represent populations by finite measures $\mu$ and $\mu'$ on some (possibly distinct) measurable spaces $I$ and $I'$, and evolutionary processes $w : \mu \mapsto \mu'$ as a measurable linear map of those measures, i.e., as any measurably-varying family of measures $w_i$ satisfying the disintegration equation
\begin{equation} \label{def_process_intro}
    \mu'(B) = \int_I w_i(B) \mu(\d i).
\end{equation}
This includes Price's discrete framework as a special case, while allowing for infinitary changes and evolution of the state space. This includes the biological settings of Sections \ref{sect_litreview} and \ref{sect_justificationofabstraction} as special cases.

In Section \ref{sect_price}, we consider evolutionary processes $w : \mu \mapsto \mu'$ satisfying the disintegration equation \eqref{eqn_disintegration}, transforming one population measure $\mu$ of size $N = \mu(I)$ to another measure $\mu'$ of size $N' = \mu'(I')$ via a transition mapping $w_i$. Write expectations on functionals by dividing by population sizes: $\E[X] := \frac{1}{N} \int_I X(i) \mu(\d i)$ and $\E'[X'] := \frac{1}{N'} \int_{I'} Y(i') \mu'(\d i')$.

Define the average change $\Delta(\bar X, \bar Y) := \E'[Y] - \E[X]$. We state and prove the general form of the Price equation (Theorem \ref{thm_price}): 
\begin{equation}
    \Delta(\bar X, \bar Y) = \partial_\NS(X) + \partial_\EC(X,Y),
\end{equation} 
for the selective and environmental changes:
\begin{equation}
    \partial_\NS(X) := \cov(X,U) \quad \mathrm{and} \quad \partial_\EC(X,Y) := \E\!\left[ \left( \langle Y \rangle_w - X \right) U \right],
\end{equation}
where $\cov(X,U) = \E[X(U-1)] = \int_I X(i) \big(U(i) -1\big) \mu(\d i)$ and $\langle Y \rangle_w(i) := \frac{1}{W(i)} \int_{I'} Y(i') w_i(\d i')$. 

The general form of Fisher's fundamental theorem (Theorem \ref{thm_fisher}) follows as a trivial consequence: $\partial_\NS(U) = \cov(U,U) = \var(U)$. If $w$ and $w'$ are composable processes with relative fitnesses $U$ and $U'$, then since both have unit mean ($\E[U] = 1 = \E'[U']$), the Price equation implies that the environmental change of relative fitness is non-positive:
\begin{equation}
    \E[(\langle U' \rangle_w - U)U] = \partial_\EC(U,U') = -\partial_\NS(U) = -\var(U) \le \frac{1}{p_*} - 1 \le 0,
\end{equation}
for the childbearing population proportion $p_* := \frac{1}{N} \mu(W>0)$

In Section \ref{sect_zerothlaw}, we introduce selective equilibrium as the case of life-or-death processes where $U$ takes exactly two values ($0$ and $\frac{1}{p_*}$). We prove the ``Weak Zeroth Law of Natural Selection'' (Proposition \ref{pro_zerothlaw}):
\begin{equation} \label{ineq_zerothlaw}
    \partial_\NS(U) = \var(U) \ge 1 - \frac{1}{p_*} \ge 0.
\end{equation}
The first inequality is saturated when $w$ is in selective equilibrium. The second inequality is saturated when $w$ is purely environmental: there is no selection or growth, and the process is just a Markov chain. Accordingly, natural selection by itself never reduces average relative fitness, though environmental change might change or eliminate it entirely.

In Section \ref{sect_pure}, we present ``pure'' processes, which are fully described by either purely selective or purely environmental change. We prove the Price representation theorem (Theorem \ref{thm_pricerepresentation}), which states that every evolutionary process factors into a purely selective process followed by a purely environmental one, i.e., $w = w_\EC \circ w_\NS$. 

In Section \ref{sect_selectivechangeofvariance_and_firstlaw}, we analyze changes of the variance using Jensen's inequality. %We introduce a repeatable technique to control convex or concave functionals with Jensen's inequality.%, which we use throughout the article. 
% Consider functionals $F(X) := \E[f(X)]$ and $G(Y) := \E'[g(Y)]$, where the variables $X$ and $Y$ may be vector valued. For example, for general variance, $\var(X) = \E[X^2 - \bar X^2] = \E[f(X,\bar X)]$, with $f(X,\bar X) = (X-\bar X)^2$.
By Fisher's theorem, the selective change of variance is the ``selective acceleration'' of relative fitness:
\begin{equation} 
    \partial_\NS^2(U) := \partial_\NS \var(U) := \cov(U^2-1,U) = \cov(U^2,U). 
\end{equation}

We prove the ``First Law of Natural Selection'' (Theorem \ref{thm_firstlaw}), showing that selective acceleration is non-negative:
\begin{equation} \label{ineq_firstlaw_intro}
    \partial_\NS^2(U) = \partial_\NS \var(U) \ge \var(U) \left( 1 + \var(U) \right) \ge 0,
\end{equation}
Just as with the Zeroth Law \eqref{ineq_zerothlaw}, the first inequality of \eqref{ineq_firstlaw_intro} is saturated when $w$ is in selective equilibrium, and the second inequality is saturated when $w$ is purely environmental. 

%We can think of the additive genetic variance $\var(U)$ as the evolutionary ``energy'' of a process. Thus \eqref{ineq_firstlaw} defines a sort of First Law of Natural Selection, describing how energy biologically evolves in time (i.e., under selection). 
% We refer to  as the First Law of Natural Selection, since it describes the evolution of the relative-fitness variance, which is akin to the net power of a process. The Price equation for relative-fitness variance  \eqref{eqn_firstlaw_priceequation} thus describes the evolution of net power: there is a non-negative heat transfer term $\partial_\NS \var(U)$ described by \eqref{ineq_firstlaw}, plus net work characterized by $\partial_\NS(\var(U), \var'(U'))$.

In Section \ref{sect_environmentalchangevariance}, we provide a lower bound on the environmental change of variance:
\begin{eqnarray}
    \partial_\EC(\var(U), \var'(U')) &=& \E[\Delta_w(U^2,(U')^2) U] \nonumber \\
    &\ge& \frac{1}{\E[U^3]} \E\!\left[ U^2 \left( \langle U' \rangle_w - U \right) \right] \E\!\left[ U^2 \left( \langle U' \rangle_w + U \right) \right], \label{ineq_environmentalchangeofvariance_lowerbound_intro}
\end{eqnarray}
%where we define the ratios ratios $R(i,i') := \frac{U(i')}{U(i)}$ and $\bar R_w(i) := \frac{\langle U' \rangle_w(i)}{U(i)}$. 
which is saturated when the processes are ``strongly stationary'', i.e., $U'(i') = U(i)$ for $w_i$-almost every $i'$ and $\mu$-almost every $i$. %, and otherwise fluctuations of child relative fitness contribute to environmental change of variance. 
The change of variance is given by the Price equation: 
\begin{eqnarray} \label{eqn_firstlaw_priceequation}
    \var'(U') - \var(U) &=& \partial_\NS \var(U) + \partial_\EC(\var(U), \var'(U')) \nonumber \\
    &\ge& \var(U) \left( 1 + \var(U) \right) + \tfrac{1}{\E[U^3]} \E\!\left[ U^2 \left( \langle U' \rangle_w - U \right) \right] \E\!\left[ U^2 \left( \langle U' \rangle_w + U \right) \right]. \qquad \quad
\end{eqnarray}
%with combined lower bound provided by \eqref{ineq_firstlaw_intro} and \eqref{ineq_environmentalchangeofvariance_lowerbound_intro}.
%This bound is saturated exactly in the case of ``strongly stationary'' processes, those for which $U'$ is exactly equal to $U$, and otherwise fluctuations amongst children contribute to environmental change of variance.

In Section \ref{sect_multilevelprice}, we prove a general version of Price's multi-level equation \cite{price1972extension}, which he used to describe group selection. In Section \ref{sect_timevaryingpriceequation}, we prove a smooth Price equation, extending Price's continuous-time equation \cite{price1972extension}, describing change on smooth spaces. 

%In Section \ref{sect_extensions}, we prove extensions of the Price equation, generalizing Price \cite{price1972extension}. This includes the multi-level Price equation, which Price used to describe group selection; and a differential Price equation, describing changes in continuous time or on smooth spaces. % The continuous-time Price equation is equivalent to the replicator-mutator and Lotka-Volterra equations of ecological dynamical systems \cite{nowak2004evolutionary}. 

In Section \ref{sect_quantum}, we present quantum versions of the Price equation, extending to the case of non-commutative observables. We define a unique fitness observable, which is used for distinct left and right quantum Price equations. The degree of non-commutativity measures the quantumness of the process. We prove quantum versions of the Zeroth and First Laws. 

In Section \ref{sect_kgs}, we present a version of the Kerr-Godfrey-Smith equation for open processes, as well as open quantum processes. The presence of orphaned children adds a third term to the Price equation: the covariance against the proportion of orphaned children, or equivalently, the negative covariance against the proportion of parented children.

\subsubsection*{Part 2 (Selective Entropy)}

%In Parts 2 and 3, we analyze the selective and environmental entropy functionals, and we prove corresponding laws of natural selection. 

% and we prove the corresponding ``Second Law of Natural Selection''. 

% In Part 2, we analyze the selective entropy functional, defined as the Kullback-Leibler divergence of relative fitness, stated as \eqref{def_intro_selectiveentropy}. We prove the Second Law of Natural Selection \eqref{ineq_secondlaw}: selective entropy is monotone under the effect of selection, with a strict quantitative bound. Thus selective entropy is a selective form of thermodynamic entropy. 

% In Part 3, we analyze the environmental entropy functional, defined as the Kolmogorov-Sinai entropy of the environmental part of a process, stated as \eqref{def_intro_environmentalentropy}. We prove the Third Law of Natural Selection: ... % as relative fitness tends to zero, the environmental entropy approaches a constant.  ??

%In Part \ref{part_selectiveentropy}, w
We introduce the selective entropy to quantitatively measure the effects of selection in a process. In Section \ref{sect_entropyNS}, we define selective entropy as the Kullback-Leibler divergence (relative entropy) of the relative fitness function:
\begin{equation} \label{def_intro_selectiveentropy}
    S_\NS := \E[-U \log U] \le 0.
\end{equation}
Selective entropy is non-positive, with saturation exactly when $w$ is purely environmental. Selective entropy is the amount of information generated by selection across an evolutionary process, and can be thought of as the ``negentropy'' of Schr\"odinger \cite{schrodinger1944life}. The negated exponential $\exp(-S_\NS) \ge 1$ represents the amount of ``selective diversity'' in a population, in accordance with the contemporary literature on entropy and diversity \cite{leinster2021entropy}. That is, the more values that the relative fitness $U$ takes, the higher the diversity $\exp(-S_\NS)$ will be.

In Theorem \ref{thm_stronggibbsNS}, we prove the strong Gibbs bounds
\begin{equation} \label{ineq_stronggibbs_intro}
    -\log\big( 1 + \var(U) \big) \le S_\NS \le \log p_* \le 0,
\end{equation}
where $p_* := \mu(U>0)/N$ is the proportion of childbearing individuals. The non-trivial inequalities of \eqref{ineq_stronggibbs_intro} are saturated when $w$ is in selective equilibrium (i.e., $U = 0$ or $1/p_*$ almost surely). In which case, we have $-\log\big( 1 + \var(U) \big) = S_\NS = \log p_*$, or equivalently $\var(U) = \e^{-S_\NS}-1 = \frac{1}{p_*}$. This implies a strong version of the Zeroth Law (Corollary \ref{cor_strongzerothlaw}):
\begin{equation}
    \partial_\NS(U) = \var(U) \ge \e^{-S_\NS} - 1 \ge \frac{1}{p_*} - 1,
\end{equation}
with saturation in the selective equilibrium case. 

% We say that the process is \emph{{\quasienvironmental}} if it takes only one positive value of fitness (i.e., $U \in \{0,1/p_*\}$), and we show that {\quasienvironmental} processes are exactly the extreme case $-\log\big( 1 + \var(U) \big) = S_\NS = \log p_*$. 

%The selective-equilibrium processes are the extreme binary, ``life-or-death'' processes. All other processes involve variation of fitness across surviving populations, thus must quantitatively exhibit ``more selection''. Formally, this means that a non-equilibrium process must take a selective entropy in its fluctuation window (i.e., $-\log\big( 1 + \var(U) \big) < S_\NS < \log p_*$).

% In Section \ref{sect_NSNS}, we analyze the changes of the selective entropy $S_\NS$ and its derivative functionals $\partial_\NS S_\NS$ and $\partial_\EC S_\NS$. % We say that a process is ``in selective equilbrium'' when the relative-fitness function takes only one positive value (in which case $U = 0$ or $1/p_*$, where $p_*$ is the population proportion for which $U>0$). 

In Section \ref{sect_NSNS}, we prove the ``Second Law of Natural Selection'' (Theorem \ref{thm_secondlaw}), which states that selective entropy can never increase solely under the effect of selection:
\begin{equation} \label{ineq_secondlaw}
\partial_\NS S_\NS \le -\var(U) \log(1 + \var(U)) \le \var(U) S_\NS \le \left(\e^{-S_\NS} - 1 \right) S_\NS \le - \left( \frac{1}{p_*} - 1 \right) \log \frac{1}{p_*} \le 0.
\end{equation}
These inequalities all vanish when $w$ is purely environmental. The non-trivial inequalities are saturated exactly when $w$ is in selective equilibrium. Consequently, selective-equilibrium processes minimize selective effects, whereas non-equilibrium processes evolves selectively at a faster rate. Thus the effects of selection tend to compound exponentially over time. 

We prove a selective speed limit (Theorem \ref{thm_selectivespeedlimit}), providing a bound for how fast selection can compound. We also provide bounds on the selective acceleration (Theorem \ref{thm_upperboundforselectiveacceleration}). 

%is selective velocity, %$\partial_\NS S_\NS = -\var(U) \log(1 + \var(U))$, 
%whereas a non-equilibrium process must evolve selectively at a faster rate. %$\partial_\NS S_\NS < -\var(U) \log(1 + \var(U))$. 
%The selective equilibrium case is the absolutely ``minimal'' amount of selection. 

% For a non-equilibrium process, selection must increase at a rate faster than any corresponding equilibrium process. 

%, in which case
%\begin{equation}
%    \partial_\NS S_\NS = -\var(U) \log(1 + \var(U)) = \var(U) S_\NS = \left(\e^{-S_\NS} - 1\right) S_\NS = \left( \frac{1}{p_*} - 1 \right) \log p_* \le 0,
%\end{equation}
%which is $<0$ when $p_* < 1$, and $=0$ when $p_* = 1$. 

%In general, $\partial_\NS S_\NS$ measures ``selective complexity'' of a process. The Second Law \eqref{ineq_secondlaw} represents a selective feedback loop, where the presence of selective effects is amplified. For a non-equilibrium process, selection must increase at a rate faster than any corresponding equilibrium process. 

% , and equilibrium solutions form a partial-ordering based on the functional $\partial_\NS S_\NS$. Namely, the 

%$\partial_\NS S_\NS = -\var(U) \log(1 + \var(U)) = \var(U) S_\NS = (\e^{-S_\NS} - 1) S_\NS$. 
% In the purely environmental case (i.e., Markov chains), all inequalities vanish. In the selective-equilibrium case (i.e., minimal non-trivial selection), all inequalities but the last are saturated, and we have
%\begin{equation}
%    \partial_\NS S_\NS = -\var(U) \log(1 + \var(U)) = \var(U) S_\NS = (\e^{-S_\NS} - 1) S_\NS < 0.
%\end{equation}

In Section \ref{sect_environmentalchangeofselectiveentropy}, we bound the environmental change of selective entropy (Theorem \ref{thm_strongupperboundECNS}):
\begin{equation} \label{ineq_entropyNSEC_upperbound}
    \partial_\EC(S_\NS,S_\NS') \le \E[U^2] + \log \E[U^3].
\end{equation}
This is saturated in the strongly stationary case ($U' = U$ jointly a.s.).  

Thus using the Price equation and combining \eqref{ineq_secondlaw} and \eqref{ineq_entropyNSEC_upperbound}, we bound total change in selective entropy solely in terms of elementary functionals of the original process: 
\begin{eqnarray} \label{ineq_controlselectivechange}
    S_\NS' - S_\NS % \Delta(S_\NS,S_\NS') 
    &=& \partial_\NS S_\NS + \partial_\EC(S_\NS,S_\NS') \nonumber \\
    &\le& -\var(U) \log(1 + \var(U)) + \log \E[U^2] + \log \E[U^3].
%    &=& \log \E[U^3] - (1+\var(U)) \log(1+\var(U)) + 1+\var(U) + \log(1+\var(U)), 
\end{eqnarray}

In Section \ref{sect_multilevel_selentropy}, we state the multi-level change of selective entropy, and prove a corresponding Multi-Level Second Law (Theorem \ref{thm_multilevel_secondlaw}).

In Section \ref{sect_quantum_selentropy}, we define selective entropy using the spectral theorem and the relative-fitness operator. 
%as $S_\NS := \E_\mu[-U \log U] = \frac{1}{N} \Tr\!\left((-U \log U) M\right)$, using the spectral theorem, where $M$ is a self-adjoint operator. 
We prove a Quantum Second Law of Natural Selection, namely that the quantum selective change of quantum selective entropy is non-positive. 

%Weak equality: $S_\NS = \frac{1}{N} \Tr\!\left((-U \log U) M\right) = ...$ 

% \frac{1}{N} \Tr(-UM \log(UM)) + \frac{1}{N} \Tr(U (M \log(UM) - (\log U) M)$. First term: $\frac{1}{N} \Tr(-UM \log(UM)) \le - \frac{1}{N} \Tr(UM) \log \Tr(UM) = 0$.}

% Weak inequality: note that the expectation $\Tr_M(X) := \E_\mu[X] := \frac{1}{N} \Tr(XM)$ is itself a trace functional on $\A$.{Proof: it suffices to show that $\E_\mu$ vanishes on commutators. Indeed, $\E_\mu[[X_1,X_2]] = \frac{1}{N} \Tr(X_1 X_2 M - X_2 X_1 M)$} By Jensen's inequality for traces, $S_\NS = \Tr_M(-U \log U) \le -\Tr_M(U) \log \Tr_M(U) = 0$. Strong inequality: define the projection operator $\pi_{U\ne 0}$ onto the open subspace defined by $U > 0$.{i.e., $\pi_{U\ne 0}$ is projection onto $\{ h : Uh > 0\} \cup \{0\}$.} Strong inequality: Write $p_* := \frac{1}{N} \Tr(\pi_{U\ne 0} M)$. Write $\Tr_*(X) := \E_*[X] := \frac{1}{p_*} \E[X \pi_{U\ne 0} M]$. 
%defined on the observables acting on the subspace $\tilde \H_M := \{ M h : h \in \H \}$. 
%left ideal $\A_M := \{XM : X \in \A\}$. selective equilibrium?} 

\subsubsection*{Part 3 (Environmental Entropy)}

%In Part \ref{part_environmentalentropy}, w
We introduce the environmental entropy to characterize the degree of environmental change in a process. 
In Section \ref{sect_environmentalentropy}, we prove basic properties about environmental entropy. Write $U_{A,B}(i) := \frac{1}{N} 1_A(i) w_i(B)$ for each $i$, and and write $\bar U_{A,B} := \E[U_{A,B}] = \frac{1}{N'} \int_A w_i(B) \mu(\d i)$. We define environmental entropy as follows:
\begin{equation} \label{def_intro_entropy_EC}
    S_\EC := \sup_{\A,\B} \sum_{A\in\A,B\in\B} \left( - \E[U_{A,B}] \log \E[U_{A,B}] \right) \ge 0.
%    S_\EC := S_\KS := \sup_{\A,\B} \sum_{A\in\A,B\in\B} \left( - \E[U_{A,B}] \log \E[\frac{U_{A,B}}{U}] \right) \ge 0.
%    OLD: S_\EC := \sup \sum \E\!\left[-U_{A,B} \log \frac{U_{A,B}}{U}\right] \ge 0, 
\end{equation}
where the supremum is over all countable, measurable partitions of $I$ and $I'$, and the sum is over partition sets. This is a one-step version of Kolmogorov-Sinai entropy. The classical KS entropy can be recovered by iterating a process indefinitely, and taking the supremum across all partition refinements over all iterates (Definition \ref{def_KSentropy}). The exponential $\exp(S_\EC) \ge 1$ represents the amount of ``environmental diversity'' in a population, i.e., the more distinct values of $U_{A,B}$ there, the greater $\exp(S_\EC)$ is. 

%{Formally, consider a process $w$ on a common space $I'=I$. For each countable, measurable partition $\A$ of $\I$, define the $T$th partition refinement $w^{(-T)} \A := \prod_{t=0}^T w^{-t} \A = \{ A : w^t_i(A) > 0 \}$. Then classical Kolomogorov-Sinai entropy of the process equals:
%\begin{equation}
%    S^{\operatorname{classical}}_\KS := \sup_\A \left( \lim_{T\to\oo} \frac{1}{T} \sum_{A \in w^{(-T)}\A} \left( - \E[1_A U] \log \E[1_A U] \right) \right). 
%\end{equation}} 

We prove a general version of Sinai's theorem (Theorem \ref{thm_sinai}) showing that the supremum in \eqref{def_intro_entropy_EC} must be realized at a ``generating joint partition'' $(\A_*,\B_*)$. This allows us to define the change of environmental entropy by evaluating the change at joint partition sets. 

%, satisfying the maximality property (i.e., the $\sigma$-algebra generated by sets of the form $A \cap w^{-1} B$ is equal to the full $\sigma$-algebra $\I$). 

% This is a ``one-step'' version of classical Kolmogorov-Sinai entropy, which can be recovered by iterating a process indefinitely, and refining the partitions across those iterates.
%{i.e., if $I'=I$, then we can define the $n$th iterate $w^{(n)} := w \circ \cdots \circ w$, then ..... define $S_\KS^{\operatorname{classical}} = $....}

%The selective entropy represents ``biological information'', and environmental entropy represents ''physical information''. 

The total entropy is the sum of the selective and environmental entropies, and can be positive or negative depending on the contribution of physical and biological forces:
\begin{eqnarray}
    S_\tot := S_\NS + S_\EC 
    &=& \E[-U \log U] + \sup_{\A,\B} \sum_{A\in\A,B\in\B} \left( - \E[U_{A,B}] \log \E[U_{A,B}] \right). %\nonumber \\
    %&=& \sup_{\A,\B} \sum \E\!\left[-U_{A,B} \log \left( \E[U_{A,B}] U \right) \right].
\end{eqnarray}

% The physical information decomposes into efficient and consistent types of information.

%The Kolmogorov-Sinai entropy is trivially non-negative (since $0 \le \E[U_{A,B}] \le 1$ and $-x \log x \ge 0$ for $0 \le x \le 1$). 

In Section \ref{sect_dispersionmixing}, we decompose environmental entropy into dispersive and mixing entropy functionals. We define the \emph{dispersive entropy} as
\begin{equation} \label{def_intro_entropy_DS}
    S_\dis := \sup_{\A,\B} \sum_{A\in\A,B\in\B} \E\!\left[-U_{A,B} \log \frac{U_{A,B}}{U}\right] \ge 0,
\end{equation}
which measures the ``inefficiency'', ``splitting'', ``stretching'', or ``clonal replication'' of a system. %By Jensen's inequality, we have the simple identity $S_\dis \le S_\EC$.  
We introduce \emph{mixing entropy} as
\begin{equation} \label{def_intro_entropy_MX}
    S_\mix := \sup_{\A,\B} \sum_{A\in\A,B\in\B} \bar U_{A,B} \E\!\left[\frac{U_{A,B}}{\bar U_{A,B} U} \log \frac{U_{A,B}}{\bar U_{A,B} U} \right] \ge 0
%   S_\SX := \sup_{\A,\B} \sum_{A\in\A,B\in\B} \bar U_{A,B} \E\!\left[-\frac{U_{A,B}}{\bar U_{A,B} U} \log \frac{U_{A,B}}{\bar U_{A,B} U} \right] \le 0
\end{equation}
where $\bar U_{A,B} := \E[U_{A,B}]$ is the averaged local relative fitness. Mixing entropy measures the ``inconsistency'', ``combining'', ``folding'', or  ``sexual reproduction'' of a system. The quantities $\exp(S_\dis) \ge 1$ and $\exp(S_\mix) \ge 1$ represent the dispersive and mixing diversities of a population, respectively. The higher $\exp(S_\dis)$ and $\exp(S_\mix)$ are, the more ways the populations disperse and mix, respectively. 

We extend Sinai's theorem (Theorem \ref{thm_sinaithm_dispersionmixing}) to show that the dispersive and mixing entropies are maximized exactly at a generating joint partition. Consequently, environmental entropy decomposes as the sum of dispersion and mixing entropies: 
\begin{equation} \label{eqn_keyidentity_KS}
    S_\EC = S_\dis + S_\mix. 
\end{equation}
%Thus environmental (Kolmogorov-Sinai) entropy can be decomposed as the sum of ``dispersive'' and ``mixing'' entropies, which fully characterizes environmental processes (i.e., Markov chains). 

In Section \ref{sect_bernoulliexamples}, we present examples based on classical Bernoulli random variables. The dispersive Bernoulli process sends one input to two outputs, and is inefficient and consistent ($S_\dis > 0$, $S_\mix = 0$). The mixing Bernoulli process sends two inputs to one output, and is efficient and inconsistent ($S_\dis=0$, $S_\mix>0$).

% In Appendixes \ref{app_sinai} and \ref{app_sinaithm_dispersionmixing}, we present a generalized form of Sinai's theorem, and show that the three entropy functionals $S_\EC$, $S_\dis$, and $S_\mix$ can all be written using a uniform ``generating joint partition'', i.e.,  a pair of countable, measurable partitions $(\A_*,\B_*)$ which optimize the extrema in \eqref{def_intro_entropy_EC}, \eqref{def_intro_entropy_DS}, and \eqref{def_intro_entropy_MX}.{A generating joint partition $(\A_*, \B_*)$ satisfies the condition that $\I$ is the smallest $\sigma$-algebra containing sets $A \cap w^{-1} B$.} We also prove the general version of identity \eqref{eqn_keyidentity_KS}.

In Section \ref{sect_ecri}, we show that the dispersive and mixing entropies characterize obstructions to invertibility. 
% The Efficiency and Consistency Theorems (Theorems \ref{thm_efficiency} and \ref{thm_consistency}) state that a purely environmental process is left- (resp. right-) invertible if and only if it is purely mixing ($S_\dis=0)$, resp. purely dispersive ($S_\mix=0$). 
The Efficiency Theorem (Theorem \ref{thm_efficiency}) shows that a purely environmental process is left-invertible if and only if it is purely mixing ($S_\dis=0$). The Consistency Theorem (Theorem \ref{thm_consistency}) shows that it is right-invertible if and only if it is purely dispersive $(S_\mix=0)$. This implies the Reversibility Theorem (Theorem \ref{thm_reversibility}): a purely environmental process is invertible if and only if environmental entropy vanishes $(S_\EC=S_\dis+S_\mix=0$). Equivalently, the Irreversibility Theorem (Theorem \ref{thm_irreversibility}) shows a purely environmental process is not invertible if and only if it exhibits dispersive or mixing effects (or both). This implies a weak form of Dollo's law of irreversibility (Corollary \ref{cor_dollo}): a full process is invertible if and only if it is purely childbearing ($p_*=1$) and environmentally reversible ($S_\EC=0$). %In particular, an evolutionary process cannot reconstruct a past state if there are any 

In Section \ref{sect_environmentalequilibrium}, we introduce environmental equilibrium and prove bounds on dispersion and mixing entropies. We also present examples of equilibrium and non-equilibrium processes.

In Section \ref{sect_thirdlaw}, we analyze the change of the environmental entropy. The ``Weak Third Law of Natural Selection'' (Theorem \ref{thm_weakthirdlaw}) shows that environmental-equilibrium processes are characterized by vanishing selective change of environmental, dispersion, and mixing entropies. %($\partial_\NS S_\EC = 0 = \partial_\NS S_\dis = \partial_\NS S_\mix$). 
The ``Strong Third Law'' (Theorem \ref{thm_thirdlaw}) provides quantitative bounds on selective changes for non-equilibrium processes, and these bounds collapse in the equilibrium case. 

%states that the selective change of environmental entropy can be written as a sum of two terms: a non-negative variance term, representing local contributions to environmental entropy as a result of selection (``the selective exhaust''), and a covariance term representing non-local fluctuations of environmental entropy. The Third Law allows for the entropy functionals to fluctuate within a window defined by certain elementary functions of certain population statistics, and these windows collapse into single values in the case of environmental equilibrium. 

In Section \ref{sect_environmentalchangeofenvironmentalentropy}, we state the Price equation for the environmental entropy. In Section \ref{sect_multilevel_enventropy}, we state the multilevel Price equation for environmental entropy. In Section \ref{sect_quantum_enventropy}, we define the quantum environmental entropy. %, though we do not explore a quantum version of Sinai's theorem in this article. %, and we conjecture that a quantum version of Sinai's holds. We do not explore quantum 
In Section \ref{sect_conclusion}, we conclude the article. %, and discuss possibilities for future work. 
%Restate your topic and why it is important,
%Restate your thesis/claim,
%Address opposing viewpoints and explain why readers should align with your position,
%Call for action or overview future research possibilities.

We hope that this work adds clarity to the mathematical biology and physics literatures, and provides a formal grounding for a unified theory of evolution and thermodynamics in the future. McCullagh reminds us that ``mathematics knows nothing about anything except mathematics, so mathematics must be instructed in the facts of rural life'' \cite[p.~1304]{mccullagh2002statistical}. We call on other scientists to use this abstract framework in the spirit of Price and Hamilton, gleaning new insights to altruistically help populations of the world.

\subsection*{Acknowledgements} T.L. gives particular thanks to Elliot Aguilar, who first introduced him to the Price equation and encouraged him to put it on a more general foundation.

T.L. also thanks Erin Beckman, Michael Betancourt, Tyler Bryson, Miguel Carri\'on \`Alvarez, David Cesarini, Dorian Goldman, Brendan Fong, George Hagstrom, Bryan  , Joseph Hirsh, Taylor Kessinger, Angela Linneman, Kellen Olszewski, Benjamin Pittman-Polletta, Javier Rodr\'{i}guez Laguna, Lisa Rogers, Leila Vaez-Azizi, Brad Weir, and Janek Wehr for helpful discussions on the Price equation.

%The author thanks Adam Brandenberger, Samantha Kappagoda, Bud Mishra, David Mordecai, Charles Newman, Daniel Stein, and Lai-Sang Young for mentorship while at the Courant Institute (NYU). 

% The author completed this work during free time while employed at Splunk, and is now employed at Google. 
T.L. was supported by NSF PIRE Grant No. OISE-07-30136 while at the Courant Institute (NYU) in 2010-2013, and is grateful to Adam Brandenberger, Bruce Bueno de Mesquita, Samantha Kappagoda, Bud Mishra, David Mordecai, Charles Newman, Alastair Smith, Daniel Stein, and Lai-Sang Young for mentorship during those years. 

T.L. finalized the work during free time while at Splunk, and is now at Google. T.L. certifies that there is no actual or potential conflict of interest in relation to this article.

% T.L. certifies that he has no affiliations with or involvement in any organization or entity with any financial interest or non-financial interest in the subject matter or materials discussed in this manuscript. T.L. has no financial or proprietary interests in any material discussed in this article.

\section{The General Price Equation and Fisher's Fundamental Theorem} \label{sect_price}

In this section, we introduce the evolutionary process framework, and prove general versions of Price's equation and Fisher's fundamental theorem. We model population states by (finite) measures, and we model processes by transition mappings between states. %Regular conditional probability is the case of a mapping from a probability measure back to itself. Dynamical systems is the case of function mappings pushing forward measures. 

Formally, let $(I,\I)$ be a measurable space (a set $I$ and a $\sigma$-algebra $\I$), representing a ``type space'' for describing a population. A point $i \in I$ represents a discrete, individual ``type'', and a measurable subset $A \in \I$ represents a more complex type or group of types. A population state (or ``data'') is modeled by a measure $\mu$ on $I$, where the value $\mu(A)$ represents the number of individuals of type $A$. The total population size is given by $N := \mu(I)$. A population variable (or observable) is a measurable function $X : I \to \R$, and the integral $\mu[X] := \int_I X(i) \, \mu(\d i)$ represents the aggregate sum of the variable across the population. The average value is given by normalizing by population size: $\bar X := \E[X] := \tfrac{1}{N} \mu[X]$. 

Let $\mu'$ be another population state of interest, defined on a (possibly different) measurable space $(I',\I')$. This framework supports both the cases of distinct and overlapping type spaces. If there is overlap, we write $I_\cap := I \cap I'$. Most authors including Price consider the case $I=I'$, but we separate the initial and final spaces for clarity and generality. Let $N' := \mu'(I')$ be the total population size, and for any measurable function $Y$ on $I'$, we define $\mu'[Y] := \int_I Y(i') \, \mu'(\d i')$ and $\bar Y := \E'[Y] := \tfrac{1}{N'} \mu'[Y]$. %Price considered the case $I' = I$, but we separate the initial and final spaces for clarity.

There are multiple ways to compare the population states $\mu$ and $\mu'$, even if there is no overlap of types. Fisher \cite{fisher1930genetical} recognized that the key quantity is the \emph{selective coefficient}, defined as the ratio of population sizes: $\bar W := \frac{N'}{N}$. 

For any measurable $X$ and $Y$, define the average change as the difference of averages: 
\begin{equation} \label{def_avgchange}
\Delta(\bar X, \bar Y) := \bar Y - \bar X = \E'[Y] - \E[X].
\end{equation}
%and the aggregate change is given by $\Delta(\bar X) := \mu'[X] - \mu'[X] = N' \E'[X] - N \E[X]$. 
When $Y = X$, we write $\Delta(\bar X) := \Delta(\bar X,\bar X)$. 

We define an \emph{evolutionary process} as a complete accounting of the second population state in terms of the first, and we write $w : \mu \mapsto \mu'$ for this transition mapping. We formalize this as a disintegration  \cite{leao2004regular,bogachev2007measure, lagatta2013continuous}, applied to the case of finite measures. 

\begin{defn}[Evolutionary Process] \label{def_process}
We say that a measure-valued function $w : i \mapsto w_i$ is a (regular) \emph{evolutionary process} if is a disintegration mapping $\mu$ to $\mu'$, i.e.,
\begin{enumerate}
%    \item The mapping is defined $I \times \I' \to \R$. 
    %\item For each $i \in I$, $w_i$ is a finite measure on $I'$.
    \item For all $B \in \I'$, $i \mapsto w_i(B)$ is a measurable function of $i \in I$.
    %\item The measure $w_i$ is supported on the child-set of $i$. i.e., for $\mu$-almost every $i$, $w_i(\chi(i))$. $
    \item For all measurable $B \subseteq I'$, the disintegration equation holds:
    %For all $\mu$-integrable functions $X : I \to \R$, the disintegration equation holds for any $B \in \I'$:
        \begin{equation} \label{eqn_disintegration}
            \mu'(B) = \int_I w_i(B) \, \mu(\d i).
        \end{equation}
\end{enumerate}
\end{defn}

The disintegration equation \eqref{eqn_disintegration} is equivalent to the following, for any integrable $Y$:
\begin{equation} \label{eqn_disintegration_observable}
    \int_{I'} Y(i') \mu'(\d i') = \int_I \int_{I'} Y(i') w_i(\d i') \mu(\d i). 
\end{equation}

% Later (Theorem \ref{thm_childsetexistence}), we will also show that child-set mappings are well-defined. i.e., there exists an essentially-unique mapping $\chi : \I \to \I'$ such that $w_i(\chi(\{i\})) = w_i(I') =: W(i)$ for $\mu$-almost every $i$. %This extends the definition of a disintegration \cite[Definition 1.1]{lagatta2013continuous}.

%Formally, an evolutionary process is defined as a measure-valued function $i \mapsto w_i$, which varies measurably in the input, and which stratifies the second population in terms of the first: 
%\begin{equation} \label{eqn_disintegration}
%\mu'(B') = \int_I w_i(B') \mu(\d i)
%\end{equation}
%for every measurable $B' \subseteq I'$. ``Varying measurably'' means that, for every measurable $B' \subseteq I'$, the real-valued function $i \mapsto w_i(B')$ is Borel-measurable. 

It is convenient to treat $\mu$ as the ``parent'' population and $\mu'$ as the ``child'' population, with $w_i$ representing the distribution of children for parent $i$.
% The measure $w_i$ represents the distribution of children for parent $i$. 
We define the \emph{fitness function} to be the total number of children:
\begin{equation}
    W(i) := w_i(I').
\end{equation}
The fitness function is itself a measurable observable, and the average fitness equals the selective coefficient, i.e., the relative population sizes: $\E[W] = \bar W = \tfrac{N'}{N}$. To see this, compute $\E[W] = \tfrac{1}{N} \int_I W(i) \mu(\d i) = \tfrac{1}{N} \int_I w_i(I') \mu(\d i) = \tfrac{1}{N} \mu'(I') = \tfrac{N'}{N}$. We write $U(i) := W(i) / \bar W$ for the relative fitness function.

\begin{rem}
The definition of an evolutionary process is a purely phenomological assumption. We make no causal, correlative or dynamical assumptions of the populations, nor do we make any assumptions on evolution away from the states $\mu$ and $\mu'$. We merely begin with the assumption that there is \emph{some} accounting function $w$, and we examine the deductive consequences of this assumption. This can be helpful in empirical analysis, to validate or falsify the assumption of a process $w$ characterizing the relationship between two states $\mu$ and $\mu'$ (e.g., a genealogy or other causal relationship), but the possibility of other processes should not be overlooked. 
% Scientific analysis works by making hypotheses about processes connecting populations, then gradually falsifying those hypotheses which are greatly at odds with empirical data. 

%, though the framework applies to any linear mapping of measures. 
%We use the parent-child language, but the framework is more general and represents any comparison of populations via some mapping $w : \mu \mapsto \mu'$.
%In a simple Popperian analysis, the scientific process works like ``measure data $\mu$ and $\mu'$, make hypotheses about processes $w$, compute functionals, and falsify hypothesis 
\end{rem}

We say that two processes $w : \mu \mapsto \mu'$ and $w' : \mu' \mapsto \mu''$ are \emph{composable} when they share the same intermediate process. In that case, we write
\begin{equation}
    (w' \circ w)_i(C) := \int_{I''} \int_{I'} w'_{i'}(C) w_i(\d i'),
\end{equation}
for any measurable $C$ and any $i \in I$.

We say that a process $w : \mu \mapsto \mu'$ is \emph{generally reversible} if it is measurably invertible, i.e., there exists another evolutionary process $w^{-1} : \mu' \mapsto \mu$ such that the compositions are the identity processes (i.e., $w^{-1} \circ w = 1_\mu$ and $w \circ w^{-1} = 1_{\mu'}$). We characterize the class of environmentally reversible processes in Section \ref{sect_ecri} via vanishing environmental entropy functionals.

\subsection{Natural Selection}

To state and prove the Price equation, we decompose the process $w$ into selective and environmental components. Price \cite{price1970selection} (echoing Fisher \cite{fisher1930genetical} before him) recognized that selective change can be described as population covariance against relative fitness $U := W/\bar W$. 

Formally, let $\cov$ denote the population covariance for the probability measure $\mu/N$. i.e., if $X_1$ and $X_2$ are two measurable functions on $I$, then:
\begin{eqnarray} 
\cov(X_1, X_2) &:=& \E[(X_1-\bar X_1)(X_2-\bar X_2)] \label{def_cov} \\
&=& \frac{1}{N} \int \left( X_1(i) - \frac{1}{N} \int X(i_1) \mu(\d i_1) \right) \left( X_2(i) - \frac{1}{N} \int X_2(i_2) \mu(\d i_2) \right) \mu(\d i). \nonumber
\end{eqnarray}
%so long as the expectation on the right is well-defined.

We say that an evolutionary process $w$ is ``finite-mean'' if the fitness is finite mean ($\E[W] = \bar W < \oo$, and ``finite-variance'' if the fitness is finite variance ($\var(W) := \cov(W,W) < \oo$). Since the populations are finite, finite-variance implies finite-mean. The class of finite-variance processes is the class for which the selective change is well defined:
%we can define $\cov(X,W)$. In particular, since the populations are finite, and so we can define the selective change of $X$: 
%, hence we can define covariance against relative fitness: $\cov(X,U) := \cov(X, W/\bar W) < \oo$
%This last term, the covariance of $X$ against the relative fitness $U$, is the key ingredient in Price's work, encoding this fact as \emph{the} change due to natural selection: 
\begin{equation} \label{def_naturalselection}
%\begin{array}{rcl}
\partial_{\NS}(X) := \cov\!\big( X, U \big) := \E\!\big[ \big(X-\E[X]\big) \big(U - 1\big) \big] = \E\!\big[ X \big( U - 1 \big) \big].
%\end{array}
\end{equation}
where the simplification follows from elementary algebra. To see this, compute $\cov(X, \tfrac{W}{\bar W}) = \E\!\big[ \big(X-\E[X]\big) \big(\tfrac{W}{\bar W} - 1\big) \big] = \E[X\tfrac{W}{\bar W}] - \E[X] \E[1] - \E[X] \E[\tfrac{W}{\bar W}] + \E[X] \E[1] = \E[X\tfrac{W}{\bar W}] - \E[X]$, since $\E[\tfrac{W}{\bar W}] = 1$. Price identified this as one half of the total change, formulated in the Price equation \eqref{eqn_price}. 

% When the context is clear, we omit the subscript $w$ and write $\partial_\NS(X) := \partial_{\NS}(X)$. 

Fisher's form of his fundamental theorem follows as a trivial consequence of this definition:
\begin{equation} \label{eqn_fisherthmclassical}
    \partial_\NS(U) = \cov(U,U) = \var(U).
\end{equation}

%While Price \cite{price1970selection} had the formalism to define \eqref{eqn_fisherthmclassical}, he did not provide his interpretion of Fisher's theorem until \cite{price1972fisher}. We state the general version of Fisher's theorem in Section \ref{sect_fisherthm}.

%\begin{rem} There is a simple geometric interpretation of selective change in terms of Hilbert spaces. Suppose that $w$ is a finite-variance process, and let $\H_\mu := \{ X : \var(X) < \oo \}$ be the Hilbert space of finite-variance random variables, equipped with the covariance inner product $\cov(X,Y)$. Then the projection of $X$ onto $U$ equals:
%\begin{equation}
%    \proj_{U,\mu}(X) = \frac{\cov(X,U)}{\var(U)} U = \frac{\partial_\NS(X)}{\var(U)} U.
%\end{equation}
%selective change $\partial_\NS(X) = \cov(X,U)$ is the coefficient of the orthogonal projection of $X$ onto $U$. 
%Therefore, the inner product $\cov(X,W) < \oo$ is finite. The orthogonal projection of any $X$ onto the line generated by $W$ equals:
%\begin{equation}
%\proj_{W,\mu}(X) = \cov(X,W) \frac{W}{\sqrt{\var(W)}} = \partial_{\NS}(X) \frac{\bar W^2}{\sqrt{\var(W)}} U.
%\end{equation}
% i.e., the projection is proportional to relative fitness $U$, with length given by $\partial_{\NS}(X) \bar W^2 / \sqrt{\var(W)}$. 
%\end{rem} 

The selective change encodes the correlative relationship between a trait and fitness. No causal claim is made: high values of the trait could cause high fitness, or high fitness could cause high values of the trait, or some third factor could be a cause of high values of both. The causal network between various traits and fitness is complex, and the effects of these pathways is a major subject in modern biology. See \cite{gregory2009understanding} for a recent introduction to natural selection. Price's views on selection can be found in \cite{price1995nature}.

\begin{rem}[Classical Regression Statistics]
Natural selection represents an ``internal'' covariance, where positive correlations between observed traits $X$ and relative fitness $U = \frac{W}{\bar W}$ are ``recorded'' by the population. Fisher \cite{fisher1930genetical} abstracted away the recording details, and summarized the relationships with population statistics. Write the standard deviations $\sigma_X := \sqrt{\var(X)}$ and $\sigma_W := \sqrt{\var(W)}$, regression coefficients $\beta_{X,W} := \cov(X,W)/\var(W)$, and correlation coefficients $\rho_{X,W} := \cov(X,W) / (\sigma_X \sigma_W)$. The selective change equals: 
\begin{equation} \label{def_popstats} %
\partial_{\NS}(X) = \frac{\cov(X,W)}{\bar W} = \frac{\beta_{X,W} \var(W)}{\bar W} = \frac{\rho_{X,W} \sigma_X \sigma_W}{\bar W},
\end{equation}
%The coefficients $\beta_{X,W}$ and $\rho_{X,W}$ are the key indicators of selective change, and these metrics should be tracked across generations. 
\end{rem}

%It is helpful to define the \emph{selective degrees of freedom} of the process on the population. The degrees of freedom equals the rank of the covariance operator $\cov$, restricted to traits which correlate with fitness. Formally, let $W^\perp := \{ X : \cov(X,W) > 0 \}$ denote the space of traits uncorrelated with fitness. The selective degrees of freedom $D := \dim W^\perp$ equals the dimension of this space. Clearly, 
%\begin{equation}
%\mbox{$\var(W) = 0$ if and only if $D = 0$.}
%\end{equation}

\subsection{Environmental Change}

Price \cite{price1970selection} introduced the environmental change component to fully capture the effects of non-selective change. We present a formalism for working with environmental change, based on conditional expectations from probability. 

First, we introduce the ``local average'' operator, defined by integrating over the child population, and normalizing by fitness:
\begin{equation} \label{def_localavg}
\langle Y \rangle_w(i) := \frac{1}{W(i)} \int Y(i') \, w_i(\d i').
\end{equation}
For each $Y$, $i \mapsto \langle Y \rangle_w(i)$ is a measurable observable. For each $i$, $Y \mapsto \langle Y \rangle_w(i)$ is a probability expectation operator. When the context is clear, we drop the subscript $w$ and write $\langle Y \rangle := \langle Y \rangle_w$. The fundamental relation is the tower property,
\begin{equation} \label{eqn_towerproperty}
    \E'[Y] = \E[U \langle Y \rangle_w],
\end{equation}
adapting the tower property of conditional expectations to variable-size measures. To prove \eqref{eqn_towerproperty}, compute 
\begin{equation}
    \E'[Y] = \frac{1}{N'} \int_{I'} Y(i') \mu'(\d i') = \frac{1}{N'} \int_I \int_{I'} Y(i') w_i(\d i') \mu(\d i) = \frac{1}{N} \int_I \frac{W(i)}{\bar W} \langle Y \rangle_w(i) \mu(\d i) = \E[U \langle Y \rangle_w].
\end{equation} This allows us to compare $\E'[Y] = \E[U \langle Y\rangle_w]$ and $\E[X]$ on a common measure space $(I,\mu)$. 

% We sometimes write the ``fitness-weighted conditional expectation'' $\E'_w[Y](i) := U(i) \langle Y \rangle_w(i)$ so that $\E[\E'_w[Y]] = \E'[Y]$, with the caveat that $\E'_w[1](i) = U(i)$.{i.e., both the local average and the fitness-weighted conditional expectation generalize the classical conditional expectation: the local average has mass $1$, whereas the fitness-weighted conditional expectation satisfies the tower property. Only in the case that $w$ is a Markov chain ($U = 1$) do they coincide. We call that the ``purely environmental'' case and describe it further in Section \ref{sect_pure}.} 

We use this to define the ``local change'' operator, by subtracting the original value of $X$ from the local average:
\begin{equation} \label{def_localchange}
\Delta_w(X,Y)(i) := \langle Y \rangle_w(i) - X(i) = \frac{1}{W(i)} \int_I Y(i') \, w_i(\d i') - X(i).
\end{equation}
The local change $\Delta_w(X,Y)(i)$ is a function of $i$ (and depends on $w$), whereas the average change $\Delta(\bar X, \bar Y)$ is a single value (and does not depend on $w$). The local change measures the difference of average value $\langle Y \rangle_w(i)$ from the original value $X(i)$. If $Y = X$, we write $\Delta_w(X) := \Delta_w(X,X)$.
%That is, $\Delta_w(X)(i)$ is the average value of $X$ among descendants of $i$, minus the original value $X(i)$. We also introduce the notation $\langle X \rangle_w(i) := \frac{1}{W(i)} \int X(i') w_i(\d i')$, so that $\Delta_w(X) = \langle X \rangle_w - X$. 

We define the environmental change by weighting the local change by relative fitness, and averaging across the parent population:
\begin{equation} \label{def_environmentalchange}
\partial_{\EC}(X,Y) := \E\big[\Delta_w(X,Y) U \big] = \E[(\langle Y \rangle_w - X) U].
\end{equation}
%The value $\Delta_w(X)(i) \tfrac{W(i)}{\bar W}$ quantifies the environmental impact of a particular individual $i$, and the environmental change describes average environmental impact. 
The environmental change is the second half of the Price equation \eqref{eqn_price}. 
%When the context is clear, we drop the subscript and write $\partial_\EC(X,Y) := \partial_{\EC}(X,Y)$.

%\begin{rem}
%Before stating and proving the Price equation, we introduce some useful formalism based on probability theory. Define the ``process conditional expectation'' (PCE) by 
%\begin{equation} \label{def_LCE}
%\E'_w[X|i] := U(i) \langle X \rangle_w(i) = \frac{1}{\bar W} \int X(i') w_i(\d i'). 
%\end{equation}
%That is, the PCE is the local change $\langle X \rangle_w$ weighted by the relative fitness $U$. The key fact is that the PCE is an unbiased estimator of the future average value:
%\begin{equation}
%\E[\langle X\rangle_w U] = \E[\E'_w[X|\cdot]] = \frac{1}{N} \int \frac{1}{\bar W} \int X(i') w_i(\d i') \mu(\d i) = \E'[X].
%\end{equation}
%This generalizes the ``tower property'' of conditional expectations. 
%\end{rem}

Identity \eqref{eqn_towerproperty} lets us rewrite the environmental change as follows:
\begin{equation}
    \partial_\EC(X,Y) = \E[U \langle Y \rangle_w - UX] = \E'[Y] - \E[UX]. 
\end{equation}

\subsection{Price Equation}

The Price equation follows as an immediate consequence. This extends the discrete Price equation to the case of general finite measures (with no topological constraints), as well as separate functions $X$ and $Y$. 

\begin{thm}[General Price Equation] \label{thm_price}
Let $w$ be a finite-variance process. If $X$ and $Y$ are measurable functions on $I$ and $I'$, respectively, then the Price equation holds:
\begin{equation} \label{eqn_price}
\Delta(\bar X, \bar Y) = \partial_{\NS}(X) + \partial_{\EC}(X,Y) = \cov(X,U) + \E[ \Delta_w(X,Y) U].
\end{equation}
\end{thm}
\begin{proof}
The proof is similar to \eqref{eqn_discretepriceproof}. Using the definition \eqref{def_naturalselection} and the tower property \eqref{eqn_towerproperty}, we have:
\begin{eqnarray}
    \Delta(\bar X, \bar Y) &=& - \E[X] + \E'[Y] = \left( \E[UX] - \E[X] \right) + \left( \E[U \langle Y \rangle_w] - \E[UX] \right) \nonumber \\
    &=& \cov(X,U) + \E[\Delta_w(X,Y) U].
\end{eqnarray}
\end{proof}

%We state the multi-level and differential Price equations in Section \ref{sect_extensions}, and the quantum version in Section \ref{sect_quantum}.

%specific forms of the Price equation in Section \ref{sect_extensions}, including a multi-level version, a discrete-time version, a differential version, and an abstract version. We state the more general quantum Price equation in Section \ref{sect_quantum}.

It can be useful to write various aggregate forms of the Price equation, rather than averaged. We state this as the following corollary. The aggregate Price equation includes similar terms to \eqref{eqn_price} with relative fitness replaced by absolute fitness, plus an additional term. 
%$U = \frac{W}{\bar W}$ replaced by absolute fitness $W$, as well as an additional term $(N'-N) \E[X]$.

\begin{cor}[Aggregate Price Equation] \label{cor_aggregateprice}
Let $w$ be a finite-variance process. If $X$ and $Y$ are measurable functions on $I$ and $I'$, respectively, then the aggregate Price equation holds:
\begin{eqnarray}
    \int_{I'} Y \mu' - \int_I X \mu &=& N' \E'[Y] - N \E[X] \nonumber \\
    &=& N' \partial_\NS(X) + N' \partial_\EC(X,Y) + (N'-N) \E[X] \nonumber \\
    &=& N \cov(X,W) + N \E\!\left[\Delta_w(X,Y) W\right] + (N'-N) \E[X] \label{eqn_aggregateprice1} \\
    &=& \int_I \left( X(W-\bar W) + \Delta_w(X,Y) W + (\bar W-1) X \right) \mu \label{eqn_aggregateprice2} \\
    &=& \int_I \left( X(W-1) + \Delta_w(X,Y) W  \right) \mu. \label{eqn_aggregateprice3}
\end{eqnarray}
\end{cor}

We can analyze the evolution of population measures using the Price equation. %, by applying to indicator functions of measurable sets. 

\begin{cor}[Evolution of Population Measures]
Let $w$ be finite variance. Then for any measurable $A \subseteq I$ and $B \subseteq I'$, %$\langle 1_B \rangle_w = \frac{1}{W} w_i(B)$, hence
\begin{eqnarray}
    \E'[1_B] - \E[1_A] &=& \cov(1_A, U) + \E[\Delta_w(1_A,1_B) U] \\
    \mu'(B) - \mu(A) &=& N \cov(1_A, W) + N \E[\Delta_w(1_A,1_B) W] + (N'-N) \E[1_A].
    %\int_A (W(i)-1) \mu(\d i) + \int_I \left( \frac{w_i(B)}{W(i)} - 1_A \right) W(i) \mu(\d i)
\end{eqnarray}
\end{cor}
\begin{proof}
Apply the standard and aggregate Price equations with %indicator functions 
$X = 1_A$ and $Y = 1_B$. 
\end{proof}

\begin{rem}[Locally Finite Case]
If $\mu$ and $\mu'$ are locally-finite measures satisfying a disintegration equation \eqref{eqn_disintegration}, then the non-covariant aggregate Price equation \eqref{eqn_aggregateprice3} still holds. This can be verified directly: $\int_{I'} Y \mu' - \int_I X \mu = \int_I \left( X(W-1) + (\langle Y \rangle_w - X) W  \right) \mu$, for $XW \in L^1(\mu)$. Thus much of this article can be adapted to the locally-finite case.
\end{rem}

%\section{Fisher's Fundamental Theorem and the Zeroth Law of Natural Selection} \label{sect_zerothlaw}
%\section{Fisher's Fundamental Theorem, Selective Equilibrium, and the Zeroth Law of Natural Selection} \label{sect_zerothlaw}
%\section{Fisher's Fundamental Theorem and the Zeroth Law of Natural Selection} 

\subsection{Fisher's Fundamental Theorem} \label{sect_fisherthm}

Fisher's fundamental theorem \eqref{eqn_fisherthmclassical} states that selective change of relative fitness is equal to the variance of relative fitness:
\begin{equation}
    \partial_\NS(U) = \cov(U,U) = \var(U). 
\end{equation}
The aggregate version states that selective change of fitness is proportional to its variance:
\begin{equation}
    \partial_\NS(W) = \cov(W,U) = \frac{\var(W)}{\bar W}.
\end{equation}
The general version incorporates the environmental change to analyze the change of the fitness functions across time. 

% We introduce selective equilibrium to describe life-or-death processes, which we show to be the class of minimally selective processes. We introduce the Zeroth Law of Natural Selection as a lower bound for variance of relative fitness, saturated in the case of selective equilibrium. %i.e., $\partial_{\NS}(U) = \var(U)$. 

%\subsection{Fisher's Fundamental Theorem}

Consider three measures $\mu$, $\mu'$ and $\mu''$ on measurable spaces $I$, $I'$ and $I''$, with population sizes $N$, $N'$ and $N''$, respectively. Consider two composable processes $w : \mu \mapsto \mu'$ and $w' : \mu' \mapsto \mu''$. Define the fitness functions $W(i) := w_i(I')$ and $W(i') := w'_{i'}(I'')$, with selective coefficients $\bar W := N' / N$ and $\bar W' := N''/N'$. Define the relative fitness functions $U := W/\bar W$ and $U' := W'/\bar W'$. 

By construction, $U$ and $U'$ both have unit mean ($\E[U] = 1 = \E'[U'])$. When we apply the Price equation, the left side vanishes and so environmental change equals negative variance.

%Our general version of Fisher's theorem incorporates the environmental change. By construction, $\Delta(U,U') = \E'[U'] - \E[U] = 1 - 1 = 0$. Therefore, the Price equation ensures that the selective change term $\var(U)$ is exactly opposite the environmental change term $\E[ \Delta_w(U,U') U] = \E[(\langle U' \rangle - U)U]$. 

\begin{thm}[Generalized Fisher's Fundamental Theorem] \label{thm_fisher}
Let $w$ and $w'$ be composable processes, with $w$ finite-variance. Then: 
\begin{equation} \label{eqn_fisherthm}
    0 = \Delta(\bar U,\bar U') = \partial_{\NS}(U) + \partial_{\EC}(U,U') = \var(U) + \E[ \Delta_w(U,U') U].
\end{equation}
Equivalently,
\begin{equation}
    \E[ \Delta_w(U,U') U] = - \var(U).
    % \partial_{\EC}(U,U') = \E[ \Delta_w(U,U') U] = - \var(U) = - \partial_{\NS}(U).
\end{equation}

\end{thm}
\begin{proof}
This follows from the Price equation, setting $X := U$ and $Y := U'$. 
\end{proof}

When we apply this to the full fitness functions, we have:
\begin{equation}
    \bar W' - \bar W = \partial_\NS(W) + \partial_\EC(W,W') = \frac{\var(W)}{\bar W} + \frac{\E[\Delta_w(W,W') W]}{\bar W}.
\end{equation}

\section{Selective Equilibrium and the Zeroth Law of Natural Selection} \label{sect_zerothlaw}

%\subsection{Zeroth Law of Natural Selection}

%\subsection{Lower Bound on Relative-Fitness Variance}

We introduce selective equilibrium to understand the extreme case of ``minimally selective'' processes. Selective equilibrium is the extreme case where all selection is due to life and death and no other population variance. The Weak Zeroth Law (Proposition \ref{pro_zerothlaw}) states that variance is minimized in the case of selective equilibrium. In Section \ref{sect_entropyNS}, we state and prove a Strong Zeroth Law, improving upon the lower bound using selective entropy.

%The variance vanishes exactly when $U = 1$ almost surely. 
% $w$ is purely environmental (i.e., $\var(U) = 0$ if and only if $U = 1$ a.s.).

% We strengthen this result. 
\begin{defn}[Selective Equilibrium]
We say that a process $w$ is ``in selective equilibrium'' if $W$ takes exactly two values: $0$ and a single positive value $\bar U_* := 1/p_*$, where $p_* := \frac{1}{N} \mu(U>0)$ is the proportion of childbearing population. 
\end{defn}

% In selective equilibrium, many critical functionals like variance vanish, including selective entropy (which we will define and prove in Part \ref{part_selectiveentropy}). 

Define the childbearing population $\mu_*(A) := \mu(A \cap \{U>0\})$ and expectation operator $\E_*[X] := \frac{1}{p_*} \E[1_{U>0} X]$. The measures $\mu_*$ and $U \mu$ are mutually absolutely continuous. i.e., $\mu_*(A) = 0$ if and only if $(U \mu)(A) := \int_A U(i) \mu(\d i) = 0$.

%We now use Jensen's inequality to prove lower bounds for the variance and the selective change in variance, with the saturation conditions being equivalent to {\quasienvironmentality}. %In this way, {\quasienvironmentality} is a ``selective equilibrium'', 

\begin{pro}[Weak Zeroth Law of Natural Selection] \label{pro_zerothlaw} %[Strong Lower Bound for $\var(U)$] \label{pro_variancelowerbound}
Let $w$ be a finite-variance process. Then:
\begin{equation} \label{ineq_stronglowerboundforvariance}
\partial_{\NS}(U) = \var(U) \ge \frac{1}{p_*} - 1,
\end{equation}
with saturation exactly when $w$ is in selective equilibrium (in which case, $\var(U) = \frac{1}{p_*} - 1$).
\end{pro}
\begin{proof}
We write the variance as:
\begin{equation}
\var(U) = \E[(U-1)^2] = (1-p_*) + p_* \E_*[(U-1)^2].
\end{equation}
We now apply Jensen's inequality to the second term, since $\E_*$ is a probability expectation and $(x-1)^2$ is convex, and we rearrange:
\begin{eqnarray}
\var(U) &\ge& (1-p_*) + p_* (\E_*[U] - 1)^2 = (1-p_*) + p_* \left(\frac{1}{p_*} - 1\right)^2 \nonumber \\
&=& (1-p_*) + (1 - p_*) (\frac{1}{p_*} - 1) = \frac{1}{p_*} - 1,
\end{eqnarray}
since $\E_*[U] = \frac{1}{p_*} \E[U] = \frac{1}{p_*}$. Saturation of this inequality occurs exactly when $U$ is constant $\mu_*$-almost surely, i.e., the selective-equilibrium case.
\end{proof}

Using the general version of Fisher's theorem \eqref{eqn_fisherthm}, this implies an upper bound on the environmental change of relative fitness. 

\begin{cor}
Let $w$ and $w'$ be composable processes, with $w$ finite-variance. Then:
\begin{equation}
    \partial_{\EC}(U,U') = \E[ \Delta_w(U,U') U] = -\var(U) \le - \left(\frac{1}{p_*} - 1\right),
\end{equation}
with saturation when $w$ is in selective equilibrium.
\end{cor}

%The aggregate weak Zeroth Law holds:
%\begin{equation}
%    \partial_\NS(W) = \frac{\var(W)}{\bar W} \ge \frac{1}{\bar W} \left( \frac{1}{p_*} - 1 \right).
%\end{equation}

% \partial_\NS^2(W) = \cov\!\left( \frac{(W-\bar W)^2}{\bar W}, U \right).

%\begin{rem}
%In this effort, the notation $\langle X \rangle_w$ and $\E'_w[X|\cdot]$ separately generalize probability concepts. When $U = 1$ almost surely, then $W = \bar W$ a.s., so $\langle X \rangle_w = \E'_w[X|\cdot]$ a.s. Conversely, if $U \ne 1$ with positive population, then we have $W \ne \bar W$ with positive population, and therefore for some $X$ we have $\langle X \rangle_w \ne \E'_w[X|\cdot]$ with positive population. Thus the case of variable fitness corresponds to the failure of these probability concepts to coincide. 
%\end{rem}

%\begin{cor}
%Let $w : \mu \mapsto \mu'$ be a reversible process, and let $w^{-1} : \mu' \mapsto \mu$ denote its inverse. Let $U$ and $U^{-1}$ denote the relative fitnesses of $w$ and $w^{-1}$, respectively. 

%Then $\langle U^{-1} \rangle_w(i) U(i) = 1$ $\mu$-almost surely and $\langle U \rangle_{w^{-1}}(i') U^{-1}(i') = 1$ $\mu'$-almost surely.
%\end{cor}

\section{``Pure'' Processes and the Price Representation Theorem} \label{sect_pure}

We consider the extreme cases of purely selective and purely environmental processes. We show that purely selective processes correspond to absolutely continuous measures, and purely environmental processes correspond to Markov chains. We then prove a novel representation theorem (Theorem \ref{thm_pricerepresentation}), showing that every evolutionary process can be written as the composition of a purely selective process followed by a purely environmental process. 

\begin{defn}[Pure Processes] \textbf{ }
\begin{enumerate}
\item (Purely Selective) Consider measures $\mu$ and $\mu'$ defined on the same space. We say that a process $w$ is \emph{purely selective} if the average change of an observable is fully described by selective change: $\Delta(\bar X, \bar X) = \partial_{\NS}(X) = \cov(X,U)$. Equivalently, environmental change vanishes ($\partial_{\EC}(X,X) = \E[\Delta_w(X,X) U] = 0$.) 
%with $\frac{\partial \mu'}{\partial \mu} = W$. 

\item (Purely Environmental) Consider measures $\mu$ and $\mu'$ on (possibly different) spaces $I$ and $I'$. We say that $w$ is \emph{purely environmental} if its average change between observable $X$ and $Y$ is fully described by environmental change: $\Delta(\bar X, \bar Y) = \partial_{\EC}(X,Y) = \E[\Delta_w(X,Y) U]$. Equivalently, selective change vanishes ($\partial_\NS(X) = \cov(X,U) = 0$).
\end{enumerate}
\end{defn}

\begin{thm}[Characterization of Pure Processes] \label{thm_pureprocesses}
\textbf{ }
\begin{enumerate}
    \item (Purely Selective iff Absolute Continuity) Consider measures $\mu$ and $\mu$' on the same space. A process $w : \mu \mapsto \mu'$ is purely selective if and only if $\mu'$ is absolutely continuous to $\mu$ ($\mu' \ll \mu$) with Radon-Nikodym density equal to fitness $(\frac{\d \mu'}{\d \mu} = W$ a.s.). 
    
    % A process $w : \mu \mapsto \tilde \mu$ is purely selective if and only if $\tilde \mu$ is absolutely continuous to $\mu$ ($\tilde \mu \ll \mu$) with Radon-Nikodym density equal to fitness $(\frac{\d \tilde \mu}{\d \mu} = W$ a.s.). 
    
    %\item (OLD VERSON OF SELECTIVE) Suppose that $w : \mu \mapsto \mu'$ is purely selective. Then $\mu'$ is absolutely continuous to $\mu$, with density function equal to fitness ($\tfrac{\d \mu'}{\d \mu} = W$ a.s.), and transition kernel equal to $w_i(\d i') = W(i) \delta_i(\d i')$, where $\delta_i$ denotes the Dirac delta measure. Formally, $\delta_i(B) = 1$ if $i \in B$ and $=0$ if $i \notin B$. % The Price equation reduces to $\Delta(\bar X) = \cov(X, W/\bar W)$.

    %Conversely, suppose that $\mu'$ is absolutely continuous to $\mu$, with density function $\rho$. Then the process $w : \mu \mapsto \mu'$ defined by $w_i(\d i') := \rho(i) \delta_i(\d i')$ is purely selective, with fitness function $\rho$. In this case, $\Delta(\bar X) = \cov(X, \tfrac{\rho}{\E[\rho]})$.
    
    \item (Purely Environmental iff Markov Chain) Consider measures on possibly distinct spaces $I$ and $I'$. A process $w : \mu \mapsto \mu'$ is purely environmental if and only if the fitness function $W$ and relative fitness $U$ are almost surely constant (with $W = \bar W$ and $U = 1$ a.s.). In this case, $w$ is a Markov chain with transition kernel $w_i(\d i')$, with uniform scaling by $\bar W$. % For purely environmental processes, the Price equation reduces to $\Delta(\bar X) = \E[\Delta_w(X)]$.
\end{enumerate}

%, and Fisher's theorem states $\Delta(W) = \var(W) / \bar W$. %The aggregate changes are given by $\Delta(\bar X) = N \cov(X,W) + N(\bar W - 1) \E[X]$ and $\Delta(W) = N\var(W) + N(\bar W - 1) \bar W$.
\end{thm}
\begin{proof}
Proof of (1). Suppose that $w$ is purely selective, so $0 = \partial_{\EC}(X,X) = \E[ \Delta_w(X) U ]$ for each $X$. Thus $\E'[X] = \E[XU]$ for every $X$. Hence $\mu'/N'$ (resp. $\mu'$) is absolutely continuous with respect to $\mu/N$ (resp. $\mu$), with density $U$ (resp. $W$). 

% Suppose that $w$ is purely selective, so that $0 = \partial_{\EC}(X,Y) = \E[ \Delta_w(X,Y) U ]$ for all $X$. Thus $\E'[X] = \E[XU]$ for every $X$, so $\mu'/N'$ is absolutely continuous with respect to $\mu/N$, with density $U$. Therefore, $\mu'$ is absolutely continuous with respect to $\mu$, with density $W$. 

We show that absolutely continuous measures define a purely selective process. Suppose that $\mu' \ll \mu$ with Radon-Nikodym density $W := \tfrac{\d \mu'}{\d \mu}$. Define the purely selective $w : \mu \mapsto \mu'$ by weighting with the density function $W$, i.e., $w_i(A) := W(i) \delta_i(A)$, where $\delta_i(A)$ is the Dirac point-mass concentrated on $i$. i.e., $\delta_i(A) = 1$ if $i \in A$, and $=0$ if $i \notin A$. Then $\mu'(A) = \int_A W(i) \mu(\d i) = \int_{I'} w_i(A) \mu(\d i)$, proving (1).

% Next, we demonstrate that absolutely continuous measures admit a purely selective process. Suppose that $\mu' \ll \mu$ with $\tfrac{\d \mu'}{\d \mu} = \rho$. Define the purely selective process $\sigma_\rho : \mu \mapsto \mu'$ using the density function $\rho$, i.e., $(\sigma_\rho)_i(A) = \rho(i) 1_A(i)$. Then $\mu'(A) = \int_A \rho(i) \mu(\d i) = \int_{I'} (\sigma_\rho)_i(A) \mu(\d i)$, which proves the result.

Proof of (2). Suppose $w$ is purely environmental, so $0 = \cov(X, U) = \E[X(U - 1)]$ for all $X$. Since this holds for all $X$, we must have that $U = 1$ a.s. (hence $W = \bar W$). This is a standard functional argument. Let $H(\mu) = \{ X : \var(X) < \oo \}$ denote the Hilbert space of finite-variance observables, equipped with the covariance inner product. Since $H$ is closed, if $\cov(X,U) = 0$ for all $X$, then $\var(U) = 0$, hence $U$ is constant a.s. and equal to $\E[U] = 1$. Consequently, $W = \bar W$ a.s. Conversely, suppose $W$ is a.s. constant (with $W = \bar W$). Then $\partial_{\NS}(X) = \cov(X, \tfrac{W}{\bar W}) = 0$ since $W = \bar W$ almost everywhere. Thus $w$ is purely environmental.

If $w : \mu \mapsto \mu'$ is a Markov process between two probability distributions with kernel $w_i(\d i')$, then it describes a purely environmental process since $W(i) := w_i(I') = 1$ is the conditional probability of $I'$ given $i$. Conversely, if $w : \mu \mapsto \mu'$ is purely environmental, then $N' = N$ and $w_i(I') = 1$, so $w$ is a regular conditional probability hence a one-step Markov chain.  
\end{proof}

\begin{rem}
If one measure is absolutely continuous with respect to another ($\mu' \ll \mu$), there exists a unique purely selective process $w^\NS : \mu \to \mu'$, given by the Radon-Nikodym derivative $W := \frac{\d \mu'}{\d \mu}(i)$, but there can also exist purely environmental or general processes between these measures. For example, suppose that $\mu$ assigns mass $1/2$ to each of $\{0\}$ and $\{1\}$, and $\mu'$ assigns mass $1$ to $\{0\}$. Then $\mu' \ll \mu$ with $W^\NS(0) := \frac{\d \mu'}{\d \mu}(0) = 2$ and $W^\NS(1) := \frac{\d \mu'}{\d \mu}(1) = 0$. There also exists a Markov chain $w : \mu \mapsto \mu'$ with $w_0(0) = 1$, $w_1(0) = 1$, and $W(0) = 1 = W(1)$.
\end{rem}

%\subsection{Price Representation Theorem}

The Price equation is equivalent to the following representation theorem, decomposing any process as a selective process $w_\NS$ followed by an environmental one $w_\EC$. 

%A simple and powerful consequence is that we can express any finite-variance process $w$ as the composition of a purely selective process $w_\NS$ with a purely environmental process $w_\EC$. 

%the Price representation theorem ensures that we can always decompose processes into separate selective and environmental components, and fully account for changes in terms of these components. 

\begin{thm}[Price Representation Theorem] \label{thm_pricerepresentation}
Let $w$ be a finite-variance process with fitness $W$. Define the purely selective process $w_\NS : \mu \mapsto W \mu$ and the purely environmental process $w_\EC : W \mu \mapsto \mu'$ as follows:
\begin{equation}
    w_{\NS,i}(\d \tilde i) := W(\tilde i) \delta_i(\d \tilde i) \qquad \mathrm{and} \qquad w_{\EC,\tilde i}(\d i') := \frac{w_{\tilde i}(\d i)}{W(\tilde i)},
\end{equation}
where $\delta_i$ is the Dirac point-mass measure on $I$. Then 
\begin{equation}
w = w_\EC \circ w_\NS.
\end{equation}

The selective changes of $w$ and $w_\NS$ are equal:
\begin{equation}
    \partial_{w,\NS}(X) = \cov(X,U) = \partial_{w_\NS,\NS}(X);
\end{equation}
and the environmental changes of $w$ and $w_\EC$ are equal:
\begin{equation}
\partial_{w,\EC}(X,Y) = \E[\Delta_w(X,Y) U] = \tilde \E[\Delta_{w_\EC}(X,Y)] = \partial_{w_\EC, \EC}(X,Y),
\end{equation}
where $\tilde \E[Y] := \E[U Y]$.
\end{thm}
\begin{proof}
Using Theorem \ref{thm_pureprocesses}, we have that $w_\NS$ is purely selective,  since it is absolutely continuous to $\mu$ with density equal to the fitness function; and $w_\EC$ is purely environmental, since $w_{\EC,i}(I') = \frac{w_i(I')}{W(i)} = 1$ for all $i$. We compute:
\begin{eqnarray}
(w_\EC \circ w_\NS)_i(\d i') &=& \int w_{\EC,\tilde i}(\d i') \, w_{\NS,i}(\d \tilde i) = \int \frac{w_i(\d i')}{W(\tilde i)} W(\tilde i) \, \delta_i(\d \tilde i) \nonumber \\
&=& \frac{w_i(\d i')}{W(i)} W(i) = w_i(\d i'),
\end{eqnarray}
where $\delta_i(\d \tilde i)$ is the Dirac delta distribution on $I$. i.e., $\delta_i(A) = 1$ if $i \in A$ and $=0$ if $i \notin A$. Clearly, the selective changes are equal, since both $w$ and $w_\NS$ have fitness function $W$ on population $\mu$. 

Define the intermediate population $\tilde \mu := W \mu$ on $I$ (i.e., $\tilde \mu(A) := \int_A W(i) \mu(\d i)$), and intermediate expectation $\tilde \E[Y] := \E[U Y]$. Note that $\tilde \mu$ has population size $N'$. We compute the adaptive local change
\begin{eqnarray}
\Delta_{w_\EC}(X,Y)(\tilde i) &:=& \langle X' \rangle_{w_\EC}(\tilde i) - X(\tilde i) = \int \frac{X'}{W(\tilde i)} w_{\tilde i}(\d i') - X(\tilde i) \nonumber \\
&=& \langle X' \rangle_w(\tilde i) - X(\tilde i) = \Delta_w(X,Y)(\tilde i)
\end{eqnarray}
Averaging with $\tilde \E$, we have:
\begin{eqnarray}
\partial_{w_\EC, \EC}(X,Y) &=& \tilde \E[\Delta_{w_\EC}(X,Y)] \nonumber \\
&=& \E[\Delta_\EC(X,Y) U] = \partial_{\EC}(X,Y).
\end{eqnarray}
%as desired. 
\end{proof}

\begin{rem}[Reversibility]
We say that $w$ is selectively reversible when $w_\NS$ is invertible. This holds if and only if $p_* = 1$, in which case $w^{-1}_\NS$ is given by multiplication by the reciprocal fitness $\frac{1}{W}$. We may always recover the childbearing population by multiplying by $\frac{1}{W}$: $\mu_* = \frac{1}{W} \tilde \mu$. We say that $w$ is environmentally reversible when $w_\EC$ is invertible, and analyze that case in Section \ref{sect_ecri}.
\end{rem}

As an immediate consequence, we can decompose any composition $w' \circ w$ as a single selective piece followed by a purely environmental one. If $X$ is $I$-measurable and $Y$ is $\mu'$-integrable, define the composable product $(Y \circ X)(i) := \langle Y \rangle_w(i) X(i)$. 

\begin{cor}
Let $w : \mu \mapsto \mu'$ and $w' : \mu' \mapsto \mu''$ be two composable processes. Define the composed fitness function
\begin{equation}
    W^{(2)}(i) := (W' \circ W)(i) := \int_{I'} W'(i') w_i(\d i') = \langle W' \rangle_w(i) W(i),
\end{equation}
and define the purely selective process by multiplying by $W^{(2)}$:
\begin{equation}
    (w' \circ w)_{\NS,i}(A) := W^{(2)}(i) \delta_i(A),
\end{equation}
where $\delta_i$ is the Dirac delta function on $i$. Define the purely environmental process
\begin{equation}
    (w' \circ w)_{\EC,i}(C) := \frac{(w' \circ w)_i(C)}{(W' \circ W)(i)} = \frac{1}{W^{(2)}(i)} \int_{I'} w'_{i'}(C) w_i(\d i').
\end{equation}
for measurable $C \subseteq I''$. Then the composed process equals:
\begin{equation}
    w' \circ w = (w' \circ w)_\EC \circ (w' \circ w)_\NS.
\end{equation}
The composed process has fitness function $W^{(2)} = W' \circ W = \langle W' \rangle_w W$, and relative fitness function $U^{(2)} := U' \circ U := \langle U' \rangle_w U = W' \circ W / (\bar W' \bar W)$. 
\end{cor}

% In Part \ref{part_selectiveentropy}, we show that a process is purely environmental if and only if its selective entropy vanishes, and in Part \ref{part_environmentalentropy}, we show that a process is purely selective if and only if its environmental entropy vanishes. 

% \subsection{Reversible Processes}

\subsection{Application to Matrix Theory}

The Price representation theorem admits a simple form in terms of matrices. Consider Price's context of discrete evolutionary processes, as in Section \ref{sect_discretepriceequation}. Finite discrete populations are encoded by finite-dimensional vectors $\mu,\mu'$, and an evolutionary process as a finite-dimensional matrix $w = (w(i,i'))$, with $\mu' = w\mu$. These vector spaces are equipped with the $L^1$ norm, so $N = |\mu| = \sum_{i=1}^K \mu_i$ and $N' = |\mu'| = \sum_{i'=1}^{K'} \mu'_{i'}$. The discrete Price representation theorem is equivalent to the matrix identity
\begin{equation} \label{eqn_matrixpricedecomposition}
    w = w_\EC w_\NS,
\end{equation} 
where $w_\NS$ is a diagonal matrix, and $w_\EC$ is a right stochastic (Markov transition) matrix. 

%(Markov transition matrix). %Decomposition \eqref{eqn_matrixpricedecomposition} is equivalent to the discrete Price equation \eqref{eqn_discreteprice}.

% \section{Variance Changes and the First Law of Natural Selection}

\section{Selective Change of Variance and the First Law of Natural Selection} \label{sect_selectivechangeofvariance_and_firstlaw}

We state a functional form of the Price equation, and we use this to analyze selective change of relative-fitness variance. We prove a First Law of Natural Selection, showing that $\partial_\NS \var(U) \ge \var(U) \left( 1 + \var(U) \right) \ge 0$, with saturation of the first inequality in the selective-equilibrium case.

\begin{cor}[Functional Price Equation] \label{cor_functionalprice}
Let $w : \mu \mapsto \mu'$ denote an evolutionary process. Let $\X$ and $\Y$ be vector-valued observables, and let $F(\X)$ and $G(\Y)$ be integrable functionals. Formally, we assume that $X : I \to V$ and $Y' : I' \to V'$ are Borel-measurable functions to topological vector spaces $V,V'$, and that $f : V \to \R$ and $g : V' \to \R$ are Borel-measurable real-valued functions. 
Define the selective functional change $\partial_\NS F(\X) := \cov(f(\X), U)$ and the environmental functional change $\partial_\EC(F(\X),G(\Y)) := \E[\Delta(f(\X),g(\Y)) U]$. The functional Price equation holds:
\begin{equation}
    \Delta(F(\X),G(\Y)) = \partial_\NS F(\X) + \partial_\EC(F(\X),G(\Y)) = \cov(f(\X), U) + \E[\Delta_w(f(\X),g(\Y)) U]. 
\end{equation}
\end{cor}

\subsection{Functional Change of Variance}
%\subsection{Selective Change of Relative-Fitness Variance}

We now analyze the change of variance across generations. Let $w : \mu \mapsto \mu'$ and $w' : \mu' \mapsto \mu''$ be composable processes, with fitness functions $W$ and $W'$ and relative fitnesses $U = W/\E[W]$ and $U' = W'/\E'[W]$. Write the two variance functionals as $\var(U) = \E[U^2 -1]$ and $\var'(U') = \E'[(U')^2 - 1]$. %The First Law provides a lower bound for $\partial_\NS \var(U)$. 
% We could also consider the general variance functionals $\var(X) = \E[X^2]-\E[X]^2$ and $\var'(Y) = \E'[Y^2]-\E'[Y]^2$. That approach requires keeping track of variable mean, which we do not do in this section.
We write the difference of variances as follows: 
%as the change from their integrands $U^2-1$ to $(U')^2-1$ (or equivalently, from $U^2$ to $(U')^2$):
\begin{eqnarray}
    \Delta(\var(U)) &:=& \Delta(\var(U), \var'(U')) := \var'(U') - \var(U) = \E'[(U')^2] - \E[U^2]. %\nonumber \\
    %&=& \Delta(U^2-1, (U')^2 - 1) = \E'[(U')^2 - 1] - \E[U^2 - 1] \nonumber \\
    %&=& \Delta(\bar{U^2}, \bar{(U')^2}') = \E'[(U')^2] - \E[U^2].
\end{eqnarray}

% Selective change: $\partial_\NS(\var(X)) = \cov(X^2-\bar X^2, U)$. 
% Environmental change: $\partial_\EC(\var(U)) := \E\!\left[\Delta_w(X^2-\bar X^2, Y^2-\bar Y') U\right] = \E\!\left[\left(\left\langle Y^2 \right\rangle_w - X^2\right)U\right] - \left( (\bar Y')^2 - \bar X^2 \right).$

Fisher's theorem states that selective velocity is given by relative-fitness variance: $\partial_\NS(U) = \var(U)$. We define the selective change of variance, or selective acceleration, by
\begin{equation} \label{eqn_fitnesschangeNS}
    \partial_\NS^2(U) := \partial_\NS \var(U) := \cov(U^2, U) = \E[U^2(U-1)],
%    \partial_\NS \var(U) := \cov(U^2-1, U) = \E[(U^2-1)(U-1)] % = \E[(U^2-1)(U-1)] = \E[(U+1)(U-1)^2] \ge 0
\end{equation}
and environmental change of variance
\begin{equation} \label{eqn_fitnesschangeEC}
    \partial_\EC(\var(U),\var'(U')) := \E\!\left[\Delta_w(U^2, (U')^2) U\right] = \E\!\left[\left(\left\langle (U')^2 \right\rangle_w - U^2\right)U\right].
\end{equation}

The functional Price equation ensures that the change of variance decomposes as the sum of the selective and environmental changes:
\begin{equation} \label{eqn_fitnesschangeprice}
%\Delta(\var(U)) 
\var'(U') - \var(U)
= \partial_\NS \var(U) + \partial_\EC(\var(U),\var'(U')).
\end{equation}

The aggregate difference of variances follows from the vector form of the functional Price equation (since $\var(W) = \E[(W-\bar W)^2] = \E[f(X,\bar X)]$). Consequently: 
\begin{eqnarray}
    \var'(W') - \var(W) &=& \partial_\NS \var(W) + \partial_\EC(\var(W), \var'(W')) \nonumber \\
    &=& \cov((W-\bar W)^2, U) + \E[\Delta_w((W-\bar W)^2, (W'-\bar W')^2) U].
\end{eqnarray}

%\subsection{First Law of Natural Selection: Strong Bounds for Selective Change of Variance}

\subsection{First Law of Natural Selection}

Recall that the Zeroth Law (Proposition \ref{pro_zerothlaw}) states that $\partial_\NS(U) = \var(U) \ge 0$. This is a monotonically upward trend for relative-fitness under the effect selection, and shows that selection acts in the direction of never decreasing relative fitness, though the effect of the environment can be arbitrary. 

We strengthen this result, and show that there is a non-negative lower bound on the selective acceleration, compounding effects of selection upon itself. We prove weak and strong versions, saturated in the purely environmental and selective-equilibrium cases, respectively.

%Formally, the selective change of variance is the selective acceleration of relative-fitness: 
%\begin{equation}
%    \partial_\NS^2(U) := \partial_\NS \var(U) = \cov(U^2-1,U) = \cov(U^2,U) = \E[U^2 (U-1)].
%\end{equation}
%We can also consider higher-order selective derivatives: $\partial_\NS^n(U) := \E[(U^2-\E[U^2])(U-1)^{n-1}]$. 

%The next result provides a non-negative lower bound for $\partial_\NS \var(U)$ is always non-negative, vanishing exactly in the purely environmental case. 

%Since by Fisher's theorem, we have $\partial_\NS(U) = \var(U)$, the selective change of variance is the selective acceleration of relative fitness. We write $\partial_\NS^2(U) := \partial_\NS \var(U) \ge 0$. 

%That is, the presence of any selective effects amplify selection. 

\begin{pro}[Weak First Law of Natural Selection] \label{pro_weakfirstlaw}
%[Weak Lower Bound for $\partial_\NS \var(U)$] \label{lem_lowerbound_variance}
Let $w$ have finite third moment $\E[U^3] < \oo$. The selective change of relative fitness variance is non-negative:
\begin{equation}
 \partial_\NS^2(U) \ge \tfrac{1}{2} \var(U)^2 = \tfrac{1}{2} \partial_\NS(U)^2 \ge 0.
\end{equation}
Both inequalities are saturated exactly when $w$ is purely environmental (in which case $\partial_\NS^2(U) = \tfrac{1}{2} \var(U)^2 = \frac{1}{2} \partial_\NS(U)^2 = 0$), otherwise the inequalities are strict.
\end{pro}
\begin{proof}
%Observe that covariance is an affine form, i.e., invariant under any additive constants: $\cov(X-A,Y-B) = \cov(X,Y)$ for any a.s.-constants $A$ and $B$. 

%Non-negativity for $\partial_\NS \var(U)$ is trivial:
%\begin{equation} \label{eqn_variablecovariance}
%    \partial_\NS \var(U) = \cov(U^2-1,U) = \E\!\left[(U^2-1)(U-1)\right] = \E\!\left[(U+1)(U-1)^2\right] \ge 0,
%\end{equation}
%since $U+1 > 0$ and $(U-1)^2 \ge 0$.

The bound for $\partial_\NS \var(U)$ uses Jensen's inequality for the quadratic term:
\begin{eqnarray} 
    \partial_\NS \var(U) &=& 2 \tfrac{1}{2} \E\!\left[(U+1)(U-1)^2\right] \ge 2 \left( \tfrac{1}{2} \E[(U+1)U] - 1 \right)^2 \nonumber \\
    &=& 2 \left( \tfrac{1}{2} (\var(U) + 2) - 1 \right)^2 = \tfrac{1}{2} \var(U)^2, \label{eqn_variablecovariance_strong}
\end{eqnarray}
since $\tfrac{1}{2} \E[U+1]=1$ and $\E[(U+1)U] = \E[U^2] + \E[U] = \var(U) + 2$. This is saturated when $U+1$ is constant $\mu$-a.s., i.e., $U + 1 = \bar U + 1 = 2$, the purely environmental case.
%since $U+1 > 0$ and $(U-1)^2 \ge 0$. 

% Strong bound: \partial_\NS \var(U) = 2 \frac{1}{2} \E[(U+1)(U-1)^2] \ge 2 \left( \E[(U+1)U] - 1 \right)^2 = 2 \var(U)^2

% If $w$ is purely environmental, then $U = 1$ almost surely hence $\partial_\NS \var(U) = 0$. Conversely, if $\partial_\NS \var(U) = 0$, then since $U+1 > 0$ and $(U-1)^2 \ge 0$, we have $(U-1)^2 = 0$ almost surely. This occurs only when $U=1$ almost surely, hence $w$ is purely environmental.
\end{proof}

%We strengthen this result and prove a bound which is saturated in selective equilibrium. %We refer to this as the First Law of Natural Selection. %The quantity $\partial_\NS^2(U)$ is essentially the ``selective energy'' of a system.

We strengthen this result by applying Jensen's inequality to the child-bearing population.

\begin{thm}[Strong First Law of Natural Selection] %[Strong Lower Bound for $\partial_\NS \var(U)$] 
\label{thm_firstlaw}
Let $w$ have finite third moment $\E[U^3] < \oo$. Then:
\begin{eqnarray} \label{ineq_variancechangeNSbound}
\partial_\NS^2(U) 
= \partial_\NS \var(U) &\ge& \partial_\NS(U) \left( 1 + \partial_\NS(U) \right) \\
&=&\var(U) \left(1 + \var(U) \right) = \var(U) \E[U^2] \ge 0, 
\end{eqnarray}
with saturation of the first inequality exactly when $w$ is in selective equilibrium. In that case, $\partial_\NS \var(U) = \frac{1}{p_*} (\frac{1}{p_*} - 1) = \frac{1}{p_*^2} - \frac{1}{p_*}$.
\end{thm}
\begin{proof}
% Is covariance invariant under sums with constants? YES $\cov(X+C,Y) = \E[(X+C-\bar X)(Y - \bar Y)] = \E[(X-\bar X)(Y-\bar Y)] + C \E[Y-\bar Y] = \cov(X,Y) + 0$.

% \partial_\NS \var(U) = 

We change measure to the intermediate population and rewrite the covariance:
\begin{equation}
    \partial_\NS \var(U) = \cov(U^2,U) = \E[(U^2-\E[U^2])(U-1)] = \E[(U^2-\E[U^2]) U] = \tilde \E[U^2] - \E[U^2]. 
\end{equation}
where $\tilde \E[X] = \E[UX]$. Using Jensen's inequality, we have
\begin{equation}
    \partial_\NS \var(U) \ge \tilde \E[U]^2 - \E[U^2] = \E[U^2]^2 - \E[U^2] = \E[U^2] \left( \E[U^2] - 1 \right),
\end{equation}
since $\tilde \E[U] = \E[U^2] = 1 + \var(U)$, proving \eqref{ineq_variancechangeNSbound}.
This inequality is saturated exactly when $U$ is constant $U \mu$-a.s.. Since $U \mu$ is mutually absolutely continuous with $\mu_*$, saturation is equivalent to $U$ being constant $\mu_*$-a.s., i.e., $w$ is in selective equilibrium. 
\end{proof}

We use the same technique to analyze higher-order selective derivatives of relative fitness, and the exponential of relative fitness. The Higher-Order First Law shows that all these selective changes are non-negative, meaning that selection acts monotonically upon all scales of fitness.

\begin{pro}[Higher-Order First Law] \label{pro_higherorder_firstlaw}
Suppose that $\E[U^{n+1}]<\oo$ for $n \ge 1$. The higher-order selective changes are non-negative:
\begin{equation} \label{ineq_higherorder_firstlaw_moments}
    \partial_\NS^n(U) \ge \begin{cases} 
    \frac{1}{p_*^{n-1}} \left(\var(U) + 1 - p_*\right)^n\ge 0, & \mbox{$n$ even,} \\
    \frac{1}{p_*^n} \left(1 - p_*\right)^{n+1} \ge 0, & \mbox{$n$ odd,}
    \end{cases}
\end{equation}
with saturation of the left inequality when $w$ is in selective equilibrium. 

Suppose that $\E\!\left[\e^U\right] < \oo$. The selective change of the exponential $\e^U$ is non-negative:
\begin{equation} \label{ineq_higherorder_firstlaw_exponential}
    \partial_\NS\!\left( \e^U \right) = \cov\!\left(\e^U, U\right) \ge (1-p_*) \left(\e^{1/p_*} - 1 \right) \ge 0,
\end{equation}
with saturation of the first inequality when $w$ is in selective equilibrium.

%If $n$ is even, then:
%\begin{equation}
%    \partial_\NS^n(U) = p_* \frac{1}{2} \E_*[U(U-1)^n] \ge p_* (\E_*[U^2]-1)^n = p_* (\frac{1}{p_*} \E[U^2] - 1)^n = \frac{1}{p_*^{n-1}} (\var(U) + 1 - p_*)^n
%\end{equation}
%If $n$ is odd, then:
%\begin{equation}
%    \partial_\NS^n(U) = 
%\end{equation}
\end{pro}
\begin{proof}
%This follows from the identities
%$\partial_\NS^n(U) = \partial_\NS^{n-1}(\var(U)) = \E[U(U-1)^n] = \E[(U-1)^{n+1}]$. 
In the even case, $\partial_\NS^n(U) = \partial_\NS^{n-1}(\var(U)) = p_* \E_*[U(U-1)^n] \ge p_* (\E_*[U^2]-1)^n = p_* (\frac{1}{p_*} \E[U^2] - 1)^n = \frac{1}{p_*^{n-1}} (\var(U) + 1 - p_*)^n$. In the odd case, $\partial_\NS^n(U) = \partial_\NS^{n-1}(\var(U) = p_* \E_*[(U-1)^{n+1}] \ge p_* (\E_*[U]-1)^{n+1} = \frac{1}{p_*^n} (1 - p_*)^{n+1}$. For \eqref{ineq_higherorder_firstlaw_exponential}, we compute:
\begin{eqnarray}
    \partial_\NS\!\left(\e^U\right)
    &=& \cov\!\left(\e^U, U \right) 
    = \E\!\left[(U-1) \e^U)\right] 
    = \E\!\left[U \e^U\right] - \E\!\left[\e^U\right] \nonumber \\
    &=& p_* \E_*[U \exp(U)] + p_* \E_*[-\exp(U)] - (1-p_*) \nonumber \\
    &\ge& p_* \E_*[U] \exp(\E_*[U]) - p_*  \exp(\E_*[U]) - 1 + p_* \nonumber \\
    %&=& p_* \frac{1}{p_*} \exp \frac{1}{p_*} - p_* \exp(\frac{1}{p_*}) - 1 + p_* \nonumber \\
    &=& \e^{1/p_*} - p_* \e^{1/p_*} - 1 + p_* 
    = (1-p_*) \left( \e^{1/p_*} - 1 \right).
\end{eqnarray}
\end{proof}

\begin{cor}[Aggregate First Law]
Let $w$ have finite third moment. Then:
\begin{equation} \label{ineq_aggregatefirstlaw}
    \partial_\NS^2 W \ge \bar W \var(U) \left( 1 + \var(U) \right) = \frac{1}{\bar W^3} \var(W) \left( \bar W^2 + \var(W) \right).
\end{equation}
\end{cor}
\begin{proof}
Fisher's theorem states $\partial_\NS W = \frac{\var(W)}{\bar W}$. Thus aggregate selective acceleration equals:
\begin{equation}
    \partial_\NS^2(W) := \partial_\NS \partial_\NS W = \cov\!\left( \frac{(W-\bar W)^2}{\bar W}, \frac{W}{\bar W} \right) 
    % = \frac{1}{\bar W^2} \cov((W-\bar W)^2, W)% \nonumber \\
    = \bar W \cov((U-1)^2, U) = \bar W \partial_\NS^2(U).
\end{equation}
Inequality \eqref{ineq_aggregatefirstlaw} follows from the First Law.
\end{proof}

\section{Environmental Change of Variance} \label{sect_environmentalchangevariance}

We rearrange the environmental change of variance with intergenerational fitness ratios, then prove a lower bound, which is saturated in a certain stationarity case. %the strongly stationary case of Section \ref{sect_environmentalchangevariance}.

%below using Jensen's inequality %$\partial_\EC\!\left( \var(U), \var'(U') \right)$ $= \E[\Delta_w(U^2,(U')^2)U]$. 
%We use Jensen's inequality to prove a lower bound for this quantity, which is saturated in a certain stationarity condition. 

%We next analyze the adaptive environmental change of fitness variance. To formalize this, let $w$ and $w'$ be composable processes, with respective relative fitnesses $U = W/\bar W$ and $U' = W'/\bar W'$. The adaptive environmental change equals:
%\begin{eqnarray}
%\partial_\EC^*(\var(U)) &:=& \partial_\EC^*(\var'(U')|\var(U)) \nonumber \\
%&:=& \E'[(U')^2 - 1] - \E[U^2 - 1] = \E[U(\langle (U')^2 \rangle_w - U^2)].
%\end{eqnarray}

%First, we consider relative versions of relative fitness across generations, not just within a population. 

\begin{defn}[Intergenerational Fitness Ratios]
Let $w$ and $w'$ be finite-mean. For $\mu$-a.e. $i$ and $w_i$-a.e. $i'$, define the intergenerational relative fitness as the ratio of relative fitnesses:
\begin{equation}
    R := R(i,i') := \frac{U'(i')}{U(i)},
\end{equation}
which need not be defined when $U(i) = 0$. Define the averaged intergenerational relative fitness as the average value of $R(i,i')$ across the children of $i$. That is, for $\mu$-a.e. $i$, define:
\begin{equation}
    \bar R_w := \bar R_w(i) := \langle R \rangle_w(i) = \frac{\langle U' \rangle_w(i)}{U(i)}.
\end{equation}  
\end{defn}

%We define stationarity conditions of the joint process $(w,w')$. 
%qualitative properties of the joint process $(w,w')$ using the intergenerational relative fitnesses. 

\begin{defn}[Stationarity of Joint Processes] \label{def_stationary}
Let $w$ and $w'$ be composable. Then:
\begin{enumerate}
    \item The pair $(w,w')$ is strongly stationary if $R(i,i') = 1$ for $\mu$-a.e. $i$ and $w_i$-a.e. $i'$, i.e., $U'(i') = U(i)$ for $w_i$-a.e. child of $i$. 
    \item The pair $(w,w')$ is weakly stationary if $\bar R_w(i) = 1$ for $\mu$-a.e. $i$, i.e., the average relative fitness among children of $i$ equals $U(i)$.
    \item The pair $(w,w')$ is locally homogeneous if $R$ is constant jointly-a.s.. That is, there exists a constant $\lambda$ so that for $\mu$-a.e. $i$ and $w_i$-a.e. $i'$, $R(i,i') = \lambda$ (so $U'(i') = \lambda U(i)$). 
    \item The pair $(w,w')$ is locally constant if $U'$ is $w_i$-a.s. constant for $\mu$-a.e. $i$. That is, for $\mu$-a.e. $i$, $U'(i') = \langle U' \rangle_w(i)$ for $w_i$-a.e. $i'$.
\end{enumerate}
\end{defn}

Strong stationarity is equivalent to being both weakly stationary and locally homogeneous. Proof: If $(w,w')$ is strongly stationary, then $R$ is constant jointly-a.s. and equal to $1$, so $\bar R_w$ is constant a.s. and equal to $1$. Conversely, if $R = \lambda$ jointly-a.s. and $\bar R_w = 1$ a.s., then $\bar R_w = \langle R \rangle_w = \langle \lambda \rangle_w = \lambda$ a.s.
%\begin{lem}
%    The pair $(w,w')$ is strongly stationary if and only if it is both weakly stationary and locally homogeneous.
%\end{lem}
%\begin{proof}
%If $(w,w')$ is strongly stationary, then $R$ is constant and equal to $1$, so $\bar R_w$ is constant and equal to $1$. Conversely, if $R = \lambda$ and $\bar R_w = 1$ a.s., then $\bar R_w = \langle R \rangle_w = \langle \lambda \rangle_w = \lambda$ a.s. 
%\end{proof}
We relate the joint stationarity conditions to marginal environmental conditions. 
%We use this to characterize the saturation condition for $\partial_\EC\!\left( \var(U), \var'(U') \right)$, and we will use this again in Section \ref{sect_environmentalchangeofselectiveentropy} when we analyze selective change of selective entropy.  

\begin{lem} \label{lem_stationary}
\textbf{ }
\begin{enumerate}
    \item The joint process $(w,w')$ is weakly stationary if and only if $w$ is purely environmental and $\langle U' \rangle_w = 1$ $\mu$-a.s. %In this case, $\partial_\EC S_\NS = \E[U^2 \langle -R \log R \rangle]$. 

    \item The joint process $(w,w')$ is strongly stationary if and only if both $w$ and $w'$ are purely environmental. %In this case, $\partial_\EC S_\NS = 0$.
\end{enumerate}
\end{lem}
\begin{proof}
Suppose $(w, w')$ is weakly stationary, so $\bar R_w = 1$ $\mu$-a.s. Thus:
\begin{equation}
\var(U) = \E[U^2] - 1 = \E[U^2 \bar R_w] - 1 = \E'[U'] - 1 = 0.
\end{equation}
Thus $U$ is a.s. constant, so $w$ is purely environmental ($U=1$ a.s.), and $\langle U'\rangle_w = \bar R_w U = 1$ almost surely. Conversely, if $U = 1$ a.s. and $\langle U' \rangle_w = 1$ a.s., then $\bar R_w = 1$ a.s.

If $w$ and $w'$ are both purely environmental, then $U' = 1 = U$, hence strongly stationary. Conversely, if $(w,w')$ is strongly stationary, then then $w$ is purely environmental (since strong implies weak stationarity), so $U' = U = 1$ a.s. So both $w$ and $w'$ are purely environmental. 
\end{proof}

%\subsection{Lower Bound on Environmental Change of Variance}

We prove a strong lower bound for the environmental change of variance. %$\partial_\EC(\var(U), \var'(U)) = \E[\Delta_w(U^2,(U')^2)U]$. 
%We prove a strong lower bound for this quantity.

\begin{pro}[Strong Lower Bound for $\partial_\EC\!\left( \var(U), \var'(U') \right)$] \label{pro_stronglowerboundforECvariance}
Let $w$ and $w'$ be composable finite-variance processes. Then:
\begin{equation}
\partial_\EC\!\left( \var(U), \var'(U') \right) \ge \frac{1}{\E[U^3]} \E\!\left[U^3 (\bar R_w-1)\right] \E\!\left[U^3 (\bar R_w+1)\right],
% \frac{1}{\E[U^3]} \E\!\left[U^3 (\bar R_w-1)\right] \left( \E[U^3 \bar R_w]^2 - \E[U^3]^2 \right)
\end{equation}
which is saturated in the strongly stationary case (in which case, $\partial_\EC\!\left( \var(U), \var'(U') \right) = 0$). 

%The lower bound is minimized when $w'$ is purely environmental (in which case $\partial_\EC\!\left( \var(U), \var'(U') \right) = -\E[U^3]$), otherwise $\partial_\EC\!\left( \var(U), \var'(U') \right) \ge \frac{1}{\E[U^3]} (\E[U^3 (\bar R_w-1)]) > -\E[U^3]$.
\end{pro}
\begin{proof}

We use Jensen's inequality twice. First, observe that $\langle R^2 \rangle_w(i) \ge \bar R_w^2(i)$ for each $i$, since the function $R \mapsto R^2$ is convex and $\langle \cdot \rangle_w$ is a probability expectation for each $i$. Second, $\bar R_w \mapsto \bar R_w^2$ is convex and $\E$ is a probability expectation, so we compute: 
\begin{eqnarray}
\partial_\EC\!\left(\var(U),\var'(U')\right) &=& \E[U^3(\langle R^2 \rangle_w - 1)] \ge \E[U^3 (\bar R_w^2 - 1)] \nonumber \\
&\ge& \E[U^3] \left(\left(\frac{\E[U^3 \bar R_w]}{\E[U^3]}\right)^2 - 1\right) 
= \frac{1}{\E[U^3]} \left(\E[U^3 \bar R_w]^2 - \E[U^3]^2\right) \nonumber \\
&=& \frac{1}{\E[U^3]} (\E[U^3 \bar R_w] - \E[U^3]) (\E[U^3 \bar R_w] + \E[U^3]) \nonumber \\
&=& \frac{1}{\E[U^3]} \E\!\left[U^3 (\bar R_w-1) \right] \E\!\left[U^3 (\bar R_w+1)\right].
\end{eqnarray}

The first inequality is saturated when $R$ is constant $\mu'$-almost surely for a.e. $i$, i.e., when $w$ is locally homogeneous. The second inequality is saturated when $\bar R_w$ is constant $\mu$-almost surely, i.e., when $w$ is weakly stationary. Since strong stationarity implies weak stationarity, both inequalities are saturated exactly in the strongly stationary case. 
\end{proof}

By combining Theorem \ref{thm_firstlaw} and Proposition \ref{pro_stronglowerboundforECvariance} via the Price equation, we have the following lower bound on the average change.

\begin{cor}
Let $w$ and $w'$ be composable processes with $w$ finite-variance. Then:
\begin{equation}
    \Delta(\var(U), \var'(U')) \ge \var(U) \big( 1 + \var(U) \big) +  \frac{1}{\E[U^3]} (\E[U^3 (\bar R_w-1)]) (\E[U^3 (\bar R_w+1)]), 
\end{equation}
with saturation when $(w,w')$ is strongly stationary. Strong stationarity implies purely environmental, hence selective equilibrium.
\end{cor}

% Aggregate version? 

%\begin{eqnarray}
%    \partial_\EC\!\left( \var(U), \var'(U') \right)
%\end{eqnarray}

% \subsection{Variance Across Generations}

%\section{Classical Extensions of the Price Equation} \label{sect_extensions}

%In \cite{price1972extension}, Price proved some useful special cases (``extensions'') of the Price equation: a hierarchical Price equation, incorporating information at earlier stages into a process change, and a differential Price equation relating infinitesimal changes in population, fitness, and local changes. We extend these results to the formal, measure-theoretic case. 

% The remainder of this section can be skipped by the reader who is not concerned with these particular cases or extensions. 

\section{Multi-Level Price Equation} \label{sect_multilevelprice}

We present the general form of Price's multi-level equation \cite{price1972extension} which is useful in applications. Our version of the multi-level Price equation relaxes the assumption of additivity, and enables a hierarchical decomposition for any composed process. Additivity arises from measures being linear objects, and admitting a linear disintegration into different scales. No other additivity assumptions are required. 

%Hamilton \cite{hamilton1975innate} notably used the multi-level Price equation to analyze selective effects happening between populations, integrating models of group selection and kin selection. A recent application is \cite[(3)]{aguilar2018gene}, where the authors consider phenotypes with additive genetic and cultural effects, and use the hierarchical Price equation to decompose selective forces into separate, interacting genetic and cultural parts. 

%This enables us to decompose the Price equation hierarchically: we first compute a marginal covariance of conditional variables, plus a marginal average, where we average over a conditional covariance plus a conditional local change term. Price originally stated and proved the multi-level Price equation \cite[(A 11)]{price1972extension}, and Price's work gives a formal basis to Hamilton's approach \cite{hamilton1964genetical} on integrating group selection with kin selection. 

%\begin{rem}

%\end{rem}

Consider the case of composable processes $w : \mu \mapsto \mu'$ and $w' : \mu' \mapsto \mu''$, with fitness functions $W(i) := w_i(I')$ and $W'(i') := w'_{i'}(I'')$, selective coefficients $\bar W := \E[W] = \frac{N'}{N}$ and $\bar W' := \E'[W'] = \frac{N''}{N'}$, and relative fitnesses $U := \frac{W}{\bar W}$ and $U' := \frac{W'}{\bar W'}$. 

Define the composed process $w^{(2)} := w' \circ w : \mu \mapsto \mu''$ by $w^{(2)}_i(C) := \int_{I'} w'_{i'}(C) w_i(\d i')$, and the composed fitness function
\begin{equation}
    W^{(2)}(i) := \langle W' \rangle_w(i) W(i) = \int_{I'} W'(i') w_i(\d i'),
\end{equation}
with selective coefficient $\bar W^{(2)} := \E[W^{(2)}] = \frac{N''}{N}$ and relative fitness $U^{(2)} := \frac{W^{(2)}}{\bar W^{(2)}} = \langle U' \rangle_w U$.

Consider observables $X$, $Y$, and $Z$ on $I$, $I'$, and $I''$, respectively. The individual Price equation gives us the following for the processes $w$, $w'$, and $w^{(2)}$:
\begin{eqnarray}
    \Delta(\bar X, \bar Y) &=& \cov(X,U) + \E[\Delta_w(X,Y) U] =: \partial_\NS(X) + \partial_\EC(X,Y) \\
    \Delta(\bar Y, \bar Z) &=& \cov'(Y,U') + \E'[\Delta_{w'}(Y,Z) U'] =: \partial'_\NS(Y) + \partial'_\EC(Y,Z) \\
    \Delta(\bar X, \bar Z) &=& \cov(X,U^{(2)}) + \E[\Delta_{w^{(2)}}(X,Z) U^{(2)}] =: \partial^{(2)}_\NS(X) + \partial^{(2)}_\EC(X,Z)
\end{eqnarray}
The hierarchical Price equation allows us to decompose the composed selective and environmental changes in terms of those of the first process, and a conditioning of the second process with a certain ``drift'' term. 

Write $\E'_w[Y](i) := \langle Y \rangle(i) U(i)$ for the conditional expectation (satisfying the tower property $\E[\E'_w[Y]] = \E'[Y]$), and $\cov'_w(Y,Y')(i) := \E'_w[Y Y'](i) - \E'_w[Y](i) \E'_w[Y'](i)$ for the conditional covariance (satisfying $\E[\cov'_w(Y,Y')] = \cov(Y,Y') + \E'[Y] \E'[Y'] - \E[\E'_w[Y] \E'_w[Y']]$). We have:
\begin{equation}
    \E'_w[U'](i) = \langle U' \rangle_w(i) U(i) = U^{(2)}(i)
\end{equation}
% Conditional covariance identity: \E[\cov'_w(Y,Y')] = \E[\E'_w[Y Y'] - \E'_w[Y] \E'_w[Y']] = \E'[Y Y'] - \E'[Y] \E'[Y'] + \E'[Y] \E'[Y'] - \E[\E'_w[Y] \E'_w[Y']]

%Define the multi-level commutator term, which measures the selective information for $Y$ not already transmitted from $w$:
%\begin{eqnarray}
%    D_{w'|w}(Y) &:=& \E'[Y U'] - \E[\E'_w[Y] \E'_w[U']] = \E[\E'_w[Y U'] - \E'_w[Y] \E'_w[U']] \nonumber \\
%    &=& \cov'(Y,U') - \cov(\E'_w[Y], \E'_w[U']). \label{eqn_multilevelcommutator}
%\end{eqnarray}

\begin{thm}[Multi-Level Price Equation] \label{thm_multilevelprice}
Let $w$ and $w'$ be composable processes, with $w^{(2)} := w' \circ w$. For any observables $Y$ and $Z$ on $I'$ and $I''$, respectively, we have the multi-level selective and environmental changes:
\begin{eqnarray}
    \partial'_\NS(Y) &=& \cov(\E'_w[Y], \E'_w[U']) + \E\!\left[\cov_w(Y,U')\right] \label{eqn_multilevelpriceNS} \\
    \partial'_\EC(Y,Z) &=& \E\!\left[\E'_w[\Delta_w(Y,Z) U']\right] \label{eqn_multilevelpriceEC}
\end{eqnarray}
and the multi-level Price equation:
\begin{equation} \label{eqn_multilevelprice}
    \Delta(\bar Y, \bar Z) = \cov(\E'_w[Y], \E'_w[U']) + \E\!\left[\cov_w(Y,U') + \E'_w[\Delta_w(Y,Z) U']\right]
\end{equation}
\end{thm}
\begin{proof}
For the proof of \eqref{eqn_multilevelpriceNS}, we compute:
\begin{eqnarray}
    \cov(\E'_w[Y], \E'_w[U']) + \E[\cov'_w(Y, U')]
    &=& \E[\E'_w[Y] \E'_w[U'] - \E'[Y] + \E'_w[YU'] - \E'_w[Y] \E'_w[U']] \nonumber \\
    &=& \E[\E'_w[YU'] - \E'[Y]] = \E'[YU'] - \E'[Y] \nonumber \\
    &=& \cov'(Y,U') = \partial'_\NS(Y).
\end{eqnarray}

For the proof of \eqref{eqn_multilevelpriceEC}, we compute:
\begin{eqnarray}
    \E[\E'_w[\Delta_w(Y,Z)U']] &=& \E'[\Delta_w(Y,Z)U'] = \partial'_\EC(Y,Z). 
\end{eqnarray}

%The multi-level Price equation \eqref{eqn_multilevelprice} follows as a trivial consequence. 
\end{proof}

\begin{cor}[Composed Multi-Level Price Equation]
For any observables $X$, $Y$ and $Z$ on $I$, $I'$ and $I''$, respectively, we have %the composed multi-level selective and environmental changes:
%\begin{eqnarray}
%    \partial''_\NS(X) &=& \label{eqn_composedmultilevelpriceNS}
%    \partial''_\EC(X,Z) &=& \label{eqn_composedmultilevelpriceEC}
%\end{eqnarray}
%hence 
the composed multi-level Price equation:
\begin{equation} \label{eqn_composedmultilevelprice}
    \Delta(\bar X, \bar Z) = \cov(X,U) + \cov(\E'_w[Y], \E'_w[U']) + \E[\Delta_w(X,Y) U] + \E\!\left[\cov_w(Y,U') + \E'_w[\Delta_w(Y,Z) U']\right]
\end{equation}
\end{cor}

\begin{cor}[Multi-Level Fisher's Theorem]
We have the multi-level versions of Fisher's fundamental theorem:
\begin{eqnarray}
%    0 &=& \Delta(\bar U, \bar U') = \var(U) + \E[\Delta_w(U,U') U], \\
    0 &=& \Delta(\bar U', \bar U^{(2)}) = \var(\E'_w[U']) + \E\!\left[\var'_w(U') + \E'_w[\Delta_w(U',U^{(2)}) U']\right], \\
    0 &=& \Delta(\bar U, \bar U^{(2)}) = \var(U) + \var(\E'_w[U']) + \E[\Delta_w(U,U') U] + \E\!\left[\var'_w(U') + \E'_w[\Delta_w(U',U^{(2)}) U']\right].
\end{eqnarray}
\end{cor}
\begin{proof}
We set $X = U$, $Y = U'$, and $Z = U^{(2)} = \langle U' \rangle U = \E'_w[U']$, and apply the multi-level Price equation.
\end{proof}

\begin{rem}[Stochastic Price Equation]
Rice \cite{rice2020universal} interprets the multi-level Price equation in a stochastic framework. Suppose that $I$ is some statistical parameter space with a distribution $\mu$, and let $w_i$ be a measure on $I'$ varying measurably in the parameter value $i \in I$. Let $\mu'$ be the measure given by integrating over parameter values $i$ (i.e., $\mu'(B) = \int_I w_i(B) \mu(\d i)$). Consider a process $w' : \mu' \mapsto \mu''$, for some measure $\mu''$ on $I''$. Then Rice's stochastic Price equation \cite[(2.2)]{rice2020universal} is exactly the multi-level Price equation \eqref{eqn_multilevelprice}. 

% Price's multi-level framework includes the cases of stochastic, thermodynamic, and time-varying processes as special cases, since the initial space $I_0$ can represent any external parameters of interest, in which case $\kappa_{i_0}$ represents a parameter-varying population measure. In the stochastic case, the three terms of \eqref{eqn_twolevelpriceeqn} correspond to the three terms of Rice's stochastic Price equation \cite[(2.2)]{rice2020universal}.
\end{rem}

\section{Smooth Price Equation} \label{sect_timevaryingpriceequation}

We state and prove a general time-varying version of the Price equation, generalizing Price's informal time-varying equation Price \cite[(A 23)]{price1972extension}. Page and Nowak \cite[(4)]{page2002unifying} stated Price's time-varying equation without precise definition as follows:
\begin{equation} \label{eqn_continuoustime}
    \dot \E(P) = \cov(P,U) + \E(\dot P),
\end{equation}
where $\E$ is a time-varying expectation, $U$ is a time-varying fitness function, and $P$ is a time-varying observable. However, this is ambiguous and 
%can lead to different conclusions based on regularity of spaces, and 
needs a more precise formalism. 

%``hides'' considerable nuance in the derivation, regarding different conclusions based on differentiability of the spaces and functions involved. We present a precise version of this result.

Let $T \subseteq \R$ be an open set representing a time domain, and let $(I^t,\I^t,\mu^t)$ be a (possibly varying) family of measurable spaces. For each pair $t \le t'$ in $T$, let $w^{t,t'} : \mu^t \mapsto \mu^{t'}$ denote a time-varying evolutionary process, i.e., a transition kernel satisfying \eqref{eqn_disintegration} and the temporal consistency condition $w^{t,t''} = w^{t',t''} \circ w^{t,t'}$ for any $t,t',t'' \in T$. Define the time-varying fitness $W^{t,t'}(i^t) := w_{i^t}(I^{t'})$ and the relative fitness $U^{t,t}(i^t) := W^{t,t'}(i^t)/\bar W^{t,t'}$. 

Let $\X  = (X^t)$ be a time-varying finite-variance family of observables, with means $\bar X^t := \E^t[X^t]$. Define the local-average $\langle X^{t'} \rangle^{t,t'}_w(i^t) := \frac{1}{W^{t,t'}(i^t)} \int_{I^t} X^{t'}(i^{t'}) w^{t,t'}_{i^t}(\d i^{t'})$ and local-change $\Delta^{t,t'}_w(X^t,X^{t'})(i^t) := \langle X^{t'} \rangle^{t,t'}_w(i^t) - X^t(i^t)$. 

For any $t \le t'$, the discrete-time Price equation holds:
\begin{equation} \label{eqn_discretetimeprice}
    \Delta^{t,t'}(\bar X^t, \bar X^{t'}) = \cov^t(X^t, U^{t,t'}) + \E^t[\Delta_w^{t,t'}(X^t,X^{t'}) U^{t,t'}].
\end{equation}
%This follows immediately from Theorem \ref{thm_price}.

\begin{defn}[Smooth Evolutionary Processes] \label{def_smoothprocesses}
Consider a time-varying process $w := (w^{t,t'})$ and a time-varying family of observables $\X := (X^t)$. We say that $w$ is a \emph{smooth evolutionary process} at $X^t$ if the following hold: 
\begin{enumerate}
    \item (Smooth Expectations) The time-varying average is smooth at $t$, i.e., the limit of real numbers is well-defined:
        \begin{equation} \label{eqn_smoothexpectation}
            \frac{\d \E^t}{\d t}[X^t] := \lim_{t' \downarrow t} \frac{\Delta^{t,t'}(\bar X^t,\bar X^{t'})}{|t'-t|} = \lim_{t' \downarrow t} \frac{\E^{t'}[X^{t'}] - \E^t[X^t]}{|t'-t|}.
        \end{equation}
    \item (Relative-Fitness Density) The time-varying relative fitness admits a density at $t$, where we take the $L^2(\mu^t)$-limit:
        \begin{equation} \label{eqn_relativefitnessdensity}
             \Upsilon^t(i) := \Ltwolim_{t' \downarrow t} \frac{U^{t,t'}(i)}{|t'-t|}.
        \end{equation}
    \item (Local-Change Density) The time-varying local change admits a density at $t$, where we take the $L^1(\mu^t)$-limit:
        \begin{equation} \label{eqn_localchangedensity}
            \delta^t(\X)(i) := \Lonelim_{t' \downarrow t} \frac{\Delta^{t,t'}_w(X^t,X^{t'})}{|t'-t|} := \Lonelim_{t' \downarrow t} \frac{\langle X^{t'} \rangle_w^{t,t'}(i) - X^t(i)}{|t'-t|}. %\frac{\Delta^{t,t'}(X^t,X^{t'})(i)}{|t'-t|} = \Lonelim_{t' \downarrow t} \frac{\langle X^{t'} \rangle_w^{t,t'}(i) - X^t(i)}{|t'-t|}.
        \end{equation}
\end{enumerate}
\end{defn}

The local-change density depends on $w^{t,t'}$ and $X^{t'}$ for $t'$ in an infinitesimal vicinity of $t$. 

%\begin{lem}
%Suppose $w = (w^{t,t'})$ is a smooth evolutionary process at $X^t$. Then $\E^t[\Upsilon^t] = 1$, and for any observable $X$, $\E^t[\delta^t(\X)]$ = 
%\end{lem}
%\begin{proof}
%First, note that $\Upsilon^t$ is also the $L^1(\mu^t)$-limit of $U^{t,t'}$, since by Jensen's inequality, $\E^t[|\U^{t,t'}-\Upsilon^t|] \le \sqrt{\}$
%For each $t,t'$, we have $\E^t[U^{t,t'}] = 1$, hence the bounded convergence theorem lets us take 
%\end{proof}

\begin{thm}[Time-Varying Price Equation] \label{thm_smoothprice}
Suppose $w = (w^{t,t'})$ is a smooth process at $X^t$. %Consider the relative-fitness density $\Upsilon^t$, and for any time-varying observables $X = (X^t)$, consider the local-change density $\delta^t(\X)$. 
Then the smooth Price equation holds at $t$:
\begin{equation} \label{eqn_timevaryingprice}
    \frac{\d \E^t}{\d t}\!\left[X^t\right] = \cov^t\!\left(X^t,\Upsilon^t\right) + \E^t\!\left[\delta^t(\X)\right].
\end{equation}

Furthermore, if $w$ is a process satisfying two of the assumptions of Definition \ref{def_smoothprocesses} at $t$, then the third assumption holds, and so $w$ is smooth at $t$ and the Price equation \eqref{eqn_timevaryingprice} holds.
%smooth Price equation holds at $t$ and the third assumption is satisfied, so $w$ is a smooth process at $X^t$.
\end{thm}
\begin{proof}
We begin with the time-varying Price equation \eqref{eqn_discretetimeprice}, and divide both sides by $|t'-t|$:
\begin{equation} \label{eqn_priceequation_discretedifferential}
    \frac{\Delta^{t,t'}(\bar X^t, \bar X^{t'})}{|t'-t|} = \cov^t\!\left(X^t, \frac{U^{t,t'}}{|t'-t|} \right) + \E^t\!\left[\frac{\Delta(X^t,X^{t'}) U^{t,t'}}{|t'-t|} \right].
\end{equation}

Assuming smooth expectations \eqref{eqn_smoothexpectation}, for the left side, when we take $t' \downarrow t$, we get $\frac{\d \E^t}{\d t}[X^t]$. 

Assuming a relative-fitness density \eqref{eqn_relativefitnessdensity}, for the selective term, we have the $L^2$-limit: $\frac{U^{t,t'}}{|t'-t|} \to \Upsilon^t$ in $L^2$ as $t' \downarrow t$. Since $\mu^t$ is a finite measure, this implies convergence in covariance, hence for the linear operator $\cdot \mapsto \cov(X^t,\cdot)$. This proves $\cov^t\!\left(X^t, \frac{U^{t,t'}}{|t'-t|} \right) \to \cov^t(X^t,\Upsilon^t)$

Assuming a local-change density \eqref{eqn_localchangedensity}, for the environmental term, we use the product rule for $L^1$-differentiation as $t' \downarrow t$:
\begin{eqnarray}
    \lim_{t' \downarrow t} \E^t\!\left[\frac{\Delta_w^{t,t'}(X^t,X^{t'}) U^{t,t'}}{|t'-t|}\right] = \E^t\!\left[\delta^t(\X) U^{t,t} + \Delta_w^{t,t}(X^t,X^t) \Upsilon^t\right] = \E^t\!\left[\delta^t(\X) \right],
\end{eqnarray}
since $U^{t,t}(i) = 1$ and $\Delta^{t,t}(X^t,X^t)(i) = 0$ almost surely. Thus 
%the smooth Price equation 
\eqref{eqn_timevaryingprice} holds at $t$. 

If we assume only two assumptions of Definition \ref{def_smoothprocesses}, then we use \eqref{eqn_discretetimeprice} to put the two Cauchy sequence terms on one side, ensuring that the remaining term is also a Cauchy sequence, and the corresponding limits \eqref{eqn_smoothexpectation}, \eqref{eqn_relativefitnessdensity}, and \eqref{eqn_localchangedensity} all exist. This proves the missing assumption, and so $w$ is smooth.
\end{proof}

\begin{conj}
The author conjectures that there exist non-smooth processes satisfying any one of the assumptions of Definition \ref{def_smoothprocesses} without satisfying the other two assumptions. For those processes, the Price equation \eqref{eqn_timevaryingprice} would not hold. 
\end{conj}

Consider the situation where we have information of a process up to time $t$, and wish to understand new infinitesimal evolution happening at time $t$. We can apply the multi-level Price equation to each time interval, then take $t' \downarrow t$. We follow the notation of the previous section for conditional expectation and covariance. 

\begin{cor}[Time-Varying Multi-Level Price Equation]
Suppose $w = (w^{t,t'})$ is a smooth evolutionary process at $X^t$. Then for any time-varying family of observables $X = (X^t)$, the time-varying, multi-level Price equation holds at $t$:
\begin{equation} \label{eqn_timevaryingmultilevelpriceeqn}
    \frac{\d \E^t}{\d t}[X](i) = \cov^t\!\left(\E^t_{w^{t_0,t}}[X^t], \E^t_{w^{t_0,t}}[\Upsilon^t]\right) + \E^t\!\left[ \cov^t_{w^{t_0,t}}(X^t,\Upsilon^t) + \E^t_{w^{t_0,t}}[\Delta_{w^{t,t'}}(X^t,X^{t,t'}) U^{t,t'}] \right].
\end{equation}
\end{cor}
\begin{proof}
For each $t_0 \le t \le t'$, we apply the discrete multi-level Price equation:
\begin{equation}
    \Delta(\bar X^t, \bar X^{t'}) = \cov^t(\E^t_{w^{t_0,t}}[X^t], \E^t_{w^{t_0,t}}[U^{t,t'}]) + \E^t\!\left[\cov^t_{w^{t_0,t}}(X^t,U^{t,t'}) + \E_{w^{t_0,t}}[\Delta_{w^{t,t'}}(X^t,X^{t,t'}) U^{t,t'}]\right].
\end{equation}

By the bounded convergence theorem, for fixed $t_0, t$, when we take $t' \downarrow t$, we have $\cov^t(\E^t_{w^{t_0,t}}[X^t], \E^t_{w^{t_0,t}}[U^{t,t'}]) \to \cov^t\!\left(\E^t_{w^{t_0,t}}[X^t], \E^t_{w^{t_0,t}}[\Upsilon^t]\right)$; $\E^t\!\left[ \cov^t_{w^{t_0,t}}(X^t,U^{t,t'}) \right] \to \E^t\!\left[ \cov^t_{w^{t_0,t}}(X^t,\Upsilon^t) \right]$; and $\E^t\!\left[ \E^t_{w^{t_0,t}}[\Delta_{w^{t,t'}}(X^t,X^{t,t'}) U^{t,t'}] \right] \to \E^t\!\left[ \E_{w^{t_0,t}}[\delta^t(\X)] \right]$. This proves \eqref{eqn_timevaryingmultilevelpriceeqn}.
\end{proof}

\section{Quantum Price Equation} \label{sect_quantum}

% REMOVED ABSTRACT PRICE EQUATION -- Start immediately with quantum structure, no need to distract

% We first prove abstract Price equation for non-commutative observables. There are two versions: one where selection acts upon the left, and another where selection acts upon the right. We then apply it to the quantum setting, where populations are defined by density operators on Hilbert spaces and observables are self-adjoint operators. Because of non-commutativity of operators, the abstract Price equation implies a quantum Price equation.

% \subsection{Non-Commutative Price Equation}

% Let $\A$ and $\A'$ be arbitrary algebras of observables (possibly non-commutative). Let $\E : \A \to \R$ and $\E' : \A' \to \R$ be arbitrary functions, representing expectation functionals in place of populations.

% \begin{defn}
% We define a \emph{left (resp. right)abstract evolutionary operator} to be a tuple $w := (U, \langle \cdot \rangle)$ consisting of:
% \begin{enumerate}
     % \item (Relative Fitness) An observable $U \in \A$ satisfying $\E[U] = 1$,  and
     % \item (Local Average) For each observable $Y \in \A'$, there exists an observable $\langle Y \rangle \in \A$ satisfying $\E'[Y] = $
% \end{enumerate}
% \end{defn}

% \begin{thm}[Non-Commutative Price Equation]
    % test
% \end{thm}

%\subsection{Quantum Price Equation}

We present a novel quantum form of the Price equation. Note: our presentation is unrelated to the ``quantum evolution'' of Simpson \cite{simpson1944tempo}. For a brief overview of quantum mechanics in general, see \cite[p.~65]{takhtadzhian2008quantum} or \cite{milz2017introduction}.

Let $\H$ and $\H'$ be arbitrary Hilbert spaces (not necessarily separable). Let $\A := \A(\H)$ and $\A' := \A'(\H')$ denote the spaces of observables, i.e., the self-adjoint operators on the Hilbert spaces. Let $\mu : \H \to \H$ and $\mu' : \H' \to \H'$ be trace-class, self-adjoint density operators with non-negative finite traces: $0 \le N := \Tr \mu < \oo$ and $0 \le N' := \Tr' \mu' < \oo$. We allow for unbounded observables and unbounded, trace-class density operators. Write $\bar W := \frac{N'}{N}$ for the ratio of population sizes, i.e., the quantum selective coefficient.  %satisfying $\Tr \mu = 1$ and $\Tr \mu^2 < \oo$, and $\Tr' \mu' = 1$ and $\Tr'(\mu')^2 < \oo$. 

The operators $\mu$ and $\mu'$ represent ``mixed states'' of quantum populations, and the non-negative real numbers $N$ and $N'$ represent the quantum ``sizes'' of the population. Just as classical evolutionary theory allows for populations of variable size, quantum evolutionary theory allows for quantum populations of variable size, and this variability is what drives quantum selective effects. The case where $N = 1$ and $N' \le 1$ is common in quantum computation, representing \cite{nielsen2002quantum,weedbrook2012gaussian}. 

Any population operator $\mu$ defines a measure on its Hilbert space $\H$ via the push forwards of the volume measures: $\mu_*(\lambda)(E) := \lambda(\mu^{-1} E)$. Consequently, we can make statements up to $\mu$-almost everywhere on $\H$ and $\mu'$-almost everywhere on $\H'$.

Let $\M := \M(\H)$ and $\M' := \M'(\H')$ denote the spaces of density operators. Define the population mean operators by normalizing the trace operators by quantum population sizes:
\begin{equation}
\E_\mu[X] := \frac{1}{N} \Tr(X\mu) \quad \mathrm{and} \quad \E'_{\mu'}[Y] := \frac{1}{N'} \Tr'(Y\mu').
\end{equation}

Define the average change between observables by
\begin{equation}
\Delta(X,Y) := \E'_{\mu'}[Y] - \E_\mu[X] = \frac{1}{N'} \Tr'(Y\mu') - \frac{1}{N} \Tr(X\mu).
\end{equation}

Define the population covariance operator by
\begin{equation}
\cov_\mu(X_0,X_1) := \E_\mu[X_1 X_0] - \E_\mu[X_1] \E_\mu[X_0] = \frac{1}{N} \Tr(X_1 X_0 \mu) - \frac{1}{N^2} \Tr(X_0 \mu) \Tr(X_1 \mu).
\end{equation}

% A ``measurement'' of observable $A$ in state $\mu$ is governed by a measure-valued correspondence $(A,\mu) \mapsto \mu_A$, which is essentially the quantum push-forward of $\mu$ through $A$.  
% The Born-von Neumann formula \cite[p.~70, Chapter 2, Formula (1.3)]{takhtadzhian2008quantum} states that this measure is the quantum push-forward of $\mu$ through $A$:
%\begin{equation}
	% \mu_A(E) = \Tr (P_A(E) \mu)
%\end{equation}
%for any Borel measurable $E \in \B(\R)$, where $P_A$ is the projection-valued measure on $\R$ corresponding to $A$. 
% The expected value of $A$ given state $\mu$ is defined by $\E_\mu[A] := \Tr(A\mu) = \int_\R \alpha \mu_A(\alpha)$. 

We define a quantum evolutionary operation to be a measurable linear map which sends the non-negative cone $\M$ to the non-negative cone $\M'$. We do not need quantum operations to be trace-preserving or completely positive.

% We demonstrate a quantum form of the Price equation, decomposing the average change of observables into selective and environmental components. This corresponds to a decomposition of the process as first multiplication by the fitness operator (the selective part), followed by a trace-preserving map (the environmental part). 

%For the quantum Price equation, 

% quantum operation which does not necessarily preserve trace nor change it monotonically. Trace-decreasing maps are already standard objects in quantum computation; our model allows for trace-increasing maps too.

%Trace-preserving maps are exactly those with no selective part. 

%We will show in Corollary that a quantum evolutionary operator can be written as the composition of a diagonal operator, representing a selective quantum population scaling, followed by a trace-preserving, memoryless quantum channel, representing a quantum environmental change. 

%The quantum Price equation will let us decompose the average change in terms of the selective part and the environmental part of a quantum evolutionary operator, with no other terms. 

\begin{defn}[Quantum Evolutionary Operator as Quantum Channel]
We say that $\W : \M \to \M'$ is a \emph{quantum evolutionary operator} transforming $\mu$ into $\mu'$ if the following hold:
\begin{itemize}
\item The linear operator $\W : \M \to \M'$ is measurable. %and trace-class. %a trace-class, compact operator of Hilbert spaces,
\item The child population is fully accounted for by the parent population via the process: 
\begin{equation} \label{eqn_disintegration_quantum}
    \mu' = \W(\mu).
\end{equation}
% \item The child population is fully accounted for by the parent population via the process: $\mu' \W = \W \mu$ Almost everywhere and $\mu' = 0$ away from the range $\W \H$. i.e. $\mu' \W h = \W \mu h$ for $\mu$-almost every $h$ and $\mu'(h') = 0$ for $\mu'$-almost every $h'$.
%\item For every $A'$, there exists a self-adjoint ``local average'' operator $\langle A' \rangle_\W$ on $\H$ satisfying $\Tr'(A' \mu') = \Tr(\langle A' \rangle_\W \mu)$. 
% $\Tr'(A' \mu') =  \Tr(A' \W \mu)$. Because $\W$ is trace-class and compact, the trace of the operator $A' \W \mu : \H \to \H'$ is well-defined and comparable to the trace of $A' \mu' : \H' \to \H'$.
% \item \muaybe: $\mu' \W = \W \mu$ ?
% NOT WELL DEFINED AS FUNCTIONS ON SA\muE SPACE: \item  $\mu' = \W \mu$, i.e., $\mu' h = (\W \mu) h$ as elements of $H'$.,
% OPTION: \item $\mu' \circ \W = \mu$, i.e., $\mu'(\W h) = \mu h$ for  all $h \in \H$, % Want something like $\mu' = w \circ \mu$, and $\E'[Y] = \E[\langle Y \rangle_w U]$. Those formulas would correspond to $\Tr'(A'\mu') = \Tr\!\left((A' \circ \W)\mu\right) ?
%\item $\W$ is completely non-negative, in the sense that $\W$ maps the non-negative cone in $\\mu$ to the non-negative cone in $\\mu'$, and if $C$ is an ``ancilla'', i.e., a finite-dimensional non-negative operator on some finite Hilbert space $\H_{\operatorname{ancillary}}$, then the tensor operator $C \otimes \W$ is non-negative on the tensor space $\H_{\operatorname{ancillary}} \otimes \H$.
\end{itemize}

We say that $\W_\NS : \M \to \M$ is purely left- (resp. right-) selective if it is given by left- (resp. right) multiplication by a self-adjoint operator (i.e., $\W_\NS(\mu) = W^\opleft \mu$ for some $W^\opleft \in \A$, resp. $\W_\NS(\mu) = \mu W^\opright$ for some $W^\opright \in A$), and that $\W_\EC : \M \to \M'$ is purely environmental if it is trace-preserving. 
%is a memoryless quantum channel, i.e., a map between density operators satisfying the following properties: $\W$ is a trace-preserving linear operator which preserves the positive cone of density operators, and is completely positive.
%    We say that $\W : \H \to \H'$ is a \emph{quantum evolutionary operator} from $\mu$ to $\mu'$ (and write $\W : \mu \mapsto \mu'$) if $\mu' \circ \W = \mu$. Formally, if $\mu'(\W h) = \mu h$ for all $h \in H$.
\end{defn}

\begin{rem}
Classical quantum channels are the trace-preserving, completely positive maps, which are a subclass of purely environmental maps. Not-completely-positive, trace-preserving operations are still purely environmental, because they preserve trace. Trace-decreasing quantum operations ha
ve $\bar W < 1$, and therefore admit selective effects. %and a two-term Price equation. 
\end{rem}

%For each $h \in \H$, let $\pi_h$ denote the projection operator onto the subspace generated by $h$. The density operator $\pi_h$ represents the pure state corresponding to $h$. 

% What if quantum fitness is $\W^\dagger \Id'$? Then $\Tr(WM) = \Tr((\W^\dagger \Id') M) = \Tr'(\Id' \W(M)) = N'$. x 

\subsection{Quantum Selective Change}

A quantum evolutionary process admits an adjoint process, which defines a fitness operator. 

\begin{lem}[Quantum Adjoint]
Let $\W$ be a quantum evolutionary operator. There exists an adjoint operator $\W^\dagger : \A' \to \A$ satisfying 
\begin{equation}
    \Tr(\W^\dagger(Y) \mu) = \Tr'(Y \W(\mu))
\end{equation}
for all $\mu \in \M$. The adjoint does not depend on $\mu$. 
\end{lem}
\begin{proof}
Since $\M$ and $\M'$ are topological linear spaces, with dual spaces $\A$ and $\A'$ and dual product given by the trace functionals $\Tr$ and $\Tr'$, the adjoint is well-defined. 
\end{proof}

%\subsection{Quantum Selective Change}

%\begin{defn}[Quantum Fitness Operator]
Define the quantum fitness operator for $\W$ as the pullback of the identity $\Id'$ on $I'$ via the adjoint, and the quantum relative-fitness operator by scaling by the selective coefficient:
\begin{equation}
    W := \W^\dagger(\Id') \quad \mathrm{and} \quad U := \frac{1}{\bar W} W = \frac{1}{\bar W} \W^\dagger(\Id').
\end{equation}
%\end{defn}

\begin{lem}
The fitness operator has mean equal to the selective coefficient, and the relative fitness has mean equal to one:
\begin{equation} \label{eqn_quantumaveragefitness}
    \E_\mu[W] = \bar W \quad \mathrm{and} \quad \E_\mu[U] = 1. 
\end{equation}
\end{lem}
\begin{proof}
Using the property of the adjoint, we have
%\begin{equation}
    $N \E_\mu[W] = \Tr(W\mu) =  \Tr(\W^\dagger(\Id') \mu) =  \Tr'(\Id' \W(\mu)) =  \Tr'(\mu') = N' = \bar W N$.
%\end{equation}
\end{proof}

Define the quantum selection changes $\partial_\NS^\opleft(X) := \cov(X,U)$ and $\partial_\NS^\opright(X) := \cov(U,X)$. In general non-commutative settings, these functionals are distinct, and are related by 
\begin{equation}
    \cov(X,U) = \cov(U,X) + \E_\mu\!\left[ [X,U] \right],
\end{equation}
for the commutator $[X,U] = XU-UX$. Proof: $\cov(X,U) = \Tr(XU\mu) - \E_\mu[X] \E_\mu[U] = \Tr(UX\mu) - \E_\mu[U] \E_\mu[X] + \Tr([X,U]\mu) = \cov(U,X) + \E_\mu\!\left[ [X,U] \right]$.

%$ = \Tr(XU\mu) - \Tr(X\mu) \Tr(U\mu) = \Tr(UX\mu) - \Tr(U\mu) \Tr(X\mu) + \Tr((XU-UX)\mu) = \cov(U,X) + \E_\mu\!\left[ [X,U] \right]

\subsection{Quantum Environmental Change}

% Suppose things are separable, and let $\e_\alpha$ be a basis of $\H$, and let $\e^\alpha$ denote the dual basis. Then we can write $\mu$ in coordinates as \mu_{\alpha \beta} = \langle \e_\beta, \mu \e_\alpha \rangle$. Then $(W\mu)_{\alpha \beta} = W_{\alpha \alpha} \mu_{\alpha \beta}$. 

We define left and right local-average operators, by pre-composing or post-composing the adjoint operator with the inverse-fitness operator. Formally, for each $Y \in \A'$, we define the left local-average $\langle Y \rangle_\W^\opleft \in \A$ on the subspace $W \H$ and the right local-average $\langle Y \rangle_\W^\opright \in \A$ on the full space $\H$ by:
\begin{equation}
    \langle Y \rangle_\W^\opleft := (\W^\dagger Y) W^{-1} \qquad \mathrm{and} \qquad \langle Y \rangle_\W^\opright := W^{-1} (\W^\dagger Y).
\end{equation}
Formally, the left and right local-average operators are related by the identity $\langle Y \rangle^\opleft_\W W = \W^\dagger Y = W \langle Y \rangle^\opright_\W$, and satisfy the quantum tower property:
\begin{equation} \label{eqn_quantumtowerproperty}
    \E_\mu[\langle Y \rangle_\W^\opleft U] = \E'_{\mu'}[Y] = \E_\mu[U \langle Y \rangle_\W^\opright].
\end{equation}
We define the left and right local-change operators from $X$ to $Y$ by:
\begin{equation}
    \Delta_\W^\opleft(X,Y) := \langle Y \rangle^\opleft_\W - X \qquad \mathrm{and} \qquad \Delta_\W^\opright(X,Y) := \langle Y \rangle^\opright_\W - X.
\end{equation}

Define the quantum environmental changes by $\partial_\EC^\opleft(X,Y) := \E_\mu[\Delta_\W^\opleft(X,Y) U]$ and $\partial_\EC^\opright(X,Y) := \E_\mu[U \Delta_\W^\opright(X,Y)]$. These are related by:
\begin{equation}
    \partial_\EC^\opleft(X,Y) = \frac{1}{\bar W} \E_\mu[(\W^\dagger Y) - XW] = \partial_\EC^\opright(X,Y) + \E_\mu\!\left[ [U,X] \right].
\end{equation}

\subsection{Quantum Price Equations}

%The left and right local changes are given as the difference of the local-average with another observable. 

% We define the local-average operator by applying the adjoint, pre-composed with the inverse fitness operator. The local change is given as the difference of the local-average with another observable. 

% \begin{defn}[Quantum Local Average and Local Change Operators]~
% \begin{enumerate}
%     \item For each $Y \in \A'$, define the left quantum local-average operator $\langle Y\rangle_\W \in \A$ on the subspace $W \H$ and the right quantum local 
        % \begin{equation}
            % \langle Y \rangle_\W := (\W^\dagger Y) W^{-1}
        % \end{equation}
    % \item For each $X \in \A$ and $Y \in \A'$, define the quantum local-change operator on $W \H$ by
        % \begin{equation}
            % \Delta_\W(X,Y) := \langle Y \rangle_\W - X.
        % \end{equation}
% \end{enumerate}
% \end{defn}

% We define the quantum selective change and quantum environmental change for $\W$ as follows:
% \begin{equation}
    % \partial_\NS(X) := \cov_\mu(X,U) \quad \mathrm{and} \quad \partial_\EC(X,Y) := \E_\mu[\Delta_\W(X,Y) U].
% \end{equation}

\begin{thm}[Quantum Price Equations]
For each $X \in \A$ and $Y \in \A'$, the left and right quantum Price equations are satisfied: 
\begin{equation}
    \Delta(\bar X, \bar Y) = \partial_\NS^\opleft(X) + \partial_\EC^\opleft(X,Y) 
    = \partial_\NS^\opright(X) + \partial_\EC^\opright(X,Y),
\end{equation}
that is,
\begin{eqnarray} 
    \Delta(\bar X, \bar Y) &=& \cov_\mu(X,U) + \E_\mu[\Delta_\W^\opleft(X,Y) U] \label{eqn_quantumprice_left} \\
    &=& \cov_\mu(U,X) + \E_\mu[U \Delta_\W^\opright(X,Y)]. \label{eqn_quantumprice_right}
    %&=& \partial_\NS(X) + \partial_\EC(X,Y).
\end{eqnarray}
\end{thm}
\begin{proof}
Using the definitions constructed, the proof is trivial by adding and subtracting $XU$ (resp. $UX$) from the integrand, as with the classical case:
\begin{eqnarray}
    \Delta(X,Y) &=& \frac{1}{N'} \Tr'(Y\mu') - \frac{1}{N} \Tr(X\mu) \nonumber \\
    &=& \frac{1}{N} \Tr(XU\mu - X\mu) + \frac{1}{N'} \Tr\!\left(\langle Y \rangle^\opleft_\W U\mu - XU\mu\right) \\
    &=& \frac{1}{N} \Tr(UX\mu - X\mu) + \frac{1}{N'} \Tr\!\left(U \langle Y \rangle^\opright_\W \mu - UX\mu\right)
\end{eqnarray}
which yields \eqref{eqn_quantumprice_left} and \eqref{eqn_quantumprice_right} since $\E_\mu[U] = 1$ and $\E_\mu[\langle Y \rangle^\opleft_\W U] = \E'_{\mu'}[Y] = \E_\mu[U \langle Y \rangle^\opright_\W]$.
\end{proof}

The quantum version of Fisher's theorem follows. %Note that $[U,U] = 0$ so the left and right selective changes are equal.

%the quantum selective change of relative fitness does not depend on left or right orientation: $\partial_\NS(U) := \partial_\NS^\opleft(U) = \partial_\NS^\opright(U) = \var(U)$. 

\begin{cor}[Quantum Fisher's Theorem]
Let $\W : \mu \mapsto \mu'$ and $\W' : \mu' \mapsto \mu''$ be composable quantum evolutionary processes, with relative fitness operators $U$ and $U'$, respectively. Then the quantum form of Fisher's fundamental theorem holds:
\begin{eqnarray}
    0 = \Delta(\bar U, \bar U') = \var(U) + \E_\mu[\Delta^\opleft_\W(U,U') U] 
    = \var(U) + \E_\mu[U \Delta^\opright_\W(U,U')].
\end{eqnarray}
% We have \Delta^\opleft_\W(U,U') = 
\end{cor}

Applied to the difference of traces, we have:
\begin{eqnarray} 
\Tr'(Y\mu') - \Tr(X\mu) &=& \Tr(X(W-1)\mu) +  \Tr\!\left((\langle Y \rangle_\W^\opleft - X)W\mu\right) \label{eqn_quantumpricetrace_left} \\
    &=& \Tr\!\left((W-1)X\mu) + \Tr(W(\langle Y \rangle_\W^\opright - X)\mu\right). \label{eqn_quantumpricetrace_right}
    %&=& \partial_\NS(X) + \partial_\EC(X,Y).
\end{eqnarray}
following the same algebra as the classical case (Corollary \ref{cor_aggregateprice}).

\subsection{Quantum Price Representation Theorem} 

A quantum Price representation theorem follows. Define the purely selective operation $\W_\NS : \M \to \M$ as left-multiplication by $W$ (i.e., $\W_\NS(\mu) := W\mu$), and the purely environmental operation $\W_\EC : \M \to \M'$ by precomposing $\W$ with inverse fitness (i.e., $\W_\EC(\tilde \mu) := \W(W^{-1} \tilde \mu)$. 

\begin{cor}[Quantum Price Representation Theorem]
The Price decomposition holds:
\begin{equation} \label{eqn_quantumpricerepresentation}
    \W = \W_\EC \circ \W_\NS.
\end{equation}
The selective (resp. environmental) change of $\W$ equals that of $\W_\NS$ (resp. $\W_\EC$).
\end{cor}
\begin{proof}
Decomposition \eqref{eqn_quantumpricerepresentation} follows trivially from the definition. The operation $\W_\EC$ is trace-preserving since $\Tr'(\W_\EC(\tilde \mu)) = \Tr'(\W(W^{-1} \tilde \mu)) = \Tr(\W^\dagger(\Id') W^{-1} \tilde \mu) = \Tr(\tilde \mu)$. 

Note that $\Tr(\W_\NS^\dagger(\Id) \mu) = \Tr(W\mu)$ hence $W_\NS = W$. Thus $\partial_{\W_\NS,\NS}(X) = \cov_\mu(X,U) = \partial_{\W,\NS}(X)$. %We also have $\partial_{\W_\NS,\EC}(X,X) = $
Note that $\Tr(\W_\EC^\dagger(\Id') W\mu) = \Tr'(\W(W^{-1} W\mu)) = \Tr'(\mu') = \Tr(W\mu)$ hence $W_\EC = \Id$ and so $\langle Y \rangle_{\W_\EC} = \langle Y \rangle_\W$. Thus $\partial_{\W_\EC,\EC}(X,Y) = \E_{W\mu}[\Delta_{\W_\EC}(X,Y) \Id] = \E_\mu[\Delta_\W(X,Y) U] = \partial_{\W,\EC}(X,Y)$. 
\end{proof}

\subsection{Quantum Jensen's Inequality}

We present a lower bound for the variance, corresponding to a quantum version of the Zeroth Law (Proposition \ref{pro_zerothlaw}). First, we need a version of Jensen's inequality for weighted trace functionals which we can apply to the quantum setting. If $f$ is a real-valued function and $X$ is a self-adjoint operator, let $f(X)$ denote the self-adjoint operator defined using the spectral theorem.

\begin{lem}[Quantum Jensen's Inequality] \label{lem_quantumjensen}
Let $\mu$ be a finite-trace density operator, with expectation $\E_\mu[X] := \frac{1}{N} \Tr(X\mu)$. For any convex function $f$ and self-adjoint operator $X$:
\begin{equation}
    \E_\mu[f(X)] \ge f\!\left(\E_\mu[X]\right).
\end{equation}
Saturation holds if and only if the operator $X$ is constant $\mu$-almost everywhere.
\end{lem}

The proof is similar to the standard measure-theoretic proof. See Appendix \ref{app_quantumjensenproof}.

\subsection{Zeroth Law of Quantum Selection}

The quantum Jensen's inequality allows us to quantize inequalities for convex functionals presented in this article. 
Let $\pi_* := \pi_{U\ne 0} = \Id - \pi_{U=0}$ be the projection operator onto the subspace orthogonal to the null space. Write $\mu_* := \pi_* \mu$, $p_* := \E_\mu[\pi_*] = \frac{1}{N} \Tr(\mu_*)$, and $\E_*[X] := \E_{\mu_*}[X] = \Tr(X\mu_*) = \Tr(X \pi_* \mu)$. We say that $\W$ is in \emph{quantum selective equilibrium} if the fitness operator is constant $\mu_*$-almost surely (in which case $W = \frac{1}{p_*} \Id$ $\mu_*$-a.s.), or equivalently, if $W \in \{0, \frac{1}{p_*} \Id\}$ $\mu$-a.s.

\begin{pro}[Weak Zeroth Law of Quantum Selection]
Let $\W$ be a quantum evolutionary process. Then:
\begin{equation}
    \partial_\NS(U) = \var_\mu(U) \ge \frac{1}{p_*} - 1.
\end{equation}
This is saturated when $\W$ is in quantum selective equilibrium. 
\end{pro}
\begin{proof}
The proof is similar to that of Proposition \ref{pro_zerothlaw}, \emph{mutatis mutandis}, including the saturation condition. Write $\var_\mu(U) = \E_\mu[(U-1)^2] = (1-p_*) + p_* \E_*[(U-1)^2]$. Then by quantum Jensen's inequality, $\var_\mu(U) \ge (1-p_*) + p_* (\E_*[U]-1)^2) = \frac{1}{p_*}-1$. 
\end{proof}

\subsection{First Law of Quantum Selection}

Define the quantum selective acceleration $\partial_\NS^2(U) := \partial_\NS \var_\mu(U) = \cov_\mu(U^2,U)$. Because $U^2$ commutes with $U$, this is unhanded.

\begin{thm}[First Law of Quantum Selection]
Let $\W$ be a quantum evolutionary process. Then:
\begin{equation*}
    \partial_\NS^2(U) := \partial_\NS \var_\mu(U) \ge \var_\mu(U) \left( 1 + \var_\mu(U) \right) \ge 0,
\end{equation*}
with saturation of the first inequality exactly when $\W$ is in selective equilibrium.
\end{thm}
\begin{proof}
The proof is similar to that of Theorem \ref{thm_firstlaw}, \emph{mutatis mutandis}, including the saturation condition. Write $\tilde \E_\mu[X] := \E_{UM}[X] =  \E_\mu[XU]$. We have $\partial_\NS \var_\mu(U) = \cov(U^2,U) = \tilde \E_\mu[U^2] - \E[U^2]$. By quantum Jensen's inequality, $\partial_\NS \var(U) \ge \tilde \E_\mu[U]^2 - \E[U^2] = \E_\mu[U^2]^2 - \E[U^2] = \E[U^2] \left(\E[U^2]-1 \right)$.
\end{proof}

%At the end of Parts 2 and 3, we will define quantum selective and environmental entropies, and illustrate similar 

\subsection{Time-Varying Quantum Price Equation}

We consider a time-varying quantum evolutionary process $\W^{t,t'} : \mu^t \to \mu^{t'}$, relating a time-varying family of population density operators $\mu := (\mu^t)$. Let $N^t := \Tr(\mu^t)$ be the population at time $t$. Write $\E^t_\mu[A] := \frac{1}{N^t} \Tr(AM^t)$ and $\cov^t_\mu(A_1,A_2) := \E_\mu^t[A_1 A_2] - \E_\mu^t[A_1] \E_\mu^t[A_2]$.
Write the relative-fitness observable $U^{t,t'} := \frac{1}{\bar W} (\W^{t,t'})^\dagger(\Id^{t'})$ for $\W^{t,t'}$. Define left local-average and local-change observables:
\begin{equation}
    \langle X^{t'} \rangle_\W^{\opleft,t,t'} := (\W^{t,t'})^\dagger(X^{t'}) (W^{t,t'})^{-1} \qquad \mathrm{and} \qquad \Delta_\W^{\opleft,t,t'}(X^t,X^{t'}) := \langle X^{t'} \rangle_\W^{\opleft,t,t'} - X^t,
\end{equation}
and the right local-average and local-change observables:
\begin{equation}
    \langle X^{t'} \rangle_\W^{\opright,t,t'} := (W^{t,t'})^{-1} (\W^{t,t'})^\dagger(X^{t'}) \qquad \mathrm{and} \qquad \Delta_\W^{\opright,t,t'}(X^t,X^{t'}) := \langle X^{t'} \rangle_\W^{\opright,t,t'} - X^t,
\end{equation}
%, the left local-average observable $$, and the left local-change observable $$.
The discrete-time quantum Price equations hold for each $t < t'$ and observables $X^t$ and $X^{t'}$:
\begin{eqnarray} \label{eqn_quantumdiscretetimeprice}
    %\Delta^{t,t'}(\bar X^t, \bar X^{t'}) = 
    \E_\mu^{t'}[X^{t'}] - \E_\mu^t[X^t] 
    &=& \cov_\mu^t(X^t, U^{t,t'}) + \E_\mu^t[\Delta^{\opleft,t,t'}_\W(X^t,X^{t'}) U^{t,t'}] \\
    %&=& \cov_\mu^t(X^t, U^{t,t'}) + \E_\mu^t[U^{t,t'} \Delta^{\opright,t,t'}_\W(X^t,X^{t'})] + \E_\mu^t\!\left[ [\Delta^{\opleft,t,t'}_\W(X^t,X^{t'}), U^{t,t'}] \right] \qquad \qquad \\
    &=& \cov_\mu^t(U^{t,t'}, X^t) + \E_\mu^t[\Delta^{\opright,t,t'}_\W(X^t,X^{t'}) U^{t,t'}]
\end{eqnarray}
%The discrete-time right quantum Price equation is similar.

We say that $\W := (\W^{t,t'})$ is a \emph{smooth left quantum evolutionary process} if the equivalent conditions to Definition \ref{def_smoothprocesses} hold in the quantum case. Specifically, for each time-varying family of observables $\X := (X^t)$ we have:
\begin{enumerate}
    \item (Smooth Expectations) The time-varying average is smooth at $t$: $\frac{\d \E_\mu^t}{\d t}[X^t] := \lim_{t' \downarrow t} \frac{\E^{t'}_M[X^{t'}] - \E^t_\mu[X^t]}{|t'-t|}$. 
    \item (Relative-Fitness Density) The time-varying relative fitness admits a density at $t$, where we take the $\cov^t_\mu$-limit: $\Upsilon^t := \Ltwolim_{t' \downarrow t} \frac{U^{t,t'}}{|t'-t|}$.
    \item (Left Local-Change Density) The time-varying left and right local changes admit densities at $t$, where we take the $\E^t_\mu$-limit:
%    \begin{equation}
%        \delta^{t}(\X) := \Lonelim_{t' \downarrow t} \frac{\Delta_\W^{\opleft,t,t'}(X^t,X^{t'})}{|t'-t|} = \Lonelim_{t' \downarrow t} \frac{(\W^{t,t']\Delta_\W^{\opleft,t,t'}(X^t,X^{t'})}{|t'-t|}
%    \end{equation}
    \begin{eqnarray}
        \delta^{\opleft,t}(\X) &:=& \Lonelim_{t' \downarrow t} \frac{\Delta_\W^{\opleft,t,t'}(X^t,X^{t'})}{|t'-t|} = \Lonelim_{t' \downarrow t} \frac{(\W^{t,t'})^\dagger(X^{t'}) (W^{t,t'})^{-1} - X^t}{|t'-t|} \\
        \delta^{\opright,t}(\X) &:=& \Lonelim_{t' \downarrow t} \frac{\Delta_\W^{\opright,t,t'}(X^t,X^{t'})}{|t'-t|} = \Lonelim_{t' \downarrow t} \frac{(W^{t,t'})^{-1} (\W^{t,t'})^\dagger(X^{t'}) - X^t}{|t'-t|} \qquad
    \end{eqnarray}

%    \begin{equation}
%        \delta^{\opleft,t}(\X) := \Lonelim_{t' \downarrow t} \frac{\Delta_\W^{\opleft,t,t'}(X^t,X^{t'})}{|t'-t|} \qquad \mathrm{and} \qquad \delta^{\opright,t}(\X) := \Lonelim_{t' \downarrow t} \frac{\Delta_\W^{\opright,t,t'}(X^t,X^{t'})}{|t'-t|}
%    \end{equation}
\end{enumerate}

\begin{thm}[Smooth Quantum Price Equations] \label{thm_smoothprice_quantum}
Suppose $\W = (\W^{t,t'})$ is a smooth quantum process at $X^t$. The smooth quantum Price equations hold at $t$:
\begin{eqnarray} \label{eqn_lefttimevaryingquantumprice}
    \frac{\d \E_\mu^t}{\d t}\!\left[X^t\right] &=& \cov_\mu^t\!\left(X^t,\Upsilon^t\right) + \E_\mu^t\!\left[\delta^{\opleft,t}(\X)\right] \\
%    &=& \cov_\mu^t\!\left(X^t,\Upsilon^t\right) + \E_\mu^t\!\left[\delta^{\opright,t}(\X)\right] + \E_\mu^t\!\left[ [\delta^{\opleft,t}(\X), \delta^{\opright,t}(\X)] \right]  \\
    &=& \cov_\mu^t\!\left(\Upsilon^t, X^t\right) + \E_\mu^t\!\left[\delta^{\opright,t}(\X)\right]
\end{eqnarray}

%Furthermore, if $w$ is a process satisfying two of the assumptions of Definition \ref{def_smoothprocesses} at $t$, then the smooth Price equation holds at $t$ and the third assumption is satisfied, so $w$ is a smooth process at $X^t$.
\end{thm}
\begin{proof}
The proof is similar to that of Theorem \ref{thm_smoothprice}, \emph{mutatis mutandis}. We begin with the discrete-time left quantum Price equation \eqref{eqn_quantumdiscretetimeprice}. By $\cov^t_\mu$-convergence, we have $\cov_\mu^t(X^t, \frac{U^{t,t'}}{|t'-t|}) \to \cov^t_\mu(X^t,\Upsilon^t)$. By $\E^t_\mu$-convergence and the product rule for derivatives, we have
\begin{equation}
    \lim_{t' \downarrow t} \E^t_\mu\!\left[ \frac{\Delta^{\opleft,t,t'}_\W(X^t,X^{t'}) U^{t,t'}}{|t'-t|} \right] 
    = \E_\mu^t\!\left[ \delta^{\opleft,t}(X) U^{t,t} + \Delta_\W^{\opleft,t,t}(X^t,X^t) \Upsilon^t \right] %\nonumber \\
    = \E^t[\delta^{\opleft,t}(\X)]
\end{equation}
since $U^{t,t} = \Id^t$ and  $\Delta_\W^{\opleft,t,t}(X^t,X^t) = 0$. This proves the left time-varying Price equation.
\end{proof}

%The right time-varying Price equation states $\frac{\d \E_\mu^t}{\d t}\!\left[X^t\right] = \cov_\mu^t\!\left(\Upsilon^t, X^t\right) + \E_\mu^t\!\left[\delta^{\opright,t}(\X)\right]$, where the right local-change density is defined with the right local-average operator instead of left. 
%The proof is similar, so we omit it.

%When we apply the quantum Price 

%We can apply the Price equation to the density operator $\mu$ itself. Note that $\E_\mu[\mu] = \Tr(\mu^2)$. In the discrete case, we have:
%\begin{eqnarray}
%    \E_{\mu'}[\mu'] - \E_{\mu}[\mu] &=& \cov(\mu,U) + \E_\mu[\Delta_\W(\mu,\mu') U] \\
%     = \E_\mu[\mu (U-1)] = \Tr((U-\Id) \mu^2) \\
%    \partial_\EC(\mu,\mu') &=&  = \Tr\!\left( \langle \mu' \rangle_w - \mu) U \mu \right) \\
%    blah
%\end{eqnarray}

%In the continuous case, we have:
%\begin{eqnarray}
%    \frac{\d \E^t}{\d t}[\mu^t] &=& \cov^t_\mu(\mu^t, \Upsilon^t) + \E_\mu^t[\delta^{\opleft,t}(\boldsymbol{\mu})] \\
%    &=& \Tr\!\left( (\Upsilon^t-\Id) (\mu^t)^2 \right) + \Tr\!\left(  \right)
%\end{eqnarray}

\section{Open Evolutionary Processes and the Kerr-Godfrey-Smith Equation} \label{sect_kgs}

\renewcommand{\opin}{{\pi}}
\renewcommand{\opout}{{\nu}}

Kerr and Godfrey-Smith \cite{kerr2009generalization} relaxed the assumption \eqref{eqn_disintegration} of a full accounting of the child population, allowing for orphaned children with unaccounted parents. Such open processes have a Price-like equation with a third term. We generalize their approach for open measure-theoretic and quantum processes. 

% ** Do we want to run through the classical Kerr-Godfrey-Smith approach? I don't think so! I think it adds complexity. I personally don't feel like an unweighted graph fully captures the intuition between parent-child relationships, especially because there should be some weighting  **

\begin{exa} \label{exa_kgs}
Kerr and Godfrey-Smith considered the combinatorial case of discrete populations $(I,\mu)$ and $(I',\mu')$ with population sizes $N$ and $N'$, and an edge set $C$ from $I$ to $I'$, allowing for orphaned descendent types. They consider the number of edges $C_*(i)$ \emph{from} parent $i$, and the number of edges $C^*(i')$ \emph{to} child $i'$. The classical Kerr-Godfrey-Smith equation states that for any observables $X$ and $Y$:
\begin{equation} \label{eqn_kgs_classical}
    \E'[Y] - \E[X] = \cov(X,U) + \E[\Delta_C(X,Y) U] - \cov'(Y,C^*),
\end{equation}
for relative fitness $U(i) = \frac{C(i)}{N'/N}$ and local change $\Delta_C(X,Y)(i) = \sum _i\sum_{(i,i') \in C} Y(i')  - X(i) C(i)$.
\end{exa}

% We relax these assumptions, and allow for arbitrary open evolutionary processes. We introduce a certain observable to account for inverse fitness. The third term of \eqref{eqn_kgs_classical} admits three representations: an average against the orphaned population, covariance of the parented population minus a fractional term, or negative covariance of the inverse fitness. 

%We present the general approach. 
We define an open evolutionary process $w : \mu \mapsto \mu'$ to consist of the following:
\begin{enumerate}
    \item Sub-populations (``demes'') of parented and orphaned children $\mu'_{\opin}$ and $\mu'_{\opout} = \mu' - \mu'_{\opin}$;
    \item A (closed) evolutionary process $w_\pi : \mu \mapsto \mu'_{\opin}$, mapping parents to their children.
\end{enumerate}
%\end{defn}

Write the child deme sizes $N'_{\opin} := \mu'_{\opin}(I')$ and $N'_{\opout} := \mu'_{\opout}(I') = N'-N'_{\opin}$. Write the proportions $p'_\pi := \frac{N'_\pi}{N'}$ and $p'_\nu := \frac{N'_\nu}{N'} = 1-p'_\pi$. The selective coefficient of the closed process is $\bar W_{\opin} := \frac{N'_{\opin}}{N} = p'_\pi \bar W$. 
%Write $\bar W_\opout := \bar W - \bar W_{\opin} = \frac{N'_\opout}{N'}$. 
A type can have parented and orphaned children (i.e., $\mu'_{\opin}(B) > 0$ and $\mu'_{\opout}(B) > 0$), for example a child with two parents, one accounted for and one not. Write the deme expectations $\E'_\pi[Y] := \frac{1}{N'_\pi} \int Y \mu'_\pi$ and $\E'_\nu[Y] := \frac{1}{N'_\nu} \int Y \mu'_\nu$, so that $\E'[Y] = p'_\pi \E'_\pi[Y] + p'_\nu \E'_\nu[Y]$. 

%The closed tower property holds: $%\E'[Y \pi] = 
%\E'_\pi[Y] = \E[\langle Y \rangle_{w_\pi} U]$, as does the open tower property:
%\begin{equation}
%    \E'[Y] = p'_\pi \E[\langle Y \rangle_{w_\pi} U] + p'_\nu \E'_\nu[Y].
%\end{equation}

% $\E'_\pi[Y] := p'_\pi \E'[Y \pi]$ and $\E'_\nu[Y] := \p'_\nu \E'[Y \nu]$. 

\begin{lem}[Open Tower Property] \label{lem_childdemes} %[Absolute Continuity of Child Demes]
The child demes $\mu'_{\opin}$ and $\mu'_{\opout}$ are absolutely continuous with respect to $\mu'$, with non-negative Radon-Nikodym densities $\pi = \frac{\d \mu'_{\opin}}{\d \mu'}$ and $\nu = \frac{\d \mu'_{\opout}}{\d \mu'}$,
%\begin{equation}
%    \pi = \frac{\d \mu'_{\opin}}{\d \mu'} \quad \mathrm{and} \quad \nu = \frac{\d \mu'_{\opin}}{\d \mu'},
%\end{equation}
%$\pi = \frac{\d \mu'_{\opin}}{\d \mu'}$ and $\nu = \frac{\d \mu'_{\opin}}{\d \mu'}$, 
and $\pi + \nu = 1$ $\mu'$-a.s. The expectations satisfy $\E'[Y\pi] = p'_\pi \E'_\pi[Y]$ and $\E'[Y\nu] = p'_\nu \E'_\nu[Y]$, so $\E'[\pi] = p'_\pi$ and $\E'[\nu] = p'_\nu$.
The tower property holds for the parented children:
\begin{equation} \label{eqn_closedtower}
    \E'[Y\pi] = p'_\pi \E'_\pi[Y] = p'_\pi \E[\langle Y \rangle_{w_\pi} U],
\end{equation}
and the open tower property holds for the entire population: 
\begin{eqnarray}
    \E'[Y] &=& p'_\pi \E[\langle Y \rangle_{w_\pi} U] + \E'[Y \nu] 
    = p'_\pi \E[\langle Y \rangle_{w_\pi} U] + \cov(Y,\nu) + p'_\nu \E'[Y] \label{eqn_opentower_base} \\
    &=& \E[\langle Y \rangle_{w_\pi} U] + \tfrac{1}{p'_\pi} \cov(Y,\nu) = \E[\langle Y \rangle_{w_\pi} U] - \tfrac{1}{p'_\pi} \cov(Y,\pi). \label{eqn_opentower_renormalized}
\end{eqnarray}

%The closed tower property holds on the parented child deme: $\E'[Y \pi] = p'_\pi \E'_\pi[Y] = p'_\pi \E[\langle Y \rangle_{w_\pi} U]$. The open tower properties hold:
%\begin{eqnarray}
%    \E'[Y] - p'_\pi \E[\langle Y \rangle_w U] &=& p'_\nu \E'[Y \nu] = (1-p'_\pi) \E'[Y (1-\pi)] \label{eqn_opentower_base} \\
%    &=& p'_\nu \cov'(Y,\nu) + (p'_\nu)^2 \E'[Y] \label{eqn_opentower_nu} \\
%    &=& - (1-p'_\pi) \cov'(Y, \pi) + (1-p'_\pi)^2 %\E'[Y] . \label{eqn_opentower_pi}
%\end{eqnarray}
\end{lem}
\begin{proof}
Child deme sizes are non-negative, so $\max\{\pi(B),\nu(B)\} \le \mu'(B)$. If $\mu'(B) = 0$, then $\pi(B) = 0 = \nu(0)$, hence the child demes are absolutely continuous. We compute $\E'[Y \pi] = \frac{1}{N'} \int Y \pi \mu' = \frac{p'_\pi}{N'_\pi} \int Y \mu'_\pi = p'_\pi \E'_\pi[Y]$, and similarly for $\nu$.
We compute:
$\E'[Y] = \frac{1}{N'} \int Y \mu' = \frac{1}{N'} \iint Y w_i \mu + \frac{1}{N'} \int Y \nu \mu'  = \E[\langle Y \rangle_w U] + \E'[Y \nu] = \E[\langle Y \rangle_w U] + \E'[Y (1-\pi)]$, proving the first identity of \eqref{eqn_opentower_base}. The second identity follows since $\E'[Y\nu] = \cov(Y,\nu) + \E'[\nu] \E'[Y]$. The first identity of  \eqref{eqn_opentower_renormalized} follows from moving the third term of \eqref{eqn_opentower_base} to the left side, and dividing by $1-p'_\nu = \pi'_\pi$. The second identity follows from $\cov(Y,\nu) = \cov(Y,1-\pi) = -\cov(Y,\pi)$. 
\end{proof}

\begin{thm}[Kerr-Godfrey-Smith Equation] \label{thm_kgs}
Let $w$ be an finite-variance open process. For any observables $X$ on $I$ and $Y$ on $I'$, we have:
\begin{eqnarray} 
    \E'[Y] - \E[X] &=& \cov(X,U) + \E[\langle Y \rangle_{w_\pi} U] + \tfrac{1}{p'_\pi} \cov'(Y,\nu) \label{eqn_kgs_nu} \\
    &=& \cov(X,U) + \E[\langle Y \rangle_{w_\pi} U] - \tfrac{1}{p'_\pi} \cov'(Y,\pi). \label{eqn_kgs_pi}
\end{eqnarray}
%If $I'$ is discrete, then $\pi(i') = \const \cdot  \mu(w^{-1}(\{i'\}))$. 
\end{thm}
\begin{proof}
We write $\E'[Y] - \E[X] = \E[X(U-1)] + \E'[Y] - \E[XU]$, then apply the open tower property \eqref{eqn_opentower_renormalized}. 
\end{proof}

%\begin{pro}
%Let $w$ be a finite-variance open process. 
%\end{pro}

% Claim: if $I'$ is discrete, then $\pi(i') = C^{-1}(i')$. 

The classical Kerr-Godfrey-Smith equation \eqref{eqn_kgs_classical} is recovered when the process $w_\pi$ can be represented a kernel against some background measure $\lambda'$. This holds if each $w_i$ is absolutely continuous to $\lambda'$, with $w_\pi(i,i') := \frac{\d w_i}{\d \lambda'}(i')$ satisfying the $w_i(B) = \int_B w_\pi(i,i') \lambda'(\d i')$. We write the dual fitness $W^*_\pi(i') := \int_I w_\pi(i,i') \mu(\d i)$, representing the infinitesimal amount of parented child population at $i'$. The dual fitness satisfies the identity $\frac{1}{p'_\pi} \E'[W^*_\pi] = \E'_\pi[W^*_\pi] = 1$. Proof. We interchange integrals by Fubini's theorem: $\E'_\pi[W^*_\pi] = \frac{1}{N'_\pi} \int_I \int_{I'} w_\pi(i,i') \mu'_\pi(\d i') \mu(\d i) = \frac{1}{N'_\pi} \int_I W_\pi(i) \mu(\d i) = \frac{1}{\bar W_\pi} \bar W_\pi = 1$.
% and $\E'_\nu[W^*_\pi] = 0$. 

For example, if the parented child population $\mu'_\pi$ is discrete, then any process admits a kernel $w_\pi(i,i') = w_i(\{i'\})$ against counting measure, with dual fitness $W^*_\pi(i') = \int_I w_i(\{i'\}) \mu(\d i)$. 
%is kernable. Proof: if $\mu'_\pi$ is discrete, then there exists countable $I'_\pi$ such that $\mu'_\pi(B) = \sum_{i' \in I'_\pi} \pi(i') \mu'(\{i'\}) \delta_{i'}(B)$. In this case, the kernel is $w_\pi(i,i') := w_i(\{i'\})$, so that $w_i(B) = \sum_{i' \in B} w_i(\{i'\}) \pi(i') \mu'(\{i'\})$}

%Write the reference covariance 
%$\E'_{\lambda'}[Y] := $

\begin{cor}[Dual-Fitness Version of Kerr-Godfrey-Smith Equation]
If $w_\pi$ admits a kernel $w_\pi(i,i')$ relative to $\lambda'$, then $\mu'_\pi$ is absolutely continuous with respect to $\lambda'$, with $\frac{\d \mu'_\pi}{\d \lambda'}(i') = W^*_\pi(i')$. 
Consequently, for any observables $X$ and $Y$:
\begin{equation} \label{eqn_kgs_kernel}
    \E'[Y] - \E[X] = \cov(X,U) + \E[\langle Y \rangle_{w_\pi} U] - \frac{1}{N'_\pi} \int_{I'} Y(i') W^*_\pi(i') \lambda'(\d i') + 1.
\end{equation}

%In particular, if $\mu'$ is 
% Chain rule: assume $\mu' \ll \lambda'$.
% Then: \frac{\d \mu'_\pi}{\d \lambda'} = \frac{\d \mu'_\pi}{\d \mu'} \frac{\d \mu'}{\d \lambda'}
\end{cor}
\begin{proof}
For any observable $Y$, we use the kernel property and interchange integrals by Fubini's theorem to show that $\frac{\d \mu'_\pi}{\d \lambda'} = W^*_\pi$ $\mu'$-a.s.:
\begin{eqnarray}
    \int_{I'} Y(i') \mu'_\pi(\d i') 
    &=& \int_I \int_{I'} Y(i') w_i(\d i') \mu(\d i) 
    = \int_I \int_{I'} Y(i') w_\pi(i,i') \lambda'(\d i') \mu(\d i) \nonumber \\
    &=& \int_{I'} Y(i') \left( \int_I w_\pi(i,i') \mu(\d i) \right)  \lambda'(\d i') = \int_{I'} Y(i') W^*_\pi(i') \lambda'(\d i').
\end{eqnarray}
Consequently,
%\begin{equation}
    $-\cov'\!\Big(Y, \frac{\pi}{p'_\pi}\Big) = -\frac{1}{N'_\pi} \int_{I'} Y(i') \mu'_\pi(\d i') + 1 = -\frac{1}{N'_\pi} \int_{I'} Y(i') W^*_\pi(i') \lambda'(\d i') + 1.$
%\end{equation}

\end{proof}

%Kerr and Godfrey-Smith used the assumption of discreteness and \eqref{eqn_pi_discrete} in their proof of their classic equation \eqref{eqn_kgs_classical}, with $\pi'$ playing the role of dual fitness $C^*$. However, the assumption of discreteness is not needed, as we see from the proof of Theorem \ref{thm_kgs}.

%\begin{rem}[Dual Fitness]
%The parented child density $\pi'(i')$ itself can be thought of as a dual relative fitness: it counts the proportion of total child population at $i'$ which can be accounted for by parents. When $\mu'_\pi$ is discrete, this density admits a pointwise representation
%\begin{equation} \label{eqn_pi_discrete}
%    \pi'(i') = \int_I w_i(\{i'\}) \mu(\d i).
%\end{equation}
%Kerr and Godfrey-Smith used the assumption of discreteness and \eqref{eqn_pi_discrete} in their proof of their classic equation \eqref{eqn_kgs_classical}, with $\pi'$ playing the role of dual fitness $C^*$. However, the assumption of discreteness is not needed, as we see from the proof of Theorem \ref{thm_kgs}.
%\end{rem}

Recall the quantum evolutionary formalism of Section \ref{sect_quantum}. An open quantum process consists of parented and orphaned child density operators $\mu'_\pi = \pi \mu'$ and $\mu'_\nu = \nu \mu'$ satisfying $\pi + \nu = \Id_{\H'}$, and a closed quantum process $\W_\pi : \mu \mapsto \mu'_\pi$. By a similar proof as Lemma \ref{lem_childdemes}, the quantum open tower properties hold: $\E'_{\mu'}[Y] - \E_\mu[\langle Y \rangle^\opleft_{\W_\pi} U] = \frac{1}{p'_\pi} \cov'_{\mu'}(Y,\nu) = - \frac{1}{p'_\pi} \cov'_{\mu'}(Y,\pi)$. 

\begin{thm}[Quantum Kerr-Godfrey-Smith Equation]
Let $\W$ be an open quantum process. 
The left and right quantum Kerr-Godfrey-Smith equations hold: 
\begin{eqnarray}
    \E'_{\mu'}[Y] - \E_\mu[X] &=& \cov_\mu(X,U) + \E[\langle Y \rangle_{\W_\pi}^\opleft U] + \tfrac{1}{p'_\pi} \cov'_{\mu'}(Y,\nu) \label{eqn_kgs_nu_quantum_left} \\
    &=& \cov(X,U) + \E[\langle Y \rangle_{\W_\pi}^\opleft U] - \tfrac{1}{p'_\pi} \cov'_{\mu'}(Y,\pi) \label{eqn_kgs_pi_quantum_left} \\
    &=& \cov_\mu(U,X) + \E[U \langle Y \rangle_{\W_\pi}^\opright] + \tfrac{1}{p'_\pi} \cov'_{\mu'}(\nu, Y) \label{eqn_kgs_nu_quantum_right} \\
    &=& \cov(U,X) + \E[U \langle Y \rangle_{\W_\pi}^\opright] - \tfrac{1}{p'_\pi} \cov'_{\mu'}(\pi, Y) \label{eqn_kgs_pi_quantum_right}
\end{eqnarray}
\end{thm}

% \newpage 
\part{Selective Entropy (Kullback-Leibler Divergence of Relative Fitness)} \label{part_selectiveentropy}

% We then introduce two entropy functionals, representing separate notions of ``information'' of an evolutionary process. First we introduce ``selective entropy'' (or ``biological entropy''), a non-positive quantity which represents the structuring of information by a process. This generalizes classical Kullback-Leibler divergence, and encodes the Fisher information. Selective entropy vanishes if and only if a process is purely environmental. 

% Second, we introduce ``environmental entropy'', a non-negative quantity which represents the mixing of information by a process. This generalizes classical Kolmogorov-Sinai entropy, and encodes the Shannon information. Environmental entropy vanishes if and only if a process is purely selective.  

% For notational clarity, we drop the subscript $w$ from $\partial_\NS := \partial_{\NS}$ and $\partial_\EC := \partial_{\EC}$.

\section{Definition of Selective Entropy and Gibbs's Inequalities} \label{sect_entropyNS}

%% ** STOPPED HERE IN FINAL REVIREW, SUNDAY, DECEMBER 26, 2021 ** %%

%We now return to the general case of the Price equation for arbitrary evolutionary processes. 

In Part 2, we introduce the selective entropy to quantify ``the amount of selection'' of a process, by generalizing Kullback-Leibler divergence to the case of evolutionary processes. This represents a ``biological entropy'' or ``negentropy''. Selective entropy is non-positive, and bounded above by a negative value in the strict selective-equilibrium case. We prove a Second Law, showing that the selective change of selective entropy is non-positive, as well as a speed limit. Both inequalities are saturated in the selective equilibrium case.
Write $U := W/\bar W$ for the relative fitness function, i.e., $U(i) := w_i(I')/\bar W$.

\begin{defn}[Selective Entropy]
We say that a process is finite-entropy if $\E[|U \log U|] < \oo$. For any finite-entropy process, we define the selective entropy as the average of $-U \log U$:
\begin{equation} \label{def_entropyNS} 
S_\NS := \E[-U \log U] = \frac{1}{N} \int \left( - \frac{W(i)}{\bar W} \log \frac{W(i)}{\bar W} \right) \, \mu(\d i).
\end{equation}
\end{defn}

The selective entropy is fully concentrated in the selective part of a process. That is, if $w$ and $\hat w$ both have the same selective part $w_\NS$, i.e., the same relative fitness function $U$, then they have the same selective entropy $S_\NS$. The environmental part does not contribute to selective entropy. 
For purely selective processes (where $\E'[X] = \E[U X]$), the selective entropy is the Kullback-Leibler divergence of $\E'$ relative to $\E$. That is, selective entropy is exactly the familiar relative entropy from information theory. For purely environmental processes (i.e., $U=1$ $\mu$-a.s.), selective entropy vanishes. For all other processes, selective entropy measures the degree to which selective effects are present in the process $w$, and the Price equation ensures all remaining effects are environmental. 

% This extends relative entropy to general evolutionary processes. 

%Markov's inequality implies the following:
%\begin{equation}
%
%\end{equation}

\begin{rem}(Sign Convention)
We choose sign convention to be consistent with classical information theory and statistical mechanics. The non-positive selective entropy $S_\NS$ has the biological role of ``negentropy'' \cite{schrodinger1944life}, and in Section \ref{sect_environmentalentropy}, we introduce the non-negative environmental entropy $S_\EC$ to represent the classical physical role of dynamical entropy. %The signs of these quantities satisfy $S_\NS \le 0$ and $S_\EC \ge 0$, representing two distinct entropy concepts which point in opposite directions. 
The total entropy is given by $S_\tot := S_\NS + S_\EC$, which is negative or positive depending on whether selective effects outweigh environmental effects. 
\end{rem}

\subsection{Selective Entropy Bounds} \label{sect_selectiveentropybounds}

We now state and prove that $S_\NS \le 0$, which corresponds to the Gibbs' inequality in classical information theory. This inequality is saturated exactly in the case of purely environmental processes (in which case $S_\NS = 0$, otherwise $S_\NS < 0$). Thus the selective entropy is a proxy for ``selectiveness'' in a process.

\begin{lem}[Weak Gibbs Inequality] %[Gibbs' Upper Bound] 
Let $w$ be a finite-entropy process with relative fitness $U$, and let $S_\NS := \E[-U \log U]$ be the selective entropy of $w$. The Gibbs' inequality holds:
\begin{equation} \label{ineq_gibbsNS}
S_\NS \le 0.
\end{equation}
This is an equality ($S_\NS = 0$) if and only if $w$ is purely environmental (i.e., $U$ is a.s. constant with $U = 1$). 
\end{lem}
\begin{proof}
Observe that $-U \log U$ is a concave function of $U$, and $\E[U] = 1$. By Jensen's inequality, we have:
\begin{equation}
S_\NS = \E[-U \log U] \le - \E[U] \log \E[U] = - 1 \log 1 = 0,
\end{equation}
with equality if and only if $U$ is constant (with $U = \E[U] = 1$ almost surely).
\end{proof}

We strengthen \eqref{ineq_gibbsNS} and derive a window \eqref{ineq_stronggibbsNS} in which selective entropy can fluctuate. Recall from Section \ref{sect_zerothlaw} that $w$ is in selective equilibrium when $U \in \{0,1/p_*\}$ a.s. The window \eqref{ineq_stronggibbsNS} collapses to a single value when $w$ is in selective equilibrium (in which case $S_\NS = \log p^*$), and otherwise the inequalities are strict. Since lost population does not contribute to entropy ($0 \log 0 = 0$), all selective entropy is generated by the childbearing population. 
%We state and prove that $S_\NS \le \log p^*$, where $p_* := \frac{1}{N} \mu(W > 0)$ is the proportion of the childbearing population. The stronger inequality improves upon the strong variance lower bound \eqref{ineq_stronglowerboundforvariance}. 

Define the childbearing population $\mu_*(A) := \mu(A \cap \{W>0\})$, with population size $N_* = p_* N_*$. Define the childbearing expectation $\E_*[X] := \frac{1}{p_*} \E[1_{U>0} X]$, and the childbearing variance $\var_*(X) := \E_*[X^2] - \E_*[X]^2$. Recall $U_* = p_* U$. Note that the general variance and childbearing variance are related as follows:
\begin{equation}
    \var(U) = \E[U^2]-1 = p_* \E_*\!\left[\left(\frac{U_*}{p_*}\right)^2\right] - 1 = \frac{1}{p_*} \E_*\!\left[U_*^2\right] - 1 = \frac{1}{p_*} \var_*(U_*) + \frac{1}{p_*} - 1.
\end{equation}

\begin{thm}[Strong Gibbs Inequality] \label{thm_stronggibbsNS} %[Strong Inequalities for Selective Entropy] \label{thm_stronggibbsNS}
Let $w$ be a finite-entropy process with selective entropy $S_\NS$, and let $p_* = \mu(W>0)/N$ be the childbearing population proportion. Then:
\begin{equation} \label{ineq_stronggibbsNS}
\log p_* - \log(1 + \var_*(U_*)) = - \log \big( 1 + \var(U) \big) \le S_\NS \le \log p_*, % \le 0.
\end{equation}
with saturation in the selective-equilibrium case (in which case $S_\NS = \log p_* = -\log \big(1+\var(U)\big)$), and otherwise the inequalities are strict.% If $w$ is finite-entropy but not finite-variance, then the lower bound is $-\oo$ and the upper bound still holds: $S_\NS \le \log p_*$.
\end{thm}
\begin{proof}
We decompose the expectation into the sum of childless and childbearing parts:
\begin{equation} \label{eqn_childbearingdecomposition}
\E[X] = p_0 \E_0[X] + p_* \E_*[X]
\end{equation}
where $p_0 = 1 - p_*$, $\E_*[X] := 1/p_* \E[1_{U>0} X]$ and $\E_0[X] := 1/p_0 \E[1_{U=0} X]$. Note that $\E_*[U] = 1/p_*$.
We apply the decomposition \eqref{eqn_childbearingdecomposition} to $-U \log U$. Since $0 \log 0 = 0$, we have:
\begin{equation}
S_\NS = 0 + p_* \E_*[-U \log U].
\end{equation}
Since $\E_*$ is a probability expectation and $-U \log U$ is concave, we use Jensen's inequality:
\begin{equation}
S_\NS \le - p_* \E_*[U] \log \E_*[U] = - p_* \frac{1}{p_*} \log \frac{1}{p_*} = \log p_* \le 0. 
\end{equation}
This proves the upper bound for \eqref{ineq_stronggibbsNS}. This is saturated when $U$ is constant $\mu_*$-a.s., i.e., the selective-equilibrium case. 
For the lower bound, observe that $\E[U\cdot]$ is a probability expectation and $-\log x$ is convex, therefore by Jensen's inequality:
\begin{equation}
S_\NS = \E[U (-\log U)] \ge - \log \E[U^2].
\end{equation}
This is saturated exactly when $U$ is constant $U \mu$-almost surely. Since $U \mu$ and $\mu_*$ are mutually absolutely continuous, the saturation condition is equivalent to {\quasienvironmentality}. 
\end{proof}

This implies a strong version of the Zeroth Law (Proposition \ref{pro_zerothlaw}), with an improved lower bound based on selective entropy. 

\begin{cor}[Strong Zeroth Law] \label{cor_strongzerothlaw}
The inequalities \eqref{ineq_stronggibbsNS} are equivalent to the following:
\begin{equation} \label{ineq_exponentialentropyrelation}
    \partial_\NS(U) = \var(U) \ge \e^{-S_\NS} - 1 \ge \frac{1}{p_*} - 1 \qquad \mathrm{and} \qquad p_* \ge \e^{S_\NS} \ge \frac{1}{1 + \var U},
%\frac{1}{1 + \var U} \le \e^{S_\NS} \le p_* \qquad \mathrm{and} \qquad \frac{1}{p_*} - 1 \le \e^{-S_\NS} - 1 \le \var(U) = \partial_\NS(U),
\end{equation}
with saturation in the selective-equilibrium case. %The variance lower bound \eqref{ineq_exponentialentropyrelation} improves upon the Weak Zeroth Law (Proposition \ref{pro_zerothlaw}) when $S_\NS < \log p_*$.
\end{cor}

\section{Selective Change of Selective Entropy and the Second Law of Natural Selection} \label{sect_NSNS}

In this section, we analyze the change of the selective entropy functional across generations. We decompose the selective change and environmental change of the functional, and prove bounds showing the ``typical'' tendency of change. The selective change of selective entropy is negative, representing that selection always amplifies existing selective effects. %The environmental change of selective entropy can be positive or negative, but is limited by certain bounds described below. 

% Some of these results are an ``intrisinic'' form of the Selective Second Law: 

Consider composable processes $w : \mu \mapsto \mu'$ and $w' : \mu' \mapsto \mu''$, with relative fitnesses $U$ and $U'$. We define the change of selective entropy as the difference in selective entropies:
\begin{equation}
    \Delta(S_\NS,S_\NS') := S_\NS' - S_\NS = \Delta(-U \log U,-U' \log U') = \E'[-U'\log U'] - \E[-U\log U]. 
\end{equation}

Define the selective and environmental changes of selective entropy:
\begin{eqnarray}
    \partial_\NS S_\NS &:=& \partial_\NS(-U \log U) = \cov(-U \log U, U) = -\E[(U-1) U \log U] \label{eqn_defNSNS} \\
    \partial_\EC(S_\NS,S'_\NS) &:=& \partial_\EC(-U \log U,-U' \log U')  \nonumber \\
    &=& \E[\Delta_w(-U \log U, -U' \log U') U] = \E[(\langle -U'\log U' \rangle_w + U \log U)U]. \qquad
\end{eqnarray}

The functional Price equation (Corollary \ref{cor_functionalprice}) decomposes the selective-entropy change as the sum of selective and environmental changes:
\begin{eqnarray}
    \Delta(S_\NS,S_\NS') &=& \partial_\NS S_\NS + \partial_\EC(S_\NS,S_\NS') \\
    &=& \cov(-U \log U, U) + \E[\Delta_w(-U \log U,-U' \log U') U].
\end{eqnarray}

\subsection{Selective Change of Selective Entropy}

Our main result is that the selective-change term $\partial_\NS S_\NS$ is always non-positive, i.e., vanishing in the purely environmental case and otherwise strictly negative. The meaning is that \emph{under the effect of natural selection, selective entropy cannot increase}. %We state this as Weak and Strong Second Laws of Natural Selection.
%This result is a selective form of the Second Law of Thermodynamics: 
%\begin{itemize} 
%\item Under purely selective processes, selective entropy never increases. 
%\item Under more general processes, the selective change (that is, the part of change which is purely selective) of selective entropy never increases.
%\end{itemize}
We state a Weak Second Law showing non-positivity, saturated in the purely environmental case, and a strong Second Law providing a chain of inequalities, saturated in the selective equilibrium case. The Weak Second Law follows from the non-positivity of the function $-(x-1) x \log x$, and does not rely on concavity. The Strong Second Law does rely on concavity of the functions $-x \log x$ and $\log x$. 

%The Weak Second Law follows from the non-positivity of the function $-(x-1) x \log x$, and does \emph{not} rely on concavity. When averaged, this is the upper bound $\partial_\NS S_\NS = \E[-(U-1) U \log U] \le 0$. We state this as the following result. The Strong Second Law relies on concavity, with a similar argument as the Strong First Law. 

%After this result, we prove a stronger upper bound. In Part \ref{sect_youngfenchel}, we prove a lower bound.

\begin{pro}[Weak Second Law of Natural Selection] %[Weak Upper Bound of $\partial_\NS S_\NS$] \label{pro_nonposNSNS}
Let $w$ be an evolutionary process with $\E[|U^2 \log U|] < \oo$. The selective change in selective entropy is non-positive:
\begin{equation} \label{ineq_entropychangeNSNS}
\partial_\NS S_\NS \le 0.
\end{equation}
The inequality is saturated exactly for purely environmental processes, and is otherwise strictly negative.
\end{pro}
\begin{proof}
Observe that the real-valued functions $x-1$ and $\log x$ always have the same sign, therefore the function $-x(x-1) \log x$ is non-positive. Thus:
\begin{equation}
\partial_\NS S_\NS = \E[-(U-1) U \log U] \le 0.
\end{equation}
%proving the result. 
The function $-x(x-1) \log x$ vanishes only at $x=0$ and $x=1$. If $w$ is purely environmental ($U=1$ a.s.), then $\partial_\NS S_\NS = \E[(-1-1) 1 \log 1] = 0$. 
If $w$ is not purely environmental, then there exists $\epsilon > 0$ and measurable $A \subseteq I$ on which $U \notin \{0,1\}$ and $-U (U-1) \log U < -\epsilon$. Thus 
\begin{equation}
    \partial_\NS S_\NS = -\E[1_{I-A} (U-1) U \log U] - \E[1_{A} (U-1) U \log U] \le 0 - \E[1_A] \epsilon = -\frac{\mu(A)}{N} \epsilon < 0.
\end{equation}
%Consequently, the saturation condition is equivalent to $w$ being purely environmental. 
\end{proof}

%\subsection{Strong Upper Bound for $\partial_\NS S_\NS$} \label{sect_upperboundforselectivevelocity}

\begin{rem}[Selective-Equilibrium Case]
Recall that $p_* = \mu(U > 0)/N$ is the proportion of childless population. If $w$ is in selective equilibrium, then we can specify selective change of selective entropy explicitly. Since relative fitness takes exactly two values $0$ and $1/p_*$ almost surely, the selective change of selective entropy takes the form:
\begin{eqnarray}
\partial_\NS S_\NS &=& \cov(-U \log U, U) = -\E[U(U-1) \log U] \nonumber \\
&=& - p_* \E_*[U(U-1) \log U] = - p_* \frac{1}{p_*} \big( \frac{1}{p_*} - 1 \big) \log \frac{1}{p_*} \nonumber \\
&=& - \left( \frac{1}{p_*} - 1 \right) \log \frac{1}{p_*} < 0. \label{eqn_quasienvchangeNSNS}
\end{eqnarray}
Recall that for selective-equilibrium processes, $S_\NS = -\log \frac{1}{p_*}$ and $\var(U) = \frac{1}{p_*} - 1$. Consequently, in the selective-equilibrium case,
\begin{equation}
\partial_\NS S_\NS = \var(U) S_\NS.
\end{equation}
\end{rem}

This provides a baseline for improving the Weak Second Law, by proving $\partial_\NS S_\NS \le \var(U) S_\NS$, which itself is further bounded by the (non-positive) expression $(\e^{-S_\NS} - 1) S_\NS$. The meaning of this statement is that $\partial_\NS S_\NS$ is ``maximally controlled'' in the selective-equilibrium case (in which case it equals $\big( \frac{1}{p_*} - 1 \big) \log p_*)$, and otherwise it is strictly bounded by these quantities. This provides a ``minimal velocity'' for selective entropy, achieved only in selective equilibrium.
%controlled by the quantities $\partial_\NS S_\NS \le \var(U) S_\NS \le (\e^{-S_\NS} - 1) S_\NS$. This constrains selective change not only in terms of the product variance times entropy (a non-positive quantity), but furthermore by the negative exponential. This provides a ``minimal velocity'' for selective entropy.

% In general, this is a functional inequality which exploits convexity in two ways, each of which is satisfied in the selective-equilibrium case. We formulate this as the following theorem. 

%In general, we can ``wiggle'' the process and demonstrate a strong inequality using variances of fitness, as well as a log-linear function of entropy. This is analogous to the inequality $S_\NS \le \log p_*$, but applied to the quantity $\partial_\NS S_\NS$. 
% 

% We now improve upon the upper bound of $\partial_\NS S_\NS$ using the childless population.  %, by expressions in terms of variance, selective entropy and the (non-negative) exponential entropy $\exp(-S_\NS) \ge 0$. 

% In the next result, we demonstrate that this is the extreme case, and otherwise $\partial_\NS S_\NS \le \var(U) S_\NS$. This means that both fitness variance and selective entropy play an extreme role in driving selective change. 

\begin{thm}[Strong Second Law of Natural Selection] %[Strong Upper Bound for $\partial_\NS S_\NS$]
\label{thm_secondlaw}
Let $w$ be an evolutionary process with $\E[|U^2 \log U|] < \oo$. Let $\partial_\NS S_\NS := \cov(-U \log U, U)$ be the selective change of selective entropy. The following upper bound holds:
\begin{equation} \label{ineq_upperboundNSNS}
\partial_\NS S_\NS \le -\var(U) \log(1 + \var(U)) \le \var(U) S_\NS \le (\e^{-S_\NS} - 1) S_\NS \le -\left( \frac{1}{p_*} - 1 \right) \log \frac{1}{p_*} \le 0,
\end{equation}
where all but the last inequalities are saturated in the selective-equilibrium case, and are otherwise strict inequalities. All quantities vanish exactly in the purely environmental case, otherwise they are all strictly negative. 

%The upper bound vanishes in the purely environmental case, and is otherwise strictly negative. 
\end{thm}
\begin{proof}
We write $\partial_\NS S_\NS$ as the sum of two terms, and we analyze an upper bound for each separately. Observe that: 
\begin{equation}
\partial_\NS S_\NS = \cov(-U \log U, U) = \E((-U \log U - S_\NS) U] = \E[-U^2 \log U] - S_\NS. 
\end{equation}
We rewrite the first term as a weighted probability expectation of the concave function $-x \log x$. Observe that relative fitness $U$ is a probability density function, since $U \ge 0$ and $\E[U] = 1$. Consequently, Jensen's inequality implies:
%we prove a strong upper bound using Jensen's inequality and the weighted expectation $\E[U \cdot]$:
\begin{equation} \label{ineq_firsttermNSNSupper}
\E[-U^2 \log U] = \E[U(-U \log U)] \le -\E[U^2] \log \E[U^2].
\end{equation}
The strong inequality \eqref{ineq_firsttermNSNSupper} is saturated exactly in the case that $U$ is constant $U \mu/N$-a.s. Since $\mu_*$ and $U/N$ are mutually absolutely continuous, saturation is equivalent to the selective-equilibrium case (i.e., $U = 1$ $\tilde \mu$-a.s.). 

For the second term, we use the strong lower bound for selective entropy, which translates into a strong \emph{upper} bound for the negative selective entropy:
\begin{equation} \label{ineq_secondtermNSNSupper}
-S_\NS \le \log(1 + \var(U)) = \log \E[U^2],
\end{equation}
with saturation in the selective-equilibrium case. Combining \eqref{ineq_firsttermNSNSupper} and \eqref{ineq_secondtermNSNSupper}, we have:
\begin{eqnarray}
\partial_\NS S_\NS &\le& -\E[U^2] \log \E[U^2] + \log \E[U^2] = -\big( \E[U^2] - 1 \big) \log \E[U^2] \nonumber \\
&=& -\var(U) \log(1 + \var(U)), \label{ineq_strongupperboundNSNSproof1}
\end{eqnarray}
since $\E[U^2] = 1 + \var(U)$. The bound \eqref{ineq_stronggibbsNS} states that $-\log(1+\var(U)) \le S_\NS$. When we apply this to \eqref{ineq_strongupperboundNSNSproof1}, we have:
\begin{equation}
\partial_\NS S_\NS \le \var(U) S_\NS.
\end{equation}
The variance lower bound \eqref{ineq_exponentialentropyrelation} ($\var(U) \ge \e^{-S_\NS} - 1$) becomes an \emph{upper} bound when we multiply by the non-positive $S_\NS$. Thus $\var(U) S_\NS \le (\e^{-S_\NS} - 1) S_\NS$. 

For the final non-trivial inequality, we have $S_\NS \le \log p_*$ from \eqref{ineq_stronggibbsNS}, hence $(\e^{-S_\NS} - 1) S_\NS \le (\e^{-S_\NS} - 1) \log p_*$. Similarly, we have $\e^{-S_\NS} - 1 \ge \frac{1}{p_*} - 1$. When we multiply by the non-positive $\log p_*$, we obtain the inequality $(\e^{-S_\NS} - 1) S_\NS \le \left(\frac{1}{p_*} - 1 \right) \log p_*$. %, which completes the proof.
\end{proof}

This leads to a selective feedback loop. %, which is a essentially a selective analogue of the Second Law of Thermodynamics. 
If a process is purely environmental, then selective entropy does not change. However, in the presence of even minimal selective effects (such as selective equilibrium), then the strictly negative quantity $(\e^{-S_\NS} - 1) S_\NS$ ``drives'' selective entropy change downward. This forces \emph{some} change across evolutionary processes. Thus in the presence of any selective effects, a system is driven to have even more selection, as measured by more negative $S_\NS$. Nonetheless, environmental effects can effect $S_\NS$ arbitrarily. 

\begin{cor}
Suppose that $\partial_\EC(S_\NS,S_\NS') = 0$ (such as the purely-selective case). % Note: being purely selective means that $\partial_\EC(X,Y) = 0$ for all $X, Y$. 
Then: 
\begin{equation}
    S'_\NS - S_\NS = \partial_\NS S_\NS \le -\var(U) \log(1 + \var(U)) \le \var(U) S_\NS \le (\e^{-S_\NS} - 1) S_\NS \le 0.
\end{equation}
\end{cor}

\subsection{The Selective Speed Limit} \label{sect_selectivespeedlimit}

The Second Law \eqref{ineq_upperboundNSNS} provides a minimal speed that selection must occur at, driving $S_\NS$ ever more negative. We use a similar technique to prove a speed limit, showing that $S_\NS$ cannot change in an unbounded way. The saturation condition is again given by selective equilibrium. This requires additional moment assumptions. 

%\begin{lem}
%Suppose $w$ is finite-variance. If $\E[|U^{2+c} \log U|] < \oo$ for $c > 0$, then $\E[|U^{2+c-\delta} \log U|] < \oo$ for $\delta \in $
%\end{lem}
%\begin{proof}
%
%\end{proof}

%Unlike the other major theorems of this Part, the speed limit is not saturated by processes in selective equilibrium. Thus one needs to experience a multitude of selective effects 

\begin{thm}[Selective Speed Limits] \label{thm_selectivespeedlimit}
Let $w$ be finite-variance. %a process with $\E[|U^2 \log U|] < \oo$. 
\begin{enumerate}
%\item Suppose $\E[U^{2+c}] < \oo$ for all $c$ in a set $C \subseteq (0,\oo)$. Then:
%\begin{equation}
%    \partial_\NS S_\NS \ge -\log p_* + \frac{\E[U^2]}{c} \log \frac{\E[U^{2+c}]}{\E[U^2]}.
%\end{equation}
%In particular, 
\item %(Basic Speed Limit) 
Suppose $\E[U^{2+c'}] < \oo$ for some $c' > 0$. % in a set $C \subseteq (0,\oo)$. 
The basic speed limit holds:
\begin{equation} \label{ineq_speedlimit_basic}
    \partial_\NS S_\NS \ge \log \frac{1}{p_*} + \sup_{c \in (0,c')} \left\{ - \frac{\E[U^2]}{c} \log \frac{\E[U^{2+c}]}{\E[U^2]} \right\} 
    %\sup_{c\in C} \left\{ - \frac{\E[U^2]}{c} \log \frac{\E[U^{2+c}]}{\E[U^2]} \right\} 
    % \ge \log \frac{\E[U^{3+\var(U)}]}{(1+\var(U)) p_*}, %\label{ineq_speedlimit_var}
\end{equation}
with saturation in the selective-equilibrium case (in which case $\partial_\NS S_\NS = - \left(\frac{1}{p_*} - 1\right) \log \frac{1}{p_*}$). 

\item %(Continuum Speed Limit) 
Suppose $\E[|U^{2+c'} \log U|] < \oo$ for some $c'>0$.
%in an open set $C \subseteq (0,\oo)$. 
If there exists $c_* \in (0,c')$ satisfying the functional equation
\begin{equation} \label{eqn_speedlimit_functionalidentity}
1 = \frac{\E[U^{1+c_*}]^{c_*}}{\E[U^2]^{c_*-1} \, \E[U^{2+c_*}]},
\end{equation}
then the continuum speed limit at $c_*$ holds:
\begin{equation} \label{ineq_speedlimit_continuum}
    \partial_\NS S_\NS \ge 
    %-\log p_* - \frac{\E[U^2]}{c_*} \log \frac{\E[U^{2+c_*}]}{\E[U^2]} = 
    \frac{1}{p_*}  - \frac{\E[U^2]}{c_*} \log \frac{\E[U^{2+c_*}]}{\E[U^2]} 
    = \log \frac{1}{p_*} - \frac{\E[U^2]}{c_*} \log \frac{\E[U^{1+c_*}]^{c_*}}{\E[U^2]^{c_*-1} \E[U^2]},
    %= -\log p_* - \E[U^2] \log \frac{\E[U^{1+c_*}]}{\E[U^2]},  
\end{equation}
with saturation in the selective-equilibrium case.

\item %(Infinitary Speed Limit). 
Suppose $\E[|U^{2+c' } \log U|] < \oo$ for some $c' > 0$. 
% $c$ in a set $C \subseteq (0,\oo)$ which is dense at $0$. 
The infinitary speed limit holds:
\begin{equation} \label{ineq_speedlimit_infinitary}
    \partial_\NS S_\NS \ge \log \frac{1}{p_*} - \E[U^2 \log U].
\end{equation}
with saturation in the selective-equilibrium case. 
\end{enumerate}
\end{thm}
\begin{proof}
% Note: $\E[|U^{2+c'} \log U}] < \oo$ implies $\E[|U^{2+c} \log U}] < \oo$ for $c \in (0,c')$. This follows from finite-variance and the Cauchy-Schwarz inequality:
% \E[|U^{2+c} \log U}]
% = \E[U^{c-c'} |U^{2+c'} \log U}]
% \le \sqrt{\E[U^{c-c'}} \sqrt{\E[|U^{2+c'} \log U}]} 

Proof of (1). Consider arbitrary $c \in C$. We apply Jensen's inequality to the convex functionals $-\log x$ and $x \log x$, and compute:
\begin{eqnarray}
    \partial_\NS S_\NS = \cov(-U \log U,U) &=& \frac{1}{c} \E[-U^2 \log U^c] + \E[U \log U] \nonumber \\
    &=& \frac{p_* \E_*[U^2]}{c} \frac{1}{\E_*[U^2]} \E_*[U^2 (-\log U^c)] + p_* \E_*[U \log U] \nonumber \\
    &\ge& - \frac{p_* \E_*[U^2]}{c} \log \frac{\E_*[U^{2+c}]}{\E_*[U^2]} + p_* \E_*[U] \log \E_*[U] \nonumber \\
    &=& - \frac{\E[U^2]}{c} \log \frac{\E[U^{2+c}]}{\E[U^2]} - \log p_* 
\end{eqnarray}
since $\E[UX] = p_* \E_*[UX]$ for any observable $X$, in particular, $\E_*[U] = \frac{1}{p_*}$. Saturation holds when $U$ and $U^c$ are constant $\tilde \mu$-almost surely, which is equivalent to $U$ being constant $\mu_*$-almost surely, i.e., the case of selective equilibrium. Taking suprema over all $c$ yields the first identity of \eqref{ineq_speedlimit_basic}. Setting $c = \E[U^2] = 1 + \var(U)$ yields the second identity of \eqref{ineq_speedlimit_basic}.

Proof of (2). To optimize \eqref{ineq_speedlimit_basic}, we find stationary points $c_*$ by differentiating the argument of the supremum in $c$ and solving for zero:
\begin{equation}
    0 = \frac{\d}{\d c} \left( - \frac{\E[U^2]}{c} \log \frac{\E[U^{2+c}]}{\E[U^2]} \right) = \frac{\E[U^2]}{c^2} \log \frac{\E[U^{2+c}]}{\E[U^2]} - \frac{\E[U^2]}{c} \log \frac{\E[U^{1+c}]}{\E[U^2]},
\end{equation}
hence
\begin{equation}
    0 = \log \left( \frac{\E[U^{1+c}]^c}{\E[U^2]^{c-1} \, \E[U^{2+c}]} \right)
\end{equation}
which is equivalent to \eqref{eqn_speedlimit_functionalidentity}. For optimal $c_*$, apply \eqref{eqn_speedlimit_functionalidentity} to \eqref{ineq_speedlimit_basic}, which yields \eqref{ineq_speedlimit_continuum}.

%compute:
%\begin{eqnarray}
%    \partial_\NS S_\NS &\ge& -\log p_* - \frac{\E[U^2]}{c_*} \log \frac{\E[U^{2+c_*}]}{\E[U^2]} \nonumber \\
%    &=& -\log p_* - \frac{\E[U^2]}{c_*} \log \frac{\E[U^2]^{c_*-1}}{\E[U^2] \E[U^{1+c_*}]^{c_*}} \nonumber \\
%\end{eqnarray}

%use the fact that 
%\begin{equation}
%    \frac{1}{c_*} \log \frac{\E[U^{2+c_*}]}{\E[U^2]} = \frac{1}{c_*} \log \frac{\E[U^{2+c_*} \log U]^c}{\E[U^2]^c} = \log \frac{\E[U^{2+c_*} \log U]}{\E[U^2]},
%\end{equation}
%proving \eqref{ineq_speedlimit_continuum}.

Proof of (3). When we take the limsup as $c \to 0$ in \eqref{ineq_speedlimit_basic}, the expression is indeterminate. Note that $\frac{\d}{\d c} \log F(c) = \frac{1}{F(c)} \frac{\d F}{\d c}(c)$ by the chain rule, and $\frac{\d}{\d c} \E[U^{2+c}] = \frac{\d}{\d c} \E[\e^{(2+c) \log U}] = \E[U^{2+c} \log U]$ by bringing the limit into the expectation \cite{folland2013real}. Thus by L'H\^opital's rule: %to compute: 
\begin{eqnarray}
    \partial_\NS S_\NS &\ge& -\log p_* + \limsup_{c\to 0} \left\{ - \frac{\E[U^2]}{c} \log \frac{\E[U^{2+c}]}{\E[U^2]}  \right\} \nonumber \\
    &=& - \log p_* + \limsup_{c \to 0} \left\{ -\E[U^2] \frac{\E[U^2]}{\E[U^{2+c}]} \frac{\E[U^{2+c} \log U]}{\E[U^2]} \right\} \nonumber \\
    &=& - \log p_* - \E[U^2 \log U].
\end{eqnarray}
%proving \eqref{ineq_speedlimit_infinitary}.
\end{proof}

Combining the Second Law and the Selective Speed Limit, we have the following.

\begin{cor}
Let $w$ be finite-variance with $\E[|U^{2+c'} \log U|] < \oo$ for some $c' > 0$. %in a set $C \subseteq (0,\oo)$ which is dense at $0$. 
Then:
\begin{equation}
    \log \frac{1}{p_*} - \E[U^2 \log U] \le \partial_\NS S_\NS \le \log \frac{1}{p_*} - \frac{1}{p_*} \log \frac{1}{p_*},
\end{equation}
with saturation when $w$ is in selective equilibrium. In particular, selective equilibrium is equivalent to the identity
\begin{equation}
    -S_\NS - \E[U^2 \log U] = \partial_\NS S_\NS = -S_\NS + \E[U^2] S_\NS = \var(U) S_\NS.
\end{equation}
\end{cor}
\begin{proof}
By the Weak Zeroth Law (Proposition \ref{pro_zerothlaw}) and Strong Gibbs' inequality (Theorem \ref{thm_stronggibbsNS}), 
selective equilibrium is equivalent to $\E[U^2] = 1 + \var(U) = \frac{1}{p_*}$ and $S_\NS = - \log \frac{1}{p_*}$. 
\end{proof}

\subsection{Selective Acceleration of Selective Entropy} \label{sect_selectiveacceleration}

We define the selective acceleration of selective entropy as the selective change of the selective change:
\begin{equation}
\partial_\NS^2 S_\NS := \partial_\NS \partial_\NS S_\NS := \cov(-(U-1) U \log U, U) = \E[-(U-1)^2 U \log U]. 
\end{equation}

The next result gives an upper bound on the selective acceleration of selective entropy, amplifying the selective feedback loop: in the presence of non-trivial selective effects, the selective velocity in the second generation is more negative than the selective velocity in the first generation. 

%Unlike Corollary \ref{cor_selectiveaccelerationoffitness}, the selective acceleration can 

%As with Corollary \ref{cor_selectiveaccelerationoffitness}, selection accelerates selection

\begin{thm}[Strong Bounds for Selective Acceleration] \label{thm_upperboundforselectiveacceleration}
Let $w$ be a process for which $\E[|(U-1)^2 U \log U|] < \oo$. Then:
\begin{equation} \label{ineq_selectiveacceleration_upperbound}
    \partial_\NS^2 S_\NS \le - \frac{1}{2} \var(U)^2 \log \var(U)^2 \le 0,
\end{equation}
and
\begin{equation} \label{ineq_selectiveacceleration_lowerbound}
    \partial_\NS^2 S_\NS \ge - \frac{1}{2} \var(U)^2 \log \var(U)^2 - \var(U)^2 \log \frac{\var(U^2) + \var(U)^2}{\var(U)^3} % \nonumber \\
    = \var(U)^2 \log \frac{\var(U^2) + \var(U)^2}{\var(U)^4}
\end{equation}
with saturation of the first inequalities of \eqref{ineq_selectiveacceleration_upperbound} and \eqref{ineq_selectiveacceleration_lowerbound} exactly in the purely environmental case (in which case $\partial_\NS^2 S_\NS = 0$), or in the selective-equilibrium case with $p_* = 1/2$ (in which case $\partial_\NS^2 S_\NS = -2 \log 4 \approx -1.204$). In all other cases, $\partial_\NS^2 S_\NS < - \frac{1}{2} \var(U)^2 \log \var(U)^2 < 0$.
\end{thm}
\begin{proof}
Observe that $\E[(U-1)^2] = \var(U)$. We use Jensen's inequality for the upper bound: %to compute:
\begin{eqnarray}
\partial_\NS^2 S_\NS &=& \E[(U-1)^2 \big( -U \log U \big) ] = \var(U) \frac{1}{\var(U)} \E[(U-1)^2 \big( -U \log U \big) ] \nonumber \\
&\le& - \E[(U-1)^2 U] \log \frac{\E[(U-1)^2 U]}{\var(U)}, \label{proof_selectiveaccelerationofentropy1}
\end{eqnarray}
with saturation when $(U-1)^2 = \var(U)$ almost surely, i.e., when $U = 1 \pm \sqrt{\var(U)}$ almost surely. We again use Jensen's inequality to compute
\begin{equation} \label{proof_selectiveaccelerationofentropy2}
\E[U (U-1)^2] \ge (\E[U^2] - 1)^2 = \var(U)^2,
\end{equation}
with saturation when $U$ is constant $U \mu$-almost surely, i.e., the selective-equilibrium case. 

For each $x$, the function $y \mapsto -x \log y$ is decreasing. Then for each $y'$, the function $x \mapsto -x \log y'$ is also decreasing. Applying \eqref{proof_selectiveaccelerationofentropy2} to \eqref{proof_selectiveaccelerationofentropy1} twice, we have:
\begin{eqnarray}
- \E[(U-1)^2 U] \log \frac{\E[(U-1)^2 U]}{\var(U)} &\le& - \E[(U-1)^2 U] \log \frac{\var(U)^2}{\var(U)} \le - \var(U)^2 \log \frac{\var(U)^2}{\var(U)} \nonumber \\
&=& - \var(U)^2 \log \var(U) = - \frac{1}{2} \var(U)^2 \log \var(U)^2.
\end{eqnarray}
%which proves \eqref{ineq_selectiveacceleration_upperbound}.

For the lower bound, we compute via Jensen's inequality:
\begin{eqnarray}
    \partial_\NS^2 S_\NS &=& \cov(-(U-1) U \log U, U) = \E[(U-1)^2(-U \log U)] \nonumber \\
    %&=& \E[(U-1)^2 U] \frac{1}{\E[(U-1)^2 U]} \E[(U-1)^2 U (- \log U)] \nonumber \\
    &\ge& - \E[(U-1)^2 U] \log \frac{\E[(U-1)^2 U^2]}{\E[(U-1)^2 U]} \nonumber \\
    &=& \E[(U-1)^2 U] \log \frac{\E[(U-1)^2 U]}{\var(U^2) + \var(U)^2},
\end{eqnarray}
where we use the identity
\begin{eqnarray}
\E[(U-1)^2 U^2] &=& \E[U^4 - 2U^2 + 1] 
= \var(U^2) + \E[U^2]^2 - 2 \E[U^2] + 1 \nonumber \\
&=& \var(U^2) + (\var(U)+1)^2 - 2 (1 + \var(U)) + 1 \nonumber \\
&=& \var(U^2) + \var(U)^2.
\end{eqnarray}
Combine \eqref{proof_selectiveaccelerationofentropy2} and the fact that $(x,y) \mapsto x \log y$ is increasing in both arguments, so
\begin{eqnarray}
    \partial_\NS^2 S_\NS
    &\ge& \var(U)^2 \log \frac{\var(U)^2}{\var(U^2) + \var(U)^2} %\nonumber \\
    %«&=& 
    = - \var(U)^2 \log \frac{\var(U^2) + \var(U)^2}{\var(U)^2} \nonumber \\
    &=& - \frac{1}{2} \var(U)^2 \log \var(U)^2 - \var(U)^2 \log \frac{\var(U^2) + \var(U)^2}{\var(U)^3}.
\end{eqnarray}
%which proves \eqref{ineq_selectiveacceleration_lowerbound}.

For the saturation conditions, suppose that $U = 1 \pm \sqrt{\var(U)}$ a.s. and $w$ is in selective equilibrium, so that $1/p_* = 1 + \sqrt{\var(U)} = 1 + \sqrt{1/p_* - 1}$. Solving for $p_*$, we have $1/p_* - 1 = \sqrt{1/p_* - 1}$, so $1/p_* - 1 = 0$ or $1/p_* - 1 = 1$. In the first case, $p_* = 1$ so $w$ is purely environmental. In the second case, $p_* = 1/2$.
\end{proof}

%\subsection{Environmental Change of Selective Entropy} \label{sect_environmentalchangeNS}

%Formally, let $w : \mu \mapsto \mu'$ and $w' : \mu' \mapsto \mu''$ be composable processes, with $U = W/\bar W$ and $U' = W'/\bar W'$. Let $S_\NS = \E[-U \log U]$ and $S_\NS' = \E'[-U' \log U']$ be the selective entropies. Define the adaptive environmental change:
%\begin{eqnarray}
%    \partial_\EC^* S_\NS := 
%    \partial_\EC(S_\NS, S_\NS') &:=& \partial_\EC(-U' \log U' |\!-\!U \log U) \nonumber \\
%    &:=& \E[\Delta_w(-U' \log U' |\!-\!U \log U) U] \nonumber \\
%    &=& \E[(\langle -U' \log U' \rangle_w + U \log U) U].
%\end{eqnarray}

%The Price equation implies that the difference of selective entropies decomposes as the selective change plus the environmental change:
%\begin{equation}
%    S'_\NS - S_\NS = \Delta(-U'\log U'|\!-\!U \log U) = \partial_\NS S_\NS + \partial_\EC(S_\NS, S_\NS').
%\end{equation}
%The first term is controlled by \eqref{ineq_upperboundNSNS}; we analyze the second term in Section \ref{sect_ECNS}. 

%When the adaptive environmental change vanishes, then the difference in selective entropies is non-positive, vanishing only in the purely environmental case. We state this as the following corollary.

\section{Environmental Change of Selective Entropy} \label{sect_environmentalchangeofselectiveentropy}

% \subsection{Relative Form for $\partial_\EC(S_\NS, S_\NS')$}

We analyze the environmental change of selective entropy. Recall the intergenerational relative fitness and its average from Section \ref{sect_environmentalchangevariance}:
\begin{equation}
R(i,i') := \frac{U'(i')}{U(i)} \qquad \mathrm{and} \qquad \bar R_w(i) := \langle R(\cdot|i) \rangle_w := \frac{\langle U' \rangle_w(i)}{U(i)}.
\end{equation}
Using the definition, we have the identity:
\begin{equation}
\E[U^2 \bar R_w] = \E[U \langle U' \rangle_w] = \E'[U'] = 1,
\end{equation}
since $U'$ is the relative fitness for $w'$.

Observe the elementary pointwise identity for any observables $X$ and $Y$:
\begin{equation} \label{eqn_decomposition_entropydiffs_XY}
    -Y \log Y + X \log X = - X \frac{Y}{X} \log \frac{Y}{X} - \left( \frac{Y}{X} - 1 \right) X \log X
\end{equation}
In particular, when $X = U$ and $Y = U'$, we have
\begin{equation} \label{eqn_decomposition_entropydiffs_relativefitness}
    -U' \log U' + U \log U = - U R \log R - (R-1) U \log U.
\end{equation}

When we average \eqref{eqn_decomposition_entropydiffs_relativefitness}, this implies for environmental change:
\begin{eqnarray} 
\partial_\EC(S_\NS,S_\NS') &=& \E\!\left[ \left( \langle -U' \log U'\rangle_w + U \log U \right) U \right] \nonumber \\
&=& \E\!\left[ \langle - U R \log R - (R-1) U \log U \rangle_w U \right] \nonumber \\
&=& \E[\langle -R \log R \rangle_w U ] - \E[(\bar R_w - 1) U^2 \log U].
\label{eqn_adaptiveenvchange}
\end{eqnarray}

%This summarizes the microscopic ``forces'' of environmental change on selective entropy: one term is $-R \log R$ weighted by $U$, and the other term is $-U \log U$ weighted by $R-1$. Both of these terms can be estimated using concavity and convexity.

% \subsection{Upper Bound for $\partial_\EC(S_\NS, S_\NS')$} \label{sect_ECNS}

We now state and prove the upper bound. We use a double Jensen's inequality approach, first leveraging concavity of the function $-R \log R$ and the measure $\langle \cdot \rangle_w(i)$ for each $i$, then concavity of $\bar R_w \log \bar R_w$ against a certain weighted measure. Recall stationarity conditions from Section \ref{sect_environmentalchangevariance}: the coupled process $(w,w')$ is strongly stationary if $R = 1$ a.s., weakly stationary if $\bar R_w = 1$ a.s., and locally homogeneous if $R$ is constant a.s.. 

%Note: this inequality controls the first term of environmental change. We will strengthen this inequality in the next subsection, then we will analyze the second term later.

\begin{thm}[Strong Upper Bound for $\partial_\EC(S_\NS,S_\NS')$] \label{thm_strongupperboundECNS}
Let $w$ and $w'$ be composable processes. Then the environmental change of selective entropy satisfies the upper bound:
\begin{equation}
\partial_\EC(S_\NS, S_\NS') \le \log \E[U^2] + \log \E[U^3]
\end{equation}
This inequality is saturated exactly in the case that $(w,w')$ is strongly stationary (in which case $\partial_\EC(S_\NS,S_\NS') = 0$).

%\begin{eqnarray}
%\partial_\EC^* S_\NS &\le& \E[U^2 (- \bar R_w \log \bar R_w \rangle)] - \E[(\bar R_w - 1) U^2 \log U] \\
%&\le& \log E[U^2] + \log \E[U^3] .
%\end{eqnarray}

%The first inequality is saturated exactly in the case that $(w,w')$ is locally homogeneous. The second inequality is saturated exactly in the case that $(w,w')$ is weakly stationary. 
\end{thm}
\begin{proof}
We start by analyzing the first term of \eqref{eqn_adaptiveenvchange}. First we apply Jensen's inequality to $-R \log R$ using the measure $\langle \cdot \rangle_w$, then we apply Jensen's inequality to $-\bar R_w \log R$ using the measure $\frac{1}{\E[U^2]} \E[U^2 \cdot]$. Note $\E[U^2 \bar R_w] = \E'[U'] = 1$. We compute:
\begin{eqnarray}
\E[U^2 \langle -R \log R \rangle] &\le& \E[U^2 (-\bar R_w \log \bar R_w)] = \E[U^2] \frac{1}{\E[U^2]} \E[U^2 (-\bar R_w \log \bar R_w)] \label{ineq_proofECNS1} \\
&\le& - \E[U^2 \bar R_w] \log \frac{\E[U^2 \bar R_w]}{\E[U^2]} = - \E'[U'] \log \frac{\E'[U']}{\E[U^2]} = \log \E[U^2]., \label{ineq_proofECNS2} 
\end{eqnarray} 

Inequality \eqref{ineq_proofECNS1} is saturated exactly in the case that, for each $i$, $R$ is constant (and equal to $\bar R_w$), meaning locally homogeneous. Inequality \eqref{ineq_proofECNS2} is saturated exactly in the case that $\bar R_w = 1$ a.s., meaning weakly stationary. Both occur exactly in the strongly stationary case.

We split the second term of \eqref{eqn_adaptiveenvchange} into two terms, and apply Jensen's inequality to each:
\begin{eqnarray}
-\E[(\bar R_w - 1) U^2 \log U] &=& \frac{1}{2} \E[\bar R_w (-U^2 \log U^2)] + \E[U^2 \log U] \nonumber \\
&\le& -\frac{1}{2} \E[\bar R_w U^2] \log \E[\bar R_w U^2] + \log \E[U^3] = \log \E[U^3] \label{ineq_proofECNS3}
\end{eqnarray}
since $\E[\bar R_w U^2] = \E[U \langle U' \rangle_w] = \E'[U'] = 1$ and $1 \log 1 = 0$. The first inequality of \eqref{ineq_proofECNS3} is saturated when $U$ is constant $\bar R_w \mu$-a.s., i.e., the strongly stationary case; and the second inequality of \eqref{ineq_proofECNS3} is satisfied when $U^2$ is constant $U \mu$-a.s., i.e., selective equilibrium.  
\end{proof}

An upper bound for the full change $\Delta(S_\NS, S_\NS')$ immediately follows.
%in terms of the bounds for selective and environmental change.

\begin{cor}
Let $w$ and $w'$ be composable processes. Then:
\begin{eqnarray}
    \Delta(S_\NS, S_\NS') = S_\NS' - S_\NS &=& \partial_\NS S_\NS + \partial_\EC(S_\NS,S_\NS') \nonumber \\
    &\le& -\var(U) \log(1+\var(U)) + \log E[U^2] + \log \E[U^3] \nonumber \\
    &=& (1-\var(U)) \log(1 + \var(U)) + \log \E[U^3],
\end{eqnarray}
with saturation exactly when $w$ is strongly stationary. 
\end{cor}
\begin{proof}
This follows from Theorems \ref{thm_secondlaw} and \ref{thm_strongupperboundECNS}. Observe that $\log E[U^2] = \log(1+ \var(U))$. If $(w,w')$ is strongly stationary, then $w$ is purely environmental hence in selective equilibrium, so the bound on the first term is saturated. 
\end{proof}

\section{Multi-Level Selective Entropy} \label{sect_multilevel_selentropy}

We apply the multi-level Price equation to selective entropy, allowing us to isolate the selective information generated in the second stage of the process, as distinct from the initial selective information. 
Consider composable processes $w : \mu \mapsto \mu'$, $w' : \mu' \mapsto \mu''$, and $w'' : \mu'' \mapsto \mu'''$, with selective entropies $S_\NS = \E[-U \log U]$, $S'_\NS = \E'[-U' \log U']$, and $S''_\NS = \E''[-U'' \log U'']$, respectively.
Write the secondary selective change of selective entropy as $\partial'_\NS S'_\NS := \cov'(-U' \log U', U') = \E'[(-U' \log U') (U'-1)]$. The Strong Second Law of Natural Selection (Theorem \ref{thm_secondlaw}) ensures that $\partial'_\NS S'_\NS \le -\var'(U') \log\!\left(\var'(U')+1\right)$, with saturation in the case that $w'$ is in selective equilibrium. We improve upon this by incorporating multi-level information. 
Define the conditional expectation $\E'_w[Y] := \langle Y \rangle_w U$ and conditional covariance $\cov_w(Y,Y') := \E'_w[Y Y'] - \E'_w[Y] \E'_w[Y']$. 

The multi-level Price equation (Theorem \ref{thm_multilevelprice}) ensures that 
\begin{eqnarray}
    \partial'_\NS S'_\NS &=& \cov\!\left( \E'_w[-U' \log U'], \E'_w[U'] \right) + \E\!\left[ \cov_w(-U' \log U',U') \right] \label{eqn_multilevelprice_selectivenentropy_NS} \\
    \partial'_\EC(S'_\NS, S''_\NS) &=& \E\!\left[ \E'_w[\Delta_{w'}(-U' \log U', -U'' \log U'') U' ] \right] \label{eqn_multilevelprice_selectivenentropy_EC} \\
    \Delta(S'_\NS, S''_\NS) &=& \cov\!\left(\E'_w[-U' \log U'], \E'_w[U'] \right) + \E\!\left[ \cov_w(-U' \log U',U') \right] \nonumber \\
    && ~+~ \E\!\left[ \E'_w[\Delta_{w'}(-U' \log U', -U'' \log U'') U' ] \right]. \label{eqn_multilevelprice_selectivenentropy}
\end{eqnarray}

% We compute bounds on the both conditional terms, and refer to these as the Multi-Level Second Laws of Natural Selection. This constrains part of the environmental change using selective information from the first two terms. This does \emph{not} require any information from the subsequent process $w''$, unlike the remaining part of the environmental change. 

The following identity allows us to relate variances at different levels. 

\begin{lem}[Multi-Level Variance Identity] \label{lem_multilevelvariance}
Let $w$ and $w'$ be composable processes. Then:
\begin{equation} \label{eqn_multilevelvariance}
    \var'(U') = \var\!\left(U^{(2)}\right) + \E[\var'_w(U')]. 
\end{equation} 
\end{lem}
\begin{proof}
We compute:
\begin{eqnarray}
    \var'(U') - \var\!\left(U^{(2)}\right) &=& \E'[(U')^2] - \E[(U^{(2)})^2] = \E[\E'_w[(U')^2] - (U^{(2)})^2] \nonumber \\
    &=& \E[\E'_w[ (U')^2 - \E'_w[U']^2 ]] = \E[\var'_w(U')],
\end{eqnarray}
since $\E'[U'] = 1 = \E[U^{(2)}]$, $\E'_w[U'] = U^{(2)}$, and $\var'_w(U') = \E'_w[(U')^2] - \E'_w[U']^2 = \E'_w[(U')^2] - (U^{(2)})^2$.
\end{proof}

By applying the Second Law (Theorem \ref{thm_secondlaw}) and the multi-level variance identity \eqref{eqn_multilevelvariance}, we have the following multi-level version of the Second Law.

% We prove a ``Multi-Level Second Law,'' specifying bounds on the selective covariance term \eqref{eqn_multilevelprice_selectivenentropy_NS}. This is a strong bound, and the inequality is saturated in the case of selective equilibrium of all processes involved ($w$, $w'$ and $w' \circ w$), as well as local constancy of the joint process. This allows us to measure the amount of ``new'' selective information generated in a subsequent process, which was not already present in the initial process. 

% Claim: $U'$ is constant $\mu'_*$-almost surely if and only if: $U^{(2)}$ is constant $\mu_{**}$-almost surely AND for $\mu$-almost every $i$, $U'$ is constant $w_i$ almost surely.

\begin{thm}[Multi-Level Second Law of Natural Selection] \label{thm_multilevel_secondlaw}
Let $w$ and $w'$ be composable processes. Then:
\begin{eqnarray}
    \partial'_\NS S'_\NS &\le& -\var'(U') \log\!\left( 1 + \var'(U') \right) \label{ineq_multilevelsecondlaw1} \\
    &=& -\left(\var(U^{(2)}) + \E[\var'_w(U')]\right) \log\!\left(1 + \var(U^{(2)}) + \E[\var'_w(U')]\right) \label{ineq_multilevelsecondlaw2} %\\
    %&=& -\var(U^{(2)}) \log(1+ \var(U^{(2)})) + 
\end{eqnarray}
which is saturated when $w'$ is in selective equilibrium ($U'$ is constant $\mu'_*$-a.s.).
% The saturation condition is equivalent to $U^{(2)}$ being constant on a set $A_{**}$ of full $\mu_{**}$-measure, and for all $i \in A_{**}$, $U'$ is constant $U' w_i$-almost surely. 
\end{thm}
\begin{proof}
\eqref{ineq_multilevelsecondlaw1} and the saturation condtiion follows from the Second Law (Theorem \ref{thm_secondlaw}) applied to the process $w'$. \eqref{ineq_multilevelsecondlaw2} follows from the variance identity \eqref{eqn_multilevelvariance}.

%Suppose $U^{(2)}$ is constant $\mu_{**}$-almost surely, and for $\mu_{**}$-almost every $i$, $U'$ is constant $U' w_i$-almost surely. Then $\var_{**}(U^{(2)}) = $$\var'(U') = $
\end{proof}

%\subsection{Environmental Change of Multi-Level Selective Entropy}

%\begin{thm}[Strong Upper Bound for $\partial'_\EC(S'_\NS, S''_\NS)$]
%Let $w$, $w'$ and $w''$ be composable processes. Then:
%\begin{eqnarray}
%    \partial'_\EC(S'_\NS, S''_\NS) &\le& \log \E'[(U')^2] + \log \E'[(U')^3] \\
%    &=& \log \E[\E'_w[[(U')^2]] + \log \E[\E'_w[[(U')^3]]
%\end{eqnarray}
%\end{thm}

\section{Quantum Selective Entropy} \label{sect_quantum_selentropy}

Recall the quantum formalism of Section \ref{sect_quantum}. Consider a quantum evolutionary process $\W : \mu \mapsto \mu'$, with quantum relative fitness operator $U := \frac{1}{\bar W} \W^\dagger(\Id')$. 
Define the selective entropy operator $-U \log U$ using the spectral theorem. We say that $\W$ is finite entropy if $\E_\mu[|U\log U|] = \Tr(|U \log U| \mu) < \oo$. Define the quantum selective entropy
\begin{equation}
    S_\NS := \E_\mu[-U \log U] = \frac{1}{N} \Tr\!\left((-U \log U) \mu\right) = \frac{1}{N} \Tr\!\left(\mu(-U \log U)\right).
\end{equation}

Write $\pi_* = \pi_{U\ne 0} = \Id - \pi_{U=0}$ for the projection onto the childbearing subspace, orthogonal to the null space of $U$. Write the childbearing population $\mu_* = \pi_* \mu$, and the childbearing proportion $p_* := \E_\mu[\pi_*] = \frac{1}{N} \Tr(\pi_* \mu) = \frac{1}{N} \Tr(\mu_*)$. 

\begin{thm}[Strong Quantum Gibb's Inequality]
Let $\W$ be a finite-entropy quantum evolutionary process. Then:
\begin{equation}
    -\log \left( 1 + \var_\mu(U) \right) \le S_\NS \le \log p_*,
\end{equation}
with saturation in the quantum selective-equilibrium case (in which case $S_\NS= \log p_* = -\log(1+\var_\mu(U))$.
\end{thm}
\begin{proof}
The proof is similar to Theorem \ref{thm_stronggibbsNS}, \emph{mutatis mutandis}. We have $\E_\mu[X] = (1-p_*)\E_0[X] + p_* \E_*[X]$. For the upper bound, we use quantum Jensen's inequality: $S_\NS = p_* \E_*[-U \log U] \le -p_* \E_*[U] \log \E_*[U] = \log p_* \le 0$. This is saturated when $U$ is constant $\mu_*$-a.s., i.e., quantum selective equilibrium.

For the lower bound, we use quantum Jensen's inequality: $S_\NS = \E_\mu[U(-\log U)] = \frac{1}{N} \Tr((-\log U) \mu U) \ge - \log \left( \frac{1}{N} \Tr(U\mu U) \right) = -\log \E_\mu[U^2]$. This is saturated when $U$ is constant $\mu U$-a.s. Since $\mu U$ and $\mu_*$ have the same null subspace, this is equivalent to quantum selective equilibrium. 
\end{proof}

\subsection{Quantum Second Law}

Define the selective change of quantum selective entropy: $\partial_\NS S_\NS := \cov_\mu(-U \log U, U) = \E_\mu[(-U \log U)(U-1)]$. 

\begin{thm}[Strong Second Law of Quantum Selection]
Let $\W$ be a quantum evolutionary process with $\E_\mu[|U^2 \log U|] < \oo$. Then
\begin{equation}
\partial_\NS S_\NS \le -\var(U) \log(1 + \var(U)) \le \var(U) S_\NS \le (\e^{-S_\NS} - 1) S_\NS \le \left( \frac{1}{p_*} - 1 \right) \log p_* \le 0,
\end{equation}
with saturation of all but the last inequality in the quantum selective-equilibrium case. 
\end{thm}
\begin{proof}
The proof is similar to Theorem \ref{thm_secondlaw}, \emph{mutatis mutandis}. We write $\partial_\NS S_\NS  = \E_\mu[-U^2 \log U] - S_\NS$. We control the first term with quantum Jensen's inequality (Lemma \ref{lem_quantumjensen}): $\E_\mu[U (-U \log U)] \le -\E_\mu[U^2] \log \E_\mu[U^2] = -(1+\var_\mu(U)) \log (1+\var_\mu(U))$. We use the strong lower bound for $S_\NS$ for the upper bound: $-S_\NS \le \log(1 + \var_\mu(U))$. Combining these terms we have the result. The other inequalities follow by applying different versions of the strong bounds for $S_\NS$. Saturation holds when $U$ is constant $U\mu$-a.s., %which is equivalent to $\mu_*$-a.s., 
i.e., quantum selective equilibrium. 
\end{proof}

\begin{thm}[Upper Bound for Quantum Selective Acceleration]
Let $\W$ be a quantum process for which $\E_\mu[|(U-1)^2 U \log U|] < \oo$. Then:
\begin{equation}
    \partial_\NS^2 S_\NS \le - \frac{1}{2} \var(U)^2 \log \var(U)^2 \le 0,
\end{equation}
with saturation of the first inequality exactly in the quantum purely environmental case (in which case $\partial_\NS^2 S_\NS = 0$), or in the quantum selective-equilibrium case with $p_* = 1/2$ (in which case $\partial_\NS^2 S_\NS = -2 \log 4 \approx -1.204$). In all other cases, $\partial_\NS^2 S_\NS < - \frac{1}{2} \var(U)^2 \log \var(U)^2 < 0$.
\end{thm}
\begin{proof}
The proof is similar to Theorem \ref{thm_upperboundforselectiveacceleration}, \emph{mutatis mutandis}. We use quantum Jensen's inequality (Lemma \ref{lem_quantumjensen}):
$\partial^2_\NS S_\NS = \var_\mu(U) \frac{1}{\var_\mu(U)} \E_\mu[(U-1)^2(-U \log U)] \le -\E_\mu[(U-1)^2 U] \log \frac{\E_\mu[(U-1)^2 U]}{\var_\mu(U)}$. We again use Jensen's inequality to compute $\E_\mu[(U-1)^2 U] \ge (\E_\mu[U^2]-1)^2 = \var_\mu(U)^2$. Since $(x,y) \mapsto -x \log y$ is decreasing in each argument, we have: $\partial^2_\NS S_\NS \le -\var_\mu(U)^2 \log \var_\mu(U)$. Saturation holds when $U$ is constant $(U-1)^2$-a.s.
\end{proof}

% \newpage
\part{Environmental Entropy (One-Step Kolmogorov-Sinai Entropy)} \label{part_environmentalentropy}

\section{Definitions of Environmental Entropy and Total Entropy} \label{sect_environmentalentropy}

%\section{Definition of Environmental Entropy and Total Entropy (97\% Done)} \label{sect_environmentalentropy}

%\subsection{Definition of Environmental Entropy}

In this section, we introduce \emph{environmental entropy} $S_\EC$ to measure environmental effects along a process, defined as the Kolmogorov-Sinai entropy of the environmental part of the process. 
Unlike the selective entropy, the environmental entropy is defined by measuring local redistributions between pairs of sets. We then take the sum over any partition, and define the general environmental entropy as the supremum of this quantity over all partitions. 

%We first define a form of ``local environmental entropy''. 
Let $A \subseteq I$ and $B \subseteq I'$ be measurable sets. We define the \emph{local fitness function} from $A$ to $B$ by restricting the process to parent set $A$ and child set $B$: 
\begin{equation}
    W_{A,B}(i) := 1_A(i) w_i(B)
\end{equation}
%, i.e., it only counts fitness for individuals $i$ of type $A$ based on their children of type $B$. 
That is, $W_{A,B}(i)$ is the number of children of an individual $i$ of parent set $A$ who are members of child set $B$. Note that $W_{A,B} \le W$. 
We define the \emph{local relative fitness} (LRF) by dividing the local fitness by the selective coefficient:
\begin{equation}
    U_{A,B}(i) := 1_A(i) \frac{w_i(B)}{\bar W} \ge 0.
\end{equation}
Note that $U_{A,B}(i) \le U(i)$. %The LRF $U_{A,B}$ represents the amount of fitness contributed to type $B$ \emph{given} that $i$ is of type $A$, and is essentially the conditional expectation of $U$. 
Write the average LRF as $\bar U_{A,B} = \E[U_{A,B}]$. % Write $\mu_{A,B} := \frac{U_{A,B}}{\bar U_{A,B} U}$.

The LRF decomposes the relative fitness into four local pieces:
\begin{equation} \label{eqn_relativefitnessdecomp}
    U = U_{I,I'} = U_{A,B} + U_{A^c,B} + U_{A,B^c} + U_{A^c,B^c}.
\end{equation}
%for all $i$. This identity is helpful in proofs and describes completeness of fitness in terms of countable partitions. 

% Observe that $U_{I,I'} = U$. Since the LRF is defined using a measure, we have $U_{B^c|A} + U_{A,B} = U_{I'|A}$ and $U_{A,B^c} + U_{A,B} = U_{B|I}$. In particular, relative fitness is preserved:
% \begin{equation}
% U = U_{I,I'} = U_{B|I} + U_{B^c|I} = U_{A,B} + U_{A,B^c} + U_{B^c|A} + U_{B^c|A^c}.
%\end{equation}
% for all $i$. 

We define \emph{environmental entropy} as a one-step version of the familiar Kolmogorov-Sinai entropy from probability and dynamical systems. This is defined locally relative to parent and child sets; partitionally relative to countable, measurable partitions; and generally by taking suprema over all partitions. 

%As We define environmental entropy at three levels: local environmental entropy, relative to sets $A$ and $B$; partition environmental entropy, relative to countable, measurable partitions $\A$ and $\B$; and general environmental entropy, as the supremum over all countable, measurable partitions. 

\begin{defn}[Environmental Entropy]
Let $w : \mu \to \mu'$ be an evolutionary process.
\begin{enumerate}
    \item Consider measurable sets $A \subseteq I$ and $B \subseteq I'$. We define the \emph{local environmental entropy} from $A$ to $B$ as: 
    % of $B$ given $A$ as:
    % the average of relative surprisal $-\!\log U_{A,B} / U$, weighted by the local fitness $U_{A,B}$:
        \begin{equation} \label{def_localentropyEC}
            S_\EC(A,B) := - \bar U_{A,B} \log \bar U_{A,B} \ge 0.
            % S_\EC(A,B) := \E\!\left[-U_{A,B} \log \frac{U_{A,B}}{U}\right] := \tilde \E\!\left[- \frac{U_{A,B}}{U} \log \frac{U_{A,B}}{U} \right], %\ge 0,
        \end{equation}
    % where $\tilde \E[X] := \E[UX]$.
    
    \item Consider countable, measurable partitions $\A$ of $I$ and $\B$ of $I'$. We define the \emph{partition environmental entropy} from $\A$ to $\B$ as:
        \begin{equation} \label{def_partitionentropyEC}
            S_\EC(\A,\B) := \sum_{A \in \A, B \in \B} S_\EC(A,B) = \sum_{A \in \A, B \in \B} \left(- \bar U_{A,B} \log \bar U_{A,B}\right) \ge 0.
            % S_\EC(A,B) := \E\!\left[-U_{A,B} \log \frac{U_{A,B}}{U}\right] := \tilde \E\!\left[- \frac{U_{A,B}}{U} \log \frac{U_{A,B}}{U} \right], %\ge 0,
        \end{equation}
        
    \item We define the \emph{general environmental entropy} from $\mu$ to $\mu'$ as:
        \begin{equation} \label{def_entropyEC}
            S_\EC := \sup_{\A,\B} S_\EC(\A,\B) = \sup_{\A,\B}  \sum_{A \in \A, B \in \B} \left(- \bar U_{A,B} \log \bar U_{A,B}\right) \ge 0,
        \end{equation}
    where the supremum is over all countable, measurable partitions $\A$ of $I$ and $\B$ of $I'$.
\end{enumerate}
\end{defn}

The environmental entropy functionals are non-negative. To see this, note that $U_{A,B} \le U$, hence $\bar U_{A,B} \le 1$ and so $-\bar U_{A,B} \log \bar U_{A,B} \ge 0$. Sinai's Theorem (Theorem \ref{thm_sinai}) ensures this supremum can be realized for a certain pair of partitions.

%The local environmental entropy from $A$ to $B$ vanishes exactly when the purely environmental part of the process is fully concentrated from $A$ to $B$. The partition environmental entropy vanishes when this condition holds for one pair of partition sets. We will show later (Theorem \ref{thm_reversibility}) that vanishing of environmental entropy corresponds to the purely environmental process being reversible, in the sense of invertible processes.

\begin{lem} \label{lem_vanishinglocalselectiveentropy}
For measurable $A \subseteq I$ and $B \subseteq I'$: $S_\EC(A,B) = 0$ if and only if $A \cap w^{-1} B = I$ or $\varnothing$. For countable, measurable partitions $\A$ and $\B$ of $I$ and $I'$: $S_\EC(\A,\B) = 0$ if and only if $A \cap w^{-1} B = I$ for a single pair $(A,B) \in \A \times \B$. 
\end{lem}
\begin{proof}
If $A \cap w^{-1} B = I$, then $U_{A,B} = U$ a.s. and $\bar U_{A,B} = 1$, so $S_\EC(A,B) = 0$. If $A \cap w^{-1} B = \varnothing$, then $U_{A,B} = 0$ a.s. and $\bar U_{A,B} = 0$, so $S_\EC(A,B) = 0$. If $A \cap w^{-1} B$ is non-empty and $\ne I$, then $U_{A,B} \in (0,U)$, and so $\bar U_{A,B} \in (0,1)$ and $S_\EC(A,B) > 0$.% This proves the first statement.

If $S_\EC(\A,\B) = 0$, then $S_\EC(A,B) = 0$ for all $(A,B) \in \A \times \B$, and only one of these pairs can satisfy $A \cap w^{-1} B = I$. If $S_\EC(\A,\B) > 0$, then $S_\EC(A,B) > 0$ for some $(A,B) \in \A \times \B$. For that pair, at least one of $U_{A^c,B}$, $U_{A,B^c}$, and $U_{A^c,B^c}$ must be positive. Let $(A',B') \in \A \times \B$ be another pair such that $U_{A',B'} > 0$, then $S_\EC(A',B') > 0$. %This proves the second statement.
\end{proof}

\begin{rem}[Kolmogorov-Sinai Entropy]
The environmental entropy is a ``one step'' form of Kolmogorov-Sinai entropy, which is instead optimized over all iterates of a process. To see this formally, let $w^t : \mu^{t-1} \mapsto \mu^t$ be a family of composable evolutionary processes on the spaces $I^t$. Classically, the Kolmogorov-Sinai entropy is defined by iterating a single process $w$ on a static state space. Let $w^{(T)} : \mu \mapsto \mu^T$ be the $T$-step iterated process $w^{(T)} := w^T \circ \cdots w^1$ from $I^0$ to $I^T$. Let $w^{-(T)} : \I^T \to \I^0$ be the $T$-step parent-set mapping. Define the $T$-step selective coefficient $\bar W^{(T)} := \bar W^1 \cdots \bar W^T = \frac{N^T}{N}$. %Write the set product of partitions $\vec \A^{(T)} := \A^0 \times \cdots \times \A^T$, and the set tuple $\vec A^T := (A^0,\cdots,A^T) \in \A^{(T)}$. 
Define the $T$-step local relative fitness for $i \in I^0$,
    \begin{equation}
        U_{A^0,\cdots,A^T}(i) := 1_{A^0 \cap w^{-1}A^1 \cap w^{-(2)} A^2 \cdots \cap w^{-(T)}A^T}(i) \frac{w^{(T)}_i(A^T)}{\bar W^{(T)}}.    
        %U_{\vec A^T}(i) := 1_{A^0 \cap w^{-1}A^1 \cap w^{-(2)} A^2 \cdots \cap w^{-(T)}A^T}(i) \frac{w^{(T)}_i(A^T)}{\bar W^{(T)}}.    
    \end{equation}
    The Kolmogorov-Sinai entropy %of the environmental part of the process 
    is the supremum over $T$-step environmental entropies:
        \begin{equation} \label{def_KSentropy}
            S_\KS := \sup_{T \ge 1} \sup_{\A^0, \cdots, \A^T} \sum_{(A^0,\cdots,A^T) \in \A^0 \times \cdots \times \A^T} \left( - \E[U_{A^0,\cdots,A^T}] \log \E[U_{A^0,\cdots,A^T}] \right) \ge S_\EC,
            %S_\KS := \sup_{T \ge 1} \sup_{\A^0,\cdots,\A^T} \sum_{A^0 \in \A^0, \cdots, A^T \in \A^T} \left( - \E[U_{A^0,\cdots,A^T}] \log \E[U_{A^0,\cdots,A^T}] \right),
            % \vec A^T \in \vec \A^{(T)}
        \end{equation}
    where the first supremum is over natural numbers $T \ge 1$, the second supremum is over countable, measurable partitions $\A^0, \cdots, \A^T$ of $I^0, \cdots, I^T$, and the sum is over partition sets. We have $S_\KS = S_\EC$ if each $T$-step environmental entropy is at most the $1$-step environmental entropy, otherwise $S_\KS \ge S_\EC$. The Kolmogorov-Sinai entropy is an invariant of a sequence of processes (i.e., $S_\KS(w,w^2,\dots) = S_\KS(w^2,w^3,\dots)$).
\end{rem}

\subsection{Generalized Sinai's Theorem}

Recall Sinai's classic theorem \cite{sinai1959notion}, which states that Kolmogorov-Sinai entropy on a static space is realized by a \emph{generating partition}, a single countable, measurable partition $\A_*$ which realizes \eqref{def_KSentropy}. We state a generalized version of Sinai's theorem for environmental entropy, allowing for generating joint partitions which realize environmental entropy and generalized Kolmogorov-Sinai entropy.

\begin{thm}[Generalized Sinai's Theorem] \label{thm_sinai}
    (One-Step) Let $w : \mu \mapsto \mu'$ be an evolutionary process. There exist countable, measurable partitions $A_*$ of $I$ and $\B_*$ of $I'$ so that
        \begin{equation}
            S_\EC = S_\EC(\A_*,\B_*).
        \end{equation}
    A joint partition $(\A_*, \B_*)$ is generating if and only if $\I$ is the smallest $\sigma$-algebra containing sets $A \cap w^{-1} B$ for $A \in \A_*, B \in \B_*$.
    
    (Iterated) Let $w$ be an evolutionary process on the same space $I$. There exists a countable, measurable partition $\A_*$ of $I$ such that
        \begin{equation}
            S_\KS = \sup_{T \ge 1} \sum_{A^0,\cdots,A^T \in \A_*} \left( - \E[U_{A^0,\cdots,A^T}] \log \E[U_{A^0,\cdots,A^T}] \right)
        \end{equation}
\end{thm}

We prove both forms of Sinai's theorem in Appendix \ref{app_sinai}. Our argument generalizes the proofs of \cite[Theorems 5,6]{le2017notes} and \cite[Section 1.7]{downarowicz2011entropy} from the classical Sinai's theorem.

\subsection{Local Selective Entropy}

Recall the definition of selective entropy $S_\NS = \E[-U \log U]$ from Part 2. For any measurable $A \subseteq I$ and $B \subseteq I'$, we define local selective entropy $S_\NS(A,B) := \E[-U_{A,B} \log U]$. For any countable, measurable partitions $\A$ and $\B$, the partition selective entropy equals the general selective entropy: $\sum_{A \in \A, B \in \B} S_\NS(A,B) = S_\NS$. This follows from additivity: $\sum_{A,B} S_\NS(A,B) = \E\!\left[ (-\sum_{A,B} U_{A,B} \log U) \right] = S_\NS$ since $\sum_{A,B} U_{A,B} = U$.

\begin{defn}
We say that $w$ is locally purely selective from $A$ to $B$ when local environmental entropy vanishes ($S_\EC(A,B) = 0$), meaning $A \cap w^{-1} B = \varnothing$ or $I$ by Lemma \ref{lem_vanishinglocalselectiveentropy}.
We say that $w$ is locally purely environmental from $A$ to $B$ when $U = 1$ $\mu$-a.s. on $A \cap w^{-1} B$. % This is equivalent to local selective entropy vanishing. 
% local selective entropy vanishes ($S_\NS(A,B) = 0$), which means $U_{A,B}(i) \in \{0, U(i)\}$ for $\mu$-almost every $i$. 
\end{defn}

%We say that $w$ is locally purely environmental from $A$ to $B$ if $U = 1$ almost surely on $A \cap w^{-1} B$. 

% We say $w$ is in local selective equilibrium from $A$ to $B$ if $U$ is a.s. constant on $A \cap w^{-1} B$. 

\begin{lem}[Non-Positivity of Local Selective Entropy] \label{lem_localselectiveentropynonpositive}
For any measurable $A \subseteq I$ and $B \subseteq I'$, the local selective entropy is non-positive:
\begin{equation} \label{ineq_localNSbound}
    S_\NS(A,B) \le 0. 
\end{equation}
This is saturated exactly when $w$ is locally purely environmental from $A$ to $B$.

% STRONG BOUND??????
% Furthermore, ...
%\begin{equation}
%    test
%\end{equation}
\end{lem}

We prove Lemma \ref{lem_localselectiveentropynonpositive} in Appendix \ref{app_localselectiveentropynonpositive}. The proof involves the completeness identity \eqref{eqn_relativefitnessdecomp}.

\subsection{Total Entropy}

We define the total entropy as the sum of selective entropy and environmental entropy. Total entropy can be negative or positive, depending on whether the process is ``more selective'' or ``more environmental''. 

\begin{defn}[Total Entropy]
The total entropy of a process is defined as the sum of the selective entropy and environmental entropy at each level. 
%\begin{equation}
%    S_\tot(A,B) := S_\NS(A,B) + S_\EC(A,B)
%\end{equation}
%\begin{equation}
%    S_\tot(\A,\B) := S_\NS + S_\EC(\A,\B)
%\end{equation}
%\begin{equation}
%    S_\tot := S_\NS + S_\EC
%\end{equation}
%\begin{itemize}
%    \item 
    For each $A \in \I$ and $B \in \I'$, we define local total entropy:
    \begin{equation}
        S_\tot(A,B) := S_\NS(A,B) + S_\EC(A,B).
    \end{equation}
%    \item 
    For each countable, measurable $A$ of $I$ and $\B$ of $I'$, we define partition total entropy:
    \begin{equation}
        S_\tot(\A,\B) := S_\NS + S_\EC(\A,\B);
    \end{equation}
%    \item 
and general total entropy:
\begin{equation} \label{def_totalentropy}
        S_\tot := S_\NS + S_\EC.
    \end{equation}
%\end{itemize}

%\begin{eqnarray}
%    S_\tot(A,B) &:=& S_\NS(A,B) + S_\EC(A,B) \nonumber \\
%    S_\tot(\A,\B) &:=& S_\NS + S_\EC(\A,\B) \nonumber \\
%    S_\tot &:=& S_\NS + S_\EC. \label{def_totalentropy}
%\end{eqnarray}
\end{defn}

\begin{cor} \label{cor_localextremes}
\textbf{ }
\begin{itemize}
\item $w$ is locally purely environmental from $A$ to $B$ iff $S_\tot(A,B) = S_\EC(A,B) \ge 0$.
\item $w$ is locally purely selective from $A$ to $B$ iff $S_\tot(A,B) = S_\NS(A,B) \le 0$.
\end{itemize}
\end{cor}
\begin{proof}
Suppose that $w$ is purely environmental, so that $U = 1$ almost surely. Then $S_\NS(A,B) = \E[U_{A,B} \log 1] = 0$, so $S_\tot(A,B) = S_\EC(A,B) \ge 0$. Conversely, if $S_\NS(A,B) = 0$, then $w$ is locally purely environmental.

Suppose that $w$ is purely selective. Then $w_i(B) = 1_B(i) W(i)$ for any measurable $B$ and a.e. $i$, so $U_{A,B} = 1_{A \cap B}(i) U(i)$. In that case, $S_\EC(A,B) = \E[-1_{A \cap B} U \log 1_{A \cap B}]$. If $A \cap B = \varnothing$, then $S_\EC(A,B) = 0$. If $A \cap B \ne \varnothing$, then $\log 1_{A \cap B} = 0$ on the set $A \cap B$, so $S_\EC(A,B) = 0$. Consequently, $S_\tot = S_\NS \le 0$.
\end{proof}

% Open question: what happens in cases of selective equilibrium (local from $A$ to $B$ or general)? What about if both in selective equilibrium and environmental equilibrium? Is that a good notion of total equilibrium? 

\section{Dispersion Entropy and Mixing Entropy} \label{sect_dispersionmixing}

%Our goal is to decompose the environmental entropy into simpler components. 
In this section, we introduce dispersion entropy, which represents the ``spreading'' of a process, and mixing entropy, which represents the ``coalescing'' of a process. Both these entropies are non-negative, and we show that environmental entropy decomposes as the sum of dispersion and mixing entropies. We also present strong bounds on these entropies.

We introduce dispersion and mixing coefficients $D_{A,B}$ and $M_{A,B}$ to help us quantify dispersive and mixing effects. The dispersion coefficient is the ratio of local relative fitness $U_{A,B}$ to total relative fitness $U$, measuring how much dispersion from set $A$ to $B$. The mixing coefficient is further normalized by $\bar U_{A,B}$.

% We introduce two simple observables to help us quantify dispersive and mixing changes. 

% \subsection{Dispersion and Mixing Coefficients}

\begin{defn}[Dispersion and Mixing Coefficients]
Consider measurable $A \subseteq I$ and $B \subseteq I$. 
%\begin{enumerate}
%\item 
Define the \emph{dispersion coefficient} from $A$ to $B$ as the ratio of local relative fitness to relative fitness:
\begin{equation} \label{def_dispersioncoefficient}
    D_{A,B}(i) := \frac{U_{A,B}(i)}{U(i)} = \frac{W_{A,B}(i)}{W(i)} = \frac{1_A(i) w_i(B)}{w_i(I')},
\end{equation}
and the \emph{mixing coefficient} from $A$ to $B$ by normalizing by the average local relative fitness:
\begin{equation}
    M_{A,B}(i) := \frac{D_{A,B}(i)}{\bar U_{A,B}} = \frac{U_{A,B}(i)}{\bar U_{A,B} U(i)} = \frac{W_{A,B}(i)}{\bar U_{A,B} W(i)} = \frac{1_A(i) w_i(B)}{\bar U_{A,B} w_i(I')}.
\end{equation}
\end{defn}

These coefficients satisfy the bounds $D_{A,B} \in [0,1]$ and $M_{A,B} \in \left[0, \frac{1}{\bar U_{A,B}}\right]$. When we average over the intermediate population, we have:
\begin{equation}
    \tilde \E[D_{A,B}] = \bar U_{A,B} \qquad \mathrm{and} \qquad \tilde \E[M_{A,B}] = 1.
\end{equation}
When we average over the initial population, we have the non-reduced $\E[D_{A,B}] = \E\!\left[\frac{U_{A,B}}{U}\right]$ and. $\E\!\left[M_{A,B}\right] = \frac{\E[D_{A,B}]}{\bar U_{A,B}} = \E\!\left[\frac{U_{A,B}}{\bar U_{A,B} U}\right]$.

%\subsection{Dispersion Entropy} \label{sect_dispersionentropy}
% and Environmental Entropy}

\subsection{Definitions of Dispersion and Mixing Entropies}

% We define dispersion entropy as the pointwise environmental entropy. Using Jensen's inequality, we show that dispersion entropy is always bounded above by environmental entropy. 

Dispersion entropy is the amount of environmental entropy generated by asexual or clonal reproduction, i.e., dispersion of an individual type. Mixing entropy is the amount of environmental entropy generated by sexual reproduction, i.e., mixing of two distinct types. The dispersion and mixing entropies are non-negative (Lemma \ref{lem_lowerbounddispersionmixingentropies}), and their sum is environmental entropy (Proposition \ref{pro_environmentalentropydecomposition}). 
Write $\tilde \E[Y] = \E[U Y]$ for the intermediate expectation. 

\begin{defn}[Dispersion and Mixing Entropies]
Let $w$ be a finite-entropy process. 
\begin{enumerate}

\item 
%\item 
Consider measurable $A \subseteq I$ and $B \subseteq I'$. Define the \emph{local dispersion entropy} from $A$ to $B$ as:
\begin{equation}
    S_\dis(A,B) := \E\!\left[ - U_{A,B} \log D_{A,B} \right] = \tilde \E\!\left[ - D_{A,B} \log D_{A,B} \right] \ge 0,
    % \E\!\left[ - U_{A,B} \log \frac{U_{A,B}}{U} \right] = \tilde \E\!\left[ - \frac{U_{A,B}}{U} \log \frac{U_{A,B}}{U} \right],
\end{equation}
and the \emph{local mixing entropy} from $A$ to $B$ as:
\begin{equation}
    S_\mix(A,B) := \E\!\left[U_{A,B} \log M_{A,B} \right] = \bar U_{A,B} \tilde \E\!\left[M_{A,B} \log M_{A,B}\right] \ge 0. 
\end{equation}

\item Consider countable, measurable partitions $\A$ and $\B$ of $I$ and $I'$. Define the \emph{partition dispersion entropy} from $\A$ to $\B$ by summing over partition sets:
\begin{equation}
    S_\dis(\A,\B) := \sum_{A\in\A, B\in\B} S_\dis(A,B) = \sum_{A,B} \E\!\left[ - U_{A,B} \log D_{A,B} \right] \ge 0,
\end{equation}
and \emph{partition mixing entropy} from $\A$ to $\B$ by:
\begin{equation}
    S_\mix(\A,\B) := \sum_{A\in\A, B\in\B} S_\mix(A,B) = \sum_{A,B} \E\!\left[ U_{A,B} \log M_{A,B} \right] \ge 0.
\end{equation}
% This function describes the Shannon entropy of a joint partition $(\A,\B)$.

\item Define the \emph{general dispersion entropy} by taking the supremum over all countable, measurable partitions:
\begin{equation}
    S_\dis := \sup_{\A,\B} S_\dis(\A,\B) = \sup_{\A,\B} \sum_{A,B} \E\!\left[ - U_{A,B} \log D_{A,B} \right] \ge 0,
\end{equation}
and general mixing entropy:
\begin{equation}
    S_\mix := \sup_{\A,\B} S_\mix(\A,\B) = \sup_{\A,\B} \sum_{A,B} \E\!\left[ U_{A,B} \log M_{A,B} \right] \ge 0.
\end{equation}
\end{enumerate}
\end{defn}

The extreme processes are those which exhibit only purely dispersive or purely mixing effects. We quantify those as when the dispersion or mixing coefficients are constant. 

\begin{defn}[Purely Dispersive and Purely Mixing Processes] \label{defn_dispersionmixingentropies}
Let $w$ be a finite-entropy process.
\begin{enumerate}
\item We say that $w$ is \emph{locally purely dispersive} (resp. \emph{locally purely mixing}) from $A$ to $B$ when $D_{A,B}(i) \in \{0,1\}$ (resp. $M_{A,B}(i) \in \{0,1\}$) for $\tilde \mu$-almost every $i$. %, in which case $D_{A,B} = \bar U_{A,B}$ $\tilde \mu$-almost surely. 
% $U_{A,B} = \bar U_{A,B} U$ $\tilde \mu$-almost surely. 
% Equivalently, when $M_{A,B} = 1$ $\tilde \mu$-almost surely.

\item We say that $w$ is \emph{partition purely dispersive} (resp. \emph{partition purely mixing}) from $\A$ to $\B$ when for all $A \in \A$, there exists $B \in \B$ such that $w$ is locally purely dispersive (resp. locally purely mixing) from $A$ to $B$.

\item We say that $w$ is \emph{generally purely dispersive} (resp. \emph{generally purely mixing}) when it is partition purely dispersive (resp. partition purely mixing) for all countable, measurable partitions $\A$ and $\B$ of $I$ and $I'$, respectively.
\end{enumerate}

\end{defn}

\begin{lem}[Non-Negativity of Dispersion and Mixing Entropies] \label{lem_lowerbounddispersionmixingentropies}
Let $w$ be finite-entropy.
\begin{enumerate}
\item The dispersion entropy functionals are non-negative:
\begin{equation}
    S_\dis(A,B) \ge 0, \qquad S_\dis(\A,\B) \ge 0, \quad \mathrm{and} \quad S_\dis \ge 0,
\end{equation}
and vanish when $w$ is locally, partition, or generally purely mixing, respectively.

\item The mixing entropy functionals are non-negative:
\begin{equation}
    S_\mix(A,B) \ge 0, \qquad S_\mix(\A,\B) \ge 0, \quad \mathrm{and} \quad S_\mix \ge 0,
\end{equation}
and vanish when $w$ is locally, partition, or generally purely dispersive, respectively.
\end{enumerate}
\end{lem}
\begin{proof}
Non-negativity of dispersive entropy is trivial since $U_{A,B} \le U$ hence $D_{A,B} \le 1$, and so $- D_{A,B} \log D_{A,B} \ge 0$ a.s. Saturation holds ($S_\dis(A,B) = 0$) iff $D_{A,B}(i) \in \{0,1\}$ $\tilde \mu$-a.s., i.e., the purely mixing case. 
Non-negativity of mixing entropy follows from Jensen's inequality since $x \log x$ is convex:
\begin{equation}
    S_\mix(A,B) = \bar U_{A,B}\,\tilde \E[M_{A,B}  \ge \bar U_{A,B}\,\tilde \E[M_{A,B}] \log \tilde \E[M_{A,B}] = 0,
\end{equation}
with saturation when $M_{A,B}$ is constant a.s. on the weighted measure $M_{A,B} \tilde \mu$. Saturation holds ($S_\mix(A,B) = 0$) iff $M_{A,B}(i) \in \{0,1\}$ $\tilde \mu$-a.s., i.e., the purely dispersive case.
\end{proof}

\subsection{Environmental Entropy Decomposition}

We now decompose the environmental entropy into dispersion and mixing components, with no additional factors. 

\begin{pro}[Environmental Entropy Decomposition, Local and Partition Versions]
% [Decomposition of Environmental Entropy into Dispersion and Mixing Entropies] 
\label{pro_environmentalentropydecomposition}
Let $w$ be a finite-entropy process. 
\begin{enumerate}
\item Let $A \in \I$ and $B \in \I'$. Local environmental entropy from $A$ to $B$ decomposes as:
\begin{equation} \label{eqn_localentropyECidentity}
    S_\EC(A,B) = S_\dis(A,B) + S_\mix(A,B).
\end{equation}
%Consequently, $S_\EC(A,B) = 0$ if and only if $S_\KS(A,B) = - S_\SX(A,B)$.

\item Consider countable, measurable partitions $\A$ of $I$ and $\B$ of $I'$. Partition environmental entropy from $\A$ to $\B$ decomposes as:
\begin{equation} \label{eqn_partitionentropyECidentity}
    S_\EC(\A,\B) = S_\dis(\A,\B) + S_\mix(\A,\B).
\end{equation}
%Consequently, $S_\EC(\A,\B) = 0$ if and only if $S_\KS(\A,\B) = - S_\SX(\A,\B)$.
\end{enumerate}
\end{pro}
\begin{proof}%[Proof of Proposition \ref{pro_environmentalentropydecomposition}.(1)-(2)]
When we average the pointwise identity $-\bar U_{A,B} \log \bar U_{A,B} = -\bar U_{A,B} \log D_{A,B} + \bar U_{A,B} \log M_{A,B}$, we have the local identity \eqref{eqn_localentropyECidentity}.
%\begin{equation} \label{eqn_enventropyidentity}
%    S_\dis(A,B) = S_\EC(A,B) - S_\mix(A,B)
%\end{equation}
%since $\tilde \E[M_{A,B}] = 1$. Identity \eqref{eqn_partitionentropyECidentity} follows by summing over partition sets. 
%\end{proof}

%S_\EC(A,B) &=& \E[-U_{A,B} \log \frac{U_{A,B}}{U}] = \E\!\left[U \left( -\frac{U_{A,B}}{U} \log \frac{U_{A,B}}{U} \right) \right] \nonumber \\
%&=& \tilde \E[ -\bar U_{A,B} M_{A,B} \log (\bar U_{A,B} M_{A,B})] \nonumber \\
%&=& - \bar U_{A,B} \log \bar U_{A,B} + \bar U_{A,B}\,\tilde \E[-M_{A,B} \log M_{A,B}] \nonumber \\
%&=& S_\KS(A,B) + \bar U_{A,B} \, S_\SX(A,B), \label{eqn_enventropyidentity}
%\end{eqnarray}
%since $\E[M_{A,B}] = 1$. 

%For the general identity \eqref{eqn_generalentropyECidentity}, we use the generalized version of Sinai's theorem (Theorem \ref{thm_sinaithm_dispersionmixing}, which shows that a generating joint partition for environmental entropy also serves as one for dispersion and mixing entropies. 

%\begin{proof}[Proof of Proposition \ref{pro_environmentalentropydecomposition}.(3)]

\end{proof}

For the general case, we need a variant of Sinai's theorem which allows us to use the same generating partitions for $S_\dis$ and $S_\mix$ as with $S_\EC$.% It is conceivably possible that there can be a generating partition for one of dispersion or mixing entropy but not for the other, and hence not the environmental entropy. The author has not investigated this possibility.

\begin{thm}[Generalized Sinai's Theorem for Dispersion and Mixing Entropies] \label{thm_sinaithm_dispersionmixing}
    A joint partition $(\A_*,\B_*)$ is generating for $S_\EC$ iff it is generating for both $S_\dis$ and $S_\mix$. That is, 
    \begin{equation}
        \mbox{$S_\EC = S_\EC(A_*,B_*)$ if and only if $S_\dis = S_\dis(A_*,B_*)$ and $S_\mix = S_\mix(A_*,B_*)$.}
    \end{equation}
    
    % That is, $S_\EC = S_\EC(\A_*,\B_*)$ if and only if $S_\dis = S_\dis(\A_*,\B_*)$ and $S_\mix = S_\mix(\A_*,\B_*)$. 

%    Let $(\A_*,\B_*)$ be a generating joint partition for environmental entropy: $S_\EC = S_\EC(\A_*,\B_*)$. 
\end{thm}

We prove Theorem \ref{thm_sinaithm_dispersionmixing} in Appendix \ref{app_sinaithm_dispersionmixing}. The proof of Theorem \ref{thm_sinaithm_dispersionmixing} depends on the partition identity \eqref{eqn_partitionentropyECidentity}. 

\begin{thm}[Environmental Entropy Decomposition, General Version] \label{thm_environmentalentropydecomposition}
General environmental entropy decomposes as:
\begin{equation} \label{eqn_generalentropyECidentity}
    S_\EC = S_\dis + S_\mix.
\end{equation}
\end{thm}
\begin{proof}
For the general result \eqref{eqn_generalentropyECidentity}, let $(\A_*, \B_*)$ denote a generating joint partition for environmental, dispersion, and mixing entropies simultaneously, using Theorem \ref{thm_sinaithm_dispersionmixing}. We evaluate the partition identity \eqref{eqn_partitionentropyECidentity} at the generating partitions, hence
\begin{equation}
    S_\EC = S_\EC(\A_*,\B_*) = S_\dis(\A_*,\B_*) + S_\mix(\A_*,\B_*) = S_\dis + S_\mix.
\end{equation}
\end{proof}

%\subsection{Environmental Upper Bound}

Both dispersion and environmental entropies are bounded above by environmental entropy, by non-negativity (Lemma \ref{lem_lowerbounddispersionmixingentropies}) and the environmental decomposition (Proposition \ref{pro_environmentalentropydecomposition}). 

\begin{cor}[Environmental Upper Bound for Dispersion and Mixing Entropies] \label{cor_upperboundfordispersionmixingentropy} %[Mixing Entropy and Environmental Entropy] 
Let $w$ be finite-entropy. The dispersion entropy functionals are bounded by environmental entropy:
\begin{equation}
    S_\dis(A,B) \le S_\EC(A,B), \qquad S_\dis(\A,\B) \le S_\EC(\A,\B), \quad \mathrm{and} \quad S_\dis \le S_\EC,
\end{equation}
with saturation when $w$ is locally, partition, or generally purely dispersive, respectively. 

The mixing entropy functionals are bounded by environmental entropy:
\begin{equation}
    S_\mix(A,B) \le S_\EC(A,B), \qquad S_\mix(\A,\B) \le S_\EC(\A,\B), \quad \mathrm{and} \quad S_\mix \le S_\EC,
\end{equation}
with saturation when $w$ is locally, partition, or generally purely mixing, respectively. 
\end{cor}

\subsection{Bernoulli Examples} \label{sect_bernoulliexamples}

We illustrate the extreme cases of dispersion and mixing via simple examples based on Bernoulli random variables.
The Bernoulli dispersion process as the simple one-to-two mapping splitting population from one point onto two points, which has positive dispersion entropy. The Bernoulli mixing process is the simple two-to-one mapping combining population from two points onto one point, which has positive mixing entropy.

%We show this is efficient but not consistent. Later (Example \ref{exa_bernoullidispersion2}) we will show the Bernoullie dispersion process has positive dispersion entropy.

% . We show this is consistent but not efficient. Later (Example \ref{exa_bernoullicoalescent2}) we will show it has positive mixing entropy. 

\begin{exa}[Bernoulli Dispersive Process] \label{exa_bernoullidispersion}
Define $I := \{0\}$ and $I' := \{0,1\}$. Let $q \in [0,1]$, and define the discrete measures $\mu$ and $\mu'$ by $\mu(0) := 1$, $\mu'(0) := q$, and $\mu'(1) := 1-q$. Define the Bernoulli dispersion process by $w_0(0) := q$ and $w_0(1) := 1-q$. Then $w : \mu \mapsto \mu'$. We have $\bar W = 1$. Note that $W_{0,0}(0) = q$ and $W_{0,1}(0) = 1-q$, and hence $W(0) = 1$.  Thus dispersion entropy is non-zero:
\begin{equation}
    S_\dis(0,0) = -q \log q \quad \mathrm{and} \quad S_\dis(0,1) = -(1-q) \log (1-q),
\end{equation}
representing dispersive effects. Since there is only one originating point $0$, the local fitnesses have the same values: $\bar W_{0,0} = q$ and $\bar W_{0,1} = 1-q$. Hence the environmental entropy equals the dispersion entropy, $S_\EC(0,0) = -q \log q = S_\dis(0,0)$ and $S_\EC(0,1) = -(1-q) \log (1-q) = S_\dis(0,1)$. Consequently, mixing entropy vanishes: $S_\mix(0,0) = 0 = S_\mix(0,1)$.

% environmental and KS entropies are equal: $S_\EC(0,0) = -q \log q = S_\KS(0,0)$ and $S_\EC(0,1) = -(1-q) \log (1-q) = S_\KS(0,1)$.
\end{exa}

\begin{exa}[Bernoulli Mixing Process] \label{exa_bernoullicoalescent}
Define $I := \{0,1\}$ and $I' := \{0\}$. Let $p \in [0,1]$, and define the measures $\mu$ and $\mu'$ by $\mu(0) := p$, $\mu(1) := 1-p$, and $\mu'(0) = 1$. Define the Bernoulli mixing process by $w_0(0) := 1$ and $w_1(0) := 1$. Then $w : \mu \mapsto \mu'$. We have $\bar W = 1$.
Note that $W_{0,0}(0) = 1$ and $W_{1,0}(1) = 1$, and so $W(0) = 1$ and $W(1) = 1$. Thus dispersion entropy vanishes: $S_\EC(0,0) = \E[-W_{0,0} / W \log W_{0,0}/W] = p*0 = 0$ and $S_\EC(1,0) = \E[-W_{1,0} / W \log W_{1,0}/W] = (1-p)*0 = 0$. Note that $\bar W_{0,0} = p$ and $\bar W_{1,0} = 1-p$. Hence environmental entropy equals mixing entropy: $S_\EC(0,0) = -p \log p = S_\mix(0,0)$ and $S_\EC(1,0) = -(1-p) \log (1-p) = S_\mix(1,0)$.

%We measure the strictly coalescent part of $w$ using a standalone entropy function, which we introduce in the next section as mixing entropy. %We measure the coalescent effects with mixing entropy. 

% We define this difference as mixing entropy.

% Thus KS entropy is non-trivial: $S_\KS(0,0) = -p \log p$ and $S_\KS(1,0) = -$.

%We have
\end{exa}

The Bernoulli dispersion and coalescent processes are inverses of each other: $w_\mix \circ w_\dis$ is the identity on measures on $\{0\}$, and $w_\dis \circ w_\mix$ is the identity on measures on $\{0,1\}$.

%In Section \ref{sect_ecriiciency}, we introduce this quantity $-\bar U_{A,B} \log U_{A,B}$ as Kolmogorov-Sinai entropy. Thus the Bernoulli mixing process is the simplest process with zero environmental entropy and non-negative Kolmogorov-Sinai entropy, measuring the presence of purely coalescent effects.

%This is the simplest process with non-zero environmental entropy and mixing entropy. %We show that we can summarize the purely coalescent effects as a standalone ``sexual-selective'' entropy. 

 %i.e., $w^{\operatorname{coal}} \circ w^{\operatorname{dis}} = 1_{\{0]}}$ and $w^{\operatorname{dis}} \circ w^{\operatorname{coal}} = 1_{\{0,1\}}$. 

%\section{Fundamental Theorem of Ecology? } \label{sect_fundthm_sexualselection}
% NO: The idea of a fundamental theorem of ecology was some bounds on the total entropy which integrated ``interactions'' between the selective and environmental parts. I can't figure out such a bound, instead the best bounds I can find are those which separately optimize the selective and environmental parts, i.e., the case of both selective equilibrium and environmental equilibrium
% Consequently, I don't have material for a section on this. 

\section{Vanishing Entropies: Efficiency, Consistency, Reversibility, and Irreversibility} \label{sect_ecri}

In this section, we relate the vanishing of dispersive, mixing, and environmental entropies to concepts of left, right and full invertibility, respectively. %consistency, efficiency, and reversibility. Formally, these correspond to the existence of left-, right- or total inverse processes to the environmental part of an evolutionary process. 

\begin{thm}[Efficiency Theorem] \label{thm_efficiency}
Let $w$ be finite-entropy. The following are equivalent: 
\begin{enumerate}
    \item $w$ is purely mixing (i.e., $D_{A,B} \in \{0,1\}$ $\tilde \mu$-a.s.~for all $A \in \I$ and $B \in \I'$),
    \item $S_\dis = 0$,
    \item There exists a purely-environmental section to $w_\EC$, i.e., a right-inverse process $w' : \mu' \to \tilde \mu$ (i.e., $w_\EC \circ w' = 1_{\mu'}$). If $S_\mix > 0$, then $w'$ is not unique. 
\end{enumerate}
\end{thm}

% We prove the Efficiency Theorem in Section \ref{sect_ecriiciency}.

\begin{thm}[Consistency Theorem] \label{thm_consistency}
Let $w$ be finite-entropy. The following are equiv.: 
\begin{enumerate}
    \item $w$ is purely dispersive (i.e., $M_{A,B} \in \{0,1\}$ $\tilde \mu$-a.s.~for all $A \in \I$ and $B \in \I'$),
    \item $S_\mix = 0$,
    \item There exists a unique purely-environmental retraction to $w_\EC$, i.e., a left-inverse process $w_\EC' : \mu' \to \tilde \mu$ (i.e., $w_\EC^\dagger \circ w_\EC = 1_{\tilde \mu}$). 
\end{enumerate}
\end{thm}

\begin{thm}[Reversibility Theorem] \label{thm_reversibility}
Let $w$ be finite-entropy. The following are equiv.: 
\begin{enumerate}
    \item $w$ is purely dispersive and purely mixing (i.e., $\{D_{A,B}, M_{A,B}\} \subseteq \{0,1\}$ $\tilde \mu$-a.s.~for all $A \in \I$ and $B \in \I'$),
    \item $S_\EC = 0$,
    \item There exists a unique purely-environmental inverse process $w_\EC^\dagger : \mu' \to \tilde \mu$ (i.e., $w_\EC^\dagger \circ w_\EC = 1_{\mu'}$ and $w_\EC \circ w_\EC^\dagger = 1_{\tilde \mu}$). 
\end{enumerate}
\end{thm}

\begin{thm}[Irreversibility Theorem] \label{thm_irreversibility}
Let $w$ be finite-entropy. The following are equiv.: 
\begin{enumerate}
    \item $w$ exhibits dispersive or mixing effects (or both),
    \item $S_\EC > 0$,
    \item $S_\dis >0$ or $S_\mix > 0$ (or both)
    \item There does not exist a purely-environmental inverse process $w_\EC^\dagger : \mu' \to \tilde \mu$ (i.e., every purely environmental process $w' : \mu \mapsto \tilde \mu$ must satisfy $w' \circ w_\EC \ne 1_{\tilde \mu}$ or $w_\EC \circ w' \ne 1_{\tilde \mu'}$). 
\end{enumerate}
\end{thm}

We prove Theorems \ref{thm_efficiency} through \ref{thm_irreversibility} in Appendix \ref{app_ECRIproofs}. To do so, we define a unique ``child-set mapping'' $\chi$, show partial invertibility of $\chi$ for purely dispersive or mixing processes, then define the partial inverse process such that $\chi$ is its parent-set mapping.

\begin{exa}
The Bernoulli dispersion process (Example \ref{exa_bernoullidispersion}) is purely dispersive but not purely mixing, with left-inverse given by mapping both child types back to the single parent type. The Bernoulli coalescent process (Example \ref{exa_bernoullicoalescent}) is purely mixing but not purely dispersive, with right-inverses given by mapping the child to any mix of parent types.  

A process generated by a function $f : I \to I'$ is purely mixing, and is reversible if and only if $f$ is an invertible function. %To exhibit dispersive effects, a process must send at least some point to distinct separate points.
\end{exa}

\subsection{Dollo's Law of Irreversibility}

The above results provide a mathematical framework for reasoning around reversibility and irreversibility. In the biology literature, a notable example is Dollo's law of irreversibiliity, which states that ``an organism never returns exactly to a former state, even if it finds itself placed in conditions of existence identical to those in which it has previously lived \dots it always keeps some trace of the intermediate stages through which it has passed'' \cite{dollo1893laws,gould1970dollo}. 

Dollo's result is an empirical observation without mathematical proof, and in fact researchers have documented exceptions to Dollo's law \cite{collin2008reversing}. Nonetheless Dollo's observation illustrates that there are strong restrictions on biologically reversible processes. We state and prove a weak form of Dollo's law as Corollary \ref{cor_dollo}, by combining the above theorems with a simple fact about selective reversibility. We do not explore a strong formalization of Dollo's law as he stated above, an investigation which is more empirical in nature. 

Observe that a purely selective process can always be reversed to restore the childbearing population, but never the childless population. %We say that $w$ is selectively reversible if $w_\NS$ is invertible, otherwise 
The selective inverse $w_\NS^{-1} : \tilde \mu \mapsto \mu_*$ is defined by scaling by the reverse fitness $\frac{1}{W}$, and is an inverse to the restricted selective process $w_\NS|_{W>0} : \mu_* \mapsto \tilde \mu$. Formally, define $w_\NS^{-1} : \tilde \mu \mapsto \mu_*$ by $w_{\NS,\tilde i}^{-1}(\d i) := \frac{1}{W(\tilde i)} \delta_{\tilde i, i}$. Then $w_\NS^{-1}$ is an inverse to the restricted process $w_\NS|_{W>0} : \mu_* \mapsto \tilde \mu$. Thus $w$ is selectively reversible if and only if $p_* = 1$, i.e., $\mu_* = \mu$. %Combining this fact with the Irreversibility Theorem, we have the following. 

We say that $w : \mu \to \mu'$ is \emph{fully reversible} if there exists $w^{-1} : \mu' \mapsto \mu$ such that $w^{-1} \circ w = 1_{\mu}$ and $w \circ w^{-1} = 1_{\mu'}$. We say that $w$ is \emph{childbearing reversible} if we may invert $w$ up to the childbearing population, i.e., if there exists $w^{-1} : \mu' \mapsto \mu_*$ such that $w^{-1} \circ w|_{\mu_*} = 1_{\mu_*}$ and $w \circ w^{-1} = 1_{\mu'}$. 
%We say that $w : \mu \to \mu'$ is \emph{childbearing reversible} if there exists $w^{-1} : \mu' \mapsto \mu_*$ such that $w^{-1} \circ w|_{W>0} = 1_{\mu_*}$ and $w \circ w^{-1} = 1_{\mu'}$. 
We say that $w$ is \emph{environmentally reversible} if $0 = S_\dis = S_\mix = S_\EC$. 

\begin{cor}[Weak Version of Dollo's Law of Irreversibility] \label{cor_dollo}
Let $w$ be finite-entropy. Then:
\begin{enumerate}
    \item $w$ is childbearing reversible if and only if $w_\EC$ is environmentally reversible. 
    \item $w$ is fully reversible if and only if $w_\EC$ is environmentally reversible and $p_* = 1$.
\end{enumerate}
In both cases, the inverse process is defined by $w^{-1} := w_\NS^{-1} \circ w_\EC^\dagger$. The inverse admits the Price representation $w^{-1} = \hat w_\EC \circ \hat w_\NS$, with selective part $\hat w_\NS : \mu' \mapsto \hat W \mu'$ given by scaling $\hat W(i') := \int \frac{1}{W(\tilde i)} w_{\EC,i'}^{-1}(\d \tilde i)$, and environmental part $\hat w_\EC : \hat W \mu' \mapsto \mu_*$ given by $\hat w_{\EC,\hat i'}(A) := \frac{1}{\hat W(\hat i')} w^{-1}_{\hat i'}(A) = \int_{A} \frac{1}{W(\tilde i) \hat W(\hat i')} w_{\EC,\hat i'}^{-1}(\d \tilde i)$. 
\end{cor}
\begin{proof}
If $w_\EC$ is environmentally reversible, then it is straightforward that $w^{-1} := (w_\EC \circ w_\NS)^{-1} = w_\NS^{-1} \circ w_\EC^\dagger$ is an inverse. 

Conversely, if $w$ is childbearing reversible, then $w$ is childbearing reversible, then define $w_\EC^\dagger := w_\NS \circ w^{-1}$. We first verify $w_\EC \circ w_\EC^\dagger = w_\EC \circ w_\NS \circ w^{-1} = w \circ w^{-1} = 1_{\mu'}$ as desired. Next we verify that $w_\EC^\dagger \circ w_\EC \circ w_\NS = w_\EC^\dagger \circ w = w_\NS \circ w^{-1} \circ w = w_\NS$. Since $w_\NS$ is childbearing reversible, we apply the inverse $w_\NS^{-1}$ on the right and thus $w_\EC^\dagger \circ w_\EC = 1_{\tilde \mu}$.
\end{proof}

\section{Environmental Equilibrium and Bounds on Dispersion and Mixing Entropies} \label{sect_environmentalequilibrium}

% Can we simultaneously prove two inequalities? One for dispersive case, and one for mixing case?

In this section, we present strong bounds for dispersive and mixing entropies, improving upon Corollary \ref{cor_upperboundfordispersionmixingentropy}. The bounds of that corollary are ``weak'', as they are only saturated in the extreme cases of purely dispersive and purely mixing processes. The saturation of the strong bounds here corresponds to the case of ``environmental equilibrium'', which represents an efficient flow between the parent and child populations. %, by using the fact that mixing entropy can be written as a (weighted) selective entropy, and therefore the strong bounds of Theorem \ref{thm_stronggibbsNS} apply. 

%Since mixing entropy can be written as a negative selective entropy, we can use the strong bounds on selective entropies (Theorem \ref{thm_stronggibbsNS}) to prove strong bounds for mixing entropy, which we state as Theorem \ref{thm_strongbound_dismix}. Using the fact that dispersion entropy is the difference of environmental and mixing entropies, these immediately imply strong bounds for dispersion entropy, with directionality flipped, which we state as Corollary 

% Consequently, these result in strong lower and upper bounds for environmental entropy. We summarize this as Theoorem \ref{thm_strongboundsfordispersionandmixingentropies}. 

%Since $S_\EC(A,B) - S_\KS(A,B) = S_\SX(A,B)$ is a multiple of a selective entropy, Theorem \ref{thm_stronggibbsNS} provides strong lower and upper bounds for mixing entropy. We apply these bounds to environmental entropy in the following result.

\subsection{Environmental Equilibrium}

Observe that $U_{A,B} > 0$ if and only if $M_{A,B} > 0$. Define the transmission proportion
\begin{equation}
    \tilde p_{A,B} := \frac{\tilde \mu(U_{A,B} > 0)}{N'} = \frac{\tilde \mu(M_{A,B} > 0)}{N'} = \frac{\tilde \mu(1_A > 0 \mathrm{~and~} w_i(B) > 0)}{N'} = \frac{N'_{A,B}}{N'},
\end{equation}
which describes the proportion of the full intermediate population which both starts in $A$ and ends in $B$. We have $\tilde p_{I,I'} = 1$ by definition. We also have
\begin{equation} \label{ineq_tildep_barU}
    \tilde p_{A,B} = \tilde \E[1_{A,B}] \ge \tilde \E[1_{A,B} D_{A,B}] = \bar U_{A,B}, \quad \mathrm{i.e.,} \quad \frac{\bar U_{A,B}}{\tilde p_{A,B}} \le 1.
\end{equation}

%\begin{defn}[Transmission Statistics]
% Define the transmission population size $N'_{A,B} := \tilde \mu(U_{A,B}>0)$. 
%\end{defn}

\begin{defn}[Environmental Equilibrium] \label{def_environmentalequilibrium}
Let $w$ be a finite-variance process. We say that $w$ is in \emph{environmental equilibrium} if for all $A \in \I$ and $B \in \I'$, both $D_{A,B}$ and $M_{A,B}$ are $\tilde \mu_{A,B}$-almost surely constant, that is, 
\begin{equation} \label{condition_environmentalequilibrium}
    \mbox{$D_{A,B} \in \left\{ 0, \frac{\bar U_{A,B}}{\tilde p_{A,B}} \right\}$ and $M_{A,B}(i) \in \left\{0, \frac{1}{\tilde p_{A,B}} \right\}$ for $\tilde \mu$-almost every $i$.}
\end{equation}
Note: if either the dispersion or mixing condition of \eqref{condition_environmentalequilibrium} is satisfied then both are. 

%-- OLD BELOW -- 
%\begin{enumerate}
%\item We say that $w$ is \emph{in local environmental equilibrium} from $A$ to $B$ if $D_{A,B}$ or $M_{A,B}$ is $\tilde \mu_{A,B}$-almost surely constant (hence both are $a.s.$-constant), that is, 
%\begin{equation}
%    \mbox{$D_{A,B} \in \left\{ 0, \frac{\bar U_{A,B}}{\tilde p_{A,B}} \right\}$ and $M_{A,B}(i) \in \left\{0, \frac{1}{\tilde p_{A,B}} \right\}$ for $\tilde \mu$-almost every $i$.}
%\end{equation}
%In this case, $\frac{1}{N'}\tilde \mu\!\left(M_{A,B} = \frac{1}{\tilde p_{A,B}}\right) = \tilde p_{A,B}$ and $\frac{1}{N'}\tilde \mu\!\left(M_{A,B} = 0\right) = 1 - \tilde p_{A,B}$. Proof. We have $1 = \tilde \E[M_{A,B}] = \tilde \E[1_{M_{A,B} = 1/\tilde p_{A,B}} M_{A,B}] + \tilde \E[1_{M_{A,B} = 0} M_{A,B}] = \frac{1}{N} \tilde \mu(M_{A,B} = \frac{1}{\tilde p_{A,B}}) \frac{1}{\tilde p_{A,B}}$. }
%\item We say that $w$ is \emph{in partition environmental equilibrium} from $\A$ to $\B$ if $w$ is in local environmental equilibrium for every $A \in \A$ and $B \in \B$.
%\item We say that $w$ is \emph{in general environmental equilibrium} is $w$ is in partition environmental equilibrium for each countable, measurable partition $\A$ and $\B$ of $I$ and $I'$. 
%\end{enumerate}
\end{defn}

The class of processes in environmental equilibrium includes the purely dispersive and purely mixing cases (Lemma \ref{lem_dispersionmixing_localequilibrium}). In general, there exist environmental-equilibrium processes which exhibit both dispersive and mixing effects (Example \ref{exa_bernoulliequilibriumprocess}). 

\begin{lem} \label{lem_dispersionmixing_localequilibrium}
A purely dispersive or purely mixing process is in environmental equilibrium.
%Let $A \in \I$ and $B \in \I'$. 
%Let $w$ be finite entropy. 
%Purely dispersive and purely mixing each imply environmental equilibrium.
%processes are in environmental equilibrium.
%If $w$ is purely dispersive or purely mixing, then $w$ is in environmental equilibrium. %Equivalently, $w$ is not in equilibrium only if $w$ exhibits both dispersive and mixing effects. 
\end{lem}
\begin{proof}
Let $A \in \I$ and $B \in \I'$. If $S_\dis(A,B) = 0$, then $D_{A,B} = 0$ or $1$ $\tilde \mu$-almost surely. In that case, $\bar U_{A,B} = \tilde \E[D_{A,B}] = \tilde p_{A,B}$. Thus $M_{A,B} = 0$ or $\frac{1}{\bar U_{A,B}} = \frac{1}{\tilde p_{A,B}}$ $\tilde \mu$-almost surely.

If $S_\mix(A,B) = 0$, then $M_{A,B} = 0$ or $1$ $\tilde \mu$-almost surely. In that case, $1 = \tilde \E[M_{A,B}] = \tilde p_{A,B}$, and so $M_{A,B} = 0$ or $1 = \frac{1}{\tilde p_{A,B}}$ $\tilde \mu$-almost surely. 
\end{proof}

%The next lemma says that if there is no transmission from $A$ to $B$ ($\tilde p_{A,B} = 0)$, then the dispersion and mixing entropies vanish. 

\begin{lem} \label{lem_vanishing_tildep_implies_vanishing_EC}
If $\tilde p_{A,B} = 0$, then $S_\dis(A,B) = S_\mix(A,B) = S_\EC(A,B) = 0$. 
\end{lem}
\begin{proof}
If $\tilde p_{A,B} = 0$, then $U_{A,B} = 0$ on a set of full $\tilde \mu$-measure, as well as $\mu$-measure.
\end{proof}

\begin{exa}[Reversible Transmitting Processes]
The converse of Lemma \ref{lem_vanishing_tildep_implies_vanishing_EC} does not hold. E.g., consider a process generated by an invertible function $f : I \to I'$ with $w = f_*$ given by the push-forward of measures, i.e., $\mu' := f_* \mu  := \mu \circ f^{-1}$. Since $f$ is invertible, $w$ is reversible by the Reversibility Theorem (Theorem \ref{thm_reversibility}) and so $S_\dis(A,B) = S_\mix(A,B) = S_\EC(A,B) = 0$. However for any set $B$ of positive $\mu'$-measure, we have $\tilde p_{f^{-1} B, B} = 1$. 
\end{exa}

\subsection{Strong Bounds on Dispersion and Mixing Entropies} \label{def_mixing_controls}

We present strong bounds on the dispersion and mixing entropies, by restricting expectations to the sets $A \cap w^{-1} B$ and using Jensen's inequality. These are saturated in environmental equilibrium. 

% Recall that a generating joint partition $(\A_*, \B_*)$ of $S_\EC$ is also of $S_\dis$ and $S_\mix$. 

\begin{thm}[Strong Bounds on Dispersion and Mixing Entropies] \label{thm_strongbound_dismix}
Let $w$ be a finite-variance process, and let $(\A_*, \B_*)$ be a generating joint partition. Then:
\begin{equation} \label{ineq_generalbound_dis}
    0 \le 
    \sum_{A,B} \bar U_{A,B} \log \frac{1}{\tilde \E[D_{A,B}^2]} \le S_\dis \le \sum_{A,B} \bar U_{A,B} \log \frac{\tilde p_{A,B}}{\bar U_{A,B}} 
    \le S_\EC,
\end{equation}
and
\begin{equation} \label{ineq_generalbound_mix}
    0 \le 
    \sum_{A,B} \bar U_{A,B} \log \frac{1}{\tilde p_{A,B}} \le S_\mix \le \sum_{A,B} \bar U_{A,B} \log \frac{\tilde \E[M_{A,B}^2]}{\bar U_{A,B}} 
    \le S_\EC,
\end{equation}
where the sums are over sets $(A,B) \in \A_* \times \B_*$ from the generating joint partition. The inner inequalities are saturated when $w$ is in environmental equilibrium.

The outer upper (resp. lower) bound of \eqref{ineq_generalbound_dis} and outer lower (resp. upper) bound of \eqref{ineq_generalbound_mix} are saturated if and only if $w$ is purely dispersive (resp. mixing). 

%The outer upper bound of \eqref{ineq_generalbound_dis} and outer lower bound of \eqref{ineq_generalbound_mix} are saturated if and only if $w$ is purely dispersive. 

%The outer lower bound of \eqref{ineq_generalbound_dis} and outer upper bound of \eqref{ineq_generalbound_mix} are saturated if and only if $w$ is purely mixing. 
\end{thm}
\begin{proof}
Fix $A \in \I$ and $B \in \I'$. We first prove local versions of the strong bounds:
\begin{equation} \label{ineq_localbound_dis}
    0 \le 
    \bar U_{A,B} \log \frac{\bar U_{A,B}}{\tilde \E[D_{A,B}^2]} \le S_\dis(A,B) \le \bar U_{A,B} \log \frac{\tilde p_{A,B}}{\bar U_{A,B}}
    \le S_\EC(A,B),
\end{equation}
and
\begin{equation} \label{ineq_localbound_mix}
    0 \le 
    \bar U_{A,B} \log \frac{1}{\tilde p_{A,B}} \le S_\mix(A,B) \le \bar U_{A,B} \log \tilde \E[M_{A,B}^2],
    \le S_\EC(A,B).
\end{equation}
with saturation of the inner inequalities when $w$ is in local environmental equilibrium (i.e., $M_{A,B} \in \{0,\frac{1}{\tilde p_{A,B}} \}$ $\tilde \mu_{A,B}$-a.s.).

We prove the inner mixing inequalities first. Observe that we can write mixing entropy as a negative weighted selective entropy $S_\mix(A,B) = -\bar U_{A,B} \tilde \E[-M_{A,B} \log M_{A,B}]$, so we can apply the strong bounds of Theorem \ref{thm_stronggibbsNS} apply to $\E[-M_{A,B} \log M_{A,B}]$:
\begin{equation}
    -\log \tilde \E[M_{A,B}^2] \le \tilde \E[-M_{A,B} \log M_{A,B}] \le - \log \frac{1}{\tilde p_{A,B}}.
\end{equation}
Multiplying by $-\bar U_{A,B}$ and flipping the order of inequalities yields \eqref{ineq_localbound_mix}. The lower (resp. upper) bound of that result becomes the upper (resp. lower) bound of \eqref{ineq_localbound_mix}. 

Saturation holds when $M_{A,B}$ is $\tilde \mu$-a.s. constant on the set where it is positive, i.e., when $M_{A,B} \in \{0, \frac{1}{\tilde p_{A,B}}\}$. By Lemma \ref{lem_dispersionmixing_localequilibrium}, condition $S_\mix(A,B) = S_\EC(A,B)$ also corresponds to the environmental equilibrium case. 

The non-zero lower bound is trivial: $\log \frac{1}{\tilde p_{A,B}} \ge 0$. For the mixing upper bound, observe:
\begin{equation}
    \tilde \E[M_{A,B}^2] - 1 = \frac{1}{\bar U_{A,B}^2} \E\!\left[\frac{U_{A,B}^2}{U}\right] - 1 \le \frac{1}{\bar U_{A,B}^2} \E[U_{A,B}] - 1 = \frac{1}{\bar U_{A,B}} - 1,
\end{equation}
since $\tilde \E[X] = \E[UX]$ and $U_{A,B} \le U$. Consequently, $\bar U_{A,B} \log \frac{\tilde \E[M_{A,B}^2]}{\bar U_{A,B}} \le \bar U_{A,B} \log \frac{1}{\bar U_{A,B}} = S_\EC(A,B)$, which proves the outer mixing inequalities. 

%Observe that $\tilde \var(D_{A,B}) = \tilde \E[D_{A,B}^2] - \bar U_{A,B}^2 \le 1 - \bar U_{A,B}^2$ since $D_{A,B} \le 1$ almost surely. 

The local dispersion inequalities and saturation conditions follow by the relation $S_\dis(A,B) = S_\EC(A,B) - S_\mix(A,B)$. In particular, 
\begin{eqnarray}
    S_\dis(A,B) &=& S_\EC(A,B) - S_\mix(A,B) %\nonumber \\
    \ge -\bar U_{A,B} \log \bar U_{A,B} - \bar U_{A,B} \log \tilde \E[M_{A,B}^2] \nonumber \\
    &=& -\bar U_{A,B} \log \bar U_{A,B} - \bar U_{A,B} \log \frac{\tilde \E[D_{A,B}^2]}{\bar U_{A,B}^2} %\nonumber \\
    = \bar U_{A,B} \log \frac{\bar U_{A,B}}{\tilde \E[D_{A,B}^2]},
\end{eqnarray}
since $\tilde \E[M_{A,B}^2] = \frac{\tilde \E[D_{A,B}^2]}{\bar U_{A,B}^2}$, proving the local dispersion bounds \eqref{ineq_localbound_dis}. The general bounds \eqref{ineq_generalbound_dis} and \eqref{ineq_generalbound_mix} follow by summing over partition sets of the generating joint partition. 

We have that $w$ is purely dispersive ($W_{A,B} = W$ $\tilde \mu$-a.s.) if and only if $\tilde p_{A,B} = \frac{1}{N'} \tilde \mu(W_{A,B} > 0) = \frac{1}{N'} \tilde \mu(W > 0) = 1$, Similarly, $w$ is purely mixing ($M_{A,B} = 1$ $\tilde \mu$-a.s.) if and only if $\tilde \E[M^2_{A,B}] = 1$. This proves the saturation conditions of the outer inequalities. 
\end{proof}

\subsection{Examples in Environmental Equilibrium and Non-Equilibrium}

We demonstrate examples of equilibrium and non-equilibrium processes. We show that the class of discrete processes is always in environmental equilibrium, covering the  extent of Price's work. For example, a process from two points to two points (a ``Bernoulli equilibrium process'') is in equilibrium, while exhibiting both dispersion and mixing effects. We present an example of a process on the real line which is  non-equilibrium.  

%We present a Bernoulli-like process which exhibits both dispersion and mixing effects, is not in environmental equilibrium. 

\begin{pro}[Discrete Processes Are In Equilibrium] \label{pro_discreteequilibrium}
Let $I$ and $I'$ be countable sets (possibly infinite). %We can write discrete measures as vectors and a discrete process as a matrix. 
Let $\mu = (\mu_i)$ and $\mu' = (\mu'_{i'})$ be discrete measures on $I$ and $I'$ respectively. Let $w = (w_i(i'))$ an evolutionary process such that $w : \mu \mapsto \mu'$, i.e., satisfying the linear equation $\mu'_{i'} = \sum_i w_i(i') \mu_i$ for each $i' \in I'$. Write the population sizes $N := |\mu| = \mu(I) = \sum_i \mu_i$ and $N' := |\mu'| = \mu'(I') = \sum_{i'} \mu'_{i'}$, and the selective coefficient $\bar W = \frac{N'}{N}$. 

The Price representation theorem (Theorem \ref{thm_pricerepresentation}) ensures that there exists a diagonal matrix $w_\NS = (W(i))$ and a stochastic matrix $w_\EC = \left(\frac{w_i(i')}{W(i)}\right)$ such that $w = w_\EC w_\NS$, in the sense of matrix multiplication. Let $\tilde \mu = w_\NS \mu = W \mu$ be the fitness-scaled parent population.

For each $(i,i') \in I \times I'$, the average local relative fitness and transmission proportion are:
\begin{eqnarray}
    \bar U_{i,i'} &:=& \E\!\left[\frac{W_{i,i'}}{\bar W}\right] = \frac{1}{N'} w_i(i') \mu_i \\
    \tilde p_{i,i'} &:=& \frac{1}{N'} \tilde \mu(W_{i,i'}>0) = \frac{1}{N'} 1_{W_{i,i'}>0} W(i) \mu_i.
\end{eqnarray}
The dispersion and mixing coefficients equal:
\begin{eqnarray}
    D_{i,i'}(i) &:=& \frac{W_{i,i'}(i)}{W(i)} = \frac{w_i(i')}{W(i)} = \frac{\bar U_{i,i'}}{\tilde p_{i,i'}} \\
    M_{i,i'}(i) &:=& \frac{W_{i,i'}(i)}{\bar U_{i,i'} W(i)} = \frac{w_i(i')}{\bar U_{i,i'} W(i)} = \frac{1}{\tilde p_{i,i'}}
\end{eqnarray}
for $i'$ such that $w_i(i') > 0$ (otherwise $D_{i,i'}(i) := 0 =: M_{i,i'}(i)$ if $w_i(i') = 0$), and $D_{i,i'}(j) = 0 = M_{i,i'}(j)$ for $j \ne i$. 
Thus $w$ is in environmental equilibrium. The entropy functionals satisfy:
\begin{equation}
    \sum_{i,i'} \bar U_{i,i'} \log \frac{1}{\tilde \E[D^2_{i,i'}]} = S_\dis = \sum_{i,i'} \bar U_{i,i'} \log \frac{\tilde p_{i,i'}}{\bar U_{i,i'}}
\end{equation}
and
\begin{equation}
    \sum_{i,i'} \bar U_{i,i'} \log \frac{1}{\tilde p_{i,i'}} = S_\mix = \sum_{i,i'} \bar U_{i,i'} \log \frac{\tilde \E[M^2_{i,i'}]}{\bar U_{i,i'}}.
\end{equation}
with $S_\dis = S_\dis + S_\mix = \sum_{i,i'} (-\bar U_{i,i'} \log \bar U_{i,i'})$. 
\end{pro}
\begin{proof}
The discrete joint partition of $I$ and $I'$ is a generating joint partition, since $\I$ is the minimal $\sigma$-algebra containing all sets of the form $\{i\} \cap w^{-1}(\{i'\})$. Formally, the discrete joint partition is the joint collection of singletons $(\{i\}, \{i'\})$. Thus it suffices to evaluate functionals at singleton pairs. %the singleton pair $\{i\}, \{i'\}$.
Since $D_{i,i'}(i) = \frac{\bar U_{i,i'}}{\tilde p_{i,i'}}$ or $=0$, $w$ is in environmental equilibrium. Theorem \ref{thm_strongbound_dismix} ensures that the relations for $S_\dis$ and $S_\mix$ are satisfied. 
\end{proof}

\begin{cor}[Bernoulli Equilibrium Process] \label{exa_bernoulliequilibriumprocess}
Define $I = \{0,1\}$ and $I' = \{0,1\}$. Let $\mu$ and $\mu'$ be measures on $I$ and $I'$, respectively. Any process $w : \mu \mapsto \mu'$ is in environmental equilibrium. If $w_i(i') > 0$ for all $(i,i') \in I \times I'$, then $w$ exhibits dispersive and mixing effects. 
\end{cor}

We next present an example of a process not in environmental equilibrium. By Proposition \ref{pro_discreteequilibrium}, any non-equilibrium process must be non-discrete. We describe diffusion processes on continuous sets like the circle or real line, where there are ``very small'' sets.

% The first example describes diffusion processes on continuous sets like the circle or real line, where there are ``very small'' sets. %The second example describes informatic processes on sequential sets like binary sequences or infinite directed acyclic graphs, where there are ``very large'' sets out at infinity. 

\begin{exa}[Diffusion Processes Are Non-Equilibrium]
Consider the two-point set $I := \{0,1\}$ with uniform mass $\mu(0) = \frac{1}{2} = \mu(1)$, and the unit interval $I' := [0,1]$ equipped with Lebesgue measure $\lambda$. 
For each $i \in I$, define the process $w_i := \lambda$, i.e., each $i$ diffuses its full mass uniformly onto the interval. The child measure is uniform measure but as a result of the mixture: $\mu'(B) := (w_* \mu)(B) = \sum_i w_i(B) \frac{1}{2} = \lambda(B)$. Note that $w$ is purely environmental since $W(i) := w_i(I') = \lambda(I') = 1 = \bar W$. 

Write the local fitness $W_{i,B}(i) := w_i(B) = \lambda(B) = U_{i,B}$ and $W_{i,B}(1-i) := 0$. The dispersion coefficient is given by $D_{i,B}(i) := \frac{W_{i,B}(i)}{W(i)} = \lambda(B)$. The average local relative fitness equals $\bar U_{i,B} = \E[W_{i,B}] = \frac{1}{2} \lambda(B)$ and the transmission proportion equals $\tilde p_{i,B} := \mu( W_{i,B} > 0) = 1$ when $\lambda(B) > 0$. Consequently, $D_{i,B} = \lambda(B) < \frac{1}{2} \lambda(B) = \frac{\bar U_{i,B}}{\tilde p_{i,B}}$. Since this holds for any measurable $B$ of positive Lebesgue measure, and since any partition must include sets of positive Lebesgue measure, the process is not in equilibrium. 

% We construct a  generating joint partition. Unlike the discrete case, the discrete joint partition is not measurable, hence cannot define a generating joint partition. Since the unit interval is separable, there exists a countable dense set $\{ q_z \} \substeq [0,1]$ for $z = 1,2,\cdots$. For each $z$, let $B_z = (q_z-2^{-1-z}, q_z+2^{-1-z})$ be the interval centered at $q_z$ of length $2^{-z}$.  

%Let $I, I' \subseteq \R$ be open subsets of the real line, and let $\lambda$ be the corresponding 
% measurably isomorphic to open subsets of the real line, and let $\lambda, \lambda'$, be the corresponding pullbacks of Lebesgue measure (e.g., $(0,1)$ equipped with volume measure). 
\end{exa}

\section{Selective\,Change of Environmental\,Entropy and the Third\,Law of Natural\,Selection} \label{sect_thirdlaw}
%\section{Selective\,Change of Environmental\,Entropy and the Third\,Law of Natural\,Selection} \label{sect_thirdlaw}

We analyze the selective change of environmental entropies. The Weak Law shows that in environmental equilibrium, the selective changes vanish, i.e., selection in equilibrium processes does not have environmental externalities. The Strong Law provides quantitative bounds, and the selective changes may fluctuate positively or negatively depending on interactions between selective and environmental niches. 

%We prove a Third Law of Natural Selection, which relates the selective change of environmental entropies in the environmental equilibrium case. 

Let $w$ be a finite-entropy process, and let $(\A_*, \B_*)$ be a generating joint partition, as in Section \ref{sect_environmentalentropy}. We define the selective change of the environmental entropy functionals:
\begin{eqnarray}
    \partial_\NS S_\EC &=& \sum \cov\!\left(-U_{A,B} \log \bar U_{A,B}, U)\right) \\
    \partial_\NS S_\dis &=& \sum \cov\!\left(-U_{A,B} \log \bar D_{A,B}, U)\right) \\
    \partial_\NS S_\mix &=& \sum \cov\!\left(U_{A,B} \log \bar M_{A,B}, U)\right),
\end{eqnarray}
where the sums are over partition sets in the generating joint partition $(A,B) \in \A_* \times \B_*$. Linearity and Theorem \ref{thm_environmentalentropydecomposition} ensure the following:
\begin{equation}
    \partial_\NS S_\EC = \partial_\NS S_\dis + \partial_\NS S_\mix. 
\end{equation}

The Weak Third Law ensures that these quantities vanish in environmental equilibrium, i.e., when $D_{A,B}$ and $M_{A,B}$ are constant $\tilde \mu_{A,B}$-almost surely. The converse need not be true. 

\begin{thm}[Weak Third Law of Natural Selection] \label{thm_weakthirdlaw}
If $w$ is in environmental equilibrium, then
\begin{equation} \label{eqn_weakthirdlaw}
    \partial_\NS S_\EC = 0 = \partial_\NS S_\dis = \partial_\NS S_\mix.
\end{equation}
and
\begin{equation} \label{eqn_weakthirdlaw_env}
    S'_\EC - S_\EC = \partial_\EC\!\left( S_\EC, S'_\EC \right), \quad S'_\dis - S_\dis = \partial_\EC\!\left( S_\dis, S'_\dis \right), \quad S'_\mix - S_\mix = \partial_\EC\!\left( S_\mix, S'_\mix \right).
\end{equation}
\end{thm}
\begin{proof}
If $w$ is in environmental equilibrium, then $D_{A,B}$ and $M_{A,B}$ are $\tilde \mu_{A,B}$-almost surely constant. We rewrite the selective changes in terms of the measure $\tilde \mu_{A,B}$:
\begin{eqnarray}
    \partial_\NS S_\dis
    &=& \sum \cov(-U_{A,B} \log D_{A,B}, U) 
    = \sum \E[(-U_{A,B} \log D_{A,B}) (U-1)] \nonumber \\
    &=& \sum \tilde p_{A,B} \tilde \E_{A,B}[(-D_{A,B} \log D_{A,B}) (U-1)] \nonumber \\
    &=& \sum \tilde p_{A,B} \tilde \E_{A,B}[(-\tilde\E_{A,B}[D_{A,B}] \log \tilde\E_{A,B}[D_{A,B}]) (U-1)] = 0,
\end{eqnarray}
and 
\begin{eqnarray}
    \partial_\NS S_\mix
    &=& \sum \cov(U_{A,B} \log M_{A,B}, U) 
    = \sum \E[(U_{A,B} \log M_{A,B}) (U-1)] \nonumber \\
    &=& \sum \tilde p_{A,B} \bar U_{A,B} \tilde \E_{A,B}[(M_{A,B} \log M_{A,B}) (U-1)] \nonumber \\
    &=& \sum \tilde p_{A,B} \bar U_{A,B} \tilde \E_{A,B}[(\tilde\E_{A,B}[M_{A,B}] \log \tilde\E_{A,B}[M_{A,B}]) (U-1)] = 0.
\end{eqnarray}
By linearity, $\partial_\NS S_\EC = \partial_\NS S_\dis + \partial_\NS S_\mix = 0$. The Price equation implies \eqref{eqn_weakthirdlaw_env}.
\end{proof}

We strengthen this result by showing that non-equilibrium processes can fluctuate, with the  fluctuation windows collapsing in environmental equilibrium. We introduce some selective-fluctuation coefficients to define the windows.

%We introduce some coefficients to help us describe the selective fluctuation window.

\begin{defn}[Selective Fluctuation Coefficients]
Let $A \in \I$ and $B \in \I'$. Define the local selective fluctuation coefficients:
\begin{eqnarray}
    %\nu_{A,B} &:=& \tilde \E_{A,B}[U D_{A,B}^c] = \tilde \E_{A,B}[U_{A,B}^c] = \frac{1}{\tilde p_{A,B}} \E[U U_{A,B}^c] = \frac{1}{\tilde p_{A,B}} \E[U_{A,B}^c(U_{A,B} + U_{A,B}^c)] \label{def_nu} \\
    \varphi_{A,B} &:=& \tilde \E_{A,B}[U] = \frac{1}{\tilde p_{A,B}} \E[1_{A,B} U^2] = \frac{1}{\tilde p_{A,B}} \E[(U_{A,B} + U_{A,B}^c)^2] \label{def_Harphi} \\
    \lambda_{A,B} &:=& \tilde \E_{A,B}[U D_{A,B}] = \bar U_{A,B} \tilde \E_{A,B}[U M_{A,B}] = \tilde \E_{A,B}[U_{A,B}] \label{def_lambda} \\
    &=& \frac{1}{\tilde p_{A,B}} \E[U U_{A,B}] = \frac{1}{\tilde p_{A,B}} \E[U_{A,B}(U_{A,B} + U_{A,B}^c)], \nonumber \\
    \gamma_{A,B} &:=& \tilde \E_{A,B}[U D_{A,B}^2] = \bar U_{A,B}^2 \tilde \E_{A,B}[U M_{A,B}^2] = \tilde \E_{A,B}\!\left[\frac{U_{A,B}^2}{U}\right] = \frac{1}{\tilde p_{A,B}} \E[U_{A,B}^2] %\\
    %&=& \frac{1}{\tilde p_{A,B}} \left( \tilde \E\!\left[ \frac{U_{A,B}}{\sqrt{U}} \right]^2 + \tilde \var\!\left( \frac{U_{A,B}}{\sqrt{U}} \right) \right)  = \frac{1}{\tilde p_{A,B}} \left( 1 + \var(U_{A,B}) \right) \nonumber
\end{eqnarray}
\end{defn}

%\begin{eqnarray}
%    \bar U_{A,B} = \tilde \E[D_{A,B}] &=& \tilde p_{A,B} \tilde \E_{A,B}[D_{A,B}] \le \tilde p_{A,B} \label{eqn_tildeAB_DAB} \\
%    1 = \tilde \E[M_{A,B}] &=& \tilde p_{A,B} \tilde \E_{A,B}[M_{A,B}] \le \frac{\tilde p_{A,B}}{\bar U_{A,B}}  \label{eqn_tildeAB_MAB} \\
%    %\tilde \E[M_{A,B} &=& \frac{\tilde \E[D_{A,B}]}{\bar U_{A,B}} \\
%    \E[D_{A,B}^2] &\le& \bar U_{A,B}
%\end{eqnarray}

%This lemma relates the quantities:

\begin{lem} \label{lem_gamma_lambda_varphi}
%Consequently, we have the relations
For any $A \in \I$ and $B \in \I'$:
\begin{equation} \label{ineq_gamma_lambda_varphi}
    \gamma_{A,B} \le \lambda_{A,B} \le \varphi_{A,B},
%    \varphi_{A,B} \ge \lambda_{A,B} \ge \gamma_{A,B}. %\qquad \mathrm{and} \qquad \varphi_{A,B} \ge \nu_{A,B}.
\end{equation}
with saturation when $w$ is purely dispersive from $A$ to $B$. 
\end{lem}
\begin{proof}
Since $U_{A,B}(i) \le U(i)$ for all $i$, we have  $D_{A,B}(i)^2 \le D_{A,B}(i) \le 1$ which proves \eqref{ineq_gamma_lambda_varphi}.

 For saturation, observe that $\lambda_{A,B} = \varphi_{A,B}$ if and only if $D_{A,B} = 1$ $\tilde \mu_{A,B}$-a.s., i.e., the purely dispersive case. We have $\gamma_{A,B} = \lambda_{A,B}$ if and only if $D_{A,B}^2 = D_{A,B}$ $\tilde \mu_{A,B}$-a.s., which is equivalent to $D_{A,B} = 1$ (since $D_{A,B} > 0$ $\tilde \mu_{A,B}$-a.s.), i.e., the purely dispersive case. 
\end{proof}

%We state and prove the Strong Third Law. 

\begin{thm}[Strong Third Law of Natural Selection] \label{thm_thirdlaw}
Let $w$ be a finite-entropy process, and let $(\A_*, \B_*)$ be a generating joint partition. Then:
\begin{equation}
    \sum \left( \tilde p_{A,B} \lambda_{A,B} \log \tfrac{\lambda_{A,B}}{\gamma_{A,B}} - \bar U_{A,B} \log \tfrac{\tilde p_{A,B}}{\bar U_{A,B}} \right) 
    \le \partial_\NS S_\dis %\\          
    %&& \qquad 
    \le \sum \left( \tilde p_{A,B} \lambda_{A,B} \log \tfrac{\varphi_{A,B}}{\lambda_{A,B}} - \bar U_{A,B} \log \tfrac{\bar U_{A,B}}{\tilde \E[D_{A,B}^2]} \right),  \label{ineq_generalthirdlaw_dis}
\end{equation}
\begin{equation}
    \hspace{-.55in} \sum \left( \tilde p_{A,B} \lambda_{A,B} \log \tfrac{\lambda_{A,B}}{\varphi_{A,B} \bar U_{A,B}} - \bar U_{A,B} \log \tilde \E[M_{A,B}^2] \right) \le \partial_\NS S_\mix         
    \le \sum \left( \tilde p_{A,B} \lambda_{A,B} \log \tfrac{\gamma_{A,B}}{\lambda_{A,B} \bar U_{A,B}} - \bar U_{A,B} \log \tfrac{1}{\tilde p_{A,B}} \right), \label{ineq_generalthirdlaw_mix}
\end{equation}
\begin{equation}
    \hspace{-.7in} 
    \mbox{\scalebox{0.95}{$
    \sum \left( \tilde p_{A,B} \lambda_{A,B} \log \tfrac{\lambda_{A,B}^2}{\gamma_{A,B} \varphi_{A,B} \bar U_{A,B}} - \bar U_{A,B} \log \tfrac{\tilde p_{A,B} \tilde \E[D_{A,B}^2]}{\bar U_{A,B}^3} \right) \le \partial_\NS S_\EC \le \sum \left( \tilde p_{A,B} \lambda_{A,B} \log \tfrac{\varphi_{A,B} \gamma_{A,B}}{\lambda_{A,B}^2 \bar U_{A,B}} - \bar U_{A,B} \log \tfrac{\bar U_{A,B}}{\tilde p_{A,B} \E[D_{A,B}^2]} \right),$}} \label{ineq_generalthirdlaw_EC}
\end{equation}
where the sums are over partition sets $A \in \A_*$ and $B \in \B_*$. These inequalities are saturated when $w$ is in environmental equilibrium.

The bounds in \eqref{ineq_generalthirdlaw_dis} (resp. \eqref{ineq_generalthirdlaw_mix}, \eqref{ineq_generalthirdlaw_EC}) vanish if and only if $w$ is purely dispersive (resp. mixing, reversible). 
\end{thm}

We prove the Strong Third Law in Appendix \ref{app_proofofstrongthirdlaw}. We do so by splitting each selective change into two terms, then controlling with Jensen's inequality using the measure $\tilde \mu_{A,B}$.

\section{Environmental Change of Environmental Entropy} \label{sect_environmentalchangeofenvironmentalentropy}

Consider composable processes $w : \mu \mapsto \mu'$ and $w' : \mu' \mapsto \mu''$, with generating joint partitions $(\A_*,\B_*)$ and $(\B_*', \C_*)$.
%, with shared intermediate partition $\B_*$. This can be done without loss of generality. Indeed, suppose $(\A_*,\B_*^1)$ and $(\B_*^2,\C_*)$ are generating joint partitions, and let $\B_* := \B_*^1 \wedge \B_*^2$ be the least common refinement of $\B_*^1$ and $\B_*^2$. Then $(\A_*, \B_*)$ and $(\B_* \wedge \B_*, \C_*)$ are also joint generating partitions of $w$ and $w'$, respectively.
Define the changes of the environmental entropy functionals:
\begin{eqnarray}
    \Delta(S_\EC, S'_\EC) &:=& \sum \left(-\bar U'_{B',C} \log \bar U'_{B',C} \right) + \sum \bar U_{A,B} \log \bar U_{A,B}  \\
    \Delta(S_\dis, S_\dis') &:=& \sum \E'\!\left[-U'_{B',C} \log D'_{B',C} \right] + \sum \E\!\left[ U_{A,B} \log D_{A,B} \right] \\
    \Delta(S_\mix, S_\mix') &:=& \sum \E'\!\left[U'_{B',C} \log M'_{B',C} \right] - \sum \E\!\left[ U_{A,B} \log M_{A,B} \right],
\end{eqnarray}
where the sums are over the joint generating partitions $(\B_*',\C_*)$ and $(\A_*,\B_*)$, respectively. 

The Price equation decomposes the changes into selective and environmental pieces:
\begin{eqnarray}
    \Delta(S_\EC, S'_\EC) &=& \partial_\NS S_\EC + \partial_\EC(S_\EC,S_\EC') \\
    \Delta(S_\dis, S'_\dis) &=& \partial_\NS S_\dis + \partial_\EC(S_\dis,S_\dis') \\
    \Delta(S_\mix, S'_\mix) &=& \partial_\NS S_\mix + \partial_\EC(S_\mix,S_\mix'),
\end{eqnarray}
where the selective changes are as in Section \ref{sect_thirdlaw}, and the environmental changes are as follows:
\begin{eqnarray}
    \partial_\EC(S_\EC,S_\EC') &:=& \tilde \E\!\left[ \sum \left\langle -U'_{B',C} \log \bar U'_{B',C}  \right\rangle_w + \sum U_{A,B} \log \bar U_{A,B} \right] \label{eqn_envchange_EC} \\
    \partial_\EC(S_\dis, S_\dis') &:=& \tilde \E\!\left[ \sum \left\langle -U'_{B',C} \log D'_{B',C}  \right\rangle_w + \sum U_{A,B} \log D_{A,B} \right]  \label{eqn_envchange_dis} \\
    \partial_\EC(S_\mix, S_\mix') &:=& \tilde \E\!\left[ \sum \left\langle U'_{B',C} \log M'_{B',C}  \right\rangle_w - \sum U_{A,B} \log M_{A,B} \right],  \label{eqn_envchange_mix} 
\end{eqnarray}
where the sums are over generating joint partitions, and $\tilde \E[X] = \E[UX]$.

%We prove a lower bound on the environmental change of environmental entropy. 

\begin{pro}[Intergenerational Environmental Change]
Let $w$ be finite-entropy. The environmental change of environmental entropy can be written intergenerationally as follows:
\begin{equation} \label{eqn_intergenerationalenvironmentalchange_EC}
    \partial_\EC\!\left(S_\EC, S'_\EC\right) = - \sum_{A,B;B',C} \frac{\E[U U_{A,B}]}{\E[U^2]} \bar U'_{B',C} \log \frac{\bar U'_{B',C}}{\bar U_{A,B}},
\end{equation}
where the sum is over generating joint partition sets. 
\end{pro}
\begin{proof}
Each of the environmental changes \eqref{eqn_envchange_EC}-\eqref{eqn_envchange_mix} can be written in an intergenerational form for certain observables $X_{A,B}$ and $Y_{B',C}$, and parameters $\alpha_{A,B}$ and $\beta_{B',C}$ satisfying $\sum_{A,B} \alpha_{A,B} = 1 = \sum_{B',C} \beta_{B',C}$. We compute:
\begin{eqnarray}
    &&\pm \tilde \E\!\left[ \sum_{B',C} \left\langle -U'_{B',C} \log Y_{B',C} \right\rangle_w + \sum_{A,B} U_{A,B} \log X_{A,B} \right] \nonumber \\
    &=& \pm \sum_{A,B;B',C} \alpha_{A,B} \beta_{B',C} \tilde \E\!\left[ \left\langle -\frac{U'_{B',C}}{\beta_{B',C}} \log Y_{B',C} \right\rangle_w + \frac{U_{A,B}}{\alpha_{A,B}} \log X_{A,B} \right] \nonumber \\
    &=& \pm \sum_{A,B;B',C} \alpha_{A,B} \beta_{B',C} \tilde \E\!\left[ \left\langle -\frac{U'_{B',C}}{\beta_{B',C}} \log \frac{Y_{B',C}}{X_{A,B}}  - \left( \frac{U'_{B',C}}{\beta_{B',C}} - \frac{U_{A,B}}{\alpha_{A,B}} \right) \log X_{A,B} \right\rangle_w \right]. \label{eqn_intergenerationalenvironmentalchange}
\end{eqnarray}

%Environmental Case) 
When we apply \eqref{eqn_intergenerationalenvironmentalchange} with $X_{A,B} = \bar U_{A,B}$, $Y_{B',C} = \bar U'_{B',C}$, $\alpha_{A,B} = \frac{\tilde \E[U_{A,B}]}{\tilde \E[U]}$, %= \frac{\cov(U_{A,B},U)+\bar U_{A,B}}{\var(U)+1}$, 
and $\beta_{B',C} = \bar U'_{B',C}$, then the second term vanishes, and the first term equals \eqref{eqn_intergenerationalenvironmentalchange_EC}.

%Dispersive Case) When we apply \eqref{eqn_intergenerationalenvironmentalchange} with $X_{A,B} = D_{A,B}$, $Y_{B',C} = D'_{B',C}$, and ??, then
%\begin{eqnarray}
%    \partial_\EC\!\left( S_\dis, S_\dis' \right) &=& \sum_{A,B;B',C} \alpha_{A,B} \beta_{B',C} \left( \tilde \E\!\left[ - \left\langle \frac{U'_{B',C}}{\beta_{B',C}} \log \frac{D'_{B',C}}{D_{A,B}} \right\rangle_w \right] \right)
%\end{eqnarray}

%Mixing Case) When we apply \eqref{eqn_intergenerationalenvironmentalchange} with $X_{A,B} =$

\end{proof}

\begin{rem}
If the reader needs bounds on the environmental change of dispersive and mixing entropies, then apply Jensen's inequality to representation \eqref{eqn_intergenerationalenvironmentalchange}. 

%Representation \eqref{eqn_intergenerationalenvironmentalchange} can be applied to environmental change of dispersive and mixing entropies, and Jensen's inequality to prove bounds on  
\end{rem}

\section{Multi-Level Environmental Entropy} \label{sect_multilevel_enventropy}

We state the multi-level Price equation for the environmental entropy functionals, following Section \ref{sect_multilevelprice}. Let $w$, $w'$ and $w''$ be composable processes, with joint generating partitions $(\A_*,\B_*)$, $(\B_*',\C_*)$, and $(\C_*',\D_*)$, respectively. Then:

\begin{eqnarray}
    \Delta(S'_\EC, S''_\EC) &=& \cov\!\left(\E'_w\!\left[-\sum U'_{B',C} \log \bar U'_{B',C} \right], \E'_w[U'] \right) + \E\!\left[ \cov_w\!\left(-\sum U'_{B',C}  \log \bar U'_{B',C}, U'\right) \right] \nonumber \\
    && ~+~ \E\!\left[ \E'_w\!\left[\Delta_{w'}\left(-\sum U'_{B',C} \log \bar U'_{B',C}, -\sum U''_{C',D} \log \bar U''_{C',D}\right) U' \right] \right] \\
    \Delta(S'_\dis, S''_\dis) &=& \cov\!\left(\E'_w\!\left[-\sum U'_{B',C} \log D'_{B',C} \right], \E'_w[U'] \right) + \E\!\left[ \cov_w\!\left(-\sum U'_{B',C}  \log D'_{B',C}, U'\right) \right] \nonumber \\
    && ~+~ \E\!\left[ \E'_w\!\left[\Delta_{w'}\left(-\sum U'_{B',C} \log D'_{B',C}, -\sum U''_{C',D} \log D''_{C',D}\right) U' \right] \right] \\
    \Delta(S'_\mix, S''_\mix) &=& \cov\!\left(\E'_w\!\left[\sum U'_{B',C} \log M'_{B',C} \right], \E'_w[U'] \right) + \E\!\left[ \cov_w\!\left(\sum U'_{B',C}  \log M'_{B',C}, U'\right) \right] \nonumber \\
    && ~+~ \E\!\left[ \E'_w\!\left[\Delta_{w'}\left(\sum U'_{B',C} \log M'_{B',C}, \sum U''_{C',D} \log M''_{C',D}\right) U' \right] \right] 
\end{eqnarray}
with sums over joint generating partitions.

\section{Quantum Environmental Entropy} \label{sect_quantum_enventropy}

Recall the quantum formalism of Sections \ref{sect_quantum} and \ref{sect_quantum_selentropy}. Consider a quantum evolutionary process $\W : \mu \mapsto \mu'$, with quantum relative fitness operator $U := \frac{1}{\bar W} \W^\dagger(\Id')$. 

\begin{defn}
A countable quantum partition is a representation of the identity operator into countably many projection operators. Formally, let $\Pi$ and $\Pi'$ denote the spaces of projection operators in the Hilbert spaces $\H$ and $\H'$, respectively. We say $\A \subseteq \Pi$ and $\B \subseteq \Pi'$ are countable quantum partitions of $\H$ and $\H'$ if $\sum_{\pi \in \A} \pi = \Id$ and $\sum_{\pi' \in \B'} \pi' = \Id'$.
\end{defn}

For projection operators $\pi \in \Pi$ and $\pi' \in \Pi'$, define local relative fitness by $U_{\pi,\pi'} := (\W^\dagger\pi') \circ U \circ \pi$ and local density $\mu_{\pi,\pi'} := (\W^\dagger\pi') \circ \mu \circ \mu$. Average local relative fitness is given by $\bar U_{\pi,\pi'} := \E_\mu[U_{\pi,\pi'}] = \Tr(U_{\pi,\pi'} \mu)$. Define quantum partition environmental entropy:
\begin{equation}
    S_\EC(\A,\B) := \sum \Tr\!\left(- \bar U_{\pi,\pi'} \log \bar U_{\pi,\pi'} \right),
\end{equation}
where we sum over partition operators $\pi \in \A$ and $\pi' \in \B$. Define the quantum environmental entropy by taking the suprema over countable quantum partitions:
\begin{equation}
    S_\EC := \sup S_\EC(\A,\B) = \sum \Tr\!\left(- \bar U_{\pi,\pi'} \log \bar U_{\pi,\pi'} \right).
\end{equation}

\begin{conj}
We conjecture that a quantum version of Sinai's theorem holds, i.e., that $S_\EC$ be realized by a generating joint quantum partition $(\A_*,\B_*)$. This would likely satisfy a maximality relation like $X = \sum_{\pi,\pi'} \Tr\!\left(X \W^\dagger \pi' \otimes \pi \right) \W^\dagger \pi' \otimes \pi$ for any self-adjoint $X$. 
\end{conj}

We prove strong strong bounds for the partition entropies. Define dispersion and mixing operators $D_{\pi,\pi'} := U_{\pi,\pi'} U^{-1}$ and $M_{\pi,\pi'} := \frac{1}{\bar U_{\pi,\pi'}} U_{\pi,\pi'} U^{-1}$. Define partition dispersion and mixing entropies using the spectral theorem:
\begin{eqnarray}
    S_\dis(\A,\B) &:=& \sum \E_\mu[-U_{\pi,\pi'} \log D_{\pi,\pi'}] = \Tr\!\left( (-U_{\pi,\pi'} \log D_{\pi,\pi'} ) \mu \right) \\
    S_\mix(\A,\B) &:=& \sum \E_\mu[U_{\pi,\pi'} \log M_{\pi,\pi'}] = \Tr\!\left( (U_{\pi,\pi'} \log M_{\pi,\pi'} ) \mu \right).
\end{eqnarray}
The environmental entropy decomposes as the sum of dispersion and mixing entropies:
\begin{equation}
    S_\EC(\A,\B) = S_\dis(\A,\B) + S_\mix(\A,\B).
\end{equation}

We say $\W$ is partition purely dispersive (resp. mixing) if $D_{\pi,\pi'}$ (resp. $M_{\pi,\pi'})$ is $\mu_*$-a.s. constant. We say $\W$ is in partition quantum environmental equilibrium if $D_{\pi,\pi'}$ and $M_{\pi,\pi'}$ are $\mu_{\pi,\pi'}$-a.s. constant. 

\begin{conj}
%    We conjecture that quantum versions of Invertibility Theorems \ref{thm_efficiency}-\ref{thm_irreversibility} hold, where 
    We conjecture that vanishing quantum environmental (resp. dispersive, mixing) entropies correspond to fully (resp. left-, right-) invertible quantum processes.
\end{conj}

The partition operation is a contraction, i.e., $U_{\pi,\pi'} \le U$ in the ordering of self-adjoint operators. To see this, apply complete non-negativity of $\W$ to the completeness identity: $U = U_{\pi,\pi'} + U_{\pi,\Id'-\pi'} + U_{\Id-\pi,\pi'} + U_{\Id-\pi,\Id'-\pi'} \ge U_{\pi,\pi'}$.
Write the squares $D_{\pi,\pi'}^2 := U_{\pi,\pi'} U^{-1} U_{\pi,\pi'} U^{-1}$ and $M_{\pi,\pi'}^2 := \frac{1}{\bar U_{\pi,\pi'}^2} D_{\pi,\pi'}^2 = \frac{1}{\bar U_{\pi,\pi'}^2} U_{\pi,\pi'} U^{-1} U_{\pi,\pi'} U^{-1}$. By contraction, $\tilde \E_\mu[D_{\pi,\pi}^2] \le \E_\mu[D_{\pi,\pi'}] \le 1$ and $\tilde \E_\mu[M_{\pi,\pi}^2] \le \frac{1}{\bar U_{\pi,\pi'}} \E_\mu[M_{\pi,\pi'}] \le 1$.

\begin{thm}[Strong Bounds on Quantum Partition  Dispersion and Mixing Entropies]
Let $\W$ be a finite-entropy quantum process. For any joint quantum countable partition $(\A,\B)$:
\begin{equation} \label{ineq_generalbound_dis_quantum}
    0 \le 
    \sum_{\pi,\pi'} \bar U_{\pi,\pi'} \log \frac{1}{\tilde \E_\mu[D_{\pi,\pi'}^2]} \le S_\dis(\A,\B) \le \sum_{\pi,\pi'} \bar U_{\pi,\pi'} \log \frac{\tilde p_{\pi,\pi'}}{\bar U_{\pi,\pi'}} 
    \le S_\EC(\A,\B),
\end{equation}
and
\begin{equation} \label{ineq_generalbound_mix_quantum}
    0 \le 
    \sum_{\pi,\pi'} \bar U_{\pi,\pi'} \log \frac{1}{\tilde p_{\pi,\pi'}} \le S_\mix(\A,\B) \le \sum_{\pi,\pi'} \bar U_{\pi,\pi'} \log \frac{\tilde \E_\mu[M_{\pi,\pi'}^2]}{\bar U_{\pi,\pi'}} 
    \le S_\EC(\A,\B),
\end{equation}
where the sums are over partition operators $\pi \in \A$ and $\pi' \in \B$. The inner inequalities are saturated when $w$ is in partition quantum environmental equilibrium.

% $U$ is a contraction since its operator norm satisfies $\|U\|_\op := $
% Need: U_{\pi,\pi'} \le U. We have $U_{\pi,\pi'} := \W^\dagger(\pi') U \pi$. The adjoint means that $\Tr(U_{\pi,\pi'} \mu) = \Tr'(\pi' \W(U\pi \mu))$. 

%Want to show that for any further partition elements \psi,\psi', (U - U_{\pi,\pi'})_{\psi,\psi} \ge 0$. We compute: (U - U_{\pi,\pi'})_{\psi,\psi} = U_{\psi,\psi'} - U_{\pi \psi, \pi' \psi'} 
% because basis projection operators commute. 

The outer upper (resp. lower) bound of \eqref{ineq_generalbound_dis_quantum} and outer lower (resp. upper) bound of \eqref{ineq_generalbound_mix_quantum} are saturated if and only if $w$ is purely quantum dispersive (resp. mixing).
\end{thm}
\begin{proof}
The proof is similar to that of Theorem \ref{thm_strongbound_dismix}, \emph{mutatis mutandis}, using quantum Jensen's inequality (Lemma \ref{lem_quantumjensen}) and the left selective transformation $\E^\opleft_\mu[X] := \E_{W\mu}$. 
%This also relies on the identity:
%\begin{equation}
%    \tilde \E^\opleft_\mu[M^2_{\pi,\pi'}] - 1 = \frac{1}{\bar U_{\pi,\pi'}} \E_\mu\!\left[ U_{\pi,\pi'} U^{-1} U_{\pi,\pi'} \right]
%\end{equation}
\end{proof}

Define the selective changes of the entropy functionals:
\begin{eqnarray}
    \partial_\NS S_\EC(\A,\B) &:=& \cov_\mu(-U_{\pi,\pi'} \log \bar U_{\pi,\pi'}, U) \\
    \partial_\NS S_\dis(\A,\B) &:=& \cov_\mu(-U_{\pi,\pi'} \log D_{\pi,\pi'}, U) \\
    \partial_\NS S_\mix(\A,\B) &:=& \cov_\mu(U_{\pi,\pi'} \log M_{\pi,\pi'}, U).
\end{eqnarray}

\begin{pro}[Weak Partition Third Law of Quantum Selection]
If $\W$ is in quantum environmental equilibrium, then for all joint quantum countable partitions $(\A,\B)$, 
\begin{equation}
    \partial_\NS S_\EC(\A,\B) = 0 = \partial_\NS S_\dis(\A,\B) = \partial_\NS S_\mix(\A,\B).
\end{equation}
\end{pro}
\begin{proof}
The proof is similar to Proposition (\ref{thm_weakthirdlaw}), \emph{mutatis mutandis}. 
\end{proof}

Write $\tilde \E_{\pi,\pi'} := \tilde \E_{\mu_{\pi,\pi'}}$. Define the quantum local selective fluctuation coefficients:
\begin{eqnarray}
    \varphi_{\pi,\pi'} &:=& \tilde \E_{\pi,\pi'}[U] \\
    \lambda_{\pi,\pi'} &:=& \tilde \E_{\pi,\pi'}[D_{\pi,\pi'} U] = \tilde \E_{\pi,\pi'}[U_{\pi,\pi'}] \\
    \gamma_{\pi,\pi'} &:=& \tilde \E_{\pi,\pi'}[D^2_{\pi,\pi'} U].
\end{eqnarray}
For any $\pi, \pi'$,
\begin{equation}
    \gamma_{\pi,\pi'} \le \lambda_{\pi,\pi'} \le \varphi_{\pi,\pi'},
\end{equation}
by a similar argument as Lemma \ref{lem_gamma_lambda_varphi} plus contractivity of $U$. 

%Need: $D_{\pi,\pi'}^2 \le D_{\pi,\pi'} \le \Id$.
%Proof: $\Id - D_{\pi,\pi'} = \Id - U_{\pi,\pi'} U^{-1}$
%Proof: $D_{\pi,\pi'} - D_{\pi,\pi'}^2 = U_{\pi,\pi'} U_{\pi,\pi'} U^{-1} U_{\pi,\pi'} U^{-1} $

\begin{thm}[Strong Partition Third Law of Quantum Selection] \label{thm_thirdlaw_quantum}
Let $w$ be a finite-entropy quantum process, and let $(\A, \B)$ be a joint quantum countable partition. Then:
\begin{equation}
    \sum \left( \tilde p_{\pi,\pi'} \lambda_{\pi,\pi'} \log \tfrac{\lambda_{\pi,\pi'}}{\gamma_{\pi,\pi'}} - \bar U_{\pi,\pi'} \log \tfrac{\tilde p_{\pi,\pi'}}{\bar U_{\pi,\pi'}} \right) 
    \le \partial_\NS S_\dis(\A,\B) %\\          
    %&& \qquad 
    \le \sum \left( \tilde p_{\pi,\pi'} \lambda_{\pi,\pi'} \log \tfrac{\varphi_{\pi,\pi'}}{\lambda_{\pi,\pi'}} - \bar U_{\pi,\pi'} \log \tfrac{\bar U_{\pi,\pi'}}{\tilde \E[D_{\pi,\pi'}^2]} \right),  \label{ineq_generalthirdlaw_dis_quantum}
\end{equation}
\begin{equation}
    \hspace{-.55in} \sum \left( \tilde p_{\pi,\pi'} \lambda_{\pi,\pi'} \log \tfrac{\lambda_{\pi,\pi'}}{\varphi_{\pi,\pi'} \bar U_{\pi,\pi'}} - \bar U_{\pi,\pi'} \log \tilde \E[M_{\pi,\pi'}^2] \right) \le \partial_\NS S_\mix(\A,\B)         
    \le \sum \left( \tilde p_{\pi,\pi'} \lambda_{\pi,\pi'} \log \tfrac{\gamma_{\pi,\pi'}}{\lambda_{\pi,\pi'} \bar U_{\pi,\pi'}} - \bar U_{\pi,\pi'} \log \tfrac{1}{\tilde p_{\pi,\pi'}} \right), \label{ineq_generalthirdlaw_mix_quantum}
\end{equation}
\begin{equation}
    \hspace{-.7in} 
    \mbox{\scalebox{0.95}{$
    \sum \left( \tilde p_{\pi,\pi'} \lambda_{\pi,\pi'} \log \tfrac{\lambda_{\pi,\pi'}^2}{\gamma_{\pi,\pi'} \varphi_{\pi,\pi'} \bar U_{\pi,\pi'}} - \bar U_{\pi,\pi'} \log \tfrac{\tilde p_{\pi,\pi'} \tilde \E[D_{\pi,\pi'}^2]}{\bar U_{\pi,\pi'}^3} \right) \le \partial_\NS S_\EC(\A,\B) \le \sum \left( \tilde p_{\pi,\pi'} \lambda_{\pi,\pi'} \log \tfrac{\varphi_{\pi,\pi'} \gamma_{\pi,\pi'}}{\lambda_{\pi,\pi'}^2 \bar U_{\pi,\pi'}} - \bar U_{\pi,\pi'} \log \tfrac{\bar U_{\pi,\pi'}}{\tilde p_{\pi,\pi'} \E[D_{\pi,\pi'}^2]} \right),$}} \label{ineq_generalthirdlaw_EC_quantum}
\end{equation}
where the sums are over partitions $\pi \in \A_*$ and $\pi' \in \B_*$. These inequalities are saturated when $\W$ is in partition environmental equilibrium.

The bounds in \eqref{ineq_generalthirdlaw_dis_quantum} (resp. \eqref{ineq_generalthirdlaw_mix_quantum}, \eqref{ineq_generalthirdlaw_EC_quantum}) vanish if and only if $w$ is partition purely dispersive (resp. mixing, reversible). 
\end{thm}
\begin{proof}
The proof is similar to Theorem \ref{thm_thirdlaw}, \emph{mutatis mutandis}.
\end{proof}

%If a quantum version of Sinai's theorem holds, then these theorems would hold for the general quantum case too.

%The Weak Third Law holds

%We do not explore further questions about quantum environmental entropy in this article. The author conjectures that there is a quantum version of Sinai's theorem, ensuring that the environmental entropy can be realized by a certain countable quantum partition, and that many of the theorems of Part 3 have a quantum version. 

%We present a quantum form of Sinai's theorem, ensuring that we can realize the environmental entropy by a generating joint quantum partition.

%\begin{defn}
%A generating joint quantum partition is a partition pair $(\Pi_*, \Pi'_*)$ such that the collection of partition products $\{\W^{-1}(\pi') \pi\}$ is dense in the algebra of self-adjoint operators $\A$ of $\H$. 
%\end{defn}

%\begin{thm}[Quantum Sinai's Theorem]
%There exists a generating joint quantum partition $(\Pi_*,\Pi'_*)$, which realizes the environmental entropy:
%\begin{equation}
%    S_\EC = S_\EC(\Pi_*, \Pi'_*).
%\end{equation}
%\end{thm}

\section{Conclusion} \label{sect_conclusion}

% This section simply states what the researcher thinks the data mean, and, as such, should relate directly back to the problem/question stated in the introduction. This section should not offer any reasons for those particular conclusions--these should have been presented in the Discussion section. By looking at only the Introduction and Conclusions sections, a reader should have a good idea of what the researcher has investigated and discovered even though the specific details of how the work was done would not be known.

%Restate your topic and why it is important,
%Restate your thesis/claim,
%Address opposing viewpoints and explain why readers should align with your position,
%Call for action or overview future research possibilities.

Price introduced his famous equation to describe all change in terms of fitnesses, covariances, and environmental terms. In Part 1, we derive a Price equation in measure-theoretic and quantum contexts. We prove Zeroth and First Laws of Natural Selection: selection increases and accelerates the pace of selection, minimizing selective effects only in the selective equilibrium case (life-or-death processes). Otherwise, selective effects accelerate over time.

%In part 2, w
We introduce selective entropy to further quantify selective effects, and we prove the Second Law to show how selection compounds selective effects. Processes in selective equilibrium satisfy identities instead of inequalities, and can be further studied analytically. Non-selective equilibrium processes compound selective effects exponentially beyond their equilibrium counterparts.

%selection acts exponentially upon selective entropy (Second Law).  Consequently, processes in selective equilibrium collapse these inequalities into identities, and equilibrium models can be studied analytically. 

We introduce environmental entropy to quantify environmental effects, which decompose into dispersion and mixing pieces. The Weak Third Law shows that environmental-equilibrium processes have vanishing selective change of dispersion and mixing entropy functionals, while the Strong Third Law provides quantitative windows for these entropies to fluctuate within. 

In the quantum case, two quantum Price equations hold, as do quantum versions of the Zeroth, First, and Second Laws. A partition version of the Third Law holds, and it is an open question whether the Third Law fully extends to the general quantum case. 

We hope this article is helpful to mathematicians and scientists seeking to understand how selection and environmental change interact within evolutionary processes of interest.

\bibliographystyle{alpha}
\bibliography{priceequation}

\appendix
\part*{Appendices}
%to The Mathematics of Evolution: The Price Equation, Natural Selection, and Environmental Change}

\section{Proof of Quantum Jensen's Inequality (Lemma \ref{lem_quantumjensen})} \label{app_quantumjensenproof}

\begin{proof}[Proof of Lemma \ref{lem_quantumjensen}]
By convexity, the tangent line to the real-valued $f$ at $\bar X$ is below the graph of $f$. Specifically, there exist real numbers $a := f'(\bar X)$ and $b := f(\bar X) - f'(\bar X) \bar X$ such that for all real $x$, 
\begin{equation}
a x + b \le f(x) \qquad \mathrm{and} \qquad a \bar X + b = f(\bar X).
\end{equation}

Consequently, with respect to the partial ordering of self-adjoint operators on a Hilbert space, we have
\begin{equation}
f(X) \ge aX + b\Id,
\end{equation}
that is, the operator $f(X) - aX - b\Id$ is non-negative, where $\Id$ is the identity operator.\footnote{i.e., for all $h$, $\langle h, (f(X) - aX - b\Id) h \rangle \ge 0$.}

We compute:
\begin{equation} \label{ineq_quantumjensonproof}
\E_\mu[f(X)] \ge \E_\mu[aX + b\Id] = a \E_\mu[X] + b = a \bar X + b = f(\bar X) = f(\E_\mu[X]).
\end{equation}

If $X$ is $\mu$-a.s. constant (with $X = \bar X \Id$ a.s.), then the inequality \eqref{ineq_quantumjensonproof} is saturated. If $X$ is not $\mu$-a.s. constant, then for operator values away from $\bar X$, $f(X) > aX + b\Id$, and so \eqref{ineq_quantumjensonproof} is a strict inequality.  

\end{proof}

\section{Proof of Generalized Sinai's Theorem (Theorem \ref{thm_sinai})} \label{app_sinai}

%We now prove a generalized version of Sinai's theorem from classical dynamical systems theory. The classical theorem states that Kolmogorov-Sinai entropy can be realized by a generating partition, that is, a countable, measurable partition for which partition KS entropy equals the supremum. 

%We extend this result to general, one-step evolutionary processes from $(I,\I,\mu)$ to $(I',\I',\mu')$. We show that the supremum in general environmental entropy can be achieved at a particular countable, measurable joint partition, which we denote a \emph{generating joint partition}. %Furthermore, this joint partition is also a generating joint partition for the environmental entropy and hence the mixing entropy. %In particular, this ensures the identity $S_\EC = S_\dis + S_\mix$. 

\begin{proof}[Proof of Theorem \ref{thm_sinai}] 
We prove the theorem by defining certain metrics on the space of joint partitions, show that they are complete, and use this to ensure the supremum is obtained. Define the space of countable, measurable joint partitions: 
\begin{equation}
\PP := \{ (\A, \B) : \mbox{$\A, \B$ countable measurable partitions of $I, I'$, resp.} \}.
\end{equation}

Define a pseudo-metric on $\I \times \I'$ using the symmetric difference:
\begin{equation}
d_1((A,B), (A',B')) := \mu(A \symdiff A') + \mu'(B \symdiff B').
\end{equation}
where $X \symdiff X' := (X \cup X') - (X \cap X')$.

\begin{sublem} \label{sublem_d1complete}
$d_1$ is a complete pseudo-metric on the space $\I \times \I'$.
\end{sublem}
\begin{proof}
See \cite[Thm 1.12.16]{bogachev2007measure} or  \cite[Theorem 1]{bell2015symmetric} for proof of completeness.
\end{proof}

For any Cauchy sequence $(A_n,B_n)$, we write $A_\oo := \lim_n A_n$ and $B_\oo := \lim_n B_n$. This extends to a complete pseudo-metric on the space $\PP$. Indeed, define the partition difference metric as the minimal distance across all joint partition elements:
\begin{equation}
d_1((\A,\B), (\A',\B')) := \min\{ d_1((A,B),(A',B')) : (A,B;A',B') \in \A \times \B \times \A' \times \B' \}. 
\end{equation}

The limiting joint partitions are of the form 
\begin{equation}
    (\A_\oo,\B_\oo) := \{ (A_\oo,B_\oo) : \mbox{$(A_n,B_n) \in (\A_n,\B_n)$ is a Cauchy sequence} \}.
\end{equation}

Define the conditional partition environmental entropy between two partitions as follows:
\begin{equation} \label{eqn_defconditionalenvironmentalentropy}
    S_\EC(\A',\B'|\A,\B) := \sum_{A\in\A,B\in\B} \sum_{A'\in\A',B'\in\B'} \left( - \bar U_{A'\cap A,B'\cap B} \log \frac{\bar U_{A'\cap A,B'\cap B}}{\bar U_{A,B}} \right) \ge 0.
    %S_\EC(\A',\B'|\A,\B) := \sum_{A\in\A,B\in\B} \sum_{A'\in\A',B'\in\B'} \left( - \bar U_{A',B'} \log \frac{\bar U_{A',B'}}{\bar U_{A,B}} \right) \ge 0.
\end{equation}
This represents the additional environmental entropy in $(\A',\B')$ given that within $(\A,\B)$.

Observe that if $(\A',\B')$ is a refinement of $(\A,\B)$ (i.e., all joint partition elements of $(\A,\B)$ can be written as unions of those in $(\A',\B')$), then
\begin{equation} \label{eqn_conditionalenvironmentalentropyconservation}
    S_\EC(\A',\B') = S_\EC(\A,\B) + S_\EC(\A',\B'|\A,\B),
\end{equation}
meaning that environmental entropy is conserved under partition refinements.\footnote{Proof of \eqref{eqn_conditionalenvironmentalentropyconservation}. We compute: 
\begin{eqnarray*}
S_\EC(\A,\B) + S_\EC(\A',\B'|\A,\B) &=& \sum_{A,B} \sum_{A',B'} \left( -\bar U_{A'\cap A,B'\cap B} \log \bar U_{A,B} - \bar U_{A'\cap A,B'\cap B} \log \frac{\bar U_{A'\cap A,B'\cap B}}{\bar U_{A,B}} \right) \nonumber \\
&=& \sum_{A',B'} \left( - \bar U_{A',B'} \log \bar U_{A',B'} \right) = S_\EC(\A',\B').
\end{eqnarray*}}In that case, the reverse conditional entropy vanishes: $S_\EC(\A,\B|\A',\B') = 0$.

%Let $(\A,\B) \wedge (\A',\B') := \{ (A \cap A', B \cap B') \}$ be the least common refinement of two joint partitions. 

We define the \emph{environmental entropy metric} on $\PP$ between two joint partitions as the larger conditional entropy difference: % from each to the common refinement:
\begin{equation} \label{eqn_defenvironmentalentropymetric}
d_\EC((\A,\B), (\A',\B')) := \max\{ S_\EC(\A',\B' | \A,\B), S_\EC(\A,\B | \A',\B') \}.
%d_\KS((\A,\B), (\A',\B')) := \max\{ S_\KS(\A \wedge \A', \B \wedge \B'|\A,\B), S_\KS(\A \wedge \A', \B \wedge \B'|\A',\B') \}.
\end{equation}

If $(\A', \B')$ is a refinement of $(\A, \B)$, then
\begin{equation}
    d_\EC((\A,\B), (\A',\B')) = S_\EC(\A',\B' | \A,\B).
\end{equation}

% We have that d_\KS vanishes exactly when two joint partitions are equal (up to sets of measure zero). In case of one partition refining another

\begin{sublem}
$d_\EC$ is a coarsening of $d_1$, i.e., convergence in $d_\EC$ implies convergence in $d_1$
\end{sublem}
\begin{proof}
See \cite[Proposition 9]{le2017notes} or \cite[Fact 1.7.7]{downarowicz2011entropy} for a proof of coarseness, which applies to our setting, \emph{mutatis mutandis}. 
\end{proof}

Sinai's theorem follows as a simple consequence. Let $(\A_n, \B_n)$ be a sequence of joint partitions such that $S_\EC(\A_n,\B_n) \uparrow S_\EC$. Without loss of generality, we assume that each joint partition refines the previous one. Thus $\A_n, \B_n$ is a Cauchy sequence in $d_\EC$, since 
\begin{equation}
    d_\EC((\A_n,\B_n), (\A_{n'},\B_{n'})) = S_\EC(\A_{n'}, \B_{n'} | \A_n, \B_n) \downarrow 0
\end{equation} 
as $n,n' \to \oo$, and therefore by $d_1$-convergence, there exists a limiting joint partition $(\A_\oo, \B_\oo)$, and that this joint partition is unique up to measure zero. This proves Sinai's theorem for environmental entropy.

For the iterated Sinai's theorem, observe that the above result gives a generating joint partition for each $T$-step iterated process. We then take coarsenings over all $T$-step processes, with convergence guaranteed by a similar argument as above. 

We prove that a joint partition $(\A_*, \B_*)$ is generating if and only if $\I$ is the smallest $\sigma$-algebra containing sets $A \cap w^{-1} B$ for $A \in \A_*, B \in \B_*$. Suppose $(\A_*, \B_*)$ is generating, and let $\Gamma = \sigma(A \cap w^{-1} B)$ be the smallest $\sigma$-algebra containing sets $A \cap w^{-1} B$. Suppose $\Gamma$ is strictly smaller than $\I$, i.e., there exists $A' \in \I$ such that $A' \notin \Gamma$, i.e., $A'$ cannot be written as a countable operation of sets $A \cap w^{-1} B$. Define a new refined partition $\A_*'$ of sets of the form $A' \cap A \cap w^{-1} B$ and $(I-A') \cap A \cap w^{-1} B$. Since $\A_*'$ refines $\A_*$, the refined partition environmental entropy $S_\EC(\A_*',\B_*)$ equals the supremum $S_\EC$ and so $(\A_*',\B_*)$ is generating. Since the entropies are equal, we can write $S_\EC(A',B)$ as a combination $\sum_A S_\EC(A,B)$ for some partition sets $A$. However, since $A' \notin \Gamma$, we must have the strict relation $A' \subset \bigcup A \cap w^{-1} B$, thus $\sum_A S_\EC(A,B) > S_\EC(A',B)$, a contradiction. Thus $\Gamma = \I$.

Conversely, suppose $\Gamma = \sigma(A \cap w^{-1} B)$ for a joint partition $(\A_*,\B_*)$. If $S_\EC \ge S_\EC(\A_*,\B_*)$, then there exists a refinement $(\A_*', \B_*')$ such that $S_\EC(\A_*', \B_*') > S_\EC(\A_*,\B_*)$, so there exist sets $A' \in \A_*', B' \in \B_*'$ such that $A' \cap w^{-1} B' \notin \Gamma$, a contradiction. Thus $(\A_*,\B_*)$ is generating. 
\end{proof}

\section{Proof of Non-Positivity of Local Selective Entropy (Lemma \ref{lem_localselectiveentropynonpositive})} \label{app_localselectiveentropynonpositive}

Recall that $S_\NS(A,B) = \E[U_{A,B} \log U]$. We prove $S_\NS(A,B) \le 0$.

\begin{proof}[Proof of Lemma \ref{lem_localselectiveentropynonpositive}]
Define the local selective coefficient: $\bar W_{A,B} := \bar W_{A,B} = \E[W_{A,B}] = \frac{1}{N} \int_A w_i(B) \mu(\d i)$. Define the weighted local relative fitness by weighting the local fitness by $\bar W_{A,B}$ instead of $\bar W$:
\begin{equation}
    \hat U_{A,B} := \frac{W_{A,B}}{\bar W_{A,B}} := \frac{\bar W}{\bar W_{A,B}} \, U_{A,B}.
\end{equation}
Note that $\E[\hat U_{A,B}] = 1$ by construction. We define the renormalized local selective entropy:
\begin{equation} \label{ineq_tildeUinequality}
    \hat S_\NS(A,B) := \E[-\hat U_{A,B} \log \hat U_{A,B}] \le \E[-\hat U_{A,B}] \log \E[\hat U_{A,B}] = -1 \log 1 = 0.
\end{equation}
Inequality \eqref{ineq_tildeUinequality} is saturated when $\hat U_{A,B} = 1$ almost surely. %, i.e., when $U = \bar W_{A,B} / \bar W$.

% This means that $\hat U_{A,B}$ can be considered the relative fitness for a different related process $w_{A,B}$, defined as follows:
% \begin{equation}
%    w_{A,B;i}(\d i') := \hat 
% \end{equation}

We now rewrite local selective entropy in terms of the renormalized version:
\begin{eqnarray}
S_\NS(A,B) &=& \E[-U_{A,B} \log U] = \frac{\bar W_{A,B}}{\bar W} \E[-\hat U_{A,B} \log U] \nonumber \\
&=& \frac{\bar W_{A,B}}{\bar W} \E\!\left[-\hat U_{A,B} \log \left(\hat U_{A,B} + \hat U_{A^c,B} + \hat U_{A,B^c} + \hat U_{A^c,B^c}\right)\right] \nonumber \\
&=& \frac{\bar W_{A,B}}{\bar W} \left( \hat S_\NS(A,B) + \E\!\left[-\hat U_{A,B} \log \left(1 + \frac{\hat U_{A^c,B}}{\hat U_{A,B}} + \frac{U_{A,B^c}}{\hat U_{A,B}} + \frac{\hat U_{A^c,B^c}}{\hat U_{A,B}} \right)\right] \right), \qquad \qquad
\end{eqnarray}
where we decompose $U = \hat U_{A,B} + \hat U_{A^c,B} + \hat U_{A,B^c} + \hat U_{A^c,B^c}$ using \eqref{eqn_relativefitnessdecomp}. 

The first term is non-positive by \eqref{ineq_tildeUinequality}. The second term is an integral of the form $\E[-\hat U_{A,B} \log(1+Y)]$, where $Y := \frac{\hat U_{A^c,B}}{\hat U_{A,B}} + \frac{U_{A,B^c}}{\hat U_{A,B}} + \frac{\hat U_{A^c,B^c}}{\hat U_{A,B}} \ge 0$. Since $-\log(1+y) \le 0$ for $y \ge 0$, the integral is bounded above by $0$. This proves \eqref{ineq_localNSbound}. 

The first term vanishes when $\hat U_{A,B} = 1$ almost surely, and the second term vanishes when $\hat U_{A^c,B} = \hat U_{A,B^c} = \hat U_{A^c,B^c} = 0$ almost surely. % \todo{What does this condition mean?}

Saturation occurs when both $\hat S_\NS(A,B) = 0$ and $\hat U_{A^c,B} = \hat U_{A,B^c} = \hat U_{A^c,B^c} = 0$. The former means that $\hat U_{A,B} = 1$ a.s., and the latter means that $\hat U_{A,B} = U$ a.s. (by Lemma \ref{lem_vanishinglocalselectiveentropy}). This is equivalent to $U = 1$ almost surely on $A \cap w^{-1} B$, i.e., $w$ is purely environmental from $A$ to $B$. 

% For the strong version, define $p_*(A,B) := \E[1_{U_{A,B}>0}]$ and $\E_{*,A,B}[X] := \frac{1}{p_*(A,B)} \E_{*,A,B}[1_{U_{A,B}>0} X]$. Note $\E_{*,A,B}[\hat U_{A,B}] = \frac{1}{p_*(A,B)}$. 
% Thus: $\hat S_\NS(A,B) = p_*(A,B) \E_{*,A,B}[-\hat U_{A,B} \log \hat U_{A,B}] \le -p_*(A,B) \E_{*,A,B}[\hat U_{A,B}] \log \E_{*,A,B}[\hat U_{A,B}] = -\E[\hat U_{A,B}] \log \frac{\E[\hat U_{A,B}]}{p_*(A,B)} = \log p_*(A,B) \le 0$.
% And: S_\NS(A,B) \le 
% Can bound \E[-\hat U_{A,B} \log (1+Y)]  = \E[\hat U_{A,B} \log \frac{1}{1+Y}] \le \log \E[\frac{\hat U_{A,B}}{1+Y}] = \log \E[\frac{\hat U_{A,B}^2}{U}]
\end{proof}

\section{Proof of Generalized Sinai's Theorem for Dispersion and Mixing Entropies (Theorem \ref{thm_sinaithm_dispersionmixing})} \label{app_sinaithm_dispersionmixing}

\begin{proof}
Suppose that $(\A_*,\B_*)$ is a generating joint partition for environmental entropy: $S_\EC = S_\EC(\A_*, \B_*)$. We first show that $(\A_*,\B_*)$ is generating for dispersion entropy. 

%We demonstrate that the environmental generating partition is also a generating partition for environmental entropy and mixing entropy:
%    \begin{equation}
%        S_\EC = S_\EC(\A_*, \B_*) \qquad \mathrm{and} \qquad S_\SX = S_\SX(\A_*, \B_*).
%    \end{equation}
    
Recall the conditional environmental entropy from \eqref{eqn_defconditionalenvironmentalentropy}: 
\begin{equation} S_\EC(\A',\B'|\A,\B) := - \sum_{A,B} \sum_{A',B'} \E[U_{A',B'}] \log \frac{\E[U_{A',B'}]}{\E[U_{A,B}]} \ge 0.
\end{equation}

Define the conditional dispersion entropy:
\begin{equation}
    S_\dis(\A',\B'|\A,\B) := \sum_{A\in\A,B\in\B} \sum_{A'\in\A',B'\in\B'} \E\!\left[-U_{A',B'} \log \frac{U_{A',B'}}{U_{A,B}} \right] \ge 0.
\end{equation}

By Jensen's inequality, we have 
\begin{equation}
    S_\dis(\A',\B'|\A,\B) = \sum_{A,B} \sum_{A',B'} \E\!\left[U_{A,B}] \frac{1}{\E[U_{A,B}]} \E[U_{A,B} \frac{U_{A',B'}}{U_{A,B}} \log \frac{U_{A',B'}}{U_{A,B}} \right] \le S_\EC(\A',\B'|\A,\B).
\end{equation}

Recall the environmental entropy metric $d_\EC$ from \eqref{eqn_defenvironmentalentropymetric}: $d_\EC = d_\EC((\A,\B), (\A',\B')) := \max\{ S_\EC(\A',\B' | \A,\B), S_\EC(\A,\B | \A',\B') \}$.

Define the dispersion entropy metric, which refines the environmental metric:

\begin{eqnarray}
    d_\dis((\A,\B), (\A',\B')) &:=& \max\{ S_\dis(\A',\B' | \A,\B), S_\dis(\A,\B | \A',\B') \} \nonumber \\
    &\le& d_\EC((\A,\B), (\A',\B')).
\end{eqnarray}
Consequently, any Cauchy sequence for $d_\EC$ is also a Cauchy sequence for $d_\dis$, with the same limiting joint partition. 

We next analyze the case of mixing entropy. Observe that by the algebraic identity \eqref{eqn_partitionentropyECidentity}, we have:
\begin{equation}
    S_\mix(\A_*,\B_*) = S_\EC(\A_*,\B_*) - S_\dis(\A_*,\B_*) = S_\EC - S_\dis.
\end{equation}

Suppose that $(\A', \B')$ is any joint partition which improves upon the generating partition: $S_\mix(\A_*,\B_*) \ge S_\mix(\A',\B')$. Define the refinements $\A'_* := \A_* \wedge \A'$ and $\B'_* := \B_* \wedge \B'$. Then $(\A'_*,\B'_*)$ is again a generating partition of $S_\EC$ and $S_\dis$. Consequently, $S_\mix(\A'_*,\B'_*) = S_\EC - S_\dis = S_\mix(\A_*,\B_*)$. This shows that no partition can strictly improve upon a generating partition, and so $S_\mix = S_\mix(\A_*,\B_*)$. This proves that $(\A_*,\B_*)$ is generating for mixing entropy, which proves the forward direction.

For the reverse direction, suppose that $(\A_*,\B_*)$ is a generating joint partition for both $S_\dis$ and $S_\mix$. Consequently, $S_\EC(\A_*,\B_*) = S_\dis(\A_*,\B_*) + S_\mix(\A_*,\B_*) = S_\dis + S_\mix$. Suppose $(\A',\B')$ any joint partition which improves upon the generating partition: $S_\EC(\A',\B') \ge S_\EC(\A',\B')$. Define the refinements $\A'_* := \A_* \wedge \A'$ and $\B'_* := \B_* \wedge \B'$. Then $(\A'_*,\B'_*)$ is again a generating partition of $S_\dis$ and $S_\mix$. Consequently, $S_\EC(\A'_*,\B'_*) = S_\dis + S_\mix = S_\EC(\A_*,\B_*)$. This shows that no partition can strictly improve upon a generating partition, and so $S_\EC = S_\EC(\A_*,\B_*)$. This proves that $(\A_*,\B_*)$ is generating for environmental entropy.
\end{proof}

\section{Proofs of Efficiency, Consistency, Reversibility, and Irreversibility Theorems (Theorems \ref{thm_efficiency}-\ref{thm_irreversibility})} \label{app_ECRIproofs}

\subsection{The Child-Set Mapping}

%We now build toward the proofs of the Efficiency, Consistency, and Reversibility Theorems. 
Our principle technique is to show there exists a formal inverse $\chi : \I \to \I'$ to the parent-set mapping $w^{-1} : \I' \to \I$. We use this to define the inverse processes. 

Let $A \in \I$ and $B \in \I'$. Define the restricted fitness functions 
\begin{equation}
    W_{A,B}(i) : = 1_A(i) w_i(B) \qquad \mathrm{and} \qquad W_A(i) := W_{A,I'}(i) = 1_A W(i),
\end{equation}
and the averaged process $w_A(B) := N \E[W_{A,B}] = \int_{A} w_i(B) \mu(\d i)$.
% $W_{A,B}(i) : = 1_A(i) w_i(B)$ and  $W_A(i) := W_{A,I'}(i) = 1_A(i) W(i)$. 

\begin{defn}[Child-Set Mappings] \label{def_childset}
Let $w$ be an evolutionary process. We define a \emph{child-set mapping} of $w$ to be a set function $\chi : \I \to \I'$ satisfying the following two properties:
\begin{enumerate}
    \item (Nullity) $\chi(\varnothing) = \varnothing$.
    \item (Local Covering) For each $A \in \I$, $W_{A,\chi(A)}(i) = W_{A,I'}(i)$ and $W_{A, I'-\chi(A)}(i) = 0$ for $\mu$-almost every $i$. 
    % The mapping covers all child sets. For all $A' \in \I_{A,B}$, $w_{A'}(\chi(A')) = W(A')$ and $w_{A'}(B - \chi(A')) = 0$.
%\end{enumerate}
%We say $\chi$ is a \emph{strong child-set mapping} of $w$ if the following property is also true:
%\begin{enumerate}
%    \item[(3)] 
    \item (Local Minimality) For each measurable $A \subseteq I$ and $B \subseteq \chi(A)$ with $w_{A}(B) > 0$, $W_{A,\chi(A) - B}(i) < W_A(i)$ and $W_{A,I' - (\chi(A) - B)}(i) > 0$ for $\mu$-almost every $i$.\footnote{If $w_{A}(I') = 0$, then $W_{A,\chi(A)} = 0$ and $W_{A,I'-\chi(A)}(i) = 0$ for $\mu$-almost every $i$.}
    %\item[(3)] (Local Minimality) For each $A' \in \I_A$ and $B' \in \I'_{\chi(A') \cap B}$ with $w_{A'}(B') > 0$, $W_{A',\chi(A') - B'}(i) < W_{A', B}(i)$ and $W_{A',B - (\chi(A') - B')}(i) > 0$ for $\mu$-almost every $i$.\footnote{If $w_{A'}(B) = 0$, then $W_{A',\chi(A')} = 0$ and $W_{A',B-\chi(A')}(i) = 0$ for $\mu$-almost every $i$.}
\end{enumerate}
\end{defn}

The next theorem demonstrates that a child-set mapping always exists, and is in fact essentially unique up to $\mu'$-measure zero. Thus we refer to ``the'' child-set mapping. %Thus we refer to a strong child-set mapping as ``the child-set mapping''. 
%The next results show that the weak child-set mappings form a mathematical lattice, and that strong child-set mappings (the minimal elements) always exist. 

\begin{thm}[Existence and Essential Uniqueness of The Child-Set Mapping] \label{thm_childsetexistence}
Let $w$ be a non-trivial evolutionary process. The child-set mapping $\chi$ exists and is essentially unique. That is, any two strong child-set mappings $\chi, \chi'$ agree up to sets of $\mu'$ measure zero, with $\chi(A) \approx \chi'(A)$ for all $A \in \I$. 

\end{thm}

We say that a set-mapping is a ``weak child-set mapping'' if it satisfies the Nullity and Covering properties of Definition \ref{def_childset}, but not Minimality. To prove the result, we first show that the set of weak child-set mappings forms a mathematical lattice. Child-set mappings are the minimal elements of this lattice. To prove Theorem \ref{thm_childsetexistence} we show these minimal elements exist and are unique up to $\mu'$-measure zero.

%Define the parent-restricted fitness functions $W_A(i) := W_{A,I'}(i) = 1_A(i) W(i)$. 
%Define the restricted measure $\mu'_{A,B}(B') := \int_A w_i(B \cap B') \mu(\d i)$. Define the restricted parent-set mapping $w^{-1}_{A,B}(B') := A \cap w^{-1}(B \cap B')$. 
%Define the restricted $\sigma$-algebras $\I_A := \{ A' \cap A : A' \in \I \} $ and $\I'_B := \{ B' \cap B : B' \in \I' \}$. Define the doubly-restricted $\sigma$-algebra:
%\begin{equation}
%    \I_{A,B} := \I_{w^{-1}_{A,B}(I')} := \{ A_0 \cap w^{-1}(B_0) : A_0 \in \I_A, B_0 \in \I'_B \}.
%\end{equation}
%Define the child-bearing set $I_+ := \{ i : W(i) > 0 \}$, and the child-bearing $\sigma$-algebra $\I_+ := \I_{I_+} := \{ A' : \tilde \mu(A') > 0 \} \}$.

\begin{lem}[Lattice of Weak Child-Set Mappings] \label{lem_latticechildset}
The set of weak child-set mappings is a non-empty lattice, i.e., closed under set-wise intersections and unions, and compatible with the partial order of set inclusion.\footnote{i.e., for weak child-set mappings $\chi$ and $\chi'$, the intersection and union mappings defined by $(\chi \wedge \chi')(A) := \chi(A) \cap \chi'(A)$ and $(\chi \cup \chi')(A) := \chi(A) \vee \chi'(A)$ are weak child-set mappings, and they are lattice compatible with the partial order defined by $\chi' \prec \chi$ if $\chi'(A) \subseteq \chi(A)$ for all $A$.} A child-set mapping is a minimal element of the lattice. 
\end{lem}

\begin{proof}[Proof of Lemma \ref{lem_latticechildset}]
Let $\Lambda$ denote the set of weak child-set mappings. Consider two child-set mappings $\chi, \chi' \in \Lambda$. 
Define the meet $(\chi \wedge \chi')(A) := \chi(A) \cap \chi'(A)$ and join $(\chi \vee \chi')(A) := \chi(A) \cup \chi'(A)$, and define the partial ordering $\chi' \prec \chi$ if $\chi'(A) \subseteq \chi(A)$ for all $A \in \I$. Clearly, the intersection and union satisfy the nullity property of Definition \ref{def_childset}.

By the Local Covering property, note that for each $i \in A$, $\chi(A)$ and $\chi'(A)$ each have full $w_i$ measure. The intersection and union of full measure sets is again full measure, proving the Local Covering property for the meet and join, hence they are child-set mappings.

The partial ordering is lattice-compatible with the meet and join owing to the lattice compatibility of set-wise intersections and unions relative to set inclusion. Thus $\Lambda$ is a lattice.

Weak child-set mappings always exist, e.g., the maximal covering $\chi(A) := I'$ for $A$, which is a weak child-set mapping even for trivial processes. 
\end{proof}

\begin{proof}[Proof of Theorem \ref{thm_childsetexistence}]
We first show that $\Lambda$ has a minimal element using a Zorn's lemma argument. Such a minimal element is a child-set mapping, and we show that any two such mappings are equal up to sets of measure zero. 

Consider a decreasing chain of child-set mappings $C = (\chi^t)$, for a totally ordered index set $T$. We show that $C$ has a minimal element. Let $(t_n)$ be a countable set of index elements such that $t_n \uparrow \oo$.\footnote{i.e., for any $t_* \in T$, there exists $n_*$ such that $t_n \ge t_*$ for $n \ge n_*$.} Define the child-set mapping $\chi^\oo(A) := \left( \bigwedge_n \chi^{t_n} \right)(A) = \bigcap_n \chi^{t_n}(A)$. Since the $\sigma$-algebra $\I$ is closed under countable intersections, $\chi^\oo$ is well-defined and a minimal element of $C$. By Zorn's lemma, $\Lambda$ has a global minimizer, hence child-set mappings exist. 

To see essential uniqueness, suppose $\chi$ and $\chi'$ are child-set mappings. Then $\mu'(\chi(A) \Delta \chi'(A)) = \int_A w_i(\chi(A) \Delta \chi'(A)) \mu(\d i) = 0$, proving the result. 
\end{proof}

\subsection{Proof of Efficiency, Consistency, and Reversibility Theorems} \label{sect_ECRItheorems}

We show that a purely mixing environmental process can always be inverted on the right, before the process has executed (i.e., there exists $w'$ such that $w_\EC \circ w' = 1_{\mu'}$). Essentially, we define the process by taking each child's unit of population, mapping it back arbitrarily to the parents, then mapping forward through the environmental mapping. Such a right-inverse process is not unique owing to the arbitrary selection of the parent mapping, which is then canceled out in the mixing from $w_\EC$. 

We state a simple local-to-global principle for purely mixing processes. Then we show that the child-set mapping $\chi$ for a purely mixing process is always left-invertible. This allows us to build a right-inverse process $w'$ for which $(w')^{-1} = \chi$. 

% Our key result is to show that a process is purely mixing if and only if the child-set mapping is left-invertible. In that case, we can construct a right-inverse process $w_\EC^\opright$, whose pullback set mapping is the child-set mapping (i.e., $(w_\EC^\opright)^{-1})^{-1} = \chi$). Right-invertibility of the process and left-invertibility of the child-set mapping ensures the relation $(w_\EC \circ w_\EC^\opright)^{-1} = 1_{\I'} = \chi \circ w_\EC^{-1}$. 

%\todo{do we need this?}

Note that for any expectation $\E$ and any non-negative random variable $X \ge 0$, we have
\begin{equation} \label{eqn_vanishingentropyfunctional}
    \mbox{$\E[-X \log X] = 0$ if and only if $X(i) \in \{0,1\}$ for $\mu$-almost every $i$,}
\end{equation}
because the real-valued function $x \mapsto -x \log x$ vanishes if and only if $x = 0,1$. 

\begin{lem}[Purely Mixing Local-to-Global Principle] \label{lem_localtoglobal_mix}
Let $w$ be a finite-entropy process. The following are equivalent:
%Let $A \in \I$ and $B \in \I'$. 
\begin{enumerate}
    \item \label{lem_localtoglobal_mix_global} $w$ is purely mixing and $S_\dis = 0$ (i.e., for all $A \in \I$ and $B \in \I'$, $D_{A,B} \in \{0,1\}$ $\tilde \mu$-a.s. and $S_\dis(A,B) = 0$). 
    \item \label{lem_localtoglobal_mix_local} For each $A \in \I$ and $B \in \I'$, $w$ is locally purely mixing from $A$ to $B$ and $S_\dis(A,B) = 0$ (i.e., $D_{A,B} \in \{0,1\}$ $\tilde \mu$-a.s.). 
%    \item STOP HERE
%    \item \label{lem_localtoglobal_mix_1} $w$ is purely mixing from $A$ to $B$ (i.e., $D_{A,B} \in \{0,1\}$ $\tilde \mu$-a.s. on $A \cap w^{-1} B$). 
%    \item \label{lem_localtoglobal_mix_2} $S_\dis(A,B) = 0$.
%    \item \label{lem_localtoglobal_mix_3} For every $A' \in \I_A$ and $B' \in \I'_B$, $W_{A',B'}(i) \in \{0, W(i)\}$ for $\mu$-almost every $i$,
%    \item \label{lem_localtoglobal_mix_4} For every $A' \in \I_A$ and $B' \in \I'_B$, $S_\dis(A',B') = 0$,
\end{enumerate}
\end{lem}
\begin{proof}
The proof is trivial since $S_\dis = \sup \sum S_\dis(A,B) = \sup \sum \tilde \E[-D_{A,B} \log D_{A,B}]$. If the global entropy vanishes, then all local entropies vanish, and so $D_{A,B} = 0$ or $1$ by \eqref{eqn_vanishingentropyfunctional}. If all local entropies vanish, then their sum and hence the supremum vanish. 
\end{proof}

\begin{pro} \label{pro_puremixing_chileftinverse}
Let $w$ be a finite-entropy process. Then $w$ is purely mixing if and only if the child-set mapping $\chi$ is essentially a left-inverse of $w^{-1}$ (i.e., $(\chi \circ w^{-1})(B) \approx B$ up to $\mu'$-measure zero for any $B \in \I'$). 
\end{pro}
\begin{proof}
%Let $A \in \I$ and $B \in \I'$, and 
Let $\chi$ denote the child-set mapping.\footnote{We do not need the minimality property of child-set mappings for this proposition, only the covering property.} %By the covering property of $\chi$, we have $w_{A'}(\chi(A')) = W(A')$ for every $A' \in \I_A$.
Consider measurable $A \in \I$ and $B \in \I'$, and suppose that $S_\dis(A,B) = 0$. Consider % $A' \in \I_A$ and 
$B' \in \I'_B$. 
% , so that by Lemma \ref{lem_localtoglobal_EC}($1 \implies 3$), we have $S_\dis(A',B') = 0$ and $W_{A',B'}(i) \in \{0, W(i) \}$ for $\mu_*$-almost every $i \in A'$. 
To prove left-invertibility from $A$ to $B$, we must show that $B'' := \chi(w^{-1}_{A,B} B') \approx B'$. By the covering property of the child-set mapping $\chi$, we have 
\begin{equation}
    \mbox{$W_{w^{-1} B', B''} = W_{w^{-1} B', B}$ 
    %and $W_{w^{-1} B', B-B''} = 0$ 
    a.s.}
\end{equation}
Since both $W_{w^{-1} B',B'}$ and $W_{w^{-1} B', B}$ are positive, by Lemma \ref{lem_localtoglobal_mix}, we have
\begin{equation}
    \mbox{$W_{w^{-1} B', B'} = W = W_{w^{-1} B', B}$ 
    % and $W_{w^{-1} B', B-B'} = 0$ 
    a.s.}
\end{equation} 
Consequently, $W_{A-w^{-1}B', B} = 0$. It follows that
\begin{equation}
    \mbox{$W_{A,B' \symdiff B''} = 
    W_{w^{-1} B', B' \symdiff B''} = 0$ a.s.}
\end{equation}
Thus $B'' \approx B'$, proving that $\chi$ is a left inverse of $w_\EC$ from $A$ to $B$. By combining over arbitrary partitions, we have that $\chi$ is a general left inverse. 

Conversely, suppose that $S_\dis(A,B) > 0$, so there exist $A'$ and $B'$ such that $0 < W_{A',B'} < W_{A',I'}$ and $0 < W_{A',B-B'} < W_{A',I'}$ on a set of positive measure. Thus $\chi(w^{-1}_{A,B} B') \cap (B-B') \ne \varnothing$, proving that $\chi$ is not a left inverse.
\end{proof}

% Lemma: purely mixing iff left-inverse child-set process
% Proof: 

\begin{proof}[Proof of Efficiency Theorem (Theorem \ref{thm_efficiency})]
Suppose that $w$ is purely mixing. We show that there exists a right inverse, i.e., a process $w' : \mu' \mapsto \tilde \mu$ such that for all $B$,
\begin{equation} \label{eqn_efficiency_proof1}
    \mbox{$(w_\EC \circ w')_{i'}(B) = 1$ for $i' \in B$.}
\end{equation}
By the purely mixing hypothesis ($W_{A,B} = W$ a.s.), \eqref{eqn_efficiency_proof1} is equivalent to:
\begin{eqnarray}
    1 &=& (w_\EC \circ w')_{i'}(B) = \int_I w_{\EC,\tilde i}(B) \, w'_{i'}(\d \tilde i) = \int_{w^{-1} B} \frac{W_{w^{-1} B, B}(\tilde i)}{W(\tilde i)} w'_{i'}(\d \tilde i) \nonumber \\
    &=&  w'_{i'}(w^{-1} B) \mathrm{~for~} i' \in B. \label{eqn_efficiency_proof2}
\end{eqnarray}

We construct a family of such processes by taking partition refinements. For each partition $\B$, consider the topological space $\W_\B$ of processes satisfying the condition \eqref{eqn_efficiency_proof2} for all $B \in \B$, equipped with the topology of almost-sure weak convergence of measures.\footnote{That is, a net of processes $(w'_t)$ converges to $w'$ if on a set $B$ of full $\mu'$ measure, we have $w'_{t,i'} -> w'_{\EC,i'}$ for each $i' \in B$.} The space $\W'_\B$ is closed, convex, and complete. Each process $w' \in \W'_\B$ corresponds to a measure $w'_B$ on each parent set $w^{-1} B$, with the process satisfying $w'_{i'} \equiv w'_B$ for all $i' \in B$.

Let $\B'$ be a refinement of $\B$, i.e., each set $B \in \B$ is a disjoint union of sets in $\B'$. We show $\W'_{\B'} \subseteq \W'_\B$. Consider a refined process $w' \in \W'_{\B'}$. For each $i'$, let $B_{i'}$ (resp. $B'_{i'}$) be the partition set of $\B$ (resp. $\B'$) containing $i'$. We define a coarse version $\hat w' \in \W'_\B$ by setting $\hat w'_{i'}(A) := w'_{i'}(A \cap w^{-1} B_{i'})$ for all $A$. %Thus this process is guaranteed to operate within each parent-set of the refined partition. 

Finally, let $\B^t$ be a net of partition refinements, and consider the infinitary intersection $\W_\oo := \bigcap_t \W_{\B^t}$. As the intersection of non-empty, closed, convex, complete spaces, $\W_\oo$ is non-empty. 

To see that each $w' \in \W'_\oo$ satisfies condition \eqref{eqn_efficiency_proof2} for arbitrary $B$, let $\epsilon > 0$ and consider a refinement $\B^t$ such that we can approximate $B$ by sets $\{B'\} \subseteq \B^t$ satisfying $|w'_{i'}(w^{-1} B) - w'_{i'}(\bigcup w^{-1} B')| < \epsilon$ for $\mu'$-almost all $i \in B$. Then $|w'_{i'}(w^{-1} B) - 1| \le \epsilon + |w'_{i'}(\bigcup w^{-1} B') - 1| = \epsilon$ since $w'$ must give full measure at $i'$ to the parent set containing $i'$. Since $\epsilon$ is arbitrary, \eqref{eqn_efficiency_proof2} must hold for $B$, proving that there exists a retraction.

If $S_\dis = 0$ and $S_\mix > 0$, then there exist distinct $A, A'$ such that $\chi(A) = \chi(A') = B$, so the space $\W'_\oo$ includes at least two retraction, one which sends full measure from $B$ to $A$, and the other to $A'$.

Suppose $S_\dis > 0$, so $w$ is not purely mixing. Then there exist $A \in \I$ and disjoint $B, B' \in \I'$ of positive measure such that $\chi(A) = B \cup B'$ and $w^{-1} B = A = w^{-1} B'$. Then for any environmental process $w'$, 
\begin{equation}
    1 = (w_\EC \circ w')_{i'}(I') \ge (w_\EC \circ w')_{i'}(\chi(A)) = (w_\EC \circ w')_{i'}(w^{-1} B) + (w_\EC \circ w')_{i'}(w^{-1} B'),
\end{equation}
so at least one of the terms on the right side must be $<1$, thus $w'$ is not a retraction.
\end{proof}

%\subsection{Proof of Consistency Theorem (Theorem \ref{thm_consistency})} \label{sect_consistency}

We next show that a purely dispersive environmental process can always be inverted on after the process has executed (i.e., there exists $w'$ such that $w' \circ w_\EC = 1_{\tilde \mu}$). In this case, we define the process by taking each child's unit of population, and map it back to the unique parent from which it came. Such a left-inverse process is unique, since in a purely dispersive process, every parent is unique. 

We state a simple local-to-global principle for purely dispersive processes, and we use this to show that the child-set mapping $\chi$ for a purely dispersive process is always right-invertible. This allows us to build the left-inverse process $w'$.

\begin{lem}[Purely Dispersive Local-to-Global Principle] \label{lem_localtoglobal_dis}
Let $w$ be a finite-entropy process. The following are equivalent:
%Let $A \in \I$ and $B \in \I'$. 
\begin{enumerate}
    \item \label{lem_localtoglobal_disp_global} $w$ is purely dispersive and $S_\mix = 0$ (i.e., for all $A \in \I$ and $B \in \I'$, $M_{A,B} \in \{0,1\}$ $\tilde \mu$-a.s. and $S_\mix(A,B) = 0$). 
    \item \label{lem_localtoglobal_disp_local} For each $A \in \I$ and $B \in \I'$, $w$ is locally purely dispersive from $A$ to $B$ and $S_\mix(A,B) = 0$ (i.e., $M_{A,B} \in \{0,1\}$ $\tilde \mu$-a.s.). 
\end{enumerate}
\end{lem}
\begin{proof}
The proof is trivial since $S_\mix = \sup \sum S_\mix(A,B) = \sup \sum \bar U_{A,B} \tilde \E[M_{A,B} \log M_{A,B}]$. If the global entropy vanishes, then all local entropies vanish, and so $M_{A,B} = 0$ or $1$ by \eqref{eqn_vanishingentropyfunctional}. If all local entropies vanish, then their sum and hence the supremum vanish. 
\end{proof}

\begin{pro} \label{pro_puredispersive_chirightinverse}
Let $w$ be a finite-entropy process. Then $w$ is purely dispersive if and only if the child-set mapping $\chi$ is a right-inverse of $w^{-1}$ (i.e., $(w^{-1} \circ \chi)(A) \approx A$ up to $\tilde \mu$-measure zero for any $A \in \I_+$).
\end{pro}
\begin{proof}
Let $\chi$ denote a child-set covering mapping.\footnote{We do not need the minimality property of child-set mappings for this proof.} For a given $A \in \I$, define its child set $B := \chi(A)$. Define the complement of the parent set $A^c := I-A$ and the complementary child set: $B' := \chi(A^c)$.\footnote{It is possible for the child set and the complementary child set to overlap, i.e., $B \cap B' \ne \varnothing$, namely, for children who have parents in both $A$ and $A^c$.}

Suppose that $S_\mix = 0$. We will show that consistency implies $B \cap B' \approx \varnothing$ up to $\mu$-measure zero, which implies $(w^{-1} \circ \chi)(A) = A$, i.e., that $\chi$ is a right inverse to $w^{-1}$, so $w$ is consistent.

By the covering property of child-set mappings, we have:
\begin{equation}
    \mbox{$W_{A,B} = W_{A,I'}$ and $W_{A^c,B'} = W_{A^c,I'} = W - W_{A,I'}$ a.s.}
\end{equation}
By Lemma \ref{lem_localtoglobal_dis}, we have
\begin{equation}
    \mbox{$W_{A,B} = W_{A,I'} = \bar U_{A,I'} \, W_{A}$ and $W_{A^c,B'} = W_{A^c,I'} = \bar U_{A^c,I'} \, W_{A^c}$ a.s.,}
\end{equation}
and
\begin{equation}
    \mbox{$W_{I,B \cap B'} = \bar U_{I,B \cap B'} \, W$ a.s.}
\end{equation}

Consequently,
\begin{eqnarray}
W_{I,B \cap B'}
&=& \bar U_{I,B \cap B'} W 
= \bar U_{I,B \cap B'} \left( W_{A} + W_{A^c} \right) \nonumber \\
&=& \bar U_{I,B \cap B'} \left( \frac{W_{A,B}}{\bar U_{A,I'}} + \frac{W_{A^c,B'}}{\bar U_{A^c,I'}} \right) \nonumber \\
&=& \frac{\bar U_{I,B \cap B'}}{\bar U_{A,I'}} W_{A,B} + \frac{\bar U_{I,B \cap B'}}{\bar U_{A^c,I'}} W_{A^c,B'}.
\end{eqnarray}

Now, the sets $A$ and $A^c$ are mutually exclusive, so only one term can be positive. If $W_{A,B} > 0$, then $W_{A^c,B'} = 0$, and so $\bar U_{I,B \cap B'} = 0$. Similarly, if $W_{A^c,B'} > 0$, then $\bar U_{I,B \cap B'} = 0$. Thus $W_{I,B \cap B'} = 0$ a.s., and so $B \cap B' \approx \varnothing$. 

It follows that $(w^{-1} \circ \chi)(A) \approx A$ up to measure zero, so $\chi$ is a right inverse to $w^{-1}$, and so $w$ is locally consistent from $A$ to $B$. This proves the forward direction.

Suppose $S_\mix > 0$, so there exist $A$ and $B$ such that $0 < W_{A,B} < \bar U_{A,B} \, W$ and $0 < W_{A,I'-B} < \bar U_{A,I'} \, W_{A,I'}$ on a set of positive measure. Thus $w^{-1}(\chi(B)) \cap (I'-B) \ne \varnothing$, proving that $\chi$ is not efficient. This completes the proof.
\end{proof}

\begin{proof}[Proof of Consistency Theorem (Theorem \ref{thm_consistency})]
Suppose that $w$ is purely dispersive. We show that there exists a left inverse, i.e., a process $w' : \mu' \mapsto \tilde \mu$ such that for all $B$,
\begin{equation} \label{eqn_consistency_proof1}
    \mbox{$(w' \circ w_\EC)_{\tilde i}(A) = 1$ for $\tilde \mu$-a.s. $\tilde A \in A$.}
\end{equation}
We construct the process $w'$ by sending each child to its unique parent:
\begin{equation}
    w'_{i'}(A) := \begin{cases} 1, & i' \in \chi(A), \\ 0 & i' \notin \chi(A). \end{cases}
\end{equation}
This is well-defined by the purely dispersive hypothesis. Indeed, if $A$ and $A'$ are disjoint, then $\chi(A)$ and $\chi'(A)$ are disjoint, since $w^{-1}(\chi(A) \cap \chi(A')) = w^{-1}(\chi(A)) \cap w^{-1}(\chi(A')) = A \cap A' = \varnothing$, where $\chi$ is a right-inverse by pure dispersivity. 

The process $w'$ is a left inverse, since for each $\tilde i \in A$,
\begin{equation}
(w' \circ w_\EC)_{\tilde i}(A)
= \int w'_{i'}(A) w_{\EC,\tilde i}(\d i') 
= w_{\EC,\tilde i}(\chi(A))
= \frac{W_{A, \chi(A)}(\tilde i)}{W(\tilde i)}
= 1,
\end{equation}
since w is purely dispersive and so $W_{A, \chi(A)}(\tilde i) = W(\tilde i)$ almost surely. Thus $w'$ is a left-inverse. 

We show that $w'$ is essentially unique. Suppose $w''$ is another left inverse. Consider arbitrary $A$, and let $B := \{ i' : w''_{i'}(A) - 1 \ne 0 \}$. Let $A' := w^{-1} B$. If $B$ has non-negative measure, then
\begin{equation}
0 = (w' \circ w_\EC)_{\tilde i}(A') - 1 
= \int_B (w'_{i'}(A) - 1) w_{\EC,\tilde i}(\d i') 
\ne 0, 
\end{equation}
a contradiction, so $w'$ is essentially unique.

Now suppose that $w$ is not purely dispersive. Then there exist disjoint $A,A' \in \I$ and $B \in \I'$ of positive measure such that $B \subseteq \chi(A) \cap \chi(A')$ and $w^{-1} B = A \cup A'$. Then for any purely environmental process $w'$, 
\begin{equation}
    1 = (w' \circ w_\EC)_{\tilde i}(I) \ge (w' \circ w_\EC)_{\tilde i}(A \cup A') = (w' \circ w_\EC)_{\tilde i}(A) + (w' \circ w_\EC)_{\tilde i}(A') = 2 (w' \circ w_\EC)_{\tilde i}(w^{-1} B),
    %1 = (w_\EC \circ w')_{i'}(I') \ge (w_\EC \circ w')_{i'}(\chi(A)) = (w_\EC \circ w')_{i'}(w^{-1} B) + (w_\EC \circ w')_{i'}(w^{-1} B'),
\end{equation}
hence $(w' \circ w_\EC)_{\tilde i}(w^{-1} B) < 1$, and so $w'$ is not a section.
\end{proof}

%\subsection{Proof of Reversibility Theorem (Theorem \ref{thm_reversibility})} \label{sect_reversibility}

\begin{proof}[Proof of Reversibility Theorem (Theorem \ref{thm_reversibility})]
We now prove the Reversibility Theorem. Suppose $w$ is purely mixing and purely dispersive. Since $w$ is purely dispersive, there exists a unique retraction (left-inverse process) $w'$ from the Consistency Theorem, defined by $w'_{i'}(A) = 1$ for all $i' \in \chi(A)$. Since $w$ is purely mixing, Proposition \ref{pro_puremixing_chileftinverse} implies that $\chi$ is a left inverse to $w^{-1}$. Consequently, for any $B$, we have $w'_{i'}(w^{-1} B) = 1$ if $i' \in \chi(w^{-1}(B)) \approx B$. Thus by \eqref{eqn_efficiency_proof2}, $w'$ is a section (right-inverse process). This proves that $w'$ is the unique inverse to $w_\EC$.

Now, suppose that $w'$ is an inverse process to $w_\EC$. In particular, $w'$ is the unique retraction for the Consistency Theorem, hence $w$ is purely dispersive, and $w'$ is a section for the Efficiency Theorem, hence $w$ is purely dispersive. This proves the Reversibility Theorem. 

The Irreversibility Theorem follows from the law of the excluded middle, as the contrapositive of the Reversibility Theorem.
\end{proof}

%The Irreversibility Theorem follows from the contrapositive. 

\section{Proof of Strong Third Law of Natural Selection (Theorem \ref{thm_thirdlaw})} \label{app_proofofstrongthirdlaw}

\begin{proof}[Proof of Strong Third Law (Theorem \ref{thm_thirdlaw})]
We first prove local versions of the result: 
\begin{equation} \label{ineq_localthirdlaw_dis}
    \tilde p_{A,B} \lambda_{A,B} \log \tfrac{\lambda_{A,B}}{\gamma_{A,B}} - \bar U_{A,B} \log \tfrac{\tilde p_{A,B}}{\bar U_{A,B}} \le \partial_\NS S_\dis(A,B) \le \tilde p_{A,B} \lambda_{A,B} \log \tfrac{\varphi_{A,B}}{\lambda_{A,B}} - \bar U_{A,B} \log \tfrac{\bar U_{A,B}}{\tilde \E[D_{A,B}^2]},
\end{equation}
% THIRD LAW - MIXING
\begin{equation} \label{ineq_localthirdlaw_mix}
    \hspace{-.35in} \tilde p_{A,B} \lambda_{A,B} \log \tfrac{\lambda_{A,B}}{\varphi_{A,B} \bar U_{A,B}} - \bar U_{A,B} \log \tilde \E[M_{A,B}^2] \le \partial_\NS S_\mix(A,B) \le \tilde p_{A,B} \lambda_{A,B} \log \tfrac{\gamma_{A,B}}{\lambda_{A,B} \bar U_{A,B}} - \bar U_{A,B} \log \tfrac{1}{\tilde p_{A,B}},
\end{equation}
and
% THIRD LAW - ENVIRONMENTAL
\begin{equation} \label{ineq_localthirdlaw_EC}
    \hspace{-.75in} 
    \tilde p_{A,B} \lambda_{A,B} \log \tfrac{\lambda_{A,B}^2}{\gamma_{A,B} \varphi_{A,B} \bar U_{A,B}} - \bar U_{A,B} \log \tfrac{\tilde p_{A,B} \tilde \E[D_{A,B}^2]}{\bar U_{A,B}^3} \le \partial_\NS S_\EC(A,B) \le \tilde p_{A,B} \lambda_{A,B} \log \tfrac{\gamma_{A,B} \varphi_{A,B}}{\lambda_{A,B}^2 \bar U_{A,B}} - \bar U_{A,B} \log \tfrac{\bar U_{A,B}}{\tilde p_{A,B} \E[D_{A,B}^2]},
\end{equation}
with saturation of all inequalities when $w$ is in local environmental equilibrium from $A$ to $B$. The partition versions follow by summing over partition sets, and the general versions follow by evaluating at a generating joint partition. 

Proof of \eqref{ineq_localthirdlaw_dis}) We decompose the dispersion entropy change as $\partial_\NS S_\dis(A,B) = \tilde \E[U(-D_{A,B} \log D_{A,B})] - S_\dis(A,B)$. Flipping the bounds of the dispersion-entropy estimates \eqref{ineq_localbound_dis} yields the following for the second term: 
\begin{equation} \label{ineq_localthirdlaw_dis_proof_negSdis}
    -\bar U_{A,B} \log \frac{\tilde p_{A,B}}{\bar U_{A,B}} \le -S_\dis(A,B) \le -\bar U_{A,B} \log \frac{\bar U_{A,B}}{\tilde \E[D_{A,B}^2]},
\end{equation}
with saturation in environmental equilibrium from $A$ to $B$.

Using Jensen's inequality, we compute the upper bound of the first term:
\begin{eqnarray}
    \tilde \E[U(-D_{A,B} \log D_{A,B})] &=& \tilde p_{A,B} \varphi_{A,B} \frac{1}{\varphi_{A,B}} \tilde \E_{A,B}[U(-D_{A,B} \log D_{A,B})] \nonumber \\
    &\le& -\tilde p_{A,B} \lambda_{A,B} \log \frac{\lambda_{A,B}}{\varphi_{A,B}} = \tilde p_{A,B} \lambda_{A,B} \log \frac{\varphi_{A,B}}{\lambda_{A,B}} %= \Lambda_{A,B}
    \label{ineq_localthirdlaw_dis_proof_upper}
\end{eqnarray}
since $\tilde \E_{A,B}[U] = \varphi_{A,B}$ and $\tilde \E_{A,B}[U D_{A,B}] = \lambda_{A,B}$. Saturation occurs when $D_{A,B}$ is constant $U \tilde \mu$-almost surely, which is equivalent to being constant $\tilde \mu$-almost surely, i.e., the environmental-equilibrium case. Combining \eqref{ineq_localthirdlaw_dis_proof_upper} and \eqref{ineq_localthirdlaw_dis_proof_negSdis} yields the upper bound of \eqref{ineq_localthirdlaw_dis}. 

We compute the lower bound of the first term:
\begin{eqnarray}
    \tilde \E[U(-D_{A,B} \log D_{A,B})] &=& \tilde p_{A,B} \lambda_{A,B} \frac{1}{\lambda_{A,B}} \tilde \E_{A,B}[U D_{A,B} (-\log D_{A,B})] \nonumber \\
    &\ge& - \tilde p_{A,B} \lambda_{A,B} \log \frac{\gamma_{A,B}}{\lambda_{A,B}},
    %= \tilde p_{A,B} \lambda_{A,B} \log \frac{\lambda_{A,B}}{\gamma_{A,B}},
    %= \Gamma_{A,B}
    \label{ineq_localthirdlaw_dis_proof_lower}
\end{eqnarray}
since $\tilde \E_{A,B}[U D_{A,B}] = \lambda_{A,B}$ and $\tilde \E_{A,B}[U D_{A,B}^2] = \gamma_{A,B}$. Saturation occurs when $D_{A,B}$ is constant $U D_{A,B} \tilde \mu$-almost surely, which is equivalent to being constant $\tilde \mu$-almost surely, i.e., the environmental-equilibrium case. Combining \eqref{ineq_localthirdlaw_dis_proof_lower} and \eqref{ineq_localthirdlaw_dis_proof_negSdis} yields the lower bound of \eqref{ineq_localthirdlaw_dis}. 

Proof of \eqref{ineq_localthirdlaw_mix}) We decompose the mixing entropy change as $\partial_\NS S_\mix(A,B) = \bar U_{A,B} \tilde \E[U(M_{A,B} \log M_{A,B})] - S_\mix$. Flipping the bounds of the mixing-entropy estimates \eqref{ineq_localbound_mix} yields the following for the second term: 
\begin{equation} \label{ineq_localthirdlaw_mix_proof_negSmix}
    -\bar U_{A,B} \log \tilde \E[M_{A,B}^2] \le -S_\mix(A,B) \le -\bar U_{A,B} \log \frac{1}{\tilde p_{A,B}},
\end{equation}
with saturation in environmental equilibrium from $A$ to $B$.

Using Jensen's inequality, we compute the upper bound of the first mixing term:
\begin{eqnarray}
    \bar U_{A,B} \tilde \E[U M_{A,B} (\log M_{A,B})] &=& \bar U_{A,B} \tilde p_{A,B} \frac{\lambda_{A,B}}{\bar U_{A,B}} \frac{\bar U_{A,B}}{\lambda_{A,B}} \tilde \E_{A,B}[U M_{A,B} (\log M_{A,B})] \nonumber \\
    &\le& %\bar U_{A,B} \tilde p_{A,B} \frac{\lambda_{A,B}}{\bar U_{A,B}} 
    %&=& 
    \tilde p_{A,B} \lambda_{A,B} \log \frac{\bar U_{A,B} \gamma_{A,B}}{\lambda_{A,B} \bar U_{A,B}^2} 
    = \tilde p_{A,B} \lambda_{A,B} \log \frac{\gamma_{A,B}}{\lambda_{A,B} \bar U_{A,B}},
    \label{ineq_localthirdlaw_mix_proof_upper}
\end{eqnarray}
since $\tilde \E_{A,B}[U M_{A,B}] = \frac{\lambda_{A,B}}{\bar U_{A,B}}$ and $\E_{A,B}[U M_{A,B}^2] = \frac{\gamma_{A,B}}{\bar U_{A,B}^2}$. Saturation occurs when $D_{A,B}$ is constant $U \tilde \mu$-almost surely, which is equivalent to being constant $\tilde \mu$-almost surely, i.e., the environmental-equilibrium case. Combining \eqref{ineq_localthirdlaw_mix_proof_upper} and \eqref{ineq_localthirdlaw_mix_proof_negSmix} yields the upper bound of \eqref{ineq_localthirdlaw_dis}. We compute the lower bound of the first mixing term:
\begin{eqnarray}
    \bar U_{A,B} \tilde \E[U(M_{A,B} \log M_{A,B})] &=& \bar U_{A,B} \tilde p_{A,B} \varphi_{A,B} \frac{1}{\varphi_{A,B}} \tilde \E_{A,B}[U(M_{A,B} \log M_{A,B})] \nonumber \\
    &\ge& \bar U_{A,B} \tilde p_{A,B} \varphi_{A,B}  \frac{\lambda_{A,B}}{\varphi_{A,B} \bar U_{A,B}} \log \frac{\lambda_{A,B}}{\varphi_{A,B} \bar U_{A,B}} \nonumber \\
    &=& \tilde p_{A,B} \lambda_{A,B} \log \frac{\lambda_{A,B}}{\varphi_{A,B} \bar U_{A,B}},
    %= \tilde p_{A,B} \lambda_{A,B} \log \frac{\lambda_{A,B}}{\gamma_{A,B}},
    %= \Gamma_{A,B}
    \label{ineq_localthirdlaw_mix_proof_lower}
\end{eqnarray}
since $\tilde \E_{A,B}[U] = \varphi_{A,B}$ and $\tilde \E_{A,B}[U M_{A,B}] = \frac{\lambda_{A,B}}{\bar U_{A,B}}$. Saturation occurs when $D_{A,B}$ is constant $U D_{A,B} \tilde \mu$-almost surely, which is equivalent to being constant $\tilde \mu$-almost surely, i.e., the environmental-equilibrium case. Combining \eqref{ineq_localthirdlaw_mix_proof_lower} and \eqref{ineq_localthirdlaw_mix_proof_negSmix} yields the lower bound of \eqref{ineq_localthirdlaw_dis}. 

Proof of \eqref{ineq_localthirdlaw_EC}) The environmental inequality \eqref{ineq_localthirdlaw_EC} follows by summing inequalities \eqref{ineq_localthirdlaw_dis} and \eqref{ineq_localthirdlaw_mix}, and using the identity $\tilde \E[M_{A,B}^2] = \frac{\tilde \E[D_{A,B}^2]}{\bar U_{A,B}^2}$.
\end{proof}

\end{document}